\documentclass[11pt]{amsart}
\usepackage{amssymb, amsmath,latexsym,amsfonts,amsbsy, amsthm,mathtools,graphicx,color}
\usepackage{float}
\usepackage{hyperref}

\numberwithin{equation}{section}
\allowdisplaybreaks

\usepackage{mathrsfs,enumerate,extarrows,bm,tikz}
\usetikzlibrary{decorations.markings,arrows.meta}

\let\al=\alpha

\let\g=\gamma

\let\la=\lambda

\let\f=\frac

\let\om=\omega

\let\Om=\Omega
\let\wt=\widetilde
\let\wh=\widehat

\let\pa=\partial

\newcommand{\andf}{\quad\hbox{and}\quad}
\newcommand{\with}{\quad\hbox{with}\quad}

\def\cA{{\mathcal A}}
\def\cB{{\mathcal B}}
\def\cC{{\mathcal C}}

\def\cF{{\mathcal F}}
\def\cG{{\mathcal G}}

\def\cM{{\mathcal M}}

\def\cT{{\mathcal T}}

\def\cX{{\mathcal X}}
\def\cY{{\mathcal Y}}

\def\C{\mathbb C}
\def\R{\mathbb R}

\def\Z{\mathbb Z}
\def\N{\mathbb N}

\def\u{u}

\def\no{\noindent}

\def\dive{\mathop{\rm div}\nolimits}

\newcommand{\beq}{\begin{equation}}
	\newcommand{\eeq}{\end{equation}}
\newcommand{\ben}{\begin{eqnarray}}
	\newcommand{\een}{\end{eqnarray}}
\newcommand{\beno}{\begin{eqnarray*}}
	\newcommand{\eeno}{\end{eqnarray*}}

\newtheorem{theorem}{Theorem}[section]
\newtheorem{definition}[theorem]{Definition}
\newtheorem{lemma}[theorem]{Lemma}
\newtheorem{proposition}[theorem]{Proposition}
\newtheorem{corollary}[theorem]{Corollary}
\theoremstyle{remark}
\newtheorem{remark}[theorem]{Remark}

\newtheorem*{theorem*}{\bf Theorem}

\begin{document}

\title[Self-similar blow-up solutions of  incompressible Euler]
{Self-similar blow-up solutions of $d$-dimensional incompressible Euler equations 
with $C^{1,\left(1-2/d\right)-}$ velocity}

\author[F. Shao]{Feng Shao}
\address[F. Shao]{Academy of Mathematics \& Systems Science and Hua Loo-Keng Center for Mathematical Sciences, Chinese Academy of Sciences, Beijing 100190, China}
\email{fshao@amss.ac.cn}

\author[D. Wei]{Dongyi Wei}
\address[D. Wei]{School of Mathematical Sciences, Peking University, Beijing 100871,  China}
\email{jnwdyi@pku.edu.cn}

\author[P. Zhang]{Ping Zhang}
\address[P. Zhang]{State Key Laboratory of Mathematical Sciences, Academy of Mathematics $\&$ Systems Science, The Chinese Academy of
	Sciences, Beijing 100190, China, and School of Mathematical Sciences,
	University of Chinese Academy of Sciences, Beijing 100049, China}
\email{zp@amss.ac.cn}

\author[Z. Zhang]{Zhifei Zhang}
\address[Z. Zhang]{School of Mathematical Sciences, Peking University, Beijing 100871, China}
\email{zfzhang@math.pku.edu.cn}

\date{\today}

\begin{abstract}
We investigate self-similar blow-up solutions to the $d$-dimensional axisymmetric incompressible Euler equations without swirl for $d\ge 3$. For any $\alpha\in(0, \alpha_d)$ with $\alpha_d=1-2/d$, we construct a self-similar blow-up solution whose initial velocity field satisfies $u_0\in C^{1,\alpha}_{\rm loc}(\mathbb R^d)\cap C^\infty(\mathbb R^d\setminus\{0\})$. Our construction relies on a fixed-point argument formulated for the self-similar profile equations, which form a coupled elliptic-transport system. Specifically, the transport equation recovers the vorticity profile from given data along characteristic curves, while the elliptic equation reconstructs the velocity field via Newtonian potentials defined in an auxiliary $(d+4)$-dimensional space. The main challenge consists in choosing appropriate function spaces that remain invariant under such nonlinear compositions and that simultaneously capture the exact singular behavior near the origin and the symmetry axis.

Furthermore, we establish a finite-codimensional stability result for the self-similar profiles obtained above. As a consequence, after suitable truncation and correction of finitely many unstable modes, we obtain finite-energy blow-up solutions with initial velocity in $C^{1,\alpha}(\mathbb R^d)\cap C^\infty(\mathbb R^d\setminus\{0\})\cap L^2(\mathbb R^d)$ and compactly supported initial vorticity. These solutions are asymptotically self-similar near the blow-up time.
\end{abstract}

\maketitle

\tableofcontents

\section{Introduction}
In this paper, we consider the following $d$-dimensional incompressible Euler equations:
\begin{equation}\label{Eq.Euler}
\begin{cases}
	\u_t+\u\cdot\nabla\u+\nabla P=0,\quad (t,x)\in [0,T)\times\R^d,\\
	\dive \u=0,
	\quad\\
	\u|_{t=0}=\u_0, \end{cases}
\end{equation}
where $\u=(u_1,u_2,\cdots,u_d):[0,T)\times\R^d\to \R^d$ denotes the velocity field of a perfect incompressible fluid, and $P=P(t,x):[0, T)\times\R^d\to \R$ designates the scalar pressure function (see \cite{MB2002} for instance). 

In pioneering works, Lichtenstein \cite{Lichtenstein1925} and Günther \cite{Gunther1927} proved that system \eqref{Eq.Euler} admits a unique local solution within the class of velocity fields in $C^{k,\al}\cap L^2$ for $k\geq 1$ and $\al\in(0,1)$. For more classical well-posedness results established in various function spaces, we refer the reader to \cite{Chae,KP,Pak,Vishik} and the references therein. We also highlight the recent breakthroughs in \cite{BL-gafa,BL-inve} concerning ill-posedness phenomena for \eqref{Eq.Euler} in borderline function spaces.

The standard energy method shows that a classical solution $u=u(t,x)$ to \eqref{Eq.Euler} develops a finite-time singularity at time $T_*$ if and only if 
\begin{equation}\label{Eq.basic-blowup}
	\lim_{t\uparrow T_*}\int_0^t\|\nabla u(s)\|_{L^\infty}\,\mathrm ds=+\infty.
\end{equation}
An important refinement of this blow-up criterion \eqref{Eq.basic-blowup} was obtained by Beale, Kato and Majda \cite{BKM1984}, who showed that a classical solution to the Euler system  \eqref{Eq.Euler}  blows up at time $T_*$ if and only if
\begin{equation*}
	\lim_{t\uparrow T_*}\int_0^t\|\bm\omega(s)\|_{L^\infty}\,\mathrm ds=+\infty,
\end{equation*}
where $\bm\omega:[0,T_*)\times\R^d\to\R^{d\times d}$ denotes the vorticity matrix defined by 
\begin{equation}\label{S1eq1}
\bm\omega=(\omega_{ij})_{1\leq i,j\leq d} \with \omega_{ij}:=\pa_iu_j-\pa_ju_i.
\end{equation}
In particular, the BKM blow-up criterion implies the global well-posedness of the two-dimensional Euler equations due to the transport structure of the scalar vorticity $\om=\pa_2u_1-\pa_1u_2$, which satisfies the transport equation $\partial_t\omega+u\cdot\nabla\omega=0$.

Nevertheless, it remains a major open challenge in mathematical fluid mechanics to determine whether the three-dimensional Euler equations are globally regular or develop finite-time singularities.
This problem has also motivated extensive numerical investigations \cite{HH2022,HouLi2006,Kerr1993,WBMRBMNC-etal2025,WLGB2023}, as well as analytical studies of simplified one-dimensional nonlocal models that capture the vortex-stretching mechanism; see \cite{CHH2021,ChoiHouKiselevLuoSverakYao2017,ConstantinLaxMajda1985,DeGregorio1990,HQWW2024,HQWW2025,HTW2026} and the references therein.
In \cite{LH2014-1,LH2014-2}, Hou and Luo provided convincing numerical evidence that the three-dimensional incompressible Euler equations with smooth initial data and boundary conditions may form finite-time singularities. This result was recently rigorously established by Chen and Hou in \cite{Chen-Hou2022,Chen-Hou2025} via a computer-assisted proof.

\subsection{Axisymmetric Euler equations without swirl} 

We call a vector field $u:\R^d\to\R^d$ axisymmetric without swirl if it has the form
\[u(x)=u^r(r,z)e_r+u^z(r,z)e_z,\quad x=(x_1,x_2,\cdots, x_{d-1},z),\quad r=\left(x_1^2+x_2^2+\cdots+x_{d-1}^2\right)^{1/2},\]
where $e_r=(x_1,x_2,\cdots, x_{d-1},0)/r$ and $e_z=(0,0,\cdots,0,1)$. It is well-known that such an axisymmetric swirl-free structure is preserved throughout the evolution governed by \eqref{Eq.Euler}.

Let $u=u^r(t,r,z)e_r+u^z(t,r,z)e_z$ denote an axisymmetric swirl-free solution to \eqref{Eq.Euler}. We define the scalar vorticity by $\omega:=\pa_ru^z-\pa_zu^r$. For $1\leq i\leq d-1$, we have $\omega_{di}=-\omega_{id}=-\omega\frac{x_i}{r}$, while all the other vorticity components $\omega_{ij}$ vanish identically. Under this setting, the Euler system \eqref{Eq.Euler} simplifies to
\begin{equation}\label{Eq.Euler-vorticity}
\begin{cases}
	\omega_t+u\cdot\nabla \omega=(d-2)\frac{u^r}{r}\omega,\quad (t,r,z)\in [0,T)\times\R^+\times\R,\quad\\
	 \pa_r\bigl(r^{d-2}u^r\bigr)+\pa_z\left(r^{d-2}u^z\right)=0,
	 \end{cases}
\end{equation}
where $u\cdot\nabla=u^r\pa_r+u^z\pa_z$. Accordingly, the relative vorticity $\omega_{\text{rel}}:=\omega/r^{d-2}$ satisfies a pure advection equation governed by the flow field:
\begin{equation}\label{Eq.Euler-relative-vorticity}
	\partial_t\omega_{\text{rel}}+u\cdot\nabla\omega_{\text{rel}}=0.
\end{equation}
This distinctive structure guarantees the global well-posedness of the three-dimensional axisymmetric swirl-free Euler equations with smooth initial data (see \cite{MB2002}). Ukhovskii and Yudovich \cite{UY1968} further established global well-posedness of \eqref{Eq.Euler} with Yudovich-type initial data, namely $u_0\in L^2(\R^3)$, $\omega_0\in \left(L^2\cap L^\infty\right)(\R^3)$ and $\omega_0/r\in \left(L^2\cap L^\infty\right)(\R^3)$. Danchin \cite{Danchin2007} subsequently relaxed the Ukhovskii–Yudovich conditions to $\omega_0\in \left(L^{3,1}\cap L^\infty\right)(\R^3)$ and $\omega_0/r\in L^{3,1}(\R^3)$. This result yields global regularity for three-dimensional axisymmetric swirl-free Euler flows with initial velocity $u_0\in \left(C^{1,\alpha}\cap L^2\right)(\R^3)$ for any $\alpha>1/3$. In parallel, global regularity for general $d$-dimensional axisymmetric swirl-free Euler equations has been recently verified in \cite{CJL2022,Lim-Jeong2024}, where the authors require the initial vorticity $\omega_0$ to be of compact support and $\omega_0/r^{d-2}\in L^\infty(\R^d)$ for spatial dimensions $d\in\{4,5,6\}$.

A key difference between the three-dimensional setting and its higher-dimensional counterparts lies in the fact that the relative vorticity $\omega_0/r^{d-2}$ generally fails to belong to $L^\infty(\R^d)$ for $d\geq4$, even if the initial velocity field $u_0$ is smooth. This constitutes a major obstacle to generalizing  the classical three-dimensional framework to higher spatial dimensions. Recent advances along this line were achieved by Shao, Wei and Zhang in \cite{SWZ2026}. Specifically, the authors established global regularity for $d$-dimensional axisymmetric swirl-free Euler equations with $d\in\{3,4,5,6\}$ under the following initial conditions:
\[u_0\in L^2(\R^d),\quad \frac{\omega_0}{r^{d-2}}\in L^{\frac d{d-2}, \infty}(\R^d) \quad\text{and}\quad \min\bigl\{1, r^{3-d}\bigr\}\frac{\omega_0}{r^{\alpha'}}\in L^\infty(\R^d)\]
for some $\alpha'\in(0,1)$. In particular, as pointed out in \cite[Remark 1.4]{SWZ2026}, for $d\in\{3,4,5,6\}$, any divergence-free and axisymmetric swirl-free initial velocity in $\bigl(C^{1,1-\frac2d}\cap L^2\bigr)(\R^d)$ gives rise to a unique global solution of the corresponding Euler system.

On the other hand, in a recent breakthrough work \cite{Elgindi2021},  Elgindi  established the existence of $C^{1,\al}(\R^3)$ self-similar blow-up solutions to the three-dimensional axisymmetric swirl-free Euler equations for sufficiently small $\al>0$. The self-similar profiles derived in \cite{Elgindi2021} possess stability in an appropriate sense, which allows cutoff procedures to yield finite-energy configurations and further motivated extensions to flows with nonzero swirl in \cite{Chen-Hou2021, EGM2022}.  Such blow-up solutions remain smooth away from the symmetry axis but only attain $C^{1,\al}$ regularity near the axis. Using distinct analytical approaches, the authors in \cite{Chen2023, CMZ2023} constructed finite-energy blow-up solutions to the three-dimensional incompressible Euler equations belonging to the class $C^\infty(\mathbb R^3\setminus\{0\})\cap \bigl(C^{1,\al}\cap L^2\bigr)(\R^3)$ for sufficiently small $\alpha$. In addition, the authors of \cite{CM2023} constructed smooth blow-up solutions to the non-axisymmetric three-dimensional incompressible Euler equations with external forcing, which are uniformly bounded in $\bigl(C^{1,\frac12-\varepsilon}\cap L^2\bigr)(\R^3)$.

It is therefore natural to characterize the exact range of $\alpha>0$ for which Elgindi’s three-dimensional result remains valid.  In \cite[Conjecture 8]{DE2023}, the authors put forward the following {\bf conjecture}: 

\noindent\textit{For any $\alpha\in(0,1/3)$, there exists a compactly supported $C_c^\alpha(\mathbb R^3)$ axisymmetric no-swirl local solution of the Euler equations, in the vorticity formulation, that develops a finite-time singularity.}

Solid numerical evidence supporting this conjecture has recently been presented in \cite{HZ2024}. In that reference, the authors investigated axisymmetric swirl-free Euler flows posed on the cylindrical domain $r\in[0,1]$ with periodic boundary conditions along the $z$-direction. 
The authors additionally examined the critical H\"older exponent $\al_d$ for arbitrary spatial dimensions $d\geq3$ and numerically recovered the relation  $\al_d=1-2/d$. Later, global regularity for flows with initial data  $u_0\in C^{1,\alpha_d}$ on such cylindrical domains was rigorously proven in \cite{SWZ2026}, a result holding for all integers 
 $d\in\N_{\geq 3}$.

\subsection{Main result}

The goal of this paper is to establish finite-time singularity formation below the critical H\"older threshold. We first construct self-similar profiles for the $d$-dimensional axisymmetric Euler equations without swirl in $\mathbb{R}^d$, for any $d\ge 3$ and $0<\alpha<\alpha_d:=1-2/d$, with the associated velocity fields belonging to $C^{1,\alpha}_{\rm loc}(\mathbb{R}^d)\cap C^\infty(\mathbb{R}^d\setminus\{0\})$. We then prove a finite-codimensional stability result for these profiles, which enables us to construct finite-energy, asymptotically self-similar blow-up solutions of \eqref{Eq.Euler} after suitable truncation and correction of unstable modes. 

Our first main result is as follows.


\begin{theorem}\label{Thm.blowupEuler}
{\sl	Let $d\in\N_{\geq 3}$ and $0<\alpha<\alpha_d:=1-2/d$. There exists $\gamma_{*,0}>1$ such that for any  $\gamma_*>\gamma_{*,0}$  
and $T>0$,  the system \eqref{Eq.Euler} has  a self-similar blow-up solution $u:[0, T)\times\R^d\to\R^d,$   which satisfies
	\begin{enumerate}[(i)]
		\item[(1)]  {\bf Self-similarity:} there exists a self-similar profile $$u_{\rm s}=u_{\rm s}^r(r,z)e_r+u_{\rm s}^z(r,z)e_z\in C^{1,\alpha}_{\rm loc}(\R^d;\R^d)\cap C^\infty(\R^{d}\setminus\{0\};\R^d),$$ which is divergence-free, axisymmetric and swirl-free, and independent of $T$, such that
		\begin{align*}
			&|u_{\rm s}(x)|
			\lesssim_{d,\alpha,\gamma_*}
			\langle x\rangle^{1-\frac1{\gamma_*}},\\
			u(t,x)=(T-t)^{\gamma_*-1}&u_{\rm s}\left(\frac{x}{(T-t)^{\gamma_*}}\right),\quad\forall\ (t,x)\in [0, T)\times\R^d,
		\end{align*}
		and the corresponding vorticity matrix satisfies
		\begin{equation*}
			\bm\omega(t,x)=(T-t)^{-1}\bm\omega_{\rm s}\left(\frac{x}{(T-t)^{\gamma_*}}\right),\quad\forall\ (t,x)\in [0, T)\times\R^d
		\end{equation*}
		for some self-similar profile $\bm\omega_{\rm s}\in C^{\alpha}(\R^d;\R^{d\times d})\cap C^\infty(\R^{d}\setminus\{0\};\R^{d\times d})$.
		
		\item[(2)]  {\bf Anisotropic decay:} 
		\beno
		|\bm\omega_{\rm s}(x)|\lesssim_{d,\alpha,\g_*} r^{d-2}|z||(r,z)|^{-d+1-1/\g_*}\quad \text{for all}\,\, |x|\geq 1.
		\eeno
		
		\item[(3)]  {\bf Symmetries:} the corresponding scalar vorticity profile $\omega_{\rm s}:=\partial_ru_{\rm s}^z-\partial_zu_{\rm s}^r$ is odd in $z$ and satisfies $\omega_{\rm s}(r,z)>0$ for all $r>0, z>0$.
		
		\item[(4)]  {\bf Boundary non-degenerate behavior:} there exist $m_1\in \N_+$ and $m_2\in\N_+$ such that for all $r>0$ and $z>0$, we have
		\begin{equation}\label{S1eq(4)}		\begin{split}
			\partial_r^j\omega_{\rm s}(0,z)=0 \ \ \text{for any}\ \ j\in\{0, 1,\cdots, m_1-1\},\quad\text{while}\ \  \partial_r^{m_1}\omega_{\rm s}(0,z)\neq 0,\\
			\partial_z^j\omega_{\rm s}(r,0)=0 \ \ \text{for any}\ \ j\in\{0, 1,\cdots, m_2-1\},\quad\text{while}\ \  \partial_z^{m_2}\omega_{\rm s}(r,0)\neq 0.
		\end{split}\end{equation}
	\end{enumerate}
	Moreover, for each $d,\alpha, \gamma_*$, there are uncountably many self-similar profiles $(\bm\omega_{\rm s}, u_{\rm s})$ satisfying the above four conditions.}
\end{theorem}

We highlight two recent closely related developments for the three-dimensional case of Theorem \ref{Thm.blowupEuler}. In \cite{Shkoller2026}, Shkoller introduced a novel Lagrangian formulation and established finite-time blow-up for the three-dimensional axisymmetric swirl-free Euler equations. Shkoller’s result applies to initial velocities $u_0\in \bigl(C^{1,\alpha}\cap L^2\bigr)(\R^3)$ and covers the full range $\alpha\in(0,1/3)$. The solutions constructed therein are dynamically stable but not self-similar.  In subsequent works \cite{Chen2026-1,Chen2026-2}, Chen constructed finite-time self-similar blow-up solutions within the same regularity class. A notable feature of Chen’s approach is that the self-similar profile is obtained by lifting a smooth one-dimensional profile, which was rigorously constructed in \cite{Chen2026-1} with computer-assisted verification. While the results of Shkoller, Chen, and the present work all establish finite-time singularity formation below the critical regularity threshold, they employ fundamentally different methods, produce solutions with distinct qualitative features, and were developed completely independently.
\smallskip

\begin{remark}
We have several remarks in order.
\begin{enumerate}
	\item[(1)]  Owing to the sublinear growth of the self-similar profiles at spatial infinity, the blow-up solutions constructed in Theorem \ref{Thm.blowupEuler} do not possess finite kinetic energy. Analogous to the results in \cite{Chen2026-2,EGM2022}, our self-similar blow-up solutions enjoy finite-codimensional stability. This permits suitable truncation procedures to yield asymptotically self-similar blow-up solutions to \eqref{Eq.Euler} evolving from compactly supported initial vorticity, with the corresponding initial velocity $u_0\in C^{1,\alpha}(\R^d)\cap C^\infty(\R^d\setminus\{0\})\cap L^2(\R^d)$ for any $\alpha\in(0,\alpha_d)$. This finite-energy blow-up result will be presented in Theorem \ref{Thm.finite-energy-blowup} below.
	
	\item[(2)]  For each fixed $d,\alpha,\gamma_*$, the self-similar profiles constructed in this paper form a three-parameter family, and each of these parameters ranges over an uncountable set. See Remark \ref{Rmk.parameters} for further details.

	\item[(3)]  In contrast to \cite{Chen2026-2,Shkoller2026}, our blow-up solutions are smooth away from the origin, whereas the solutions in \cite{Chen2026-2,Shkoller2026} fail to be smooth along the symmetry axis.	
	
	\item[(4)]  In contrast to \cite{Chen2026-1,Chen2026-2}, our construction of self-similar profiles is purely analytic and does not rely on any computer-assisted verification.
	
\end{enumerate}
\end{remark}

Our second main result is concerned with the existence of the finite-energy blow-up solutions of \eqref{Eq.Euler}.

\begin{theorem}[Finite-energy blow-up]\label{Thm.finite-energy-blowup}
	{\sl	Let $d\in\N_{\geq3}$ and $0<\alpha<\alpha_d:=1-2/d$. There exists $\gamma_{*,0}>1$ such that for any $\gamma_*>\gamma_{*,0}$, there exists $T>0$ so that the system \eqref{Eq.Euler}  has an axisymmetric swirl-free solution $u:[0,T)\times\R^d\to\R^d$ with initial velocity
		\[u_0\in C^{1,\alpha}(\R^d;\R^d)\cap C^\infty(\R^d\setminus\{0\};\R^d)\cap L^2(\R^d;\R^d),\]
		which blows up at time $T$ in the sense that the vorticity matrix  defined by \eqref{S1eq1} satisfies	\begin{equation}\label{S0eq4} 
			c(T-t)^{-1}\leq \|\bm\omega(t)\|_{L^\infty(\R^d)}\leq C(T-t)^{-1},\qquad 0<t<T,\end{equation}
		where  $c,C>0$ are constants depending on the constructed solution.
		Moreover, the solution is asymptotically self-similar with a concentrating scale $\lambda(t)\to0$ as $t\uparrow T$ satisfying
		\begin{equation}\label{S0eq6}
			c(T-t)^{\gamma_*}\leq \lambda(t)\leq C(T-t)^{\gamma_*},\qquad 0<t<T.
		\end{equation}
		Furthermore,  the initial vorticity $\bm{\omega}_0$ can be chosen to have compact support.
	}
\end{theorem}

\begin{remark} More precisely, let $\mu:=\gamma_*/(1+(d-2)\gamma_*)$. We shall prove in Section \ref{Sec.stability} that there exist $\delta\in(0,1/\mu-(d-2)),$ $\nu>0$
	and a   fixed profile $\Omega_*$ 
	such that 
	\[\Omega^{\rm ren}(t,r,z):=\lambda(t)^{1/\mu}\omega_{\rm rel}(t,\lambda(t)r,\lambda(t)z),\]
	satisfies
	\begin{equation}\label{S0eq5}  \begin{split}
			&\Bigl\||(r,z)|^{-\delta}\frac{\left(\Omega^{\rm ren}(t)-\Omega_*\right)}{\Omega_*}\Bigr\|_{L^\infty}+\Bigl\||(r,z)|^{-\delta}\frac{r\pa_r\left(\Omega^{\rm ren}(t)-\Omega_*\right)}{\Omega_*}\Bigr\|_{L^\infty}\\
			&+\Bigl\||(r,z)|^{-\delta}\frac{z\pa_z\left(\Omega^{\rm ren}(t)-\Omega_*\right)}{\Omega_*}\Bigr\|_{L^\infty}\leq C(T-t)^\nu, \quad \forall\ 0<t<T.
	\end{split}\end{equation}
\end{remark}

Let us emphasize several features of Theorem \ref{Thm.finite-energy-blowup}. The exact self-similar solutions constructed in Theorem \ref{Thm.blowupEuler} exhibit slow spatial decay and thus lack finite kinetic energy. Theorem \ref{Thm.finite-energy-blowup} overcomes this limitation by leveraging the finite-codimensional stability of the self-similar profile. More precisely, one first truncates the profile in the far field and then adjusts finitely many unstable components of the renormalized initial datum, so that the resulting solution lies on the stable manifold of the profile. The solutions obtained in Theorem \ref{Thm.finite-energy-blowup} are therefore not exactly self-similar, but only asymptotically self-similar. We also note that the finite-dimensional correction is supported away from the origin. Consequently, the locally nondegenerate structure of the profile near the origin and the coordinate axes is preserved.  As a result, the finite-energy solutions retain the same local $C^{1,\alpha}$ regularity near the origin and smoothness away from the origin as the exact self-similar profiles.  Meanwhile, the far-field truncation renders the corresponding initial velocity globally 
$C^{1,\alpha}$ and square-integrable.

More broadly, beyond the Euler blow-up mechanisms discussed above, self-similar techniques have played an important role in the construction of singular and coherent structures in fluid-dynamics equations  and related PDEs. Notable examples include algebraic spiral solutions to the two-dimensional Euler equations \cite{Elling2016,SWZ2025-AnnPDE,SWZ2026-AnnPDE}, imploding solutions for compressible fluids \cite{BCLGS2025,CLGSSS2025,CCSV2024,CSV,GHJ2021,GHJ2023,GHJS2022,MRRJ2022Ann1,MRRJ2022Ann2,SWZ2024,SWWZ2025}, blow-up dynamics for energy-supercritical defocusing dispersive equations \cite{Buckmaster-Chen2024,CLGSSS-2024,MRRJ2022Invent,SWZ2025-Pi}, and, more recently, finite-time singularities for kinetic equations \cite{BCGJVY2026}.

\subsection{Organization of the paper} The rest of the paper is organized as follows. 

In Section \ref{Sec.Road-map}, we give a road-map of the construction of the self-similar profiles. We first derive the self-similar profile equations and reduce them, through a stream-function formulation, to a coupled elliptic--transport system. We then introduce the functional spaces tailored to the expected singular behavior, boundary structure, and far-field decay of the profiles. Within this setting, we define the transport and elliptic solution maps, formulate the Schauder fixed-point problem, and explain how a fixed-point argument yields Theorem \ref{Thm.profile}.

In Section \ref{Sec.stability}, we give a road-map for the finite-codimensional stability argument leading to Theorem \ref{Thm.finite-energy-blowup}. 

In Section \ref{Sec.transport}, we investigate the transport equation. After a change of variables adapted to the characteristic flow, we derive a representation formula for the vorticity profile and provide the upper and lower bounds required for Proposition \ref{Prop.F}. 

In Section \ref{Sec.elliptic}, we study the elliptic equation through the Newtonian potential representation in $\mathbb R^{d+4}$. We establish the zeroth-, first-, and second-order estimates needed to prove Proposition \ref{Prop.G}. 

In Section \ref{Sec.existence-fixedpoint}, we prove the continuity of the transport map and the continuity and compactness of the elliptic map, thereby completing the proof of the fixed-point theorem.

 In Section \ref{sec:smoothness}, we prove the regularity properties of the fixed point. 
 
In Section \ref{Sec.stability-proof}, we prove the finite-codimensional stability results stated in Section \ref{Sec.stability}. 

Finally,  in Appendix \ref{AppendixA}, we collect several convolution inequalities used in the estimates for the Newtonian potential.  In Appendix \ref{app:second-log-derivative}, we  derive  the second-order derivative estimates for the fixed point, which are used in the stability analysis.

\section{A road-map of the proof: existence of self-similar profiles}\label{Sec.Road-map}

The purpose of this section is to set up the fixed-point scheme used to construct the self-similar profiles of the system \eqref{Eq.Euler}. Starting from the self-similar ansatz, we reduce the profile equations to a coupled elliptic--transport system for a normalized stream function $\psi$ and a relative vorticity profile $\Omega$. The construction is then reduced to a nonlinear fixed-point problem:  given $\psi$, the transport equation determines $\Omega$ from prescribed data on a characteristic curve, while a rescaled elliptic equation recovers  a new stream function from $\Omega$. The most crucial step is to choose function spaces that are simultaneously adapted to the singular scaling near the origin, the boundary structures at $r=0$ and $z=0$, and the slow spatial decay at infinity.

\subsection{Self-similar formulation}
Our first goal is to construct self-similar blow-up solutions of \eqref{Eq.Euler-vorticity}.  For any $\gamma_\ast>0,$ we introduce the self-similar change of variables
\begin{align*}
	&u^r(t,r,z)=(T-t)^{\gamma_*-1}C_d(\gamma_*)^{d-1}U^r\Bigl(\frac{r}{C_d(\gamma_*)(T-t)^{\gamma_*}}, \frac{z}{C_d(\gamma_*)(T-t)^{\gamma_*}}\Bigr),\\
	&u^z(t,r,z)=(T-t)^{\gamma_*-1}C_d(\gamma_*)^{d-1}U^z\Bigl(\frac{r}{C_d(\gamma_*)(T-t)^{\gamma_*}}, \frac{z}{C_d(\gamma_*)(T-t)^{\gamma_*}}\Bigr),\\
	&\omega_{\text{rel}}(t,r,z)=\frac{\omega(t,r,z)}{r^{d-2}}=(T-t)^{-1-(d-2){\gamma_*}}\Omega\Bigl(\frac{r}{C_d(\gamma_*)(T-t)^{\gamma_*}}, \frac{z}{C_d(\gamma_*)(T-t)^{\gamma_*}}\Bigr),
\end{align*}
where $\gamma_*>0$ is the self-similar parameter, $C_d(\gamma_*):=\left(1+(d-2)\gamma_*\right)^{\frac1{d-2}}$, $t\in[0, T)$, $r\geq 0$ and $z\in\R$. 

Let us denote
\begin{equation*}
	\mu:=\frac{\gamma_*}{1+(d-2)\gamma_*}\in\left(0,\frac1{d-2}\right),
\end{equation*}
then in the self-similar coordinates, \eqref{Eq.Euler-vorticity} is converted to the following system for the self-similar profiles $(U^r, U^z, \Omega)$:
\begin{equation}\label{Eq.ss-eq-UrUz}
	\begin{cases}
		\Omega+\mu\left(r\pa_r+z\pa_z\right)\Omega+U^r\partial_r\Omega
		+U^z\partial_z\Omega=0,\quad(r,z)\in\R^+\times\R,\\
		\pa_r\left(r^{d-2}U^r\right)+\pa_z\left(r^{d-2}U^z\right)=0,\\
		r^{d-2}\Omega=\pa_rU^z-\pa_zU^r.
	\end{cases}
\end{equation}
We seek solutions to \eqref{Eq.ss-eq-UrUz} possessing the following $z$-odd symmetric properties:
\begin{equation}\label{Eq.z-odd-symmetry}
	\Omega(r,-z)=-\Omega(r,z),\quad U^r(r,-z)=U^r(r,z),\quad U^z(r,-z)=-U^z(r,z).
\end{equation}
Using this symmetry, it suffices to solve \eqref{Eq.ss-eq-UrUz} in the first quadrant $\{(r,z):r>0, z>0\}$. 

 We introduce the stream function $\psi_{\text s}=\psi_{\text s}(r,z)$ by solving 
\begin{equation}\label{Eq.psi-s-stream}
	r^{d-2}U^r=\pa_z\psi_{\text s} \andf r^{d-2}U^z=-\partial_r\psi_{\text s} \with  \psi_{\text s}(0,0)=0.
\end{equation}
By \eqref{Eq.z-odd-symmetry}, we define  $\psi_{\text s}(r,z)=-\psi_{\text s}(r,-z)$ if $z<0.$ Moreover, it follows from $r^{d-2}\Omega=\pa_rU^z-\pa_zU^r$ and \eqref{Eq.psi-s-stream} that
\begin{equation}\label{S2Omeq}
	-r^{d-2}\Omega=\pa_r\left(r^{2-d}\pa_r\psi_{\text s}\right)+\pa_z\left(r^{2-d}\pa_z\psi_{\text s}\right).
\end{equation}

Let 
\begin{equation}\label{Eq.psi-def}
	\psi_{\text s}(r,z)=r^{d-1}z\psi(r,z),
\end{equation}
then the equation \eqref{S2Omeq}
 is equivalent to the following equation for $\psi$:
\begin{equation}\label{Eq.psi-elliptic}
	\Bigl(\pa_r^2+\frac{d}{r}\pa_r+\pa_z^2+\frac{2}{z}\pa_z\Bigr)\psi=-r^{d-3}\frac{\Omega}{z},\quad r>0, z>0.
\end{equation}

\begin{remark}\label{Rmk.elliptic}
	Equation \eqref{Eq.psi-elliptic} may be recast as a Poisson equation in $\mathbb R^{d+4}$. For $X=(x,y)\in \mathbb R^{d+1}\times \mathbb R^3$, set $\Psi(X)=\Psi(x,y):=\psi(|x|,|y|)$, then a direct computation yields $$\Delta_X\Psi(X)=(\Delta_x+\Delta_y)\Psi(X)=\left(\partial_r^2+\frac{d}{r}\partial_r+\partial_z^2+\frac{2}{z}\partial_z\right)\psi(r,z).$$ Consequently, \eqref{Eq.psi-elliptic} is equivalent to
	\begin{equation} \label{S2eqPsi}
	-\Delta_X\Psi(X)=|x|^{d-3}|y|^{-1}\Omega(|x|,|y|),\quad X=(x,y)\in \mathbb R^{d+1}\times \mathbb R^3. \end{equation}
		It follows from the Newtonian potential representation in $\mathbb R^{d+4}$ that
	\begin{equation} \label{Eq.Psi0-def}	\Psi(X)=\beta_d\int_{\mathbb R^{d+4}}|X-Y|^{-d-2}\,|\xi|^{d-3}|\eta|^{-1}\Omega(|\xi|,|\eta|)\,\mathrm dY \with Y=(\xi,\eta)\in \mathbb R^{d+1}\times \mathbb R^3,\end{equation}
	where $\beta_d=\frac{1}{(d+4)(d+2){\rm{v}}_{d+4}}$, and ${\rm{v}}_{d+4}$ denotes the volume of the unit ball in $\mathbb R^{d+4}$. This formula  will serve as the starting point for the estimates of $\Psi$ and $\psi$ below.
\end{remark}

By virtue of  \eqref{Eq.psi-s-stream} and \eqref{Eq.psi-def},  we can rewrite the first equation of \eqref{Eq.ss-eq-UrUz} as the following transport equation
\begin{equation}\label{Eq.Omega-rel-transport}
	\left(\mu+\psi+z\pa_z\psi\right)r\pa_r\Omega+\left(\mu-(d-1)\psi-r\pa_r\psi\right)z\pa_z\Omega=-\Omega.
\end{equation}

In what follows, we shall study the coupled equations \eqref{Eq.psi-elliptic} and \eqref{Eq.Omega-rel-transport}.  
Below we always denote
\begin{equation}\label{S2eq1}
\begin{split}
&\Pi_+:=\bigl\{(r,z)\in\mathbb R^2:r>0,\ z>0\bigr\},\qquad\overline{\Pi_+}:=\bigl\{(r,z)\in\mathbb R^2:r\geq0,\ z\geq0\bigr\} \andf\\
&D:=\Pi_+\cup\bigl\{(r,0):r>0\bigr\}\cup\bigl\{(0,z):z>0\bigr\},\qquad D_+:=\bigl\{(r,z):r\geq0\bigr\}\setminus\bigl\{(0,0)\bigr\}.\end{split} 
\end{equation}

Then we may reduce the proof of Theorem \ref{Thm.blowupEuler} to the following theorem.

\begin{theorem}[Existence of self-similar profiles]\label{Thm.profile}
{\sl	For each $d\in\mathbb N_{\geq 3}$, $0<\alpha<\alpha_d:=1-2/d$, there exists $\mu_0\in\left(0,1/(d-2)\right)$ such that for every $\mu\in(\mu_0,1/(d-2))$, the self-similar profile system	\eqref{Eq.ss-eq-UrUz} admits a solution $(U^r, U^z, \Omega)$ satisfying the	$z$-odd symmetry condition \eqref{Eq.z-odd-symmetry}. More precisely, there exists a function $\psi$ such that, for any $r>0$ and $z>0$,
	\begin{equation}\label{Eq.profile-velocity-from-psi}
		U^r=r(\psi+z\partial_z\psi),\ \ 
		U^z=-z\bigl((d-1)\psi+r\partial_r\psi\bigr) \andf \psi\in C^\infty(D_+), \ \   \Omega\in C^\infty(D_+). 	\end{equation}
	Furthermore, there hold:
	\begin{enumerate}[(i)]
		\item[(1)] let $\mathbf U_{\rm s}(x):=U^r(r,z)e_r+U^z(r,z)e_z$ {for} $x\in\R^d\setminus\{0\}$ and
\begin{equation*}
\bm{\omega}_{\rm s}=(\omega_{{\rm s}, ij})_{1\leq i,j\leq d} \with \omega_{{\rm s}, ij}=
\begin{cases}
\omega_{{\rm s}, id}=-\omega_{{\rm s}, di}=\omega_{\rm s}x_i/r\ \mbox{ for}\ 1\leq i\leq d-1,\\
0 \ \mbox{ otherwise},
\end{cases}
\end{equation*} then one has
		\begin{align*}
		&\omega_{\rm s}=r^{d-2}\Omega\in C^{\alpha}(\{(r,z): r\geq 0\}),\quad			\mathbf U_{\rm s}\in C^{1,\alpha}_{\rm loc}(\R^d;\R^d)\cap C^{\infty}(\R^d\setminus\{0\}),\\
		&\andf \bm\omega_{\rm s}\in C^{\alpha}(\R^d;\R^{d\times d})\cap C^{\infty}(\R^d\setminus\{0\});
		\end{align*}
		
		\item[(2)] $|\mathbf U_{\rm s}(r,z)|
		\lesssim \langle r,z\rangle^{d-1-\frac1\mu}$ and
		$|\omega_{\rm s}(r,z)|\lesssim r^{d-2}|z||(r,z)|^{-1-1/\mu}$ for all $r\geq 0$, $z\in\R$ with $|(r,z)|\geq 1$;
		
		\item[(3)]  $\omega_{\rm s}(r,z)>0$ for all $r>0, z>0$;
		\item[(4)]  there exist $m_1\in \N_+$ and $m_2\in\N_+$ such that \eqref{S1eq(4)}	holds  for all $r>0$ and $z>0.$
		\end{enumerate}
	Finally, for each $d,\alpha$ and $\mu$, there exist uncountably many self-similar profiles $(U^r, U^z, \Omega)$ which satisfy the above properties.}
\end{theorem}

Obviously, Theorem \ref{Thm.blowupEuler} follows directly from Theorem \ref{Thm.profile} above.

\begin{remark}\label{Rmk.parameters}
	According to our construction, for each fixed $d,\alpha$ and $\mu$, the self-similar profiles form a three-parameter family. One parameter is given by $a$ in \eqref{Eq.a-range}, which belongs to an uncountable set; see Subsection \ref{Subsec.proof-thm-profile} for details. The other two parameters are given by the functions $\Theta$ and $\chi$ introduced in Subsection \ref{Subsec.fixed-point-formulation}, which also belong to uncountable families; see Remark \ref{Rmk.vanishing-conditions} for further details.
\end{remark}

\subsection{Functional spaces}
We now introduce the functional framework in which the fixed-point argument will be implemented. The parameters below encode the expected leading-order behavior of the stream function near the origin and the asymptotic scaling dictated by the self-similar equations. The functional space for $\psi$ is designed to control its magnitude, derivatives, and far-field decay, while the functional space for $\Omega$ encodes the corresponding positivity, boundary vanishing, core lower bound,  and $(r,z)$-weighted first-order derivative estimates. These constraints are chosen so that the transport map and the elliptic map defined in the next subsection are closed within the same functional class.

We fix $d\in\N_{\geq 3}$ and the parameters $a$ and $a_1,$ so that
\begin{equation}\label{Eq.a-range}
	\begin{aligned}
		&\frac1{d(d-2)}<\frac{4}{(4d-3)(d-2)}<a<\frac{1}{(d-1)(d-2)},\\ 
		&a_1:=\frac a2+\frac1{2(d-1)(d-2)}\in \left(a, \frac{1}{(d-1)(d-2)}\right).
	\end{aligned}
\end{equation}
We choose a parameter $\mu$ so that
\begin{equation}\label{Eq.mu-range}
    \frac1d<\frac3{3d-5}<(d-1)a<(d-1)a_1<\frac{(d-1)a_1}2+\frac{1}{2(d-2)}<\mu<\frac{1}{d-2}.
\end{equation}
We also define
\begin{equation}\label{Eq.gamma-gamma1-def}
	\gamma:=\frac{\mu+1-(d-1)a}{\mu+a} \andf \gamma_1:=\frac{\mu+1-(d-1)a_1}{\mu+a_1}.
\end{equation}
Throughout the paper, all constants are allowed to depend on the parameter $a$ (and $a_1$), hence the parameter $a$ should be viewed as a fixed constant. Of course, all constants are allowed to depend on the dimension $d\in\N_{\geq3}$.  Many constants (see Subsection \ref{Subsec.notations}) are independent of the parameter $\mu$,  which satisfies \eqref{Eq.mu-range} and $(d-2)^{-1}-\mu\ll 1$. 

As a convention, throughout the rest of the paper, we always assume that 
\begin{equation}\label{d210}
d\in\mathbb{N}_{\geq 3}\ \ \ \text{and}\ \ \ \text{the parameters}\ a, \mu\mbox{ satisfy}\  \eqref{Eq.a-range} \  \text{and}\ \eqref{Eq.mu-range}\  \mbox{respectively}. 
\end{equation}

For  $\Pi_+$ given by \eqref{S2eq1}, we define the functional space
\begin{equation}\label{S2eq8}
\mathcal A^0 := \left\{\begin{aligned}
&\psi \in C\left(\overline{\Pi_+}\right),\\
&\psi \in C^2(\Pi_+)
\end{aligned}
\;\middle|\;
\begin{aligned}
&\psi(0,0)=a,\quad 0<\psi(r,z)\leq a_1, \ \ \forall\ (r,z)\in\overline{\Pi_+},\\
& \andf \text{for all}\  \ (r,z)\in\Pi_+\\
&|\nabla_{r,z}\psi(r,z)|\leq \frac1{10}\langle r,z\rangle^{-1+\frac1{d\mu}}\langle z\rangle^{-\frac1{d\mu}}\psi(r,z),\\
&
|\langle r,z\rangle\pa_r\psi(r,z)|\leq \frac1{10}\psi(r,z),\\
&|\nabla_{r,z}^2\psi(r,z)|\leq \frac1{10}\Big(\langle r,z\rangle^{2-d+\frac1\mu}r^{d-3-\frac{d-1}{d\mu}}\langle z\rangle^{-1-\frac1{d\mu}}\\
&\qquad\qquad\qquad\qquad\quad+r^{d-3-\gamma}\mathbf1_{\{r^2+z^2\leq 1\}}\Big)\psi(r,z), \end{aligned}
\right\}.
\end{equation}
Here $\langle r,z\rangle:=(1+r^2+z^2)^{1/2}$ and $\langle z\rangle:=(1+z^2)^{1/2}$. 

For any $M\in(1,+\infty)$, we define the functional space
\begin{equation}\label{defam}
    \mathcal A_M:=\left\{\psi\in \mathcal A^0: M^{-1}\langle r,z\rangle^{d-2-\frac1\mu}\leq \psi(r,z)\leq M\langle r,z\rangle^{d-2-\frac1\mu},\ \ \forall\ \ (r,z)\in \overline{\Pi_+}\right\}.
\end{equation}
Then clearly $\mathcal A_M\subset \mathcal A^0$ for all $M>1$. 

For any $M'>1$, we also define the following functional space for $\Omega=\Omega(r,z)$:
\begin{equation}\label{Eq.BM'-def}
    \mathcal B_{M'}:=\left\{\Omega\in  C^1(\Pi_+)\;\middle|\;\begin{aligned}
    &  \ \mbox{for all}\ (r,z)\in\Pi_+,\\
            &0<\Omega(r,z)\leq M'\min\left\{r^{-\frac{d-1}{d\mu}}z^{-\frac1{d\mu}}, zr^{-\gamma}+zr^{-\gamma_1}\right\},\\
      &  \Omega(r,z)\geq \frac1{M'}r^{-\frac{d-1}{d\mu}}z^{-\frac1{d\mu}}\ \ \text{for all}\ \ 2<z<r<2z,\\
      &  |r\partial_r\Omega(r,z)|+|z\partial_z\Omega(r,z)|\leq M' \Omega(r,z),\qquad\qquad\\
    \end{aligned}\right\}.
\end{equation}

With the above functional spaces, we shall construct the two solution maps whose composition will be the desired fixed-point map.

\subsection{A fixed-point formulation}\label{Subsec.fixed-point-formulation}
Given $\psi\in \mathcal A^0$, we investigate the transport equation \eqref{Eq.Omega-rel-transport} for the unknown $\Omega=\Omega(r,z)$. In order to determine a unique solution of \eqref{Eq.Omega-rel-transport}, we need to assign an ``initial datum''. For convenience, we define the initial surface (or curve, since we shall solve  \eqref{Eq.Omega-rel-transport} in a subdomain of $\R^2$) by
\begin{equation}\label{Eq.Gamma0-def}
    \Gamma_0:=\bigl\{(r,z)\in\R^2: r^2+(\mu+\psi(r,z))^2z^2=1, \ r>0, z>0 \bigr\}.
\end{equation}
Using the properties of $\psi$ in $\mathcal A^0$ and the implicit function theorem, one can directly check that $\Gamma_0$ is a well-defined $C^1$ curve and is the graph $z=f_0(r)$ with
 $$f_0\in C^1(0,1),\quad  \lim_{r\to0}f_0(r)\in (0, +\infty) \andf \lim_{r\to 1}f_0(r)=0. $$
  As a consequence, the curve $\Gamma_0$ can be parametrized by $r$, and $\Gamma_0=\bigl\{(r, f_0(r)): r\in(0,1)\bigr\}$. Alternatively, one can parametrize the curve $\Gamma_0$ by the angle variable $\sigma\in (0,\pi/2)$, hence
\[\Gamma_0=\bigl\{(\sin\sigma, f_0(\sin\sigma)): \sigma\in(0,\pi/2)\bigr\}.\]

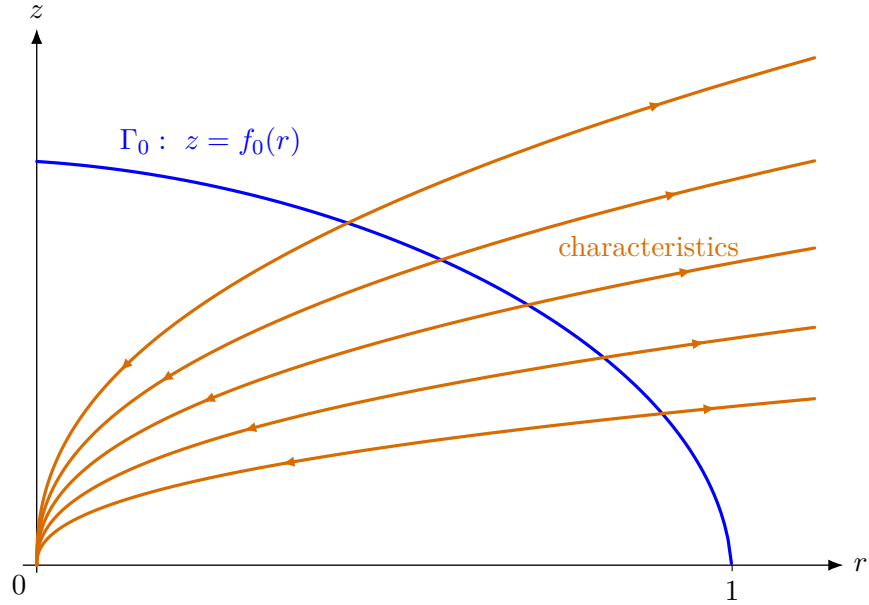
\begin{figure}[htbp]
	\centering
	\begin{tikzpicture}[x=9.2cm,y=5cm]
		
		\draw[-{Latex[length=2mm]}, line width=0.5pt] (-0.02,0) -- (1.16,0) node[right] {$r$};
		\draw[-{Latex[length=2mm]}, line width=0.5pt] (0,-0.02) -- (0,1.42) node[above] {$z$};
		
		\draw (0,0) node[below left] {$0$};
		\draw (1,0) -- +(0,-0.016) node[below] {$1$};
		
		\draw[blue,very thick,line join=round,line cap=round]
		plot[domain=0:1,samples=200]
		(\x,{1.02*sqrt(1-\x*\x)/(0.92+0.10*\x)-0.04});
		
		\node[blue] at (0.25,1.12) {$\Gamma_0:\ z=f_0(r)$};
		
		\def\pp{0.43}
		\def\mm{5}
		\def\tmax{1.023} 
		
		\tikzset{
			charcurve/.style={
				orange!85!black,
				very thick,
				line join=round,
				line cap=round
			}
		}
		
		\foreach \cc/\pin/\pout in {
			0.42/0.34/0.88,
			0.60/0.31/0.87,
			0.80/0.28/0.86,
			1.02/0.25/0.85,
			1.28/0.22/0.84
		}{
			\draw[
			charcurve,
			postaction={decorate},
			decoration={
				markings,
				mark=at position \pin with {\arrowreversed{Latex[length=1.8mm,width=1.4mm]}},
				mark=at position \pout with {\arrow{Latex[length=1.8mm,width=1.4mm]}}
			}
			]
			plot[variable=\t,domain=0:\tmax,samples=100]
			({pow(\t,\mm)},{\cc*pow(\t,\mm*\pp)});
		}
		
		\node[orange!85!black] at (0.88,0.845) {characteristics};
		
	\end{tikzpicture}
	\caption{The initial curve and characteristic curves for \eqref{Eq.Omega-rel-transport}}
	\label{Fig.transport-original}
\end{figure}

To fix the initial data for the transport equation \eqref{Eq.Omega-rel-transport}, we choose and fix a function $\Theta\in C^\infty((0,+\infty);(0,1])$ such that 
\[ \Theta(s)>0\ \ \mbox{for}\ s>0, \ \lim_{s\to0+}\Theta(s)=0 \andf \Theta(s)=1\ \ \mbox{for}\ s\geq 1. \]
For each $\lambda>0$ we define $\Theta_\lambda(s):=\Theta(\lambda s)$ for any $s>0$. Moreover,  we impose the following finite-order polynomial vanishing condition at $s=0$:
\begin{equation}\label{Eq.Theta-vanishing-condition}
	\Theta(s)=s^{\delta_d+(d-1)/(d\mu)}\quad\forall\ s\in(0, 1/10], \andf  \delta_d:=\begin{cases}
0&\text{if} \ d\ \text{is odd},\\
 1&\text{if} \ d\ \text{is even}. \end{cases} \end{equation} 

We also fix a smooth function $\eta\in C^\infty((0,\pi/2);[0,1])$ such that
\[\eta(\sigma)=1\quad\text{for }0<\sigma\le \arctan(2d),\qquad \eta(\sigma)=0\quad\text{for }\arccos\frac1{3d}\le \sigma<\pi/2.\]
Then we define
\begin{equation}\label{Eq.chi-def}
	\chi_\mu(\sigma):=\eta(\sigma)+\bigl(1-\eta(\sigma)\bigr)(\cos\sigma)^{1+\frac1{d\mu}},\quad\forall\ \sigma\in(0,\pi/2).
\end{equation}
This gives $\chi_\mu(\sigma)=1$ for $\sigma\in(0,\arctan(2d)]$, $\lim_{\sigma\to\left(\pi/2\right)_-}\chi_\mu(\sigma)=0$, and
\begin{equation}\label{Eq.chi-bound}
	1\le\frac{\chi_\mu(\sigma)}{(\cos\sigma)^{1+\frac1{d\mu}}}\le 9d^2,\quad\forall\  0<\sigma<\pi/2.
\end{equation}
As a matter of fact, if $0<\sigma\le\arctan(2d)$, then $\eta(\sigma)=1$, and hence
\begin{align*}
\frac{\chi_\mu(\sigma)}{(\cos\sigma)^{1+\frac1{d\mu}}}=(\cos\sigma)^{-1-\frac1{d\mu}}\le&\big(\cos(\arctan(2d))\big)^{-1-\frac1{d\mu}}\\
=&\bigl(\sqrt{1+4d^2}\bigr)^{1+\frac1{d\mu}}\le1+4d^2\le 9d^2,\end{align*}
where we used $1+\frac1{d\mu}<2$ and $d\ge3$. 

If $\arctan(2d)<\sigma<\arccos(1/(3d))$, then $\cos\sigma\geq 1/(3d)$. Therefore,
\[\frac{\chi_\mu(\sigma)}{(\cos\sigma)^{1+\frac1{d\mu}}}\leq (\cos\sigma)^{-1-\frac1{d\mu}}\le(3d)^{1+\frac1{d\mu}}\le(3d)^2=9d^2.\]

Finally, if $\arccos(1/(3d))\le\sigma<\pi/2$, then $\eta(\sigma)=0$, and one has
\begin{equation}\label{Eq.chi-vanishing-condition}
	\chi_\mu(\sigma)=(\cos\sigma)^{1+\frac1{d\mu}}.
\end{equation} 

Combining the above estimates, we obtain \eqref{Eq.chi-bound}. 

To simplify notation, we simply denote $\chi_\mu$ by $\chi$ if it causes no confusion.

\begin{remark}\label{Rmk.vanishing-conditions}
	We impose finite-order polynomial vanishing conditions \eqref{Eq.Theta-vanishing-condition} and \eqref{Eq.chi-vanishing-condition} on $\Theta$ and $\chi$, respectively. These conditions are irrelevant to the proof of Theorem \ref{Thm.fixed-point} concerning fixed-point existence. Nevertheless, they are essential for establishing the regularity of fixed points; see Proposition \ref{prop:smoothness-fixed-point} and Section \ref{sec:smoothness}.
	
More generally, the vanishing conditions \eqref{Eq.Theta-vanishing-condition} and \eqref{Eq.chi-vanishing-condition} may be substituted by $$\Theta(s)=s^{\delta_d+(d-1)/(d\mu)}\vartheta(s^2) \andf \chi(\sigma)=(\cos\sigma)^{1+1/(d\mu)}\kappa(\cos^2\sigma)$$ for some smooth functions $\vartheta$ and $\kappa$, valid for $s$ near the origin and for $\sigma$ near $\pi/2$. Such generalized vanishing conditions do not alter the core arguments of our proof. However, they modify the vanishing order of our self-similar solutions close to the symmetry axis. Different vanishing orders produce self-similar solutions with distinct quantitative properties, which demonstrates that our construction in fact yields a broad family of self-similar solutions. For simplicity, we restrict our attention to the nondegenerate vanishing conditions \eqref{Eq.Theta-vanishing-condition} and \eqref{Eq.chi-vanishing-condition}.
\end{remark}

The existence and uniqueness of solutions to the transport equation \eqref{Eq.Omega-rel-transport} are stated in the following lemma.

\begin{lemma}\label{Lem.transport-solve}
{\sl    For each $\psi\in\mathcal A^0$ and $\la>1$, there is a unique $\Omega=\Omega(r,z)\in C^1(\Pi_+)$ solving the linear transport equation \eqref{Eq.Omega-rel-transport} along with the initial data
    \begin{equation}\label{Eq.transport-initial-2}        \Omega\left(\sin\sigma, f_0(\sin\sigma)\right)=\Theta_{\la}(\sigma)(\sin\sigma)^{-\frac{d-1}{d\mu}}(\cos\sigma)^{-\frac1{d\mu}}\chi(\sigma)=:\Omega_{0,\la}(\sigma),\quad \forall\ \sigma\in(0,\pi/2).
    \end{equation}}
\end{lemma}

See Subsection \ref{Subsec.solve-transport} for the proof of Lemma \ref{Lem.transport-solve}. By assuming further $\psi\in\mathcal A_M\subset \mathcal A^0$ for some $M>1$, we can derive an upper  bound of $\Omega$. 

\begin{proposition}\label{Prop.F}
  {\sl   Let $M>1$ and $\psi\in \mathcal A_M\subset \mathcal A^0$, then there exist three constants $\lambda_0=\lambda_0(d,a,\mu,M)>1$,  $M'_0=M'_0(d,a)>1$ and $M_1=M_1(d,a,\mu, M)>1$, such that the unique solution $\Omega=\Omega(r,z)$ to \eqref{Eq.Omega-rel-transport} with the initial data $\Omega_{0,\la_0}$ given by \eqref{Eq.transport-initial-2} 
    satisfies $\Omega\in \mathcal B_{M'_0}$. Moreover, there holds
    \begin{equation}\label{Eq.Omega-decay}
    	\Omega(r,z)\leq M_1z|(r,z)|^{-\frac1\mu-1},\quad\forall\ (r,z)\in\Pi_+\ \ \text{with}\  |(r,z)|\geq 1.
    \end{equation}}
\end{proposition}

See Subsection \ref{Subsec.Proof-PropF} for the detailed proof of Proposition \ref{Prop.F}. The constant $\lambda_0>1$ is defined explicitly in \eqref{Eq.lambda-def}. The fact that $M'_0$ does not depend on $M$ is crucial for our purpose. Recall that we fix $d\in\N_{\geq 3}$ and $a$ satisfying \eqref{Eq.a-range}, hence $M'_0$ is also a fixed constant. From now on, we fix such a constant $M'_0$. 

Thanks to Proposition \ref{Prop.F}, for each $M>1$ and $\mu$ satisfying \eqref{Eq.mu-range}, we can define a nonlinear map 
\begin{equation} \label{S2eq4}
\mathcal F_{\mu,M}: \ \cA_M \to \cB_{M'_0} \with  \cF_{\mu, M}(\psi)=\Omega,  \end{equation}
so that $ \Omega$ is the unique solution of  \eqref{Eq.Omega-rel-transport} with the initial data \eqref{Eq.transport-initial-2}, which is constructed by Proposition \ref{Prop.F}. 

Our next goal is to define a map from $\Omega$ to $\psi$. In principle, given $\Omega\in \cB_{M'}$, we are going to solve the elliptic equation \eqref{Eq.psi-elliptic}. However, in order to show that $\psi$ belongs to $\cA_M$, we need to introduce a scaling in the equation \eqref{Eq.psi-elliptic}.

\begin{lemma}\label{Lem.Psi_0-bound}
{\sl	 Let $M'>1$ and $\Omega=\Omega(r,z)\in \cB_{M'}$. For each $X=(x,y)\in\R^{d+1}\times\R^3$, we define $\Psi_0(X)$ via \eqref{Eq.Psi0-def}.		
	Then 
 $
        \Psi_0\in C^1\bigl(\R^{d+4}\bigr)\cap C^2\bigl(\R^{d+4}\setminus\{|x|=0\}\bigr)$
    solves \eqref{S2eqPsi} whenever $\ |x|\cdot|y|\neq 0,$
    and there exists a constant $M''=M''(d,a,M')>1$ depending only on $d, a, M'$ such that
	\begin{equation}\label{Eq.Psi0-lower-upper-bound}
		\frac{1}{M''(1-\mu(d-2))}\langle X\rangle^{d-2-\frac1\mu}\leq \Psi_0(X)\leq \frac{M''}{1-\mu(d-2)}\langle X\rangle^{d-2-\frac1\mu},\quad\forall\ X\in\R^{d+4}.
	\end{equation}}
	\end{lemma}
	
In particular, if we denote $\mathfrak M(\Omega):=\Psi_0(X=0)$, then it follows from \eqref{Eq.Psi0-lower-upper-bound} that
	\begin{equation}\label{S2eq2}
		\frac{1}{M''(1-\mu(d-2))}\leq \mathfrak M(\Omega)\leq \frac{M''}{1-\mu(d-2)}.
	\end{equation}

Lemma \ref{Lem.Psi_0-bound} follows directly from Lemmas \ref{Lem.Psi0-upperbound} and \ref{Lem.Psi0-lowerbound}. By virtue of the definition of $\Psi_0$ and using also the rotational symmetry, there exists a function $\psi_0(r,z)\in C^1\left(\overline{\Pi_+}\right)\cap C^2(\Pi_+)$ such that for all $X=(x,y)\in \R^{d+1}\times \R^3$ with $r=|x|, z=|y|$, $\Psi_0(X)=\Psi_0(x,y)=\psi_0(r,z)$. Moreover,  $\psi_0$ solves \eqref{Eq.psi-elliptic}.

Next we consider a re-scaled elliptic equation.

\begin{proposition}\label{Prop.G}
  {\sl   Let $M'>1$ and $\Omega=\Omega(r,z)\in \cB_{M'}$. We consider the following elliptic equation:
    \begin{equation}\label{Eq.elliptic-scaling}
        \Bigl(\pa_r^2+\frac dr\pa_r+\pa_z^2+\frac2z\pa_z\Bigr)\psi(r,z)=-\frac{a}{\mathfrak M(\Omega)} r^{d-3}\frac{\Omega(r,z)}{z},\quad\forall\ (r,z)\in\Pi_+,
    \end{equation}
 where 
 $\mathfrak M(\Omega)$  is defined above in \eqref{S2eq2}. Then  \eqref{Eq.elliptic-scaling}  has a unique solution $$\psi(r,z)\in  C^1\left(\overline{\Pi_+}\right)\cap C^2(\Pi_+) \with \lim_{r+z\to+\infty}\psi(r,z)=0.$$  Moreover, there exists a constant $$\mu_0=\mu_0(d,a,M')\in\Bigl(\frac {(d-1)a_1}2+\frac1{2(d-2)}, \frac1{d-2}\Bigr)$$ depending only on $d,a, M'$ such that for all $\mu\in(\mu_0, 1/(d-2)),$ we have $\psi\in \cA_{(dM'')^2}$,  where $M''=M''(d,a,M')>1$ is determined by Lemma \ref{Lem.Psi_0-bound}.}
\end{proposition}

It is easy to observe from \eqref{Eq.psi-elliptic} and \eqref{Eq.elliptic-scaling} that $\psi=\frac a{\mathfrak M(\Omega)}\psi_0$. See Subsection \ref{Subsec.Proof-Prop-G} for the proof of Proposition \ref{Prop.G}.

Assume that $a$ satisfies \eqref{Eq.a-range} and $M'>1$. Let $\mu_0=\mu_0(d,a, M')$ be determined by Proposition \ref{Prop.G}, and let $M''=M''(d,a, M')>1$ be determined by Lemma \ref{Lem.Psi_0-bound}. For each $\mu\in(\mu_0, 1/(d-2))$, we define the map: 
\begin{equation} \label{S2eq3}
\cG_{\mu, M'}: \ \cB_{M'}\to \cA_{(dM'')^2} \with \cG_{\mu, M'}(\Omega)=\psi,
\end{equation}
so that   $\psi$ is the unique solution of the elliptic equation \eqref{Eq.elliptic-scaling},  which is constructed in Proposition \ref{Prop.G}. 

We now introduce the ultimate nonlinear mapping to which we apply Schauder’s fixed-point theorem.

\begin{definition}\label{Def.nonlinear-map}
   Let $M_0'=M_0'(d,a)>1$ and $M''_0=M''(d,a,M_0'(d,a))=M''_0(d,a)>1$ be given by Proposition \ref{Prop.F} and Lemma \ref{Lem.Psi_0-bound} respectively. Let $$\mu_0=\mu_0(d,a,M_0'(d,a))=\mu_0(d,a)\in\Bigl(\frac {(d-1)a_1}2+\frac1{2(d-2)}, \frac1{d-2}\Bigr)$$
    be given by Proposition \ref{Prop.G}. 
    We take $M_0=(dM'')^2$, which depends only on $d$ and $a$. For each $\mu\in (\mu_0, 1/(d-2))$, 
      let   $\cF_{\mu, M_0}$ 
   and   $ \cG_{\mu, M_0'}$ 
   be the maps defined respectively by \eqref{S2eq4} and \eqref{S2eq3}, 
  we define the  nonlinear map from $\mathcal A_{M_0}$ to itself via
    \begin{equation}\label{S2eq5}
        \mathcal T_{\mu}:=\cG_{\mu, M_0'}\circ\cF_{\mu, M_0}:  \mathcal A_{M_0}\to\mathcal A_{M_0}.
    \end{equation}
\end{definition}

The first purpose of this paper is to prove the following fixed-point statement.

\begin{theorem}[Existence of fixed-points]\label{Thm.fixed-point}
    {\sl  Let $\mu_0=\mu_0(d,a)$ and $M_0=M_0(d,a)>1$ be determined by Definition \ref{Def.nonlinear-map}. Then for all $\mu\in(\mu_0, 1/(d-2))$, the nonlinear map $\mathcal T_{\mu}$
    defined by \eqref{S2eq5}  has a fixed-point $\psi_*^{(\mu)}\in \mathcal A_{M_0}$. }
\end{theorem}

See Subsection \ref{Subsec.Proof-fixed-point} for the proof of Theorem \ref{Thm.fixed-point}. Indeed, the fixed-point $\psi_*^{(\mu)}$ depends on $d,a,\mu$.  For simplicity, we 
sometimes neglect  the dependence of $\psi_*$ on these parameters 
if it causes no confusion.

\begin{proposition}[Smoothness of fixed-points]\label{prop:smoothness-fixed-point}
   {\sl  Let $\psi_*\in\cA_{M_0}$ be the fixed-point constructed in Theorem \ref{Thm.fixed-point},
   we denote $\Omega_*:=\cF_{\mu, M_0}(\psi_*)\in \cB_{M'_0}$.   
Then $\psi_*\in C^\infty(D)$ and $\Omega_*\in C^\infty(D)$ for the domain $D$ being given by \eqref{S2eq1}.
    More precisely, near each point $(0, z_0)$ with $z_0>0$, there exists a smooth function $A$ such that $\Omega_*(r,z)=r^{\delta_d}A(r^2,z)$, where $\delta_d$ is given by \eqref{Eq.Theta-vanishing-condition} and $A(0, z_0)>0$; near each point $(r_0, 0)$ with $r_0>0$, there exists a smooth function $B$ with $B(r_0, 0)>0$ such that $\Omega_*(r, z)=zB(r, z^2)$.}
\end{proposition}

The proof of Proposition \ref{prop:smoothness-fixed-point} will be given in Subsection \ref{Subsec.smoothness}. By virtue of Proposition \ref{prop:smoothness-fixed-point}, we may extend $\Omega_*$ to the domain $r\geq 0$ and $z\in\R$ via the odd reflection $\Omega_*(r,-z)=-\Omega_*(r,z)$. In the subsequent proposition, we assume $\Omega_*$ is defined for all $r\geq 0$ and $z\in\R$. 

\begin{proposition}[H\"older regularity of the fixed point]\label{Prop.Holder-fixed-point}
	{\sl Let $\alpha_*:=d-2-1/(\mu+a)\in(0,1)$. Then under the assumptions of  Theorem \ref{Thm.fixed-point},
	we have  $$r^{d-2}\Omega_*\in C^{\alpha_*}\left(\{(r,z): r\geq0, z\in\R\}\right).$$ Moreover, for each \(1\leq i\leq d-1\), we define \(x_i r^{d-3}\Omega_*(r,z)\) to be zero on the axis \(\{r=0\}\), where \(r=(x_1^2+\cdots+x_{d-1}^2)^{1/2}\) and \(z=x_d\). Then there hold \begin{align*}
	x\mapsto x_i r^{d-3}\Omega_*(r,z)\in C^{\alpha_*}(\mathbb R^d)\andf\qquad \\
			\bigl\|r^{d-2}\Omega_*\bigr\|_{C^{\alpha_*}(\{r\geq 0\})}+\bigl\|x_i r^{d-3}\Omega_*\bigr\|_{C^{\alpha_*}(\R^d)}<+\infty.
	\end{align*}}
\end{proposition}

Proposition \ref{Prop.Holder-fixed-point} is a direct consequence of Lemmas \ref{Lem.Holder-rd-2-omega} and  \ref{lem:global-holder-cartesian-omega}.

\begin{remark}\label{Rmk.mu0-uniform}
It follows from our analysis that the constant $\mu_0=\mu_0(d,a)$ 
appearing in Definition \ref{Def.nonlinear-map} admits a uniform lower bound whenever $a$ lies in a compact subinterval of the range specified in \eqref{Eq.a-range}.
\end{remark} 

\subsection{Proof of Theorem \ref{Thm.profile}}\label{Subsec.proof-thm-profile}
With Theorem \ref{Thm.fixed-point}, Proposition \ref{prop:smoothness-fixed-point} and Proposition \ref{Prop.Holder-fixed-point}, we are in a position to complete the  proof of Theorem \ref{Thm.profile}.

\begin{proof}[Proof of Theorem \ref{Thm.profile}]
	Fix $d\in\mathbb N_{\geq3}$ and $0<\alpha<\alpha_d=1-2/d$. We first choose a non-degenerate compact interval of values of $a$ contained in the admissible range \eqref{Eq.a-range} such that
	\[
	\alpha<d-2-\frac1{(d-2)^{-1}+a}
	\]
	for every $a$ in this interval. This is possible because the right-hand side tends to $(d-2)/d=\alpha_d$ as $a\uparrow \bigl[(d-1)(d-2)\bigr]^{-1}$. By choosing this interval slightly below the endpoint, and then increasing $\mu_0$ if necessary while keeping $\mu_0<1/(d-2)$, we may assume that for every $a$ in this interval and every $\mu\in(\mu_0,1/(d-2))$,  there holds \eqref{Eq.mu-range}. Let $\alpha_*:=d-2-1/(\mu+a)$, then one has $\alpha<\alpha_*<1$. The constants in Propositions \ref{Prop.F} and \ref{Prop.G}, and hence in Definition \ref{Def.nonlinear-map} and Theorem \ref{Thm.fixed-point}, may be chosen uniformly for $a$ in this compact interval, since all the inequalities involved are strict and have a positive margin on the interval. In the rest of the proof, we shall fix one such $a$ and one $\mu\in\bigl(\mu_0,1/(d-2)\bigr)$.
	
Let $M_0=M_0(d,a)>1$ and $M'_0=M'_0(d,a)>1$ be determined by Definition \ref{Def.nonlinear-map}. By Theorem \ref{Thm.fixed-point}, the map $\mathcal T_\mu=\mathcal G_{\mu,M'_0}\circ \mathcal F_{\mu,M_0}:\mathcal A_{M_0}\to\mathcal A_{M_0}$ has a fixed point. We denote this fixed point by $\psi_*\in\mathcal A_{M_0}$, and set $\Omega_*:=\mathcal F_{\mu,M_0}(\psi_*)\in\mathcal B_{M'_0}$. Then it follows from \eqref{S2eq4} and \eqref{S2eq3} 	 that  $(\Omega_*, \psi_*)$ solves the coupled system of \eqref{Eq.Omega-rel-transport} and	 \eqref{Eq.elliptic-scaling}.	 


Since the transport equation \eqref{Eq.Omega-rel-transport} is homogeneous in $\Omega_*$, after setting $\Omega:=c_*\Omega_*$ and $\psi:=\psi_*$, where $c_*:=a/\mathfrak M(\Omega_*)$, we find
	\begin{align}\label{Eq.Omega-final-transport-proof-Thm-profile}
		&(\mu+\psi+z\partial_z\psi)r\partial_r\Omega+(\mu-(d-1)\psi-r\partial_r\psi)z\partial_z\Omega=-\Omega\quad\text{in }\Pi_+,\\	
		&\Bigl(\partial_r^2+\frac d r\partial_r+\partial_z^2+\frac2z\partial_z\Bigr)\psi=-r^{d-3}\frac{\Omega}{z}\quad\text{in }\Pi_+. \label{Eq.psi-final-elliptic-proof-Thm-profile}	\end{align}
	Thus, $(\psi,\Omega)$ solves the reduced elliptic-transport system \eqref{Eq.psi-elliptic}, \eqref{Eq.Omega-rel-transport} in $\Pi_+$.
	
	We define $U^r$ and $U^z$ via \eqref{Eq.profile-velocity-from-psi}. Let $\psi_{\rm s}=r^{d-1}z\psi$, then one has 
	\begin{align*}
	r^{d-2}U^r=\partial_z\psi_{\rm s} \andf r^{d-2}U^z=-\partial_r\psi_{\rm s}\Longrightarrow
	\partial_r(r^{d-2}U^r)+\partial_z(r^{d-2}U^z)=0. \end{align*}
	 Moreover,  by reversing the computation  from \eqref{Eq.psi-s-stream}--\eqref{Eq.psi-def} to \eqref{Eq.psi-elliptic},  we obtain
	  $$\partial_rU^z-\partial_zU^r=r^{d-2}\Omega. $$
	  Finally, as $\mu r+U^r=r(\mu+\psi+z\partial_z\psi)$ and $\mu z+U^z=z(\mu-(d-1)\psi-r\partial_r\psi)$, 
	  \eqref{Eq.Omega-final-transport-proof-Thm-profile} is exactly the first equation of \eqref{Eq.ss-eq-UrUz}. This shows that  $(U^r,U^z,\Omega)$  thus obtained solves \eqref{Eq.ss-eq-UrUz} in $\Pi_+$.
	
	By Proposition \ref{prop:smoothness-fixed-point}, $\psi\in C^\infty(D)$ and $\Omega\in C^\infty(D)$. More precisely, near each point $(0,z_0)$ with $z_0>0$, we have $\Omega(r,z)=r^{\delta_d}A(r^2,z)$ for some smooth function $A$ with $A(0,z_0)>0$, while near each point $(r_0,0)$ with $r_0>0$, we have $\Omega(r,z)=zB(r,z^2)$ for some smooth function $B$ with $B(r_0,0)>0$. We extend $\psi$ evenly in $z$ and $\Omega$ oddly in $z$, namely 
	$$\psi(r,z)=\psi(r,-z)  \andf \Omega(r, z)=-\Omega(r,-z) \quad \mbox{if}\ z<0.$$ The same smoothness statement then holds on $D_+$. Consequently,  $$U^r(r,-z)=U^r(r,z), \quad U^z(r,-z)=-U^z(r,z) \andf \Omega(r,-z)=-\Omega(r,z), $$
	so that the required symmetry \eqref{Eq.z-odd-symmetry} holds.
	
	We next verify the regularity properties (1) of Theorem \ref{Thm.profile}. As $\Omega=c_*\Omega_*$ and $c_*>0$, Proposition \ref{Prop.Holder-fixed-point} ensures $r^{d-2}\Omega\in C^{\alpha_*}(\{(r,z):r\geq0,\ z\in\mathbb R\})$. Due to $\alpha<\alpha_*$, we have $r^{d-2}\Omega\in C^\alpha(\{(r,z):r\geq0,\ z\in\mathbb R\})$. Let $\omega_{\rm s}:=\partial_rU^z-\partial_zU^r=r^{d-2}\Omega$. For $x=(x_1,\ldots,x_{d-1},z)\in\mathbb R^d$ and $r=(x_1^2+\cdots+x_{d-1}^2)^{1/2}$, the vorticity matrix of $\mathbf U_{\rm s}=U^re_r+U^ze_z$ satisfies $\omega_{{\rm s},id}=-\omega_{{\rm s},di}=x_ir^{d-3}\Omega$ for $1\leq i\leq d-1$, and all the other entries vanish. Proposition \ref{Prop.Holder-fixed-point} therefore implies $\bm\omega_{\rm s}\in C^{\alpha_*}(\mathbb R^d)\subset C^\alpha(\mathbb R^d)$. The smoothness of $\bm\omega_{\rm s}$ away from the origin follows from $\Omega\in C^\infty(D_+)$ and the axis structure above. Indeed, near $\{r=0,\ z\neq0\},$ one has $x_ir^{d-3}\Omega=x_ir^{d-3+\delta_d}A(r^2,z)$, which is smooth by the choice of $\delta_d$.
	
	The decay in part (2) of Theorem \ref{Thm.profile} follows directly from Proposition \ref{Prop.F}. Indeed, the sub-linear growth of $\mathbf U_{\rm s}$ is a direct  consequence of \eqref{Eq.profile-velocity-from-psi} and \eqref{S2eq8}. Moreover, \eqref{Eq.Omega-decay} implies $$
	|\Omega(r,z)|\lesssim |z||(r,z)|^{-1-1/\mu}\quad \mbox{ for}\ z>0 \ \text{with}\  |(r,z)|\geq1, $$
	and the odd extension gives the same estimate for all $z\in\mathbb R$. Therefore,
	$$|\omega_{\rm s}(r,z)|\lesssim r^{d-2}|z||(r,z)|^{-1-1/\mu}\quad \mbox{ for all}\ r\geq0, z\in\mathbb R \ \text{with}\  |(r,z)|\geq1. $$
	Since $\Omega_*\in\mathcal B_{M'_0}$ is positive in $\Pi_+$ and $c_*>0$, we also have $\Omega(r,z)>0$ for all $r>0$, $z>0$, hence $\omega_{\rm s}(r,z)>0$ in $\Pi_+.$ This proves  part (3) of Theorem \ref{Thm.profile}.
	
	It remains to prove the boundary non-degeneracy. We first deduce from  the expansion $\Omega(r,z)=r^{\delta_d}A(r^2,z)$ near $(0,z_0)$ and $A(0,z_0)>0$ that  
	$$\omega_{\rm s}(r,z)=r^{d-2+\delta_d}A(r^2,z)\ \ \mbox{ near}\ \ (0,z_0). $$
We now  take $m_1=d-2+\delta_d.$ Then for every $z>0$, $$\partial_r^j\omega_{\rm s}(0,z)=0\ \mbox{ for}\ 0\leq j\leq m_1-1,\ \mbox{ while}\  \partial_r^{m_1}\omega_{\rm s}(0,z)=m_1!A(0,z)\neq0. $$ Similarly, it follows from $\Omega(r,z)=zB(r,z^2)$ near $(r_0,0)$ and $B(r_0,0)>0$ that 
	 $$\omega_{\rm s}(r,z)=r^{d-2}zB(r,z^2)\ \ \mbox{near} \ \ (r_0,0).$$
One may take $m_2=1.$ Then for every $r>0$, $$\omega_{\rm s}(r,0)=0 \andf \partial_z\omega_{\rm s}(r,0)=r^{d-2}B(r,0)\neq0. $$
This proves part (4) of Theorem \ref{Thm.profile}.
	
	We finally prove the regularity of the velocity field. Away from the origin,  this is an immediate consequence of the smoothness of $\psi$ and $\mathbf U_{\rm s}(x)=U^r(r,z)e_r+U^z(r,z)e_z$. Near $\{r=0,\ z\neq0\}$, the lifted smoothness of $\psi$ implies that $\psi$ is a smooth function of $(r^2,z)$, and hence the Cartesian components $x_i(\psi+z\partial_z\psi)$ and $-z((d-1)\psi+r\partial_r\psi)$ are smooth. Near $\{r>0,\ z=0\}$, the even extension in $z$ yields the corresponding smoothness across the plane $z=0$. Thus, $\mathbf U_{\rm s}\in C^\infty(\mathbb R^d\setminus\{0\})$. 
	For every $R>0$, the bounds following from $\psi\in\mathcal A_{M_0}$ imply that $\mathbf U_{\rm s}$ is bounded on $B_{2R}$. Since $\operatorname{div}\mathbf U_{\rm s}=0$ and
	$\bm\omega_{\rm s}\in C^\alpha(\mathbb R^d)$, the local Schauder estimate for the div--curl system gives
	\[\|\mathbf U_{\rm s}\|_{C^{1,\alpha}(B_R)}\leq C_R\left(\|\mathbf U_{\rm s}\|_{L^\infty(B_{2R})}+\|\bm\omega_{\rm s}\|_{C^\alpha(B_{2R})}\right).\]
	Consequently,
	\[\mathbf U_{\rm s}\in C^{1,\alpha}_{\rm loc}(\mathbb R^d;\mathbb R^d).\]
	
	
	To prove the last assertion, we let $a$ vary over the compact interval chosen at the beginning of the proof. For each such $a$, the preceding construction gives at least one profile. If two choices of $a$ lead to the same profile, then the corresponding functions $\psi$   have the same value at the origin. This is impossible because the construction gives $\psi(0,0)=a$. Equivalently,  it follows from \eqref{Eq.profile-velocity-from-psi} and the bounds defining $\mathcal A_{M_0}$ that
	$$U^r(r,z)/r\to a \andf U^z(r,z)/z\to -(d-1)a \ \ \quad\mbox{ as}\ \ (r,z)\to(0,0)\ \mbox{ in} \ \Pi_+, $$
	so $a$ is encoded in the linear part of the velocity profile at the origin. Therefore, different values of $a$ give rise to distinct self-similar profiles. Since the interval of admissible $a$ is uncountable, there are uncountably many self-similar profiles for each fixed $d,\alpha$ and $\mu\in(\mu_0,1/(d-2))$. This completes the proof of Theorem \ref{Thm.profile}.
\end{proof}

\subsection{Existence of fixed-points}\label{Subsec.Proof-fixed-point} 
This subsection outlines the fixed-point strategy employed to prove Theorem \ref{Thm.fixed-point}. We invoke the following formulation of Schauder’s fixed-point theorem for Fr\'echet spaces, also commonly referred to as the Schauder–Tychonoff fixed-point theorem. Readers may consult \cite[\textsection 7, Theorem 1.13, p.148]{GD2003} for the more general statement valid within locally convex topological vector spaces.

\begin{theorem*}[Schauder fixed-point theorem in Fr\'echet spaces]
{\sl	Let $\mathcal X$ be a Fr\'echet space and  $\mathcal C\subset \cX$ be a nonempty, closed and convex subset. 	Assume that $\mathcal T:\mathcal C\to \mathcal C$ is continuous and compact, in the sense that $\mathcal T(\mathcal C)$ is relatively compact in $\cX$. Then $\mathcal T$ has a fixed point in $\cC$.}
\end{theorem*}

In our setting, we take
\[\cX:=C^2_{\mathrm{loc}}(\Pi_+)\cap C_{\mathrm{loc}}(\overline{\Pi_+}),\]
equipped with the topology of local uniform convergence of the functions up to $\overline{\Pi_+}$ and local uniform convergence of first and second derivatives inside $\Pi_+$. Equivalently, for each $j\in\N_+$, we define
\[K_j:=[1/j,j]^2\subset\Pi_+,\quad L_j:=[0,j]^2\subset\overline{\Pi_+},\]
and the semi-norm
\[p_j(\psi):=\sup_{L_j}|\psi|+\sup_{K_j}|\nabla\psi|+\sup_{K_j}|\nabla^2\psi|,\quad\forall\ \psi\in\cX.\]
This makes $\cX$ a Fr\'echet space.

Let $M_0=M_0(d,a)>1$ be as specified in Definition \ref{Def.nonlinear-map}. We view $\mathcal A_{M_0}$ (see \eqref{defam}) as a subset of $\cX$. The pointwise upper and lower bounds, the normalization condition $\psi(0,0)=a$, and the differential inequalities entering the definition of $\mathcal A_{M_0}$ are preserved under convergence in  $\cX$. In addition, these constraints are convex. Consequently, $\mathcal A_{M_0}$ forms a closed, convex subset of $\cX$ and is evidently nonempty. Thus, if 
the map $\mathcal T_\mu$ defined in \eqref{S2eq5}
is continuous and compact with respect to the topology of  $\cX$, Theorem \ref{Thm.fixed-point} follows immediately from Schauder’s fixed-point theorem.

The continuity and compactness of $\mathcal T_\mu$ will be deduced from the corresponding properties of its two component mappings. These properties are established in Section \ref{Sec.existence-fixedpoint}.

\begin{lemma}[Continuity of the transport map]\label{Lem.F-continuity}
	{\sl  Let $M>1$ and  $M_0'=M_0'(d,a)>1$ be the constant determined by Proposition \ref{Prop.F}. Then for each $\mu$ satisfying \eqref{Eq.mu-range}, the map $\cF_{\mu,M}$ defined by \eqref{S2eq4}	
	is continuous.  Precisely, 
	\begin{align*}
	\mbox{if}\quad \psi_n,\psi\in\mathcal A_M\andf \psi_n\to\psi\ &\mbox{ in}\ C^2_{\mathrm{loc}}(\Pi_+)\cap C_{\mathrm{loc}}\left(\overline{\Pi_+}\right)\\
	&\Longrightarrow \cF_{\mu,M}(\psi_n)\to \cF_{\mu,M}(\psi) \ \mbox{ in}\ C_{\mathrm{loc}}^1(\Pi_+). \end{align*}}
\end{lemma}

\begin{lemma}[Continuity and compactness of the elliptic map]\label{Lem.G-continuity}
	{\sl  Let $M'>1$,  let $M''=M''(d,a,M')>1$ and $\mu_0=\mu_0(d, a, M')$ be the constants determined by  Proposition \ref{Prop.G}. For each $\mu\in(\mu_0,1/(d-2))$, the map $\cG_{\mu,M'}$ defined by \eqref{S2eq3}	
	is continuous and compact. Precisely,
	\begin{align*}
	\mbox{if}\quad	 \Omega_n,\Omega\in\mathcal B_{M'} \andf
	&\Omega_n\to\Omega\ \mbox{ in}\  C_{\mathrm{loc}}^1(\Pi_+)\\
	&\Longrightarrow \cG_{\mu,M'}(\Omega_n)\to \cG_{\mu,M'}(\Omega)\ \mbox{ in}\   C^2_{\mathrm{loc}}(\Pi_+)\cap C_{\mathrm{loc}}\left(\overline{\Pi_+}\right). 
	\end{align*}
	Moreover, the set $\cG_{\mu,M'}(\mathcal B_{M'})$ is relatively compact in $C^2_{\mathrm{loc}}(\Pi_+)\cap C_{\mathrm{loc}}\left(\overline{\Pi_+}\right)$.}
\end{lemma}

Now we are ready to prove Theorem \ref{Thm.fixed-point}. 

\begin{proof}[Proof of Theorem \ref{Thm.fixed-point}]
	We apply Schauder's fixed-point theorem in the Fr\'echet space $\cX=C^2_{\mathrm{loc}}(\Pi_+)\cap C_{\mathrm{loc}}\left(\overline{\Pi_+}\right)$. As previously explained, $\mathcal A_{M_0}$ (see \eqref{defam})  is a nonempty, closed and convex subset of $\cX$. By Definition \ref{Def.nonlinear-map}, $\mathcal T_\mu=\cG_{\mu,M_0'}\circ \cF_{\mu,M_0}: \mathcal A_{M_0}\to\mathcal A_{M_0}$. 
Then it follows from	Lemmas \ref{Lem.F-continuity} and  \ref{Lem.G-continuity} that	$\mathcal T_\mu$ is continuous.  Furthermore, Lemma \ref{Lem.G-continuity} implies that $\cG_{\mu,M_0'}$ is compact, so that the composition	$\mathcal T_\mu=\cG_{\mu,M_0'}\circ \cF_{\mu,M_0}$ is compact as a map from $\mathcal A_{M_0}$ to $\cX$. Schauder's fixed-point theorem therefore gives rise to $\psi_*^{(\mu)}\in\mathcal A_{M_0}$ such that $\mathcal T_\mu(\psi_*^{(\mu)})=\psi_*^{(\mu)}$. 	This finishes the proof of Theorem \ref{Thm.fixed-point}.
\end{proof}

\subsection{Notations and conventions}\label{Subsec.notations} We gather below the notation and conventions adopted throughout the paper. For non-open sets such as $D$ and $D_+$ defined by \eqref{S2eq1}, $C^k(D)$ denotes the space of functions admitting  a $C^k$ extension to a neighborhood of every point in $D.$ The definition of $C^{k,\alpha}(D)$ follows analogously.

We employ Japanese brackets $\langle r,z\rangle:=(1+r^2+z^2)^{1/2}$ and $\langle z\rangle:=(1+z^2)^{1/2}$. When formulating the elliptic equation as a Poisson equation on $\mathbb R^{d+4}$, we set
\[
X=(x,y)\in\mathbb R^{d+1}\times\mathbb R^3,\qquad
Y=(\xi,\eta)\in\mathbb R^{d+1}\times\mathbb R^3,
\]
with $r=|x|$ and $z=|y|$. The variable $x$ introduced here should not be conflated with the spatial variable appearing in the original Euler equations. Unless specified otherwise, the subsequent analysis focuses on the reduced elliptic–transport system \eqref{Eq.psi-elliptic} and \eqref{Eq.Omega-rel-transport}. Accordingly, for 
 $X=(x,y)\in\mathbb R^{d+1}\times\mathbb R^3$, the symbol $x$ consistently refers to the first  $d+1$ components of $X$. We further define $\langle X\rangle=(1+|X|^2)^{1/2}$. 





Throughout this paper, $C>0$ and $c>0$ stand for positive constants whose values may vary from line to line. In Section \ref{Sec.transport} and Section \ref{Sec.elliptic}, unless indicated otherwise, these constants depend only on the spatial dimension $d$ and  the fixed parameter $a$ introduced in \eqref{Eq.a-range}.
Since  $a_1$ (see \eqref{Eq.a-range})
 is determined by $a$, this dependence automatically incorporates dependence on 
 $a_1$. The constants are independent of spatial variables, of functions belonging to the designated function classes, and of the parameter $\mu$, as long as $\mu$ lies within the admissible interval and is sufficiently close to the critical value $(d-2)^{-1}$. More precisely, estimates expressed with a constant $C=C(d,a)$ or using the notation $\lesssim_{d,a}$ hold uniformly for all $\mu\in(\mu_\sharp,(d-2)^{-1})$, where $\mu_\sharp=\mu_\sharp(d,a)<(d-2)^{-1}$ can be taken arbitrarily close to $(d-2)^{-1}$. During the proof, we may finitely increase this lower bound for $\mu$ without introducing new notation. Accordingly, the statement that $(d-2)^{-1}-\mu$ is sufficiently small always refers to $\mu$ belonging to such a possibly shrunk interval near $(d-2)^{-1}$. In particular, constants such as $\lambda_0=\lambda_0(d,a,\mu,M)$ and $M_1=M_1(d,a,\mu,M)$ 
 appearing in Proposition \ref{Prop.F} are not subject to this default uniformity convention.

In Section \ref{Sec.stability}, Sections \ref{Sec.existence-fixedpoint}–\ref{Sec.stability-proof}, and Appendix \ref{app:second-log-derivative}, most constants are permitted to depend on $\mu$, unless otherwise specified. If a constant relies on extra parameters, such dependence will be made explicit, for example by writing $C=C(d,a,M)$, $C=C(d,a,M')$, $C=C(d,a,\mu)$, $C=C(d,a,\mu,\delta)$, or adopting the corresponding notation $\lesssim_{d,a,M}$, $\lesssim_{d,a,M'}$, and so forth.  We write $A\lesssim B$ whenever $A\leq CB$ for some constant $C$ obeying the default uniformity rule described above, and $A\sim B$ if both $A\lesssim B$ and $B\lesssim A$ hold.

Given an open set $U$, the notation $K\Subset U$ means that $K$ is a compact subset of $U$. All statements of local convergence are interpreted to hold on every compact subset of the stated domain. Unless noted otherwise, all integrals over Euclidean spaces are taken with respect to the Lebesgue measure. Lastly, we define the commutator of two operators  $\cA$ and $\cB$ by 
 $[\cA;\cB]:=\cA\cB-\cB\cA$.

\section{A road-map of the proof: finite- codimensional stability}\label{Sec.stability}

This section explains how the self-similar profiles constructed in Section \ref{Sec.Road-map} yield the finite-energy blow-up result stated in Theorem \ref{Thm.finite-energy-blowup}. Our strategy consists in rewriting the Euler dynamics in renormalized variables, treating the self-similar profile as a steady state, and analyzing the perturbation equations centered at this equilibrium. The linearized dynamics is stable up to a finite-dimensional unstable subspace. We eliminate these unstable directions via finitely many corrections to the initial data, while nonlinear estimates combined with a bootstrap argument propagate convergence toward the self-similar profile. 

Below we present the framework for proving the finite-codimensional stability and will  postpone the proof  of all lemmas and propositions 
till Section \ref{Sec.stability-proof}.

Let $\mu_0=\mu_0(d,a)$, $M_0=M_0(d,a)>1$ and $M_0'=M_0'(d,a)>1$ be given by Definition \ref{Def.nonlinear-map}. We fix $\mu\in(\mu_0, 1/(d-2))$. 
We then deduce from Theorem \ref{Thm.fixed-point} that  the nonlinear map $\cT$ defined by \eqref{S2eq5}
has a fixed point $\psi_*=\psi_*(r,z)\in \cA_{M_0}.$  We denote $\Omega_*:=\mathcal F_{\mu, M_0}(\psi_*)\in \cB_{M_0'}$. 
Then  $(\psi_*, \Omega_*)$  solves the coupled system of \eqref{Eq.Omega-rel-transport} and	 \eqref{Eq.elliptic-scaling}.


The starting point is a dynamical re-scaling of the original Euler system. Let $L_1=L_1(s):\R\to \R_+$ and $L_2=L_2(s):\R\to \R_+$ be functions such that
\begin{equation}\label{Eq.L12-def}
	L_1'(s)/L_1(s)=1,\quad -L_2'(s)/L_2(s)=\mu,\quad\forall\ s\in\R.
\end{equation}
We introduce the following dynamical re-scaling:
\begin{equation}\label{Eq.dynamical-rescaling}
	\begin{aligned}
		&\omega_{\rm rel}(t,r,z)=L_1(s)\Omega\left(s, \frac{r}{L_2(s)}, \frac{z}{L_2(s)}\right),\\
		&u(t,r,z)=L_1(s)L_2(s)^{d-1}U[\Omega]\left(s, \frac{r}{L_2(s)}, \frac{z}{L_2(s)}\right).
	\end{aligned}
\end{equation}
Here $s=s(t)$ is a change of variable to be determined later, and $U[\Omega]$ is the linear operator defined by 
\begin{align}\label{Eq.U-operator}
	U[\Omega]=(U^r[\Omega], U^z[\Omega])=\left(r\big(\psi_0+z\pa_z\psi_0\big), -z\big((d-1)\psi_0+r\pa_r\psi_0\big)\right),
\end{align}
where $\psi_0$ solves the elliptic equation \eqref{Eq.psi-final-elliptic-proof-Thm-profile} with
$\lim_{|(r,z)|\to+\infty}\psi_0(r,z)=0.$

As in the previous section, we denote
\begin{equation}\label{Eq.M(Omega)-def}
	\mathfrak M(\Omega):=\psi_0(r=0, z=0).
\end{equation}
We remark  that $\mathfrak M(\Omega)$ is not always well-defined. Below we will choose appropriate function spaces for $\Omega$ such that $\mathfrak M(\Omega)$ is well-defined.

By plugging \eqref{Eq.dynamical-rescaling} into \eqref{Eq.Euler-relative-vorticity}
and using \eqref{Eq.L12-def}, we find 
\begin{equation*}
	\pa_s\Omega+\mu\left(r\pa_r+z\pa_z\right)\Omega+\frac{\mathrm dt}{\mathrm ds}L_1(s)L_2(s)^{d-2}U[\Omega]\cdot\nabla\Omega+\Omega=0.
\end{equation*}
We now define the change-of-variable formula $s=s(t)$ by letting
\begin{equation}\label{S0eq3}
	\frac{\mathrm dt}{\mathrm ds}L_1(s)L_2(s)^{d-2}=\frac{a}{\mathfrak M\big(\Omega(s)\big)}\quad\Longrightarrow\quad \frac{\mathrm dt}{\mathrm ds}=\frac{a}{\mathfrak M\big(\Omega(s)\big)L_1(s)L_2(s)^{d-2}}.
\end{equation}
Hence, $\Omega=\Omega(s,r,z)$ solves
\begin{equation}\label{Eq.Omega-dynamical-equation}
	\pa_s\Omega+\mu\left(r\pa_r+z\pa_z\right)\Omega+\frac{a}{\mathfrak M\big(\Omega(s)\big)}U[\Omega]\cdot\nabla\Omega+\Omega=0.
\end{equation}
Notice that $\Omega_*=\Omega_*(r,z)$ is a steady solution of \eqref{Eq.Omega-dynamical-equation}. We aim at proving the global (finite codimensional) stability of this steady solution.

In order to do so, we write $\Omega(s,r,z)=\Omega_*(r,z)+G(s,r,z).$ Then by \eqref{Eq.Omega-dynamical-equation}, $G$ solves the following perturbation equation:
\begin{equation}\label{Eq.G-eq-detailed}
	\begin{aligned}
		\pa_sG+\mu(r\pa_r+z\pa_z)G&+a\frac{U[\Omega]}{\mathfrak M(\Omega)}\cdot\nabla_{r,z}G+G\\
		&+a\frac{U[G]\mathfrak{M}(\Omega_*)-\mathfrak{M}(G)U[\Omega_*]}{\mathfrak{M}(\Omega_*)\cdot\mathfrak{M}(\Omega)}\cdot\nabla_{r,z}\Omega_*=0.
	\end{aligned}
\end{equation}
In short, we write
\begin{equation}\label{Eq.G-eq-short}
	\pa_sG+\mathscr L[G]=\mathscr N[G],
\end{equation}
where $\mathscr L$ and  $\mathscr N$ denote the linear and the nonlinear parts respectively,
\begin{align}
	&\mathscr L[G]:=\mathscr L_{\rm tr}[G]+\mathscr K[G]\with\nonumber\\
	&\qquad\mathscr L_{\rm tr}[G]:=		
	\mu(r\pa_r+z\pa_z)G+a\frac{U[\Omega_*]}{\mathfrak M(\Omega_*)}\cdot\nabla_{r,z}G+G\andf \label{Eq.mathscrL-def}\\
	&\qquad\mathscr K[G]:=a\frac{U[G]\mathfrak{M}(\Omega_*)-\mathfrak{M}(G)U[\Omega_*]}{\left(\mathfrak{M}(\Omega_*)\right)^2}\cdot\nabla_{r,z}\Omega_*,\nonumber\\
	&\mathscr N[G]:=-a\frac{U[G]\mathfrak{M}(\Omega_*)-\mathfrak{M}(G)U[\Omega_*]}{\mathfrak{M}(\Omega_*)\cdot\mathfrak{M}(\Omega)}\cdot\nabla_{r,z}G\label{Eq.mathscrN-def}\\
	&\qquad\qquad+a\frac{\mathfrak{M}(G)\left(U[G]\mathfrak{M}(\Omega_*)-\mathfrak{M}(G)U[\Omega_*]\right)}{\left(\mathfrak{M}(\Omega_*)\right)^2\mathfrak{M}(\Omega)}\cdot\nabla_{r,z}\Omega_*.\nonumber
\end{align}


\subsection{Finite co-dimensional stability of the linear operator}\label{Subsec.linear-sketch}
We first analyze the linearized operator $\mathscr L$. Its leading part is a transport operator along the characteristic field generated by the steady profile, while the remaining nonlocal part is compact in the weighted topology. This decomposition gives exponential decay for the essential linear dynamics and leaves only finitely many possible unstable eigenmodes.

The principal part of the linear operator \eqref{Eq.mathscrL-def} is the transport operator $\mathscr L_{\rm tr}[G]$  \footnote{Here the subscript ``tr'' stands for ``transport''.}  defined by \eqref{Eq.mathscrL-def}, 
and $\mathscr L-\mathscr L_{\rm tr}=\mathscr K$ is a compact perturbation. We next present the appropriate weighted $L^\infty$ functional framework.

Let $\psi_*\in \cA_{M_0}$ be the fixed point constructed in Theorem \ref{Thm.fixed-point}. We define the weight function
\begin{equation}\label{Eq.w*-def}
	w_*(r,z):=\sqrt{r^2+\left(\mu+\psi_*(r,z)\right)^2z^2},\quad\forall \ (r,z)\in\Pi_+,
\end{equation}
and the weighted $L^\infty$ space
\begin{equation}\label{Eq.mathscrX0-def}
	\mathscr X_{0,\delta}:=\bigl\{G\in C(\Pi_+): \|G\|_{\mathscr X_{0,\delta}}<+\infty\bigr\},\with \|G\|_{\mathscr X_{0,\delta}}:=\Bigl\|w_*^{-\delta}\frac{G}{\Omega_*}\Bigr\|_{L^\infty(\Pi_+)}.
\end{equation}
Here $\delta>0$ is a sufficiently small parameter, which will be further restricted later. To ease notation, we introduce the vector field
\begin{equation}\label{Eq.Q_*-def}
	Q_*=Q_*(r,z):=(\mu r, \mu z)+a\frac{U[\Omega_*]}{\mathfrak{M}(\Omega_*)}\in \R^2,\quad\forall\ (r,z)\in\Pi_+.
\end{equation}

Similar to \cite{CLGSSS2025,Chen2026-2,EGM2022}, the following outgoing property plays a fundamental role in our stability proof.

\begin{lemma}\label{Lem.outgoing}
	{\sl There exists $\mu_0'=\mu_0'(d,a)\in \left(\mu_0, 1/(d-2)\right)$ such that for all $\mu\in \left(\mu_0', 1/(d-2)\right)$ we have
	\begin{equation}\label{S0eqY}
		Q_*(r,z)\cdot\nabla_{r,z}w_*(r,z)\geq \varkappa\, w_*(r,z),\quad\forall\ (r,z)\in\Pi_+,
	\end{equation}
	for some constant $\varkappa=\varkappa(d,a)\in(0,1)$.}
\end{lemma}


See Subsection \ref{Subsec.transport-semigroup} for the  detailed proof of Lemma \ref{Lem.outgoing}. By using this outgoing property, we can establish the exponential decay estimate of the semigroup generated by $\mathscr L_{\rm tr}$.

\begin{lemma}\label{Lem.L_tr-semigroup-est}
{	\sl Let $\mu_0'\in(\mu_0, 1/(d-2))$ be given by Lemma \ref{Lem.outgoing}. For any $\delta>0$, we have
	\begin{equation}
		\left\|\mathrm e^{-s\mathscr L_{\rm tr}}\right\|_{\mathscr X_{0,\delta}\to \mathscr X_{0,\delta}}\leq \mathrm e^{-\delta\varkappa\,s},\quad\forall\ s\geq 0.
	\end{equation}
Precisely, for any $\delta>0$ and $G_0\in \mathscr X_{0,\delta}$, the following equation:
\begin{align*}
\begin{cases}
\pa_sG+\mathscr L_{\rm tr}[G]=0,\\
G(0,\cdot)=G_0,
\end{cases}
\end{align*}
 has a unique $G=G(s,r,z)$ with $G(s,\cdot)\in \mathscr X_{0,\delta}$, which satisfies 	\begin{equation}\label{S0eqcL}
		\left\|G(s)\right\|_{\mathscr X_{0,\delta}}\leq\mathrm e^{-\delta\varkappa\,s} \left\|G_0\right\|_{\mathscr X_{0,\delta}},\quad\forall\ s\geq 0.
	\end{equation}}
\end{lemma}

See Subsection \ref{Subsec.transport-semigroup} for the proof of Lemma \ref{Lem.L_tr-semigroup-est}. 
We now turn to the linear operator $\mathscr K$ defined by \eqref{Eq.mathscrL-def}.

\begin{lemma}\label{Lem.K-compact}
{	\sl There exists $\mu_0''\in\left(\mu_0', 1/(d-2)\right)$ such that for any 
\begin{equation}\label{S0eq1} \mu\in\left(\mu_0'', 1/(d-2)\right) \andf \delta\in\left(0,1/\mu-(d-2)\right), 
\end{equation} the linear operator $\mathscr K: \mathscr X_{0,\delta}\to\mathscr X_{0,\delta}$ is bounded and compact.}
\end{lemma}

See Subsection \ref{Subsec.K-compactness} for the detailed proof of Lemma \ref{Lem.K-compact}.

Let $\mu,\delta$ satisfy \eqref{S0eq1} and $\mu_0''$ be given by Lemma \ref{Lem.K-compact}. 
Then it follows from  Lemma \ref{Lem.L_tr-semigroup-est} that
\begin{equation*}
	\sigma_{\rm ess}\left(-\mathscr L_{\rm tr}\right)\subset \bigl\{\la\in\C: \operatorname{Re}\la\leq -\delta\varkappa \bigr\},
\end{equation*}
from which, Lemma \ref{Lem.K-compact} and the fact that compact perturbations preserve the essential spectrum, we infer
\begin{equation*}
	\sigma_{\rm ess}\left(-\mathscr L\right)\subset \{\la\in\C: \operatorname{Re}\la\leq -\delta\varkappa\}.
\end{equation*}
Hence for every $0<\varkappa_0<\delta\varkappa/2$ the set
$$\sigma_{\rm u,\varkappa_0}:=\sigma\left(-\mathscr L\right)\cap \bigl\{\la\in\C: \operatorname{Re}\la> -1.1\varkappa_0\bigr\}$$
consists of finitely many eigenvalues of finite algebraic multiplicity. We choose a simple closed smooth curve $$\mathcal C_0\subset\rho(-\mathscr L):=\bigl\{\lambda\in\C\big| (\lambda+\mathscr L): \mathscr X_{0,\delta}\to\mathscr X_{0,\delta}\quad\text{has a bounded inverse}\bigr\}$$ in the complex plane enclosing $\sigma_{\rm u,\varkappa_0}$, 
and define the Riesz projection
\begin{equation*}
	P_{\rm u}:=\frac1{2\pi\mathrm i}\int_{\mathcal C_0}(\lambda+\mathscr L)^{-1}\,\mathrm d\lambda.
\end{equation*}
Then the unstable subspace
\begin{equation}\label{Eq.X-unstable-def}
	\mathscr X_{{\rm u} ,\delta}:=P_{\rm u}\left(\mathscr X_{0,\delta}\right)
\end{equation}
has a finite dimension. Moreover, for any $\mu,\delta$ satisfying \eqref{S0eq1}, 
we also have
\begin{equation}\label{Eq.L-semigroup-est}
	\bigl\|\mathrm e^{-s\mathscr L}(I-P_{\rm u})\bigr\|_{\mathscr X_{0,\delta}\to \mathscr X_{0,\delta}}\leq C\mathrm e^{-\varkappa_0s},\quad\forall\ s\geq0,
\end{equation}
where $C=C(d,a,\mu,\delta)>1$ and $\varkappa_0=\varkappa_0(d,a,\mu,\delta)\in(0,1)$. See \cite{EN2000,GGA2013} for more details.

\subsection{Estimates of the nonlinear operator}
The previous subsection reduces the stability problem to controlling the nonlinear source term $\mathscr N[G]$ defined by \eqref{Eq.mathscrN-def}. As $\mathscr N[G]$ contains both the perturbation itself and its first-order derivatives, the zeroth-order norm $\mathscr X_{0,\delta}$ must be supplemented by a first-order seminorm. 
Indeed, notice  from \eqref{Eq.mathscrN-def} that  $\mathscr N[G]$ contains  a term involving $\nabla_{r,z}G$. We introduce the following weighted $\dot W^{1,\infty}$ semi-norm.
\begin{equation}
		\|G\|_{\mathscr X_{1,\delta}}:=\Bigl\|w_*^{-\delta}\frac{r\pa_rG}{\Omega_*}\Bigr\|_{L^\infty(\Pi_+)}+\Bigl\|w_*^{-\delta}\frac{z\pa_zG}{\Omega_*}\Bigr\|_{L^\infty(\Pi_+)},\quad\forall\ G\in C^1(\Pi_+), \label{S0eqG1}
\end{equation}
where $w_*$ is defined in \eqref{Eq.w*-def} and $\delta>0$. We remark that $\|\cdot\|_{\mathscr X_{1,\delta}}$ is not a norm, although we still use the notation $\|\cdot\|$. 

\begin{proposition}\label{Prop.N[G]-est}
	{\sl Let $\mu_0''\in(0,1/(d-2))$ be given by Lemma \ref{Lem.K-compact}.  Then for $\mu$ and $\delta$ satisfying \eqref{S0eq1},  there exists $\varepsilon_0\in(0,1/2)$ such that
	\begin{equation*}
 \mbox{if}\ \|G\|_{\mathscr X_{0,\delta}}< \varepsilon_0 \andf \|G\|_{\mathscr X_{1,\delta}}<+\infty\Longrightarrow 		\left\|\mathscr N[G]\right\|_{\mathscr X_{0,\delta}}\leq C\left\|G\right\|_{\mathscr X_{0,\delta}}\bigl(\left\|G\right\|_{\mathscr X_{0,\delta}}+\left\|G\right\|_{\mathscr X_{1,\delta}}\bigr).
	\end{equation*}
	Here $C=C(d,a,\mu,\delta)>1$.}
\end{proposition}

See Subsection \ref{Subsec.nonlinear} for the  proof of Proposition \ref{Prop.N[G]-est}. 
To close the bootstrap in the next subsection, we also need to estimate $\|G\|_{\mathscr X_{1,\delta}}$.

\begin{proposition}\label{Prop.G-X1-est}
	{\sl Under the assumptions of Proposition \ref{Prop.N[G]-est}, there exists $\varepsilon_0'\in(0,\varepsilon_0)$ such that for any $C^1$ solution $G$ to \eqref{Eq.G-eq-detailed} on the time interval $[0, s_0)$ with $$\sup_{s\in[0,s_0)}\|G(s)\|_{\mathscr X_{0,\delta}}<\varepsilon_0',$$ we have 
	\begin{equation}\label{SGs}
		\begin{aligned}
			&\left\|G(s)\right\|_{\mathscr X_{1,\delta}}\leq C\mathrm e^{-\varkappa_1s}\|G(0)\|_{\mathscr X_{1,\delta}}\\
			&\quad+C\int_0^s\mathrm e^{-\varkappa_1(s-\tau)}\left[\left\|G(\tau)\right\|_{\mathscr X_{1,\delta}}\bigl(\left\|G(\tau)\right\|_{\mathscr X_{0,\delta}}+\left\|G(\tau)\right\|_{\mathscr X_{1,\delta}}\bigr)+\left\|G(\tau)\right\|_{\mathscr X_{0,\delta}}\right]\,\mathrm d\tau,
		\end{aligned}
	\end{equation}
	for all $s\in[0, s_0)$, where $\varkappa_1=\varkappa_1(d,a,\mu,\delta)\in(0,1)$ and $C=C(d,a,\mu,\delta)>1$.}
\end{proposition}

See Subsection \ref{Subsec.propagation-first-order-est} for the proof of Proposition \ref{Prop.G-X1-est}.

\subsection{Bootstrap and global stability}
Let $\mu_0''\in (0,1/(d-2))$ be determined  by Lemma \ref{Lem.K-compact},  and $\mu,$  $\delta$ satisfy \eqref{S0eq1}.
Let $\varepsilon_0'=\varepsilon_0'(d,a,\mu,\delta)\in (0,1/2)$ and $\varkappa_1=\varkappa_1(d,a,\mu,\delta)\in(0,1)$ be given by Proposition \ref{Prop.G-X1-est}. Up to decreasing $\varkappa_0$ in \eqref{Eq.L-semigroup-est}, which causes no loss, we may assume that $\varkappa_0<\varkappa_1$.

Recall from our linear analysis in Subsection \ref{Subsec.linear-sketch} that the ``unstable'' subspace $\mathscr X_{{\rm u}, \delta}\subset \mathscr X_{0,\delta}$, defined in \eqref{Eq.X-unstable-def}, is finite dimensional, and every eigenvalue $\lambda_i$ of $-\mathscr L|_{\mathscr X_{{\rm u}, \delta}}$ satisfies $\operatorname{Re}\la_i> -1.1\varkappa_0$. 
As in \cite{BCLGS2025}, there exists a basis $e_1,\cdots,e_N$ such that the matrix of $-\mathscr L|_{\mathscr X_{{\rm u}, \delta}}$ in this basis is 
$$ \left[\begin{array}{cccc}
	J_1 &  &  &  \\
	& J_2 &  &  \\
	&  & \ddots &  \\
	&  &  & J_{\widetilde{n}} 
\end{array}\right]
 \with  J_i=\left[\begin{array}{cccc}
	\lambda_i & \varkappa_0/10 &  &  \\
	& \lambda_i &  \ddots&  \\
	&  & \ddots &  \varkappa_0/10\\
	&  &  & \lambda_i 
\end{array}\right].
$$ 
Hence, if we choose an equivalent finite-dimensional Hilbert norm $|\cdot|_{\mathscr H}$ on $\mathscr X_{{\rm u}, \delta}$ such that $e_1,\cdots,e_N$ are orthonormal, then
\begin{equation}\label{L1}
	\langle -\mathscr LG, G\rangle_{\mathscr H}\geq -1.2\varkappa_0|G|_{\mathscr H}^2,\quad\forall\ G\in\mathscr X_{{\rm u}, \delta}.
\end{equation}

\begin{proposition}[Bootstrap]\label{Prop.bootstrap}
	{\sl There exist three small constants $\eta_0=\eta_0(d,a,\mu,\delta)\in(0,1/2)$, $\eta_{\rm u}=\eta_{\rm u}(d,a,\mu,\delta)\in(0,1/2)$, and $\varepsilon_0''=\varepsilon_0''(d,a,\mu,\delta)\in(0,\varepsilon_0')$ such that 
	 if $G$ is a $C^1$ solution to \eqref{Eq.G-eq-detailed} on the time interval $[0, s_0]$, with $\|G(0)\|_{\mathscr X_{0,\delta}}+\|G(0)\|_{\mathscr X_{1,\delta}}\leq \eta_0\varepsilon_0''$, and
	\begin{align}\label{b1}
		&\|G(s)\|_{\mathscr X_{0,\delta}}+\|G(s)\|_{\mathscr X_{1,\delta}}\leq \varepsilon_0''\,\mathrm e^{-\varkappa_0s},\\
		&\label{b2}|P_{\rm u}G(s)|_{\mathscr H}\leq \eta_{\rm u}\varepsilon_0''\,\mathrm e^{-2\varkappa_0s},\quad \ \forall\ s\in[0, s_0],	\end{align}
then we have
	\begin{align}\label{SGXd}
		\|G(s)\|_{\mathscr X_{0,\delta}}+\|G(s)\|_{\mathscr X_{1,\delta}}\leq \frac{\varepsilon_0''}2\,\mathrm e^{-\varkappa_0s},\quad \ \forall\ s\in[0, s_0].
	\end{align}
	}
\end{proposition}

See Subsection \ref{Subsec.global-stability-proof} for the detailed proof of Proposition \ref{Prop.bootstrap}. It remains to verify that the unstable coordinates can be assigned via smooth, compactly supported perturbations localized away from the origin. This is precisely the purpose of the correction directions introduced below.

\begin{lemma}[Choice of correction directions]\label{Lem.correction-directions}
	{\sl Let $P_{\rm u}$ and $\mathscr X_{{\rm u},\delta}$ be defined by \eqref{Eq.X-unstable-def}, and $N:=\dim\mathscr X_{{\rm u},\delta}$. Let $\mathscr D$ be the class of all $\Phi\in C^1(\Pi_+)$ such that $\|\Phi\|_{\mathscr X_{0,\delta}}+\|\Phi\|_{\mathscr X_{1,\delta}}<+\infty$, and whose lifted odd extensions\footnote{For a profile \(\Phi=\Phi(r,z)\) on \(\Pi_+\), we first take the odd extension \(\Phi^{\rm odd}(r,z):=\operatorname{sgn}(z)\Phi(r,|z|)\). Its lifted odd extension is the axisymmetric vorticity tensor on \(\mathbb R^d\) defined by \[\omega^\Phi_{id}=-\omega^\Phi_{di}:=x_i r^{d-3}\Phi^{\rm odd}(r,x_d),\qquad 1\le i\le d-1,\] with all other components zero, where \(r=|x'|\). Thus the condition means that \(\omega^\Phi\in C_c^\infty(\mathbb R^d\setminus\{0\})\).} are smooth and compactly supported away from the origin. Then there exist $\Phi_1,\ldots,\Phi_N\in\mathscr D$ so that $P_{\rm u}\Phi_1,\ldots,P_{\rm u}\Phi_N$ constitute a basis of $\mathscr X_{{\rm u},\delta}$. 
	
Furthermore, there exists a linear map $\mathcal I:\mathscr X_{{\rm u},\delta}\to \operatorname{span}\{\Phi_1,\ldots,\Phi_N\}$ such that
\begin{equation}\label{S8eq19}
	\begin{split}
	&\qquad P_{\rm u}\mathcal I h=h,\qquad \forall\ h\in\mathscr X_{{\rm u},\delta}\andf\\
	&\|\mathcal I h\|_{\mathscr X_{0,\delta}}+\|\mathcal I h\|_{\mathscr X_{1,\delta}}\leq C_{\mathcal I}|h|_{\mathscr H},\qquad \forall\ h\in\mathscr X_{{\rm u},\delta}
	\end{split} \end{equation} for a constant $C_{\mathcal I}=C_{\mathcal I}(d,a,\mu,\delta)>1.$ }
\end{lemma}

See Subsection \ref{Subsec.global-stability-proof} for the detailed proof of Lemma \ref{Lem.correction-directions}.  Equipped with these correction directions, the unstable component can be eliminated by appropriately selecting finitely many parameters within the initial perturbation. The resulting codimension-$N$ family of initial data converges exponentially toward the steady profile.

\begin{theorem}[Finite co-dimensional stability]\label{Thm.global-stability} 
	{\sl Let $\mu_0''\in(0,1/(d-2))$ be given by Lemma \ref{Lem.K-compact}.  For any $\mu, \delta$ satisfying \eqref{S0eq1}, 
	let $\eta_0,\eta_{\rm u}$ and $\varepsilon_0''$ be the constants given by Proposition \ref{Prop.bootstrap}, after decreasing them if necessary depending only on $d,a,\mu,\delta$ and $\mathcal I$. 
	Then there exists $\varepsilon_{\rm tn}>0$ \footnote{Here the subscript ``tn'' stands for ``truncation''.} 
	 such that for any  $\Omega_{\rm tn}\in C^1(\Pi_+)$ satisfying
	 	\[ G_{\rm tn}:=\Omega_{\rm tn}-\Omega_*,\qquad \|G_{\rm tn}\|_{\mathscr X_{0,\delta}}+\|G_{\rm tn}\|_{\mathscr X_{1,\delta}}\leq \varepsilon_{\rm tn}, \] 
	there exists $h_{\rm tn}\in\mathscr X_{{\rm u},\delta}$ satisfying $|h_{\rm tn}|_{\mathscr H}\leq \eta_{\rm u}\varepsilon_0''$ so that \eqref{Eq.G-eq-detailed} with initial datum 
	\[G(0)=G_{\rm tn}+\mathcal I\bigl(h_{\rm tn}-P_{\rm u}G_{\rm tn}\bigr)\] 
	has a global solution $G$ which satisfies
	 \begin{equation}\label{S0eq2} \|G(s)\|_{\mathscr X_{0,\delta}}+\|G(s)\|_{\mathscr X_{1,\delta}}\leq \frac12\varepsilon_0''\,\mathrm e^{-\varkappa_0s}\andf |P_{\rm u}G(s)|_{\mathscr H}\leq \eta_{\rm u}\varepsilon_0''\,\mathrm e^{-2\varkappa_0s},\ \forall\ s\geq0. \end{equation}
	 In particular, the corresponding solution $\Omega(s)=\Omega_*+G(s)$ of \eqref{Eq.Omega-dynamical-equation} converges exponentially to $\Omega_*$ in the sense of \eqref{S0eq2}. }
\end{theorem}

See Subsection \ref{Subsec.global-stability-proof} for the detailed proof of Theorem \ref{Thm.global-stability}. 
We now proceed with the proof of Theorem \ref{Thm.finite-energy-blowup}. The core strategy consists of three steps: we first truncate the original self-similar profile far from the origin; we then implement a finite-dimensional correction to situate the truncated initial data on the stable manifold; lastly, we convert the solution back from renormalized variables to the original Euler variables.

\begin{proof}[Proof of Theorem \ref{Thm.finite-energy-blowup}] 
	Fix $d\in\N_{\geq3}$ and $0<\alpha<\alpha_d$. We choose $a$ in the admissible range \eqref{Eq.a-range} and then choose $\mu$ sufficiently close to $1/(d-2)$ so that the profile constructed in Theorem \ref{Thm.profile} has H\"older exponent strictly larger than $\alpha$. Equivalently, after setting $\gamma_*=\mu/(1-(d-2)\mu)$, this corresponds to choosing $\gamma_*$ sufficiently large. Let $\Omega_*$ and $\psi_*$ be the fixed profile presented at the beginning of Section \ref{Sec.stability}. 
	We first produce an admissible finite-energy truncation at the level of the relative vorticity. Let $\Xi\in C_c^\infty([0,+\infty))$ satisfy $\Xi=1$ on $[0,1]$ and $\Xi=0$ on $[2,+\infty)$. For $R\gg1$, set
	\[\Xi_R(r,z):=\Xi\Bigl(\frac{\sqrt{r^2+z^2}}{R}\Bigr)\andf \Omega_{{\rm tn},R}:=\Xi_R\Omega_* .\]
	Then $G_{{\rm tn},R}:=\Omega_{{\rm tn},R}-\Omega_*=(\Xi_R-1)\Omega_*$. As $\Xi_R\equiv1$ near the origin, the local non-degenerate structure of $\Omega_*$ near $r=0$ and $z=0$ is unchanged. Moreover, using $w_*(r,z)\simeq \sqrt{r^2+z^2}$, the derivative bounds for $\Omega_*$, and the bounds
	\[|r\partial_r\Xi_R|+|z\partial_z\Xi_R|\lesssim 1,\]
	we obtain
	\[\|G_{{\rm tn},R}\|_{\mathscr X_{0,\delta}}+\|G_{{\rm tn},R}\|_{\mathscr X_{1,\delta}}\lesssim R^{-\delta}\to0,\qquad R\to+\infty .\]
	Thus, for $R$ sufficiently large, Theorem \ref{Thm.global-stability} applies to $\Omega_{\rm tn}=\Omega_{{\rm tn},R}$. We thus obtain $h_{\rm tn}\in\mathscr X_{{\rm u},\delta}$ and a global solution $\Omega(s)=\Omega_*+G(s)$ to \eqref{Eq.Omega-dynamical-equation} with \[ \|G(s)\|_{\mathscr X_{0,\delta}}+\|G(s)\|_{\mathscr X_{1,\delta}}\leq C\mathrm e^{-\varkappa_0s},\qquad s\geq0. \] 
	The initial relative vorticity $\Omega(0)=\Omega_{{\rm tn},R}+\mathcal I(h_{\rm tn}-P_{\rm u}G_{\rm tn})$ is still compactly supported, as the fixed correction directions $\Phi_j$ are compactly supported. It also has the same local structure near the origin as $\Omega_*$, since all correction directions are supported away from the origin. Therefore the corresponding initial velocity belongs to $C^{1,\alpha}(\R^d)\cap C^\infty(\R^d\setminus\{0\})\cap L^2(\R^d)$.	
	
	We now return to the original Euler variables. We take $L_1(s)=\mathrm e^s$ and $L_2(s)=\mathrm e^{-\mu s}$, and define $t=t(s)$ via the ODE \eqref{S0eq3} with initial data
	$t(0)=0.$
	Since $\mathfrak M(\Omega(s))=\mathfrak M(\Omega_*)+O(\mathrm e^{-\varkappa_0s})$ and $\mathfrak M(\Omega_*)>0$, we deduce from \eqref{S0eq3} that
	 \[ T:=\lim_{s\to+\infty}t(s)<+\infty,\qquad c\mathrm e^{-(1-(d-2)\mu)s}\leq T-t(s)\leq C\mathrm e^{-(1-(d-2)\mu)s}. \] 
	 Then with $\lambda(t):=L_2(s(t))=\mathrm e^{-\mu s(t)}$, we obtain \eqref{S0eq6}.
	 Finally we define $u$ by the dynamical rescaling \eqref{Eq.dynamical-rescaling}. By the derivation of \eqref{Eq.Omega-dynamical-equation}, $u$ solves the original Euler equations on $[0,T)$.
	 
	  It remains only to record the blow-up rate. The rescaled convergence estimate gives $\Omega(s)\to\Omega_*$ locally uniformly. Since $r^{d-2}\Omega_*$ is not identically zero, there exists a point $(r_0,z_0)\in\Pi_+$ such that $r_0^{d-2}\Omega_*(r_0,z_0)>0,$  for any sufficiently large $s$, \[ |\omega(t(s),L_2(s)r_0,L_2(s)z_0)|\geq cL_1(s)L_2(s)^{d-2}. \] 
	  Yet  the weighted estimates and the pointwise bounds of the profile give $\|\bm\omega(t(s))\|_{L^\infty}\leq CL_1(s)L_2(s)^{d-2}$. Due to  $L_1(s)L_2(s)^{d-2}\sim (T-t(s))^{-1}$, we obtain \eqref{S0eq4}
	  
	  The convergence estimate in \eqref{S0eq5}  follows from the bound on $G(s)$ and the relation $\mathrm e^{-\varkappa_0s}\lesssim (T-t)^\nu$ with $\nu=\varkappa_0/(1-(d-2)\mu)$. This completes the proof of Theorem \ref{Thm.finite-energy-blowup}. 
\end{proof}

\section{Analysis of the transport equation}\label{Sec.transport}

In this section, we aim at  proving Proposition \ref{Prop.F}.  Our first objective is to introduce a change of variables that simplifies the transport field of equation \eqref{Eq.Omega-rel-transport} and makes the characteristic structure more transparent. We then derive a representation of solutions along the associated characteristics and establish the upper and lower bounds needed to show that the resulting solution belongs to $\cB_{M_0'}$. In this section (except for Lemma \ref{Lem.hat-Omega-decay}), the implicit constants are all independent of $\mu$.

\subsection{A coordinate transform and the characteristic geometry}\label{Subsec.solve-transport}
We begin by introducing a coordinate transform on $\Pi_+$ that simplifies the vector field in the transport equation \eqref{Eq.Omega-rel-transport}. Let $d \in \mathbb{N}_{\geq 3}$, and assume that the parameters $a$ and $\mu$ satisfy \eqref{Eq.a-range} and \eqref{Eq.mu-range}, respectively. Given any $\psi\in \cA^0$, we define
\begin{equation}\label{Eq.H_psi}
	H_{\psi}:\Pi_+\to \Pi_+,\quad H_{\psi}(r,z):=(R, Z):=\left(r, (\mu+\psi(r,z))z\right).
\end{equation}

The first step is to verify that $H_\psi$ is a bijection and that its inverse enjoys sufficient regularity. This will allow us to reformulate \eqref{Eq.Omega-rel-transport} to a simpler version.

\begin{lemma}\label{Lem.H_psi}
{\sl	For any $\psi\in\cA^0$, we define the map $H_{\psi}: \Pi_+\to \Pi_+$ via \eqref{Eq.H_psi}. Then $H_{\psi}$ is bijective and $H_{\psi}^{-1}\in C^2(\Pi_+)$. }
\end{lemma}
\begin{proof}
	For each $r>0$, we define $F_r(z):=(\mu+\psi(r,z))z$ for $z>0$, then 
	$$H_{\psi}(r,z)=(r, F_r(z)) \andf F_r'(z)=\mu+\psi(r,z)+z\pa_z\psi(r,z)\quad\mbox{ for} \ z>0. $$
	In view of \eqref{S2eq8}, due to $\psi\in \cA^0$ and $-1+\frac1{d\mu}<0$, for all $(r,z)\in\Pi_+,$ we have
	\begin{align}\label{Eq.z-pa-z-psi-est}
		|z\pa_z\psi(r,z)|&\leq \frac1{10}z\langle r,z\rangle^{-1+\frac1{d\mu}}\langle z\rangle^{-\frac1{d\mu}}\psi(r,z)\leq 
		\frac1{10}\psi(r,z),
	\end{align}
which together with the fact  $\psi>0$ (see \eqref{S2eq8}), ensures that	\begin{equation}\label{Eq.Fr'>0}
		F_r'(z)=\mu+\psi(r,z)+z\pa_z\psi(r,z)\geq \mu+\psi(r,z)-\frac1{10}\psi(r,z)=\mu+\frac 9{10}\psi(r,z)>\mu>0
	\end{equation} 
	for all $(r,z)\in\Pi_+.$ Moreover, for each $r>0$, we have 
	\[\lim_{z\to0+}F_r(z)=0,\quad \lim_{z\to+\infty}F_r(z)=+\infty.\]
	Hence, $F_r:(0,+\infty)\to (0,+\infty)$ is a strictly increasing bijection. Consequently, for any $(R, Z)\in\Pi_+$, the equation $Z=F_R(z)$ has a unique solution $z=\zeta_{\psi}(R, Z)>0$. This proves the bijectivity of $H_\psi: \Pi_+\to\Pi_+$ and  $H_\psi^{-1}(R, Z)=(R, \zeta_{\psi}(R, Z))$. 
	
	On the other hand, for any $\psi\in\cA^0$, we consider the map $$\tilde F_{\psi}(R, z, Z):=(\mu+\psi(R, z))z-Z, $$
	then $\zeta_\psi(R,Z)$ solves $\tilde F_{\psi}(R, \zeta_\psi(R,Z), Z)=0$. 
Moreover,  it follows from	\eqref{Eq.Fr'>0} that 	$$\pa_z\tilde F_{\psi}(R, z, Z)=\mu+\psi(R, z)+z\pa_z\psi(R,z)>\mu>0.$$ Then by $\psi\in\cA^0\subset C^2(\Pi_+)$ and the implicit function theorem, we have $\zeta_\psi\in C^2(\Pi_+)$.
\end{proof}

After passing to the new variables $(R, Z)$, we rewrite the transport equation \eqref{Eq.Omega-rel-transport} in a form that is better adapted to the characteristic method. In particular, the transformed formulation reveals the geometry of the flow more clearly and makes it possible to track the propagation of the initial data prescribed on the curve $\Gamma_0$ defined in \eqref{Eq.Gamma0-def}.
\begin{lemma}\label{Lem.transport-eq-transform}
{\sl	Let $\psi\in\cA^0$ and let $\Omega\in C^1(\Pi_+)$ be a solution of \eqref{Eq.Omega-rel-transport}. We denote  
	\begin{equation}\label{Eq.h-def}
\hat\Omega:=\Omega\circ H_{\psi}^{-1}, \andf		h:=\frac{\mu-(d-1)\psi\circ H_\psi^{-1}}{\mu+\psi\circ H_\psi^{-1}},
	\end{equation}
	then $\hat\Omega\in C^1(\Pi_+)$ solves the following transport equation:
	\begin{equation}\label{Eq.Omega-transport}
		R\pa_R\hat\Omega+hZ\pa_Z\hat\Omega=-\frac1{d\mu}\left(Z\pa_Zh+h+d-1\right)\hat\Omega\quad\mbox{for} \  (R, Z)\in\Pi_+.	\end{equation}
Conversely, if $\hat\Omega\in C^1(\Pi_+)$ solves \eqref{Eq.Omega-transport}, then $\Omega:=\hat\Omega\circ H_\psi\in C^1(\Pi_+)$ is a solution to \eqref{Eq.Omega-rel-transport}.}
\end{lemma}

\begin{proof}
	Assume that $\Omega\in C^1(\Pi_+)$ solves \eqref{Eq.Omega-rel-transport}. By Lemma \ref{Lem.H_psi}, $H_\psi$ is a $C^2$ diffeomorphism of $\Pi_+$ onto itself. Hence, $\widehat\Omega=\Omega\circ H_\psi^{-1}\in C^1(\Pi_+)$ and $h\in C^2(\Pi_+)$. In what follows, for a fixed point $(R,Z)\in\Pi_+$, we write $(r,z)=H_\psi^{-1}(R,Z)$ with $H_\psi$ being defined by \eqref{Eq.H_psi}. 
	
	For simplicity, we temporarily denote $$A(r,z):=\mu+\psi(r,z)+z\pa_z\psi(r,z) \andf B(r,z):=\mu-(d-1)\psi(r,z)-r\pa_r\psi(r,z). $$ Due to $\Omega(r,z)=\widehat\Omega(R,Z)$, the classical chain rule gives
	\[
	\pa_r\Omega=\pa_R\widehat\Omega+z\pa_r\psi\,\pa_Z\widehat\Omega,\andf
	\pa_z\Omega=A\,\pa_Z\widehat\Omega ,
	\]
	where the derivatives of $\widehat\Omega$ are taken at $(R,Z)$. Then due to 	$rz\pa_r\psi+Bz=z(\mu-(d-1)\psi)=hZ,$ we have	\[
	Ar\pa_r\Omega+Bz\pa_z\Omega
	=A\left(R\pa_R\widehat\Omega+\bigl(rz\pa_r\psi+Bz\bigr)\pa_Z\widehat\Omega\right)
	=A\left(R\pa_R\widehat\Omega+hZ\pa_Z\widehat\Omega\right).
	\]
 As $\Omega$ solves \eqref{Eq.Omega-rel-transport}, we obtain \begin{equation} \label{S3eq1} R\pa_R\widehat\Omega+hZ\pa_Z\widehat\Omega=-\frac1A\widehat\Omega. \end{equation}
 It remains to express $A^{-1}$ in terms of $h$. In order to do so, we denote $q(r,z):=\frac{\mu-(d-1)\psi(r,z)}{\mu+\psi(r,z)}$, so that $h=q\circ H_\psi^{-1}$. As a result, we find $\pa_z q=-\frac{d\mu\,\pa_z\psi}{(\mu+\psi)^2}$ and $\pa_z Z=A$, which implies
  $$Z\pa_Z h=-\frac{d\mu\,z\pa_z\psi}{(\mu+\psi)A} \andf h+d-1=\frac{d\mu}{\mu+\psi}. $$
   Thus,
	\[
	Z\pa_Z h+h+d-1
	=\frac{d\mu}{\mu+\psi}\Bigl(1-\frac{z\pa_z\psi}{A}\Bigr)
	=\frac{d\mu}{A},
	\]
	which ensures  $A^{-1}=\frac1{d\mu}(Z\pa_Zh+h+d-1)$, and we deduce \eqref{Eq.Omega-transport} from \eqref{S3eq1}.
	
	Conversely, if $\widehat\Omega\in C^1(\Pi_+)$ solves \eqref{Eq.Omega-transport} and $\Omega:=\widehat\Omega\circ H_\psi$, we get, by using the chain rule, that
	\begin{equation}\label{S3eq2}
	Ar\pa_r\Omega+Bz\pa_z\Omega
	=A\left(R\pa_R\widehat\Omega+hZ\pa_Z\widehat\Omega\right)\circ H_\psi,
	\end{equation}
	from which,  \eqref{Eq.Omega-transport} and the identity $A^{-1}=\frac1{d\mu}(Z\pa_Zh+h+d-1)\circ H_\psi$, then we deduce that the right-hand side of \eqref{S3eq2} is $-\widehat\Omega\circ H_\psi=-\Omega$. This gives  \eqref{Eq.Omega-rel-transport}. 
\end{proof}

For $\psi\in\cA^0$, we define the coordinate transform $H_\psi: \Pi_+\to\Pi_+$ by \eqref{Eq.H_psi}. Under this coordinate transform, the initial curve $\Gamma_0$ defined by \eqref{Eq.Gamma0-def}
becomes  the circle in $(R,Z)$ variables:
\begin{equation}\label{Eq.Gamma-def}
	\Gamma:=\bigl\{(R, Z)\in\R^2: R^2+Z^2=1, R>0, Z>0\bigr\}.
\end{equation}


\begin{lemma}\label{Lem.hat-Omega-expression}
{\sl	Let $\psi\in \cA^0,$  $\lambda>1$ and $\Om_{0,\la}$ be given by \eqref{Eq.transport-initial-2}. Then the transport equation \eqref{Eq.Omega-transport} with the initial data
	\begin{equation}\label{Eq.Omega-initial}
		\hat\Omega(\sin\sigma, \cos\sigma)=\Om_{0,\la}(\sigma),\quad\forall\ \sigma\in(0,\pi/2)
	\end{equation}
	admits a unique solution
	\begin{equation}\label{Eq.Omega-sol-expression}
		\hat\Omega(R,Z)=R^{-\frac{d-1}{d\mu}}Z^{-\frac1{d\mu}}\Theta_{\lambda}(\sigma_{\psi}(R, Z))\chi(\sigma_{\psi}(R, Z)) e^{J_\psi(R, Z)},\quad\forall\ (R, Z)\in \Pi_+
	\end{equation}
	for some $\sigma_{\psi}\in C^2(\Pi_+;(0,\pi/2))$ and $J_\psi\in C^1(\Pi_+; \mathbb{R})$. }
\end{lemma}
\begin{proof} We divide the proof into the following steps:\smallskip

\no{\bf Step 1.}  The existence of the characteristic curves.

	We are going to solve the (linear) transport equation \eqref{Eq.Omega-transport} by the method of characteristics. For each $\sigma\in(0,\pi/2)$, we define the characteristic curves $s\mapsto(R(s;\sigma), Z(s;\sigma))$ via
	\begin{equation}\label{Eq.characteristic}
		\begin{cases}
&\frac{\mathrm d}{\mathrm ds}R(s;\sigma)=R(s;\sigma),\\
& \frac{\mathrm d}{\mathrm ds}Z(s;\sigma)=h(R(s;\sigma), Z(s;\sigma))Z(s;\sigma),\\
			&R(0;\sigma)=\sin\sigma,\quad Z(0;\sigma)=\cos\sigma.
		\end{cases}
	\end{equation}
As $\psi\in\cA^0$, we deduce from Lemma \ref{Lem.H_psi} and  \eqref{Eq.h-def} that  $h\in C^2(\Pi_+)$; hence the ODE system \eqref{Eq.characteristic} has a unique local solution.
	
	We next prove that the characteristic curves are globally well-defined and give a $C^2$ parametrization of the whole quadrant.  We first observe from \eqref{Eq.a-range}, \eqref{Eq.mu-range},  \eqref{S2eq8} and \eqref{Eq.h-def}   that
\begin{equation}\label{Eq.h_*-def}
\begin{split}		h(R, Z)&=\frac{\mu-(d-1)\psi(r,z)}{\mu+\psi(r,z)}\geq\frac{\mu-(d-1)a_1}{\mu+a_1}>\frac12(\mu-(d-1)a_1)\\
		&>\frac14\Bigl(\frac1{d-2}-(d-1)a_1\Bigr)=\frac{d-1}{8}\Bigl(\frac1{(d-1)(d-2)}-a\Bigr)=:h_*\in(0,1),
		\end{split}
	\end{equation}
which implies
	\begin{equation}\label{Eq.h-bound}
		0<h_*<h(R,Z)<1,\qquad \forall\,(R,Z)\in\Pi_+.
	\end{equation}

	For each $\sigma\in(0,\pi/2)$, let $(R(s;\sigma),Z(s;\sigma))$ be the local solution to \eqref{Eq.characteristic}. Then we find
	\begin{align}\label{S3eq4}
		R(s;\sigma)=e^{s}\sin\sigma,\andf
	\frac{\mathrm d}{\mathrm ds}\ln Z(s;\sigma)=h(R(s;\sigma),Z(s;\sigma)).
	\end{align}
	Since $h$ is bounded, $Z(s;\sigma)$ can neither vanish nor blow up in finite time. Together with the explicit formula for $R(s;\sigma)$, this shows that the maximal existence interval of the solution is the whole real line. Hence, the map $(s,\sigma)\mapsto (R(s;\sigma),Z(s;\sigma))$ is well-defined on $\mathbb{R}\times(0,\pi/2)$. Since $h\in C^2(\Pi_+)$, the standard $C^2$ dependence of ODE solutions on the initial data implies that
	\begin{equation}\label{S3eq3}
	\cM:\ (s,\sigma)\mapsto (R(s;\sigma),Z(s;\sigma))\in C^2\big(\mathbb{R}\times(0,\pi/2);\Pi_+\big).\end{equation}
	
	\no{\bf Step 2.}  	
	The map $\cM$ defined by \eqref{S3eq3} is a bijection and a $C^2$ diffeomorphism. 
	
	For each $P=(R,Z)\in\Pi_+$, let $\Phi_\psi (s;P)$ be the flow generated by the vector field $V_\psi (R,Z):=(R,h(R,Z)Z)$, that is,
	\[\frac{\mathrm d}{\mathrm ds}\Phi_\psi (s;P)=V_\psi (\Phi_\psi (s;P)),\qquad \Phi_\psi (0;P)=P.\]
	We deduce by an argument similar to Step {\bf 1} that $s\mapsto \Phi_\psi (s;P)=(R(s;P), Z(s;P))$ is globally defined.  Furthermore, thanks to \eqref{Eq.h-bound},  we have
	\[\frac{\mathrm d}{\mathrm ds}|\Phi_\psi (s;P)|^2=2R(s; P)^2+2h(R(s;P),Z(s;P))Z(s;P)^2>0,\]
which implies that $s\mapsto |\Phi_\psi (s;P)|$ is strictly increasing. 

On the other hand,
	it follows from  \eqref{Eq.h-bound} that
	\[R(s;P)=e^{s}R,\qquad e^{h_* s}Z\le Z(s;P)\le e^{s}Z,\quad s\ge0,\]
	and the analogous bounds for $s\le0$. Consequently,
	\[\lim_{s\to-\infty}|\Phi_\psi (s;P)|=0,\qquad\lim_{s\to+\infty}|\Phi_\psi (s;P)|=+\infty.\]
	Therefore, for each $P\in\Pi_+$, there exists a unique number $s_\psi (P)\in\mathbb{R}$ such that $$|\Phi_\psi (-s_\psi (P);P)|=1.$$ For $\Gamma$  given by \eqref{Eq.Gamma-def},	we define
	\[Q_\psi (P):=\Phi_\psi (-s_\psi (P);P)\in\Gamma.
	\]
	Then there exists a unique $\sigma_\psi (P)\in(0,\pi/2)$ such that $Q_\psi (P)=(\sin\sigma_\psi (P),\cos\sigma_\psi (P))$. By construction, if $P=(R(s;\sigma),Z(s;\sigma))$, then $s_\psi (P)=s$ and $\sigma_\psi (P)=\sigma$. Hence, the inverse of the characteristic map is precisely $P\mapsto (s_\psi (P),\sigma_\psi (P))$. This proves the bijectivity of the map $\cM$ from $\mathbb{R}\times(0,\pi/2)$ onto $\Pi_+$.
	
	It remains to check that the inverse map is $C^2$. Since $h\in C^2(\Pi_+)$, the flow
	map $(s,P)\mapsto\Phi_\psi (s;P)$ is $C^2$ on $\mathbb{R}\times\Pi_+$. Consider $F(s,P):=|\Phi_\psi (s;P)|^2-1$, which belongs to $C^2(\mathbb{R}\times\Pi_+)$. At $s=-s_\psi (P)$, we have $F(-s_\psi (P),P)=0$ and
	\begin{equation}\label{Eq.partial-F-hittingtime}
		\partial_sF(-s_\psi (P),P)=2R_0^2+2h(R_0,Z_0)Z_0^2>0,
	\end{equation}
	where $(R_0,Z_0):=\Phi_\psi (-s_\psi (P);P)\in\Gamma$. Thus, the implicit function theorem ensures that $s_\psi \in C^2(\Pi_+)$. Consequently, $Q_\psi (P)=\Phi_\psi (-s_\psi (P);P)$ is also $C^2$ in $P$. Since $Q_\psi (P)\in\Gamma$, and the parametrization $\sigma\mapsto(\sin\sigma,\cos\sigma)$ is a $C^2$ diffeomorphism from $(0,\pi/2)$ onto $\Gamma$, we obtain $\sigma_\psi \in C^2(\Pi_+)$. Therefore, the inverse map $P\mapsto (s_\psi (P),\sigma_\psi (P))$ is $C^2$. We conclude that $\cM$
	is a $C^2$ diffeomorphism from $\mathbb{R}\times(0,\pi/2)$ onto $\Pi_+$.
	
	\smallskip
	
	\no{\bf Step 3.}  	 The derivation of 	
 \eqref{Eq.Omega-sol-expression}. 
 
 Since the map  $\cM$ defined by \eqref{S3eq3} is a $C^2$ diffeomorphism from $\mathbb{R}\times (0,\pi/2)$ onto $\Pi_+$, for each	$(R,Z)\in \Pi_+$, there exists a unique pair $(s_\psi(R,Z),\sigma_\psi(R,Z))\in \mathbb{R}\times (0,\pi/2)$ such that
	\begin{equation}\label{S3eq21} (R,Z)=\bigl(R(s_\psi(R,Z);\sigma_\psi(R,Z)),Z(s_\psi(R,Z);\sigma_\psi(R,Z))\bigr). 	\end{equation}	Moreover, by the $C^2$ regularity of the inverse map, we have $s_\psi,\sigma_\psi\in C^2(\Pi_+)$. 
	
	Along each characteristic curve, equation \eqref{Eq.Omega-transport} becomes
	\begin{align}\label{Eq.hat-Omega-along-curve}
		\frac{\mathrm d}{\mathrm ds}\widehat \Omega\big(R(s;\sigma), Z(s;\sigma)\big)=-\frac{1}{d\mu}
		\left[\big(Z\partial_Z h+h+d-1\big)\widehat \Omega\right]\big(R(s;\sigma), Z(s;\sigma)\big).
	\end{align}
	It is convenient to absorb the explicit powers of $R$ and $Z$. Set
	\[\mathcal{W}(s,\sigma):=R^{\frac{d-1}{d\mu}}(s;\sigma)Z^{\frac{1}{d\mu}}(s;\sigma)\widehat \Omega\big(R(s;\sigma), Z(s;\sigma)\big),\quad\forall\ (s,\sigma)\in \mathbb{R}\times(0,\pi/2).\]
	By virtue of \eqref{S3eq4},
	we infer
	\begin{align*}
		\frac{\mathrm d}{\mathrm ds}\ln \mathcal{W}(s,\sigma)&=\frac{d-1}{d\mu}+\frac{1}{d\mu}h\big(R(s;\sigma),Z(s;\sigma)\big)-\frac{1}{d\mu}\bigl(Z\partial_Z h+h+d-1\bigr)\big(R(s;\sigma),Z(s;\sigma)\big)\\
		&=-\frac{1}{d\mu}\bigl(Z\partial_Z h\bigr)\big(R(s;\sigma),Z(s;\sigma)\big),
	\end{align*}
	which implies
	\[\mathcal{W}(s,\sigma)=\mathcal{W}(0,\sigma)\exp\left\{-\frac{1}{d\mu}\int_0^s\bigl(Z\partial_Z h\bigr)\big(R(\tau;\sigma),Z(\tau;\sigma)\big)\,\mathrm d\tau\right\},\quad\forall\ (s,\sigma)\in\mathbb{R}\times(0,\pi/2).\]
	Yet it follows from  \eqref{Eq.Omega-initial} that 
	\[\mathcal{W}(0,\sigma)=(\sin\sigma)^{\frac{d-1}{d\mu}}(\cos\sigma)^{\frac{1}{d\mu}}\widehat\Omega(\sin\sigma,\cos\sigma)=\Theta_\lambda(\sigma)\chi(\sigma).\]
	We thus obtain
	\begin{equation}\label{Eq.cJ_psi-def}
	\begin{split}	
	\widehat \Omega\big(R(s;\sigma),Z(s;\sigma)\big)&=R(s;\sigma)^{-\frac{d-1}{d\mu}}Z(s;\sigma)^{-\frac{1}{d\mu}}\Theta_\lambda(\sigma)\chi(\sigma)\exp\{\mathcal{J}_\psi(s,\sigma)\},\with\\
		\mathcal{J}_\psi(s,\sigma)&:=-\frac{1}{d\mu}\int_0^s\bigl(Z\partial_Z h\bigr)\big(R(\tau;\sigma),Z(\tau;\sigma)\big)\,\mathrm d\tau.
	\end{split}\end{equation}
		Returning to the original coordinates  $(R,Z)\in \Pi_+$, we define
		\begin{equation}\label{Eq.J_psi-def}	\begin{split}		
		J_\psi(R,Z):&=\mathcal{J}_\psi(s_\psi(R,Z),\sigma_\psi(R,Z))\\
		&=-\frac1{d\mu}\int_0^{s_\psi(R,Z)}(Z\pa_Zh)\bigl(R(\tau;\sigma_\psi(R, Z)), Z(\tau;\sigma_\psi(R, Z))\bigr)\,\mathrm d\tau.
	\end{split}\end{equation}
	Since $h\in C^2(\Pi_+)$, the characteristic flow is $C^2$, and
	$(s_\psi,\sigma_\psi)\in C^1(\Pi_+)$, we deduce that $J_\psi\in C^1(\Pi_+;\mathbb{R})$.  Moreover, 
	 \eqref{Eq.Omega-sol-expression} follows from \eqref{Eq.cJ_psi-def}. 	 
	The uniqueness of the solution is standard, so we omit the details. This completes the proof of Lemma  \ref{Lem.hat-Omega-expression}.
	\end{proof}

\subsection{\texorpdfstring{Upper bounds for $\Omega$}{Upper bounds for Omega}}

Using the characteristic formulation, we have derived an explicit representation formula \eqref{Eq.Omega-sol-expression} for the solution $\Omega$. Based on this formula, we establish the required upper bounds in the definition of $\mathcal{B}_{M_0'}$.

\begin{lemma}\label{Lem.h-derivative-bounds}
{\sl	 Let $\psi\in\mathcal{A}^0,$ $H_\psi: \Pi_+\to\Pi_+$ be  the coordinate transform defined by \eqref{Eq.H_psi} and  $h=h(R, Z)$ be given by \eqref{Eq.h-def}.  Then for all $(R, Z)\in\Pi_+,$  there exists a constant $C>1$ depending only on $d$ such that
	\begin{enumerate}
		\item $|Z\partial_Z h(R,Z)|\leq C Z,$ 
		\item $|Z\partial_Z h(R,Z)|\leq C R^{-1+\frac{1}{d\mu}} Z^{1-\frac{1}{d\mu}} (1-h(R, Z)),$		\item $|R\partial_R h(R,Z)|\leq C R Z^{-1} (1-h(R, Z)),$ 
		\item $|(Z\partial_Z)^2 h(R,Z)|\leq C R^{-1+\frac{1}{d\mu}} Z^{1-\frac{1}{d\mu}} (1-h(R, Z))$ if moreover $R\geq Z$,
		\item $|R\partial_R (Z\partial_Z h)(R,Z)|\leq C (R/Z)^{d-2-\gamma} (1-h(R, Z))$ if moreover $R\leq Z$.
	\end{enumerate}}
\end{lemma}

\begin{proof}
(1) In view of \eqref{Eq.H_psi}  and \eqref{Eq.h-def}, we   compute
		\begin{align*}
			Z\partial_Z h(R, Z)&=-d\mu\frac{z\partial_z\psi(r,z)}{(\mu+\psi(r,z)+z\partial_z\psi(r,z))(\mu+\psi(r,z))}\\
			&=\frac{d\mu}{\mu+\psi(r,z)+z\partial_z\psi(r,z)}-\frac{d\mu}{\mu+\psi(r,z)},\quad\forall\ (R, Z)\in\Pi_+,
		\end{align*}
		where $(r,z)=H_{\psi}^{-1}(R, Z)\in\Pi_+$. As $\psi\in\mathcal{A}^0$, it follows from \eqref{S2eq8}  and \eqref{Eq.z-pa-z-psi-est} that
		\begin{align}
		\label{Eq.Z-pa-Z-h-est}			|Z\partial_Z h(R, Z)|\leq \frac{d}{\mu}|z\partial_z\psi(r,z)|\leq \frac{d\psi(r,z)}{10\mu}\min\{|z|,1\}\leq\frac{a_1 d}{10\mu}\min\{|z|,1\}.
		\end{align}
Due to $Z=(\mu+\psi(r,z))z$ and $\psi>0,$ we deduce that $|Z|\geq \mu|z|$ and
		\[
		|Z\partial_Z h(R,Z)|\leq \frac{a_1 d}{10\mu}|z|\leq \frac{a_1 d}{10\mu^2}|Z|\leq \frac{2d(d-2)}{5(d-1)}|Z|,
		\ \forall \ (R, Z)\in\Pi_+. \]
		Here in the last inequality we used $a_1<\frac{1}{(d-1)(d-2)}$ and $\mu>\frac{1}{2(d-2)}$, which are contained in  \eqref{Eq.a-range} and \eqref{Eq.mu-range}.
		
(2) By virtue of \eqref{S2eq8}, \eqref{Eq.a-range} and \eqref{Eq.mu-range}, for all $(R, Z)\in\Pi_+$, we deduce from \eqref{Eq.Z-pa-Z-h-est} that		\begin{equation}\label{S3eq5} \begin{split}
			|Z\partial_Z h(R,Z)|&\lesssim_d |z\partial_z\psi(r,z)|\lesssim_d |z|\langle r, z\rangle^{-1+\frac{1}{d\mu}}\langle z\rangle^{-\frac{1}{d\mu}}\psi(r,z)\\
			&\lesssim_d r^{-1+\frac{1}{d\mu}}z^{1-\frac{1}{d\mu}}\psi(r,z)\lesssim_d R^{-1+\frac{1}{d\mu}}Z^{1-\frac{1}{d\mu}}\psi(r,z)
		\end{split} 		\end{equation}
		 where $(r,z)=H_\psi^{-1}(R, Z)\in\Pi_+$.
		 
Whereas it follows from \eqref{Eq.h-def} that
		\begin{align}\label{Eq.1-h-lowerbound}
			1-h(R, Z)=\frac{d\psi(r,z)}{\mu+\psi(r,z)}\geq \frac{d}{\mu+a_1}\psi(r,z)\geq \frac{d}{2}\psi(r,z),
		\end{align}
	which along with \eqref{S3eq5} ensures (2).

(3)  It is easy to observe that for  $(r,z)=H_{\psi}^{-1}(R, Z)\in\Pi_+$
\begin{align*}
&\begin{pmatrix}
&\f{\pa R}{\pa r} &\f{\pa R}{\pa z}\\
&\f{\pa Z}{\pa r}& \f{\pa Z}{\pa z}
\end{pmatrix}(r,z)=\begin{pmatrix}
&1 &0\\
&z\pa_r\psi(r,z) &\mu+\psi+z\pa_z\psi
\end{pmatrix},\\
&\begin{pmatrix}
&\f{\pa r}{\pa R} &\f{\pa r}{\pa Z}\\
&\f{\pa z}{\pa R} &\f{\pa z}{\pa Z}
\end{pmatrix}(r,z)=\bigl(\mu+\psi+z\pa_z\psi\bigr)^{-1}\begin{pmatrix}
&\mu+\psi+z\pa_z\psi &0\\
&-z\pa_r\psi(r,z) & 1\end{pmatrix},\end{align*}
from which and \eqref{Eq.h-def}, we infer
		\begin{equation*}
			R\partial_R h(R, Z)=-d\mu\frac{r\partial_r\psi(r,z)}{(\mu+\psi(r,z)+z\partial_z\psi(r,z))(\mu+\psi(r,z))},\quad\forall\ (R, Z)\in\Pi_+,
		\end{equation*}
		where $(r,z)=H_{\psi}^{-1}(R, Z)\in\Pi_+$.  Then we deduce from \eqref{Eq.z-pa-z-psi-est}, $\psi>0$, and \eqref{S2eq8} that
		\begin{align*}
			|R\partial_R h(R, Z)|&\lesssim_d |r\partial_r\psi(r,z)|\lesssim_d r\langle r, z\rangle^{-1+\frac{1}{d\mu}}\langle z\rangle^{-\frac{1}{d\mu}}\psi(r,z)\\
			&\lesssim_d r z^{-1+\frac{1}{d\mu}}z^{-\frac{1}{d\mu}}\psi(r,z)\lesssim_d r z^{-1}\psi(r,z)\\
			&\lesssim_d R Z^{-1}\psi(r,z)\lesssim_d R Z^{-1}(1-h(R,Z)),
		\end{align*}
		where we used \eqref{Eq.1-h-lowerbound} in the last step.
		
(4) We first get, by a similar derivation of \eqref{Eq.Z-pa-Z-h-est}, that
		\begin{equation}\label{S3eq7a}
		|(Z\partial_Z)^2 h(R,Z)|\lesssim_d |z\partial_z\psi(r,z)|+|z^2\partial_z^2\psi(r,z)|,
		\end{equation}
		where $(r,z)=H_\psi^{-1}(R, Z)\in\Pi_+$.  Under the assumption that $R\ge Z,$ we have $r=R\gtrsim_d Z\sim_d z$, from which and \eqref{S2eq8}, we infer		\[
		|z\partial_z\psi(r,z)|
		\lesssim_d r^{-1+\frac{1}{d\mu}} z^{1-\frac{1}{d\mu}}\psi(r,z).
		\]
	While observing that	 $r\gtrsim_d z$, 		$1/(d\mu)<1$ and $\gamma<d-2$,\footnote{\label{Footnote.d-2-gamma>0}In fact $d-2-\gamma=\frac{(d-3)\mu+(2d-3)a-1}{\mu+a}\geq \frac{(d-3)(d-1)a+(2d-3)a-1}{\mu+a}=\frac{d(d-2)a-1}{\mu+a}>0$.} we deduce from \eqref{S2eq8} that
		\[
		|z^2\partial_z^2\psi(r,z)|
		\lesssim_d r^{-1+\frac{1}{d\mu}} z^{1-\frac{1}{d\mu}}\psi(r,z).
		\]
	By substituting the above estimates into \eqref{S3eq7a}, we obtain
		 (4).
		
(5) Similarly,  it follows from  a direct computation that
		\begin{equation}\label{S3eq7}
			\begin{aligned}
				\left|R\partial_R (Z\partial_Z h)(R,Z)\right|&\lesssim_d r z|\partial_r\partial_z\psi(r,z)|\\
				&\qquad\qquad+r|\partial_r\psi(r,z)|\left(1+|z\partial_z\psi(r,z)|+|z^2\partial_z^2\psi(r,z)|\right),
			\end{aligned}
		\end{equation}
		where $(r,z)=H_\psi^{-1}(R, Z)\in\Pi_+$.  With $R\le Z,$  $r=R\lesssim_d Z\sim_d z$. Then
		 we deduce from \eqref{S2eq8} that
		\[
		|r\partial_r\psi(r,z)|\lesssim_d \frac{r}{z}\psi(r,z)\lesssim_d \min\left\{\left(\frac{r}{z}\right)^{d-2-\gamma}\psi(r,z),\frac{r}{z}\right\}.
		\]
		Here we used $d-3<\gamma<d-2$ and $r/z\lesssim_d 1$. Indeed, \eqref{Eq.gamma-range} below implies that $$\gamma>(d-1)/(d\mu)>(d-1)(d-2)/d>d-3. $$
		Furthermore, one has		\begin{align*}
			&|r z\partial_{r}\partial_z\psi(r,z)|+|r\partial_r\psi(r,z)|\left|z^2\partial_z^2\psi(r,z)\right|\\
			&\lesssim_d |r z\partial_{r}\partial_z\psi(r,z)|+|r/z|\left|z^2\partial_z^2\psi(r,z)\right|\\
			&\lesssim_d\, r z\left(r^{d-3-\gamma}\mathbf{1}_{r^2+z^2\leq 1}+\langle r,z\rangle^{2-d+\frac{1}{\mu}}r^{d-3-\frac{d-1}{d\mu}}\langle z\rangle^{-1-\frac{1}{d\mu}}\right)\psi(r,z).
		\end{align*}
		
On the other hand,  due to $\gamma<d-2<d-1,$	 we 
have
\begin{equation}\label{Eq.elementary1}
r z\cdot r^{d-3-\gamma}\mathbf{1}_{r^2+z^2\leq 1}=(r/z)^{d-2-\gamma}z^{d-1-\gamma}\mathbf{1}_{r^2+z^2\leq 1}\lesssim (r/z)^{d-2-\gamma},
		\end{equation}
		
If $z\geq 1$ and $r\lesssim_d z$, then $\langle r,z\rangle\sim_d z\sim_d\langle z\rangle$, hence
		\begin{align*}
			r z\cdot \langle r,z\rangle^{2-d+\frac{1}{\mu}} r^{d-3-\frac{d-1}{d\mu}}\langle z\rangle^{-1-\frac{1}{d\mu}}\sim_d&\ r z\cdot z^{2-d+\frac{1}{\mu}} r^{d-3-\frac{d-1}{d\mu}} z^{-1-\frac{1}{d\mu}}\\
			\sim_d&\, (r/z)^{d-2-\frac{d-1}{d\mu}}\lesssim_d (r/z)^{d-2-\gamma}
		\end{align*}
		since $\gamma>(d-1)/(d\mu)$ and $r/z\lesssim_d 1$, recalling \eqref{Eq.gamma-range}.
		
		Whereas  	 if $z\leq 1$, then $r\lesssim_d z\lesssim_d 1$, hence $\langle r,z\rangle\sim_d 1\sim_d\langle z\rangle$, so that
		\begin{align*}
			&r z\cdot \langle r,z\rangle^{2-d+\frac{1}{\mu}} r^{d-3-\frac{d-1}{d\mu}}\langle z\rangle^{-1-\frac{1}{d\mu}}\sim_d r z\cdot r^{d-3-\frac{d-1}{d\mu}}\\
			&\sim_d\, (r/z)^{d-2-\frac{d-1}{d\mu}} z^{-\frac{d-1}{d\mu}+d-1}\lesssim_d (r/z)^{d-2-\frac{d-1}{d\mu}}\lesssim_d (r/z)^{d-2-\gamma}.
		\end{align*}
		Here we used $(d-1)/(d\mu)<\gamma<d-2$. 
		
Therefore, we obtain
\begin{align}
			r z\cdot \langle r,z\rangle^{2-d+\frac{1}{\mu}} r^{d-3-\frac{d-1}{d\mu}}\langle z\rangle^{-1-\frac{1}{d\mu}}\lesssim_d (r/z)^{d-2-\gamma},\label{Eq.elementary2}
		\end{align}		
			It follows from  \eqref{Eq.z-pa-z-psi-est} that $|z\partial_z\psi(r,z)|\leq \psi(r,z)\leq a_1<1$. 
We conclude the proof of (5) by substituting the above estimates into \eqref{S3eq7}.
	This completes the proof of Lemma \ref{Lem.h-derivative-bounds}.
\end{proof}

\begin{lemma}\label{Lem.J_psi-bound}
	Let $\psi\in\mathcal{A}^0$ and define $J_\psi\in C^1(\Pi_+)$ by \eqref{Eq.J_psi-def}. Then there exists a constant $C>1$ depending only on $d$ and $a$ such that
	\begin{equation}\label{S3eq7-1}
		\sup_{(R, Z)\in\Pi_+}|J_\psi(R, Z)|\leq C.
	\end{equation}
\end{lemma}
\begin{proof}
	
By virtue of \eqref{S3eq3} and \eqref{Eq.J_psi-def},  noting that the map $(s,\sigma)\mapsto (R(s;\sigma), Z(s;\sigma))$ is a $C^2$ diffeomorphism from $\mathbb{R}\times(0,\pi/2)$ onto $\Pi_+,$	we reduce the proof of \eqref{S3eq7-1} to 
	\begin{equation}\label{Eq.cJ-est}
		\sup_{(s,\sigma)\in\mathbb{R}\times(0,\pi/2)}|\mathcal{J}_\psi(s,\sigma)|\leq C,
	\end{equation}
	for some constant $C>1$ depending only on $d$ and $a$. 
	
	 In the case  $s\leq 0$ and $\sigma\in(0,\pi/2)$,  we deduce from \eqref{Eq.J_psi-def} that
	 \begin{equation}\label{Eq.cJ-s<0-est}	 	\begin{split}
		|\mathcal{J}_\psi(s,\sigma)|&\lesssim_d\int_s^0 \left|\bigl(Z\partial_Z h\bigr)\big(R(\tau;\sigma),Z(\tau;\sigma)\big)\right|\,\mathrm d\tau\\
		&\lesssim_d\int_{-\infty}^0 \left|\bigl(Z\partial_Z h\bigr)\big(R(\tau;\sigma),Z(\tau;\sigma)\big)\right|\,\mathrm d\tau.
			\end{split}\end{equation}
It follows from  \eqref{Eq.characteristic} and \eqref{Eq.h-bound} that
	\begin{align*}
		\frac{\mathrm d}{\mathrm d\tau}|(R(\tau), Z(\tau))|^2=2R(\tau)^2+2h(R(\tau), Z(\tau))Z(\tau)^2\geq 2h_*|(R(\tau), Z(\tau))|^2,\quad\forall\ \tau\in\mathbb{R}.
	\end{align*}
	Here and in the rest of the proof of Lemma \ref{Lem.J_psi-bound}, we denote $(R(\tau), Z(\tau)):=(R(\tau;\sigma), Z(\tau;\sigma))$ for simplicity, as each constant does not depend on $\sigma\in(0,\pi/2)$. Observing from  \eqref{Eq.Gamma-def} that $(R(0), Z(0))=(\sin\sigma, \cos\sigma)\in\Gamma$, 	
	we thus obtain
	\begin{align}\label{Eq.trajectory-radius-est}
		|(R(\tau), Z(\tau))|\leq |(R(0), Z(0))| e^{h_*\tau}=e^{h_*\tau},\quad\forall\ \tau\leq 0,
	\end{align}
 from which,  (1)  of Lemma \ref{Lem.h-derivative-bounds}, \eqref{Eq.cJ-s<0-est} and \eqref{Eq.trajectory-radius-est}, we infer
	\begin{align}\label{Eq.cJ-est-1}
		|\mathcal{J}_\psi(s,\sigma)|\lesssim_d\int_{-\infty}^0 Z(\tau;\sigma)\,\mathrm d\tau\lesssim_d\int_{-\infty}^0 e^{h_*\tau}\,\mathrm d\tau\lesssim_{d,a}1\quad\forall\ (s,\sigma)\in(-\infty, 0]\times(0,\pi/2).
	\end{align}
	
It remains to prove \eqref {Eq.cJ-est} for $s>0$.  Indeed, by \eqref{Eq.characteristic} and  \eqref{Eq.h-bound}, we have
	\begin{equation}\label{Eq.cJ-est-3-2}		\frac{\mathrm d}{\mathrm d\tau}\left(\frac{R(\tau)}{Z(\tau)}\right)=\big(1-h(R(\tau), Z(\tau))\big)\frac{R(\tau)}{Z(\tau)}>0,\quad\forall\ \tau\in\mathbb{R},
	\end{equation}
	hence the function $\tau\mapsto R(\tau)/Z(\tau)$ is strictly increasing. Let
	\begin{equation}\label{S3eq8}
		s_0:=\inf\bigl\{s>0: R(s)\geq Z(s) \bigr\}\in [0, +\infty],
	\end{equation}
	where we use the convention that $s_0=+\infty$ if the set $\bigl\{s>0: R(s)\geq Z(s) \bigr\}$ is an empty set. We claim that
	\begin{align}
		\int_{s_0}^\infty\left|\bigl(Z\partial_Z h\bigr)\big(R(\tau;\sigma),Z(\tau;\sigma)\big)\right|\,\mathrm d\tau\leq C,\label{Eq.cJ-est-2}\\
		\sup_{s\in(0,s_0)}\left|\int_0^s\bigl(Z\partial_Z h\bigr)\big(R(\tau;\sigma),Z(\tau;\sigma)\big)\,\mathrm d\tau\right|\leq C, \label{Eq.cJ-est-3}
	\end{align}
	for some constant $C>1$ depending only on $d$ and $a$. Now, our desired \eqref{Eq.cJ-est} follows directly from \eqref{Eq.cJ-est-1}, \eqref{Eq.cJ-est-2} and \eqref{Eq.cJ-est-3}. Below we prove \eqref{Eq.cJ-est-2} and \eqref{Eq.cJ-est-3} respectively.
	
	\underline{Proof of \eqref{Eq.cJ-est-2}}. We assume without loss of generality that $s_0<+\infty$. Then in view of \eqref{Eq.cJ-est-3-2} and \eqref{S3eq8}, one has $R(\tau)\geq Z(\tau)$ for all $s_0\leq\tau<+\infty$.  By (2) of Lemma \ref{Lem.h-derivative-bounds}, we thus obtain
	\begin{equation}\label{Eq.cJ-est-2-1}
		\left|\bigl(Z\partial_Z h\bigr)\big(R(\tau),Z(\tau)\big)\right|\lesssim_d \left(\frac{R(\tau)}{Z(\tau)}\right)^{-1+\frac{1}{d\mu}}\big(1-h(R(\tau), Z(\tau))\big),\quad \forall\ \tau\in[s_0, +\infty).
	\end{equation}
	Whereas it follows from \eqref{Eq.cJ-est-3-2} that for any $\tau\in\mathbb{R}$,
	\begin{align}\label{Eq.cJ-est-2-2}
		\frac{\mathrm d}{\mathrm d\tau}\left(\frac{R(\tau)}{Z(\tau)}\right)^{-1+\frac{1}{d\mu}}=\left(-1+\frac{1}{d\mu}\right)\left(\frac{R(\tau)}{Z(\tau)}\right)^{-1+\frac{1}{d\mu}}\big(1-h(R(\tau), Z(\tau))\big)<0.
	\end{align}
	By combining \eqref{Eq.cJ-est-2-1} with \eqref{Eq.cJ-est-2-2}, we find
	\begin{equation*}
		\begin{split}
			\int_{s_0}^\infty\left|\bigl(Z\partial_Z h\bigr)\big(R(\tau;\sigma),Z(\tau;\sigma)\big)\right|\,\mathrm d\tau&\lesssim_d -\int_{s_0}^\infty \frac{\mathrm d}{\mathrm d\tau}\left(\frac{R(\tau)}{Z(\tau)}\right)^{-1+\frac{1}{d\mu}}\,\mathrm d\tau\\
			&\lesssim_d\,\left(\frac{R(s_0)}{Z(s_0)}\right)^{-1+\frac{1}{d\mu}}\lesssim_d1,
	\end{split} \end{equation*}
	thanks to $-1+\frac{1}{d\mu}<0$ and $R(s_0)\geq Z(s_0)$. This proves \eqref{Eq.cJ-est-2}.
	
	\underline{Proof of \eqref{Eq.cJ-est-3}}. We assume without loss of generality that $s_0>0$. We remark that $s_0=+\infty$ is possible. Let $s\in(0,s_0)$. Then it follows from \eqref{S3eq8} that $R(\tau)<Z(\tau)$ for all $\tau\in[0,s]$. In view of \eqref{Eq.characteristic}, we compute
	\begin{equation*}
		\frac{\mathrm d}{\mathrm d\tau}h\big(R(\tau), Z(\tau)\big)=(R\partial_R h)\big(R(\tau), Z(\tau)\big)+h\big(R(\tau), Z(\tau)\big)(Z\partial_Z h)\big(R(\tau), Z(\tau)\big),\quad\forall\ \tau\in\mathbb{R},
	\end{equation*}
	which implies
	\begin{align*}
		(Z\partial_Z h)\big(R(\tau), Z(\tau)\big)=\frac{\mathrm d}{\mathrm d\tau}\ln h\big(R(\tau), Z(\tau)\big)-\frac{(R\partial_R h)\big(R(\tau), Z(\tau)\big)}{h\big(R(\tau), Z(\tau)\big)},\quad\forall\ \tau\in\mathbb{R}.
	\end{align*}
	As a result, for any $s\in(0,s_0)$, we find
	\begin{align*}
		\int_0^s\bigl(Z\partial_Z h\bigr)\big(R(\tau),Z(\tau)\big)\,\mathrm d\tau
		=&\,\ln h\big(R(s), Z(s)\big)-\ln h\big(R(0), Z(0)\big)\\
		&-\int_0^s\frac{(R\partial_R h)\big(R(\tau), Z(\tau)\big)}{h\big(R(\tau), Z(\tau)\big)}\,\mathrm d\tau.
	\end{align*}
	Noticing from  \eqref{Eq.h-bound} that
	$$ |\ln h(R, Z)|\leq \ln \frac{1}{h_*}\leq \frac{1}{h_*},\quad \forall \ (R, Z)\in\Pi_+, $$  we obtain for any $s\in(0,s_0),$	\begin{equation}\label{Eq.cJ-est-3-1}
		\left|\int_0^s\bigl(Z\partial_Z h\bigr)\big(R(\tau),Z(\tau)\big)\,\mathrm d\tau\right|\leq \frac{2}{h_*}+\frac{1}{h_*}\int_0^s\left|(R\partial_R h)\big(R(\tau), Z(\tau)\big)\right|\,\mathrm d\tau.
	\end{equation}
	Whereas due to $R(s)<Z(s)$ for $s\in(0, s_0)$, it follows from (3) of Lemma \ref{Lem.h-derivative-bounds} and \eqref{Eq.cJ-est-3-2} that
	\begin{align*}
		\int_0^s\left|(R\partial_R h)\big(R(\tau), Z(\tau)\big)\right|\,\mathrm d\tau&\lesssim_d\int_0^s \frac{R(\tau)}{Z(\tau)}\big(1-h(R(\tau), Z(\tau))\big)\,\mathrm d\tau\\
		&\lesssim_d\int_0^s\frac{\mathrm{d}}{\mathrm d\tau}\Bigl(\frac{R(\tau)}{Z(\tau)}\Bigr)\,\mathrm d\tau\lesssim_d \frac{R(s)}{Z(s)}\lesssim_d1.
	\end{align*}
	By substituting the above estimate into \eqref{Eq.cJ-est-3-1}, we obtain
	\if0\begin{align*}
		\sup_{s\in(0,s_0)}\left|\int_0^s\bigl(Z\partial_Z h\bigr)\big(R(\tau;\sigma),Z(\tau;\sigma)\big)\,\mathrm d\tau\right|\leq C, 
	\end{align*}
	for some constant $C>1$ depending only on $d$. This proves\fi  \eqref{Eq.cJ-est-3}.
\end{proof}

The following estimate \eqref{S3eq9} is a direct consequence of  Lemma \ref{Lem.J_psi-bound} and  \eqref{Eq.Omega-sol-expression}.

\begin{corollary}\label{Cor.hat-Omega-upperbound1}
{\sl	 Let $\psi\in\mathcal{A}^0$ and $\lambda>1.$ Let $\hat\Omega\in C^1(\Pi_+)$ be the unique solution to \eqref{Eq.Omega-transport} with the initial data \eqref{Eq.Omega-initial} . Then there exists a constant $C>1$ depending only on $d$ and $a$ such that
	\begin{equation}\label{S3eq9}
		0<\hat\Omega(R, Z)\leq C R^{-\frac{d-1}{d\mu}} Z^{-\frac{1}{d\mu}},\quad\forall\ (R, Z)\in\Pi_+.
	\end{equation}}
\end{corollary}

\begin{proof}
	
Due to $\Theta_\lambda(\sigma)=\Theta(\lambda\sigma)$ for all $\lambda>1$ and $\sigma>0,$ in view of \eqref{Eq.Omega-sol-expression}, we write	\begin{align*}
		\hat\Omega(R,Z)
		&=R^{-\frac{d-1}{d\mu}}Z^{-\frac{1}{d\mu}}\Theta(\lambda \sigma_{\psi}(R, Z))\chi(\sigma_{\psi}(R, Z)) e^{J_\psi(R, Z)},
	\end{align*}
	where $\sigma_\psi\in C^1(\Pi_+; (0,\pi/2))$, $J_\psi\in C(\Pi_+; \mathbb{R})$, $\Theta\in C^\infty((0, +\infty); (0, 1])$, and $\chi\in C^\infty((0, \pi/2);$ $ (0, 1])$. Then \eqref{S3eq9} follows from 
	 Lemma \ref{Lem.J_psi-bound}.
\end{proof}

\begin{lemma}\label{Lem.hat-Omega-upperbound2}
{\sl Let $\gamma$ and $\gamma_1$ be determined by \eqref{Eq.gamma-gamma1-def}. Then
	under the same assumptions of  Corollary \ref{Cor.hat-Omega-upperbound1}, there exists a constant $C>1$ depending only on $d$ and $a$ such that
	\begin{equation}\label{Eq.hat-Omega-est}
		\hat\Omega(R, Z)\leq C\Bigl(\frac{Z}{R^{\gamma}}+\frac{Z}{R^{\gamma_1}}\Bigr),\quad\forall\ (R, Z)\in\Pi_+.
	\end{equation}}
\end{lemma}
\begin{proof}
	Let 
	\begin{equation}\label{S3eq10} \hat\Omega_\gamma(R, Z):=\frac{R^{\gamma}}{Z}\hat\Omega(R, Z) \andf \hat\Omega_{\gamma_1}(R, Z):=\frac{R^{\gamma_1}}{Z}\hat\Omega(R, Z)\quad \mbox{for} \ (R, Z)\in\Pi_+.	\end{equation}	It follows from 
	 the proof of Lemma \ref{Lem.hat-Omega-expression} that the map $\cM$ defined by \eqref{S3eq3} 
	 is a $C^1$ diffeomorphism from $\mathbb{R}\times(0,\pi/2)$ onto $\Pi_+$. 
	 Once again we  shall abbreviate $(R(s), Z(s)):=(R(s;\sigma), Z(s;\sigma)). $ 
	 
	 In view of \eqref{S3eq10}, we get, by applying \eqref{Eq.characteristic} and \eqref{Eq.hat-Omega-along-curve}, that	\begin{align*}
	\frac{\mathrm d}{\mathrm ds}\ln\hat\Omega_\gamma(R(s), Z(s))=&\f\gamma{R}\frac{\mathrm d}{\mathrm ds}R(s)-\f1{Z(s)}\frac{\mathrm d}{\mathrm ds}Z(s)+\f1{\hat\Omega(R(s), Z(s))}\frac{\mathrm d}{\mathrm ds}\hat\Omega(R(s), Z(s))\\
	=&	\gamma-h-\frac1{d\mu}\left(Z\pa_Zh+h+d-1\right).
	\end{align*} 
	Whereas observing from  \eqref{Eq.h-def} that
	\begin{align*}
		h=\frac{d\mu}{\mu+\psi\circ H_\psi^{-1}}-(d-1),
		\end{align*}  from which and \eqref{Eq.mu-range}, for any $s\in\mathbb{R},$ we infer	 
	\begin{align*}
		\frac{\mathrm d}{\mathrm ds}\ln\hat\Omega_\gamma(R(s), Z(s))=&\frac{d\mu+1}{(\mu+a)\left(\mu+\big(\psi\circ H_\psi^{-1}\big)(R(s), Z(s))\right)}\\
		&\times\left(\big(\psi\circ H_\psi^{-1}\big)(R(s), Z(s))-a\right)-\frac{1}{d\mu}(Z\partial_Z h)\big(R(s), Z(s)\big).
	\end{align*}
This together with  \eqref{Eq.cJ_psi-def} ensures that
	\begin{equation}\label{Eq.hat-Omega-gamma-identity}
	\begin{split}
		&\frac{\mathrm d}{\mathrm ds}\left(\ln\hat\Omega_\gamma(R(s), Z(s))-\mathcal{J}_\psi(s)\right)\\
		&=\, \frac{d\mu+1}{(\mu+a)\left(\mu+\big(\psi\circ H_\psi^{-1}\big)(R(s), Z(s))\right)}\left(\big(\psi\circ H_\psi^{-1}\big)(R(s), Z(s))-a\right).	\end{split}\end{equation}
Similarly, for any $s\in\mathbb{R}$,  one has
	\begin{equation}\label{Eq.hat-Omega-gamma1-identity}	\begin{split}
		&\frac{\mathrm d}{\mathrm ds}\left(\ln\hat\Omega_{\gamma_1}(R(s), Z(s))-\mathcal{J}_\psi(s)\right)\\
		&=\, \frac{d\mu+1}{(\mu+a_1)\left(\mu+\big(\psi\circ H_\psi^{-1}\big)(R(s), Z(s))\right)}\left(\big(\psi\circ H_\psi^{-1}\big)(R(s), Z(s))-a_1\right).
	\end{split}\end{equation}
		
	As $\psi\in\mathcal{A}^0$, we have $0<\psi\leq a_1$,  the map
	\begin{align*}
		s\mapsto \left(\ln\hat\Omega_{\gamma_1}(R(s), Z(s))-\mathcal{J}_\psi(s)\right)
	\end{align*}
	is decreasing on $s\in\mathbb{R}$, and thus,
	\begin{align}\label{Eq.hat-Omega-gamma1-est}
		\hat\Omega_{\gamma_1}\big(R(s), Z(s)\big)\leq \hat\Omega_{\gamma_1}\big(R(0), Z(0)\big)\exp\left(\mathcal{J}_\psi(s)\right),\quad\forall\ s\geq0.
	\end{align}
	Whereas again due to  $\psi\in\mathcal{A}^0,$ for all $(r,z)\in \Pi_+,$  we have
	\begin{align*}
		|\psi(r,z)-a|=|\psi(r,z)-\psi(0,0)|\leq \|\nabla\psi\|_{L^\infty}(r+z)\lesssim_d |(r, z)|
	\end{align*}
Notice that  if $(r,z)=H_\psi^{-1}(R, Z)$, then $(R, Z)=(r, (\mu+\psi(r,z))z)$ and thus $|(R, Z)|\sim_d |(r, z)|$, hence we deduce from  \eqref{Eq.trajectory-radius-est} that for all $s\leq0$,	
\begin{align}\label{Eq.psi-a-est}
		\left|\big(\psi\circ H_\psi^{-1}\big)(R(s), Z(s))-a\right|\lesssim_d |(R(s), Z(s))|\lesssim_d e^{h_* s}.
	\end{align}
	
	 Recall that $h_*\in(0,1)$ is a constant depending only on $d$ and $a$. Plugging \eqref{Eq.psi-a-est} into \eqref{Eq.hat-Omega-gamma-identity} gives
	\begin{align*}
		\ln\hat\Omega_{\gamma}\big(R(s), Z(s)\big)-\mathcal{J}_\psi(s)-\ln\hat\Omega_{\gamma}\big(R(0), Z(0)\big)\lesssim_d \int_s^0 e^{h_*\tau}\,\mathrm d\tau\lesssim_{d,a}1
	\end{align*}
	for all $s\leq 0$, so that
	\begin{equation}\label{Eq.hat-Omega-gamma-est}
		\hat\Omega_\gamma\big(R(s), Z(s)\big)\lesssim_{d,a}\hat\Omega_\gamma\big(R(0), Z(0)\big)\exp\left(\mathcal{J}_\psi(s)\right),\quad\forall\ s\leq 0.
	\end{equation}
	
On the other hand,  in view of \eqref{Eq.Omega-initial}, and \eqref{S3eq10}, we find that for any $\sigma\in(0,\pi/2)$,	
\begin{align*}
		\hat\Omega_\gamma(\sin\sigma,\cos\sigma)&=(\sin\sigma)^{\gamma} (\cos\sigma)^{-1}\hat\Omega(\sin\sigma,\cos\sigma)\\
		&=(\sin\sigma)^{\gamma-\frac{d-1}{d\mu}}(\cos\sigma)^{-1-\frac{1}{d\mu}}\Theta_\lambda(\sigma)\chi(\sigma), \andf\\
		\hat\Omega_{\gamma_1}(\sin\sigma,\cos\sigma)&=(\sin\sigma)^{\gamma_1-\frac{d-1}{d\mu}}(\cos\sigma)^{-1-\frac{1}{d\mu}}\Theta_\lambda(\sigma)\chi(\sigma).
	\end{align*}
	 Yet it follows from \eqref{Eq.a-range} and \eqref{Eq.mu-range} that
	\begin{align}\label{Eq.gamma-range}
		\gamma=\frac{\mu+1-(d-1)a}{\mu+a}>\gamma_1=\frac{\mu+1-(d-1)a_1}{\mu+a_1}>\frac{1}{\mu+\frac{\mu}{d-1}}=\frac{d-1}{d\mu}.
	\end{align}
	As a result, it follows that
	 \[ 0<\hat\Omega_{\gamma}(\sin\sigma,\cos\sigma)\leq \hat\Omega_{\gamma_1}(\sin\sigma,\cos\sigma)\leq (\cos\sigma)^{-1-\frac{1}{d\mu}}\chi(\sigma)\ \mbox{ for any}\ \sigma\in(0,\pi/2),\]
	  where we have also used $0<\Theta\leq 1$ and $\chi>0$. Notice from \eqref{Eq.chi-bound} that $(\cos\sigma)^{-1-\frac{1}{d\mu}}\chi(\sigma)\lesssim_d 1$ for any $\sigma\in(0,\pi/2)$. We thus obtain
	\begin{equation}\label{Eq.hat-Omega-Gamma-initial-bound}
		\sup_{\sigma\in(0,\pi/2)}\left|\hat\Omega_\gamma(\sin\sigma,\cos\sigma)\right|+\sup_{\sigma\in(0,\pi/2)}\left|\hat\Omega_{\gamma_1}(\sin\sigma,\cos\sigma)\right|\leq C
	\end{equation}
	for some constant $C>1$ depending only on $d$.
	
	By using \eqref{Eq.cJ-est}, \eqref{Eq.hat-Omega-gamma1-est}, \eqref{Eq.hat-Omega-Gamma-initial-bound} and the fact that $$(s,\sigma)\in [0, +\infty)\times (0,\pi/2)\mapsto (R(s;\sigma), Z(s;\sigma))\in\{(R, Z)\in \Pi_+: R^2+Z^2\geq 1\}$$ is bijective, we find
	\begin{equation}\label{Eq.hat-Omega-est1}
		\hat\Omega_{\gamma_1}(R, Z)\leq C\ \ \Longrightarrow\ \ \hat\Omega(R, Z)\leq C\frac{Z}{R^{\gamma_1}},\quad\forall\ (R, Z)\in\{(R, Z)\in \Pi_+:  R^2+Z^2\geq 1\}
	\end{equation}
	for some constant $C>1$ depending only on $d$ and $a$.
	
	Along the same line, by using \eqref{Eq.cJ-est}, \eqref{Eq.hat-Omega-gamma-est}, \eqref{Eq.hat-Omega-Gamma-initial-bound} and the fact that $$(s,\sigma)\in (-\infty, 0]\times (0,\pi/2)\mapsto (R(s;\sigma), Z(s;\sigma))\in\{(R, Z)\in \Pi_+: R^2+Z^2\leq 1\}$$ is bijective, we obtain
	\begin{equation}\label{Eq.hat-Omega-est2}
		\hat\Omega_{\gamma}(R, Z)\leq C\ \ \Longrightarrow\ \ \hat\Omega(R, Z)\leq C\frac{Z}{R^{\gamma}},\quad\forall\ (R, Z)\in\{(R, Z)\in \Pi_+: R^2+Z^2\leq 1\}
	\end{equation}
	for some constant $C>1$ depending only on $d$ and $a$.
	
 Finally, \eqref{Eq.hat-Omega-est} follows from \eqref{Eq.hat-Omega-est1} and \eqref{Eq.hat-Omega-est2}.  
This completes the proof of Lemma \ref{Lem.hat-Omega-upperbound2}.	
	\end{proof}

Next, we turn to the derivative estimates for $\hat\Omega$. The proof shares a spirit similar to Lemma \ref{Lem.J_psi-bound}. 

\begin{lemma}\label{Lem.UV-eq}
{\sl Under the same hypothesis as Corollary \ref{Cor.hat-Omega-upperbound1}, we define
	\begin{equation}\label{S3eq12}
		U:=\frac{Z\partial_Z\hat\Omega}{\hat\Omega},\quad V:=hU+\frac{Z\partial_Z h+h}{d\mu},
	\end{equation}
	then one has
	\begin{align}
		R\partial_R U + h Z\partial_Z U + Z\partial_Z h\, U &= -\frac{1}{d\mu}\left((Z\partial_Z)^2 h + Z\partial_Z h\right),\label{Eq.U-eq}\\
		R\partial_R V + h Z\partial_Z V - \frac{R\partial_R h}{h} V &= \frac{1}{d\mu}\bigl(R\partial_R (Z\partial_Z h) + R\partial_R h - {h}^{-1}{R\partial_R h}(h+Z\partial_Z h)\bigr).\label{Eq.V-eq}
	\end{align}}
\end{lemma}
\begin{proof}
	This is a direct computation. Dividing \eqref{Eq.Omega-transport} by $\hat\Omega$ gives
	\begin{equation}\label{S3eq13}
	R\partial_R\ln\hat\Omega + h Z\partial_Z\ln\hat\Omega = -\frac{1}{d\mu}(Z\partial_Z h + h + d - 1).
\end{equation}
By 	applying $Z\partial_Z$ to the above identity and using the fact that $$
Z\partial_Z(R\partial_R\ln\hat\Omega + h Z\partial_Z\ln\hat\Omega) = R\partial_R U + h Z\partial_Z U + Z\partial_Z h\, U,$$ we obtain \eqref{Eq.U-eq}. 

Whereas by applying $R\partial_R$ to \eqref{S3eq13}, we find 
	\begin{align}\label{S3eq14}
		(R\partial_R)^2\ln\hat\Omega + h Z\partial_Z (R\partial_R\ln\hat\Omega) + R\partial_R h \, Z\partial_Z\ln\hat\Omega = -\frac{1}{d\mu}\bigl(R\partial_R (Z\partial_Z h) + R\partial_R h\bigr).
	\end{align}
It is easy to observe from \eqref{S3eq12} and \eqref{S3eq13} that
	\begin{align*}
&Z\partial_Z\ln\hat\Omega = U = h^{-1}\Bigl(V - \frac{Z\partial_Z h + h}{d\mu}\Bigr),\andf R\partial_R\ln\hat\Omega = -V - \frac{d-1}{d\mu}.
	\end{align*}
	 By substituting the above identities to \eqref{S3eq14}, we obtain  \eqref{Eq.V-eq}.
\end{proof}

\begin{remark}
It follows from Lemma \ref{Lem.hat-Omega-expression} that for $\psi\in \mathcal{A}^0$,  $\hat\Omega\in C^1(\Pi_+)$, so that  $U\in C(\Pi_+)$. Then, \eqref{Eq.U-eq} and \eqref{Eq.V-eq} hold in the sense of distribution.
\end{remark}

\begin{lemma}\label{Lem.hat-Omega-derivate-est-initial}
{\sl	Under the same hypothesis as Corollary \ref{Cor.hat-Omega-upperbound1}, there exists a constant $C>1$ depending only on $d$ and $a$ such that
	\begin{equation} \label{S3eq18}		\left|Z\partial_Z\hat\Omega(R, Z)\right| + \left|R\partial_R\hat\Omega(R, Z)\right| \leq C \hat\Omega(R, Z),\quad \forall\ (R, Z)\in\Gamma.
	\end{equation}}
\end{lemma}
\begin{proof}
	Let $(R,Z)=(\sin\sigma,\cos\sigma)\in\Gamma, $ we first observe from \eqref{Eq.Omega-initial} that		\[
	\frac{\mathrm d}{\mathrm d\sigma}\ln\hat\Omega(\sin\sigma,\cos\sigma)
	=\lambda\frac{\Theta'(\lambda\sigma)}{\Theta(\lambda\sigma)}-\frac{d-1}{d\mu}\cot\sigma+\frac{1}{d\mu}\tan\sigma+\frac{\chi'(\sigma)}{\chi(\sigma)}.
	\]
	We claim that
	\begin{equation}\label{S3eq15}
		\left|\sin\sigma\cos\sigma\,\frac{\mathrm d}{\mathrm d\sigma}\ln\hat\Omega(\sin\sigma,\cos\sigma)\right|\leq C,\qquad \forall\ \sigma\in(0,\pi/2),
	\end{equation}
	where $C>1$ depends only on $d$ and $a$.
As $\Theta$ is fixed, positive on $(0,+\infty)$, equals $1$ on $[1,+\infty)$, and has the polynomial behavior \eqref{Eq.Theta-vanishing-condition} near the origin, we have $\sup_{t>0}\left|t\Theta'(t)/\Theta(t)\right|<+\infty$. Then, using $\sin\sigma\cos\sigma\leq\sigma$, the contribution of the first term is uniformly bounded in $\lambda>1$. 
Due to $\mu>1/d,$ one has
\[ \sin\sigma\cos\sigma\Bigl|\frac{d-1}{d\mu}\cot\sigma+\frac{1}{d\mu}\tan\sigma\Bigr|\leq C.\]
 Finally,  it follows from the definition of $\chi$ that
 $$\sup_{\sigma\in(0,\pi/2)}\left|\sin\sigma\cos\sigma\,\chi'(\sigma)/\chi(\sigma)\right|\leq C_d. $$
  This is trivial where $\chi=1$; in the transition region $\cos\sigma$ is bounded from below, and near $\pi/2$ we have $\chi(\sigma)=(\cos\sigma)^{1+\frac{1}{d\mu}}$.  As a consequence, we obtain \eqref{S3eq15}. 

On the other hand, by differentiating $\hat\Omega$ along $\Gamma,$  we get 
	\[
	\frac{\mathrm d}{\mathrm d\sigma}\hat\Omega(\sin\sigma,\cos\sigma)
	=\cos\sigma\,\partial_R\hat\Omega(\sin\sigma,\cos\sigma)-\sin\sigma\,\partial_Z\hat\Omega(\sin\sigma,\cos\sigma),
	\]
	from which and \eqref{S3eq15}, we infer
 \begin{equation}\label{S3eq16}
	\left|Z^2\frac{R\partial_R\hat\Omega(R,Z)}{\hat\Omega(R,Z)}
	-R^2\frac{Z\partial_Z\hat\Omega(R,Z)}{\hat\Omega(R,Z)}\right|\leq C,\qquad \forall\ (R,Z)\in\Gamma.
	\end{equation}	
Whereas by evaluating the transport equation \eqref{Eq.Omega-transport} on $\Gamma,$ we find
	\[
	\frac{R\partial_R\hat\Omega(R,Z)}{\hat\Omega(R,Z)}
	+h(R,Z)\frac{Z\partial_Z\hat\Omega(R,Z)}{\hat\Omega(R,Z)}
	=-\frac{1}{d\mu}\big(Z\partial_Z h + h + d - 1\big)(R,Z),
	\]
from which, \eqref{Eq.h-bound}, (1) of Lemma \ref{Lem.h-derivative-bounds}, and the fact that $R^2+Z^2=1$ on $\Gamma$,  we deduce that
 \begin{equation}\label{S3eq17}
 	\Bigl|\frac{R\partial_R\hat\Omega(R,Z)}{\hat\Omega(R,Z)}
	+h(R,Z)\frac{Z\partial_Z\hat\Omega(R,Z)}{\hat\Omega(R,Z)}\Bigr|\leq C,\qquad \forall\ (R,Z)\in\Gamma.
	\end{equation}	
Noticing that $R^2+hZ^2\geq h_*(R^2+Z^2)=h_*$ on $\Gamma$, we deduce from \eqref{S3eq16} and \eqref{S3eq17} that
	\[
	\left|\frac{R\partial_R\hat\Omega(R,Z)}{\hat\Omega(R,Z)}\right|
	+\left|\frac{Z\partial_Z\hat\Omega(R,Z)}{\hat\Omega(R,Z)}\right|\leq C,\qquad \forall\ (R,Z)\in\Gamma,
	\]
which along with $\hat\Omega>0$ on $\Gamma$ ensures \eqref{S3eq18}.
\end{proof}

\begin{lemma}\label{Lem.hat-Omega-derivate-est}
{\sl 	Under the same hypothesis as Corollary \ref{Cor.hat-Omega-upperbound1}, there exists a constant $C>1$ depending only on $d$ and $a$ such that
	\begin{equation}\label{S3eq25}
		\left|Z\partial_Z\hat\Omega(R, Z)\right| + \left|R\partial_R\hat\Omega(R, Z)\right| \leq C \hat\Omega(R, Z),\quad \forall\ (R, Z)\in\Pi_+.
	\end{equation}}
\end{lemma}

\begin{proof}
By virtue of  \eqref{Eq.Omega-transport} and \eqref{Eq.Z-pa-Z-h-est}, it suffices to prove
	\begin{equation*}
		\left|Z\partial_Z\hat\Omega(R, Z)\right| \leq C \hat\Omega(R, Z),\quad\forall\ (R, Z)\in\Pi_+.
	\end{equation*}
	We shall prove the following stronger estimate:
		\begin{equation}\label{Eq.hat-Omega-paZ-est}
	|U(R,Z)|+|V(R,Z)|\leq C,\quad\forall\ (R, Z)\in\Pi_+ , 
	\end{equation}  for $U,V$  defined by \eqref{S3eq12}. 
	The proof shares a spirit similar to the proof of Lemma  \ref{Lem.J_psi-bound}. Consequently, below we only sketch the proof of \eqref{Eq.hat-Omega-paZ-est}.
	
		Let $(R(s;\sigma),Z(s;\sigma))$ be the characteristic curve determined by \eqref{Eq.characteristic}. As in the proof of Lemma \ref{Lem.J_psi-bound}, we abbreviate $(R(s),Z(s)) := (R(s;\sigma),Z(s;\sigma))$. Then we deduce from 
	\eqref{Eq.U-eq} and \eqref{Eq.V-eq} that
	\begin{equation}\label{S3eq19}		\begin{split}
		&\frac{\mathrm d}{\mathrm ds} U(R(s),Z(s)) + (Z\partial_Z h)(R(s),Z(s))\, U(R(s),Z(s)) \\
		&= -\frac{1}{d\mu}\left((Z\partial_Z)^2 h + Z\partial_Z h\right)(R(s),Z(s)),
	\end{split}\end{equation}
	and
\begin{equation}\label{S3eq20}		\begin{split}		&\frac{\mathrm d}{\mathrm ds} V(R(s),Z(s)) - \frac{(R\partial_R h)(R(s),Z(s))}{h(R(s),Z(s))}\, V(R(s),Z(s)) \\
		&= \frac{1}{d\mu}\Bigl(R\partial_R (Z\partial_Z h) + R\partial_R h - \frac{R\partial_R h}{h}(h+Z\partial_Z h)\Bigr)(R(s),Z(s)).
	\end{split}\end{equation}

	We next prove two elementary integral estimates along the same characteristic. First, on any time interval where $R(s) \geq Z(s)$, (2) and (4) of Lemma \ref{Lem.h-derivative-bounds} imply
	\[
	|Z\partial_Z h| + |(Z\partial_Z)^2 h| \lesssim_d \left(\frac{R}{Z}\right)^{-1+\frac{1}{d\mu}} (1-h).
	\]
	As $\frac{\mathrm d}{\mathrm ds}(R/Z) = (1-h)(R/Z)$ and $-1+\frac{1}{d\mu}<0$, we have
	\[
	\int \bigl(|Z\partial_Z h| + |(Z\partial_Z)^2 h|\bigr)\,\mathrm ds \lesssim_{d,a} 1
	\]
	on every such interval. Then we get, by applying Grönwall's inequality to \eqref{S3eq19},  that, on every characteristic sub-interval contained in $\{R \geq Z\}$, the bound of $U$ at one endpoint controls the bound of $U$ at the other endpoint, with a constant depending only on $d$ and $a$.
	
	Second, on any interval where $R(s) \leq Z(s)$, we deduce from  (3) and (5)  of Lemma \ref{Lem.h-derivative-bounds},   $h \geq h_*$ and $|Z\partial_Z h| \lesssim_d 1$ (using \eqref{Eq.Z-pa-Z-h-est}), that
	\[
	\Bigl|\frac{R\partial_R h}{h}\Bigr|
	+ \Bigl|R\partial_R (Z\partial_Z h) + R\partial_R h - \frac{R\partial_R h}{h}(h+Z\partial_Z h)\Bigr|
	\lesssim_{d,a} \left(\frac{R}{Z}\right)^{d-2-\gamma} (1-h).
	\]
	Here we used $0 < R/Z \leq 1$ and $0 < d-2-\gamma < 1$. As
	\[
	\frac{\mathrm d}{\mathrm ds}\left(\frac{R}{Z}\right)^{d-2-\gamma}
	= (d-2-\gamma)\left(\frac{R}{Z}\right)^{d-2-\gamma} (1-h),
	\]
	we find
	\[
	\int \Bigl|\frac{R\partial_R h}{h}\Bigr|\,\mathrm ds
	+ \int \Bigl|R\partial_R (Z\partial_Z h) + R\partial_R h - \frac{R\partial_R h}{h}(h+Z\partial_Z h)\Bigr|\,\mathrm ds \lesssim_{d,a} 1
	\]
	on every characteristic sub-interval contained in $\{R \leq Z\}$.  Then we get, by applying Grönwall's inequality to \eqref{S3eq20}, that, on such an interval, the bound of $V$ at one endpoint controls the bound of $V$ at the other endpoint.
	
	Finally, by \eqref{Eq.cJ-est-3-2}, the map $s \mapsto R(s)/Z(s)$ is strictly increasing. Thus, each characteristic crosses the curve $\{R = Z\}$ at most once. Starting from $s = 0$, where $(R(0),Z(0)) \in \Gamma,$ Lemma \ref{Lem.hat-Omega-derivate-est-initial} ensures that  $|U|\leq C$ and $|V|\leq C$ on $\Gamma$,
 we propagate the bound by the previous two estimates. If the characteristic segment stays in one of the two regions $\{R \geq Z\}$ or $\{R \leq Z\}$, this is immediate. If it crosses $\{R = Z\}$, we first propagate the appropriate quantity up to the crossing time, then use $V = hU + \frac{Z\partial_Z h + h}{d\mu}$, $h \geq h_*$ and $|Z\partial_Z h| \lesssim_d 1$ on $\{R = Z\}$ to pass from the bound of $U$ to the bound of $V$, or vice versa, and then continue the propagation in the other region.
 We thus complete the proof of \eqref{Eq.hat-Omega-paZ-est} and the lemma.
 \end{proof}

\subsection{Lower bound in the core region}

We next prove the lower bound for $\hat\Omega$ in the region $\{(R, Z)\in\Pi_+: |(R, Z)|>1, Z < 2R < 4d Z\}$, which will give the lower bound for $\Omega$ and ensure that $\Omega$ belongs to $\mathcal{B}_{M_0'}$ for some $M_0'$. 

In order to do so, for  $\psi\in\mathcal{A}^0$, 
we denote by $(s,\sigma)\mapsto \big(R(s;\sigma), Z(s;\sigma)\big)$  the trajectory map for the autonomous ODE system \eqref{Eq.characteristic}. 
It is easy to observe from \eqref{Eq.characteristic} that
\begin{align}\label{Eq.R-Z-expression}
	&R(s;\sigma)=e^{s}\sin\sigma \andf  Z(s;\sigma)=e^{s}e^{-\mathcal{I}(s;\sigma)}\cos\sigma,
\with\\
&\label{Eq.cI-def}
	\mathcal{I}(s;\sigma):=\int_0^s \bigl(1-h\bigr)\big(R(\tau;\sigma), Z(\tau;\sigma)\big)\,\mathrm d\tau,\quad\forall\ (s,\sigma)\in\mathbb{R}\times(0,\pi/2).
\end{align}

\begin{lemma}\label{Lem.cI-est}
{\sl  Let $\psi\in\mathcal{A}^0$, and the map $\cM$ defined by \eqref{S3eq3} be the trajectory map determined by \eqref{Eq.characteristic}.  Let $\mathcal{I}(s;\sigma)$ be defined by \eqref{Eq.cI-def}. If there exists a constant $M>1$ such that
	\begin{equation}\label{Eq.psi-upperbound-assume}
		\psi(r,z)\leq M\langle r,z\rangle^{d-2-\frac{1}{\mu}},\quad\forall\ (r,z)\in\Pi_+,
	\end{equation}
	then we have
	\begin{equation}\label{S3eq22}
		0\leq \mathcal{I}(s;\sigma)\leq \left(\frac{2}{3}\right)^{d-2-\frac{1}{\mu}} \frac{d^2 M}{\left(\frac{1}{\mu}-(d-2)\right) h_*},\quad\forall\ (s,\sigma)\in[0, +\infty)\times(0,\pi/2),
	\end{equation}
	where $h_*$, defined in \eqref{Eq.h_*-def}, is a constant depending only on $d$ and $a$.}
\end{lemma}

\begin{proof}
	For each $(R, Z)$, we denote $(r,z):=H_\psi^{-1}(R, Z)$, where $H_\psi: \Pi_+\to\Pi_+$ is defined by \eqref{Eq.H_psi}. It follows from  Lemma \ref{Lem.H_psi} that $H_\psi$ is bijective and $H_\psi^{-1}\in C^2(\Pi_+)$. In view of \eqref{Eq.h-def}, we have
	\begin{equation*}
	h(R, Z)=\frac{\mu-(d-1)\psi(r,z)}{\mu+\psi(r,z)} \andf	1-h(R, Z)=\frac{d\psi(r,z)}{\mu+\psi(r,z)}.
	\end{equation*}
Whereas	by  \eqref{Eq.mu-range} and \eqref{S2eq8}, we have   $\mu>1/d$ and  $\psi(r,z)>0,$ which together with \eqref{Eq.psi-upperbound-assume} ensures that 
	\begin{equation*}
		0<1-h(R, Z)\leq d^2\psi(r,z)\leq d^2 M \langle r,z\rangle^{d-2-\frac{1}{\mu}}.
	\end{equation*}
Notice from \eqref{Eq.a-range} and \eqref{Eq.mu-range} that $(d-1)a_1<\mu<1/(d-2)$ and $d\geq 3$,
from which, \eqref{Eq.H_psi}  and $\psi(r,z)\leq a_1$, we infer
	\begin{align}\label{Eq.z-geq-2Z/3}
		z=\frac{Z}{\mu+\psi(r,z)}\geq\frac{Z}{\mu+a_1}\geq \frac{d-1}{d\mu}Z\geq \frac{(d-1)(d-2)}{d}Z\geq \frac{2}{3}Z.
	\end{align}
Whereas by  \eqref{Eq.mu-range}, we have	$d-2-\frac{1}{\mu}<0$,  so there holds
	\begin{equation*}
		\langle r,z\rangle^{d-2-\frac{1}{\mu}}\leq \left(\frac{2}{3}\right)^{d-2-\frac{1}{\mu}} \langle R,Z\rangle^{d-2-\frac{1}{\mu}}.
	\end{equation*}
	We thus obtain
		\begin{equation}\label{Eq.1-h-est}
		0<1-h(R, Z)\leq d^2 M \left(\frac{2}{3}\right)^{d-2-\frac{1}{\mu}} \langle R,Z\rangle^{d-2-\frac{1}{\mu}},\quad\forall\ (R, Z)\in\Pi_+.
	\end{equation}
	
On the other hand, we get, by a similar derivation of \eqref{Eq.trajectory-radius-est},  that
	\begin{equation}\label{Eq.trajectory-radius-est-positive}
		\left|\big(R(\tau;\sigma), Z(\tau;\sigma)\big)\right|\geq e^{h_*\tau},\quad\forall\ (\tau,\sigma)\in[0, +\infty)\times(0,\pi/2),
	\end{equation}
	for $h_*$ being defined in \eqref{Eq.h_*-def}. In view of \eqref{Eq.1-h-est} and \eqref{Eq.trajectory-radius-est-positive}, for $(s,\sigma)\in[0, +\infty)\times(0,\pi/2)$, we deduce from \eqref{Eq.cI-def} that	\begin{align*}
		0\leq \mathcal{I}(s;\sigma)&\leq \int_0^\infty d^2 M \left(\frac{2}{3}\right)^{d-2-\frac{1}{\mu}} \exp\left(\Bigl(d-2-\frac{1}{\mu}\Bigr)h_*\tau\right)\,\mathrm d\tau\\
		&\leq \left(\frac{2}{3}\right)^{d-2-\frac{1}{\mu}} \frac{d^2 M}{\left(\frac{1}{\mu}-(d-2)\right) h_*},
	\end{align*}
	which leads to \eqref{S3eq22}. This completes the proof of the lemma.
\end{proof}

Now we are ready to derive the lower bound of $\hat\Omega$.

\begin{lemma}\label{Lem.hat-Omega-lowerbound}
{\sl	 Let $\psi\in\mathcal{A}^0$ be such that \eqref{Eq.psi-upperbound-assume} holds for some $M>1.$ Let
	\begin{equation}\label{Eq.lambda-def}
		\lambda := \lambda_0 = 8 \exp\left( \frac{\left(2/3\right)^{d-2-1/\mu} d^2 M}{\left(1/\mu - (d-2)\right) h_*} \right) > 1,
	\end{equation}
	and denote by $\hat\Omega\in C^1(\Pi_+)$ the unique solution to \eqref{Eq.Omega-transport} with the initial data \eqref{Eq.Omega-initial} with this $\lambda>1$. Then there exists a constant $C>1$ depending only on $d$ and $a$ such that
	\begin{equation}\label{Eq.hat-Omega-lowerbound}
		\hat\Omega(R, Z) \geq \frac{1}{C} R^{-\frac{d-1}{d\mu}} Z^{-\frac{1}{d\mu}},\quad\forall\ (R, Z)\in\Pi_+\ \ \text{with}\ |(R, Z)|>1 \ \text{and}\ Z<2R<4d Z.
	\end{equation}}
\end{lemma}

\begin{proof}
	Let $(R, Z)\in\Pi_+$. 
	We denote by $\big(s_\psi(R, Z), \sigma_\psi(R, Z)\big)\in\mathbb{R}\times (0,\pi/2)$  the unique solution of \eqref{S3eq21},
where the map $(s,\sigma)\mapsto \big(R(s;\sigma), Z(s;\sigma)\big)$ denotes the trajectory map for the autonomous ODE system \eqref{Eq.characteristic}.

	
	For $(R, Z)\in\Pi_+$ satisfying $|(R, Z)|>1$ and $Z<2R<4d Z$, we denote simply $$(s,\sigma):=\big(s_\psi(R, Z), \sigma_\psi(R, Z)\big) \in[0, +\infty)\times (0,\pi/2),$$ so that $(R, Z)=\big(R(s;\sigma), Z(s;\sigma)\big)$. 
Notice from \eqref{Eq.R-Z-expression} that  $(R, Z)=\big(R(s;\sigma), Z(s;\sigma)\big)=\left(e^{s}\sin\sigma,\, e^{s}e^{-\mathcal{I}(s;\sigma)}\cos\sigma\right),$ 	 from which and  $Z<2R<4d Z,$  we infer
	\begin{align*}
		(\tan\sigma) \cdot e^{\mathcal{I}(s;\sigma)} = \frac{R}{Z} \in \Bigl(\frac12,\, 2d\Bigr),
	\end{align*}
	which along with \eqref{S3eq22} implies that
	\begin{equation*}
		\tan\sigma \in \left( \frac12 \exp\Bigl( -\frac{\left(2/3\right)^{d-2-1/\mu} d^2 M}{\left(1/\mu - (d-2)\right) h_*} \Bigr),\, 2d \right).
	\end{equation*}
	Since $\tan x < 2x$ for $x\in(0,1)$ and
	\begin{equation*}
		\frac12 \exp\Bigl( -\frac{\left(2/3\right)^{d-2-1/\mu} d^2 M}{\left(1/\mu - (d-2)\right) h_*} \Bigr) < \frac12 e^{-9/2} < 1,
	\end{equation*}
	we find
	\begin{equation*}
		\sigma \in \left( \frac14 \exp\Bigl( -\frac{\left(2/3\right)^{d-2-1/\mu} d^2 M}{\left(1/\mu - (d-2)\right) h_*} \Bigr),\, \arctan(2d) \right).
	\end{equation*}
	By our choice of the cut-off function $\chi$, we have $\chi(\sigma)=1.$ By the definition of $\lambda$ in \eqref{Eq.lambda-def}, we know that $\lambda\sigma>2$ and thus $\Theta_\lambda(\sigma)=\Theta(\lambda\sigma)=1$. Finally, Lemma \ref{Lem.J_psi-bound} implies that $|J_\psi(R, Z)|\leq C$ for some constant $C>1$ depending only on $d$ and $a$. Therefore,  \eqref{Eq.hat-Omega-lowerbound} follows directly from \eqref{Eq.Omega-sol-expression}.
\end{proof}

\subsection{Proof of Proposition \ref{Prop.F}}\label{Subsec.Proof-PropF}


The previous subsections provide the estimates which are needed to verify the defining bounds of $\cB_{M_0'}$: the pointwise upper bounds, the lower bound in the core region, and the first-order derivative bounds. It remains to prove the additional far-field estimate \eqref{Eq.Omega-decay}. Since this estimate is isotropic in the original variables $(r,z)$, we first establish its counterpart for $\widehat\Omega$ in the transformed variables $(R,Z)$, and then pull the estimate back by the change of variables $H_\psi$. Finally, we combine these estimates to conclude that $\Omega\in\cB_{M_0'}$ and that \eqref{Eq.Omega-decay} holds. We also emphasize that the constant $M_0'$ is independent of $M$, which is essential for the fixed-point argument developed in Subsection \ref{Subsec.fixed-point-formulation}, while the constant in \eqref{Eq.Omega-decay} may depend on $d,a,\mu$ and $M$.


\begin{lemma}\label{Lem.hat-Omega-decay}
{\sl	 Let $\psi\in\cA^0$ be such that \eqref{Eq.psi-upperbound-assume} holds for some $M>1.$
Let $\lambda=\lambda_0$ be given by \eqref{Eq.lambda-def} and $\widehat\Omega\in C^1(\Pi_+)$ be the unique solution to \eqref{Eq.Omega-transport} with the initial data \eqref{Eq.Omega-initial}. Then there exists a constant $C>1$ depending only on $d,a,\mu,M$ such that
	\begin{equation}\label{Eq.hat-Omega-decay}
		0<\widehat\Omega(R,Z)\leq C Z |(R,Z)|^{-1-\frac1\mu},\qquad
		\forall\ (R,Z)\in\Pi_+\quad\text{with}\quad |(R,Z)|\geq 1/d.
	\end{equation}}
\end{lemma}

\begin{proof}
	For $(R,Z)\in\Pi_+,$  we denote $\rho:=|(R,Z)|=(R^2+Z^2)^{1/2}$. Let $$(s,\sigma)=\big(s_\psi(R, Z), \sigma_\psi(R, Z)\big)\in\mathbb{R}\times (0,\pi/2)$$ be the unique solution of \eqref{S3eq21} so that $(R,Z)=\bigl(R(s;\sigma),Z(s;\sigma)\bigr),$ which is determined by \eqref{Eq.characteristic}.	For simplicity,  we write
	\[(R(s),Z(s)):=\bigl(R(s;\sigma),Z(s;\sigma)\bigr) \andf \rho(s):=\bigl(R(s)^2+Z(s)^2\bigr)^{1/2}.\]
	Then $\rho(0)=1$.
	
	By the definition of $\Theta$ and the polynomial vanishing condition \eqref{Eq.Theta-vanishing-condition}, we have $$\sup_{t>0} t^{-\delta_d-(d-1)/(d\mu)}\Theta(t)<+\infty,$$ which implies for $\sigma\in(0,\pi/2)$,	\[\Theta_{\lambda_0}(\sigma)=\Theta(\lambda_0\sigma)\lesssim_\mu \lambda_0^{\delta_d+(d-1)/(d\mu)}\sigma^{\delta_d+(d-1)/(d\mu)}\lesssim_\mu \lambda_0^{\delta_d+(d-1)/(d\mu)}(\sin\sigma)^{\delta_d+(d-1)/(d\mu)}.\]
Whereas by \eqref{Eq.chi-bound}, $\chi(\sigma)\lesssim_d(\cos\sigma)^{1+1/(d\mu)}$ for $\sigma\in(0,\pi/2)$. We thus deduce from \eqref{Eq.Omega-sol-expression} that
	\begin{equation}\label{Eq.initial-decay-hatOmega}
		\widehat\Omega(\sin\sigma,\cos\sigma)\leq C_{\lambda_0}(\sin\sigma)^{\delta_d} \cos\sigma\leq C_{\lambda_0} \cos\sigma,\qquad
		\forall\ \sigma\in(0,\pi/2),
	\end{equation}
	where $C_{\lambda_0}>1$ depends only on $d,\mu,\lambda_0$. Since $\lambda_0$ is given by \eqref{Eq.lambda-def}, this constant depends only on $d,a,\mu,M$. 
	
	Before proceeding, we define
\begin{equation}\label{S3eq23}
\mathcal Q(s):=\frac{\widehat\Omega(R(s),Z(s))}{Z(s)}\rho(s)^{1+\frac1\mu}. \end{equation}
Then a direct computation shows that
\begin{align*}
\frac{\mathrm d}{\mathrm ds}\ln \mathcal Q(s)=&\f1{\widehat\Omega(R(s),Z(s))}\frac{\mathrm d}{\mathrm ds}{\widehat\Omega(R(s),Z(s))}+\Bigl(1+\f1\mu\Bigr)\rho^{-1}(s)\Bigl(R(s)\frac{\mathrm d}{\mathrm ds}R(s)+Z(s)\frac{\mathrm d}{\mathrm ds}Z(s)\Bigr)\\
&-\f1{Z(s)}\frac{\mathrm d}{\mathrm ds}Z(s),
\end{align*}	 from which, \eqref{Eq.Omega-transport} and \eqref{Eq.characteristic}, we infer
	\begin{align*}
		\frac{\mathrm d}{\mathrm ds}\ln \mathcal Q(s)
		=&-\frac1{d\mu}(Z\partial_Zh)(R(s),Z(s))
		-\frac{h(R(s),Z(s))+d-1}{d\mu}
		\\
		&-h(R(s),Z(s))+\Bigl(1+\frac1\mu\Bigr)
		\frac{R(s)^2+h(R(s),Z(s))Z(s)^2}{R(s)^2+Z(s)^2}.
	\end{align*}
As $0<h<1$ and $\mu>1/d$,  we have
	\[\left|-\frac{h+d-1}{d\mu}-h+\left(1+\frac1\mu\right)\frac{R^2+hZ^2}{R^2+Z^2}\right|\leq C_d(1-h),\]
	so that there holds 
	\begin{equation}\label{Eq.Q-log-derivative}
		\left|\frac{\mathrm d}{\mathrm ds}\ln \mathcal Q(s)\right|		\leq C_d\left(|Z\partial_Zh|+1-h\right)(R(s),Z(s)).
	\end{equation}
	
	If $s\geq0$, we get, by applying  Lemma \ref{Lem.cI-est}, that
	\[
	\int_0^s(1-h)(R(\tau),Z(\tau))\,\mathrm d\tau
	\leq C(d,a,\mu,M).
	\]
	While it follows from \eqref{Eq.Z-pa-Z-h-est} and the estimates obtained in the proof of Lemma \ref{Lem.cI-est} that
	\[
	\int_0^s |(Z\partial_Zh)(R(\tau),Z(\tau))|\,\mathrm d\tau
	\leq C(d,a,\mu,M).
	\]
	By substituting the above two estimates into  \eqref{Eq.Q-log-derivative}, we obtain
	\[
	\mathcal Q(s)\leq C(d,a,\mu,M)\mathcal Q(0),\qquad s\geq0.
	\]
	
	It remains to consider the case $s<0$ under the additional assumption $\rho(s)\geq1/d$. By \eqref{Eq.trajectory-radius-est}, for $\tau\leq0$ we have $\rho(\tau)\leq \mathrm e^{h_*\tau}$. Therefore, $\rho(s)\geq 1/d$ implies $s\geq -\frac{\ln d}{h_*}$. On this finite interval, we get, by  using $0<1-h<1$ and (1)  of Lemma \ref{Lem.h-derivative-bounds}, that
	\[\int_s^0\left(|Z\partial_Zh|+1-h\right)(R(\tau),Z(\tau))\,\mathrm d\tau\leq C(d,a),\]
which together with \eqref{Eq.Q-log-derivative} ensures that  $\mathcal Q(s)\leq C(d,a)\mathcal Q(0)$ for all $s<0$ satisfying $\rho(s)\geq 1/d$.
	
	Finally, by \eqref{Eq.initial-decay-hatOmega} we have $\mathcal Q(0)\leq C(d,a,\mu,M),$
	 we thus obtain  $$\mathcal Q(s)\leq C(d,a,\mu,M)\  \mbox{for any}\ s\in \R \with \rho(s)\geq 1/d,$$ from which and \eqref{S3eq23}, we conclude the proof of  \eqref{Eq.hat-Omega-decay}.
\end{proof}

Now we are ready to prove Proposition \ref{Prop.F}.

\begin{proof}[Proof of Proposition \ref{Prop.F}]
	Let $M>1$ and  $\lambda_0>1$ be determined by  \eqref{Eq.lambda-def}. Fix $\psi\in \mathcal A_M$, which in particular satisfies \eqref{Eq.psi-upperbound-assume}. 
	Let $\widehat \Omega\in C^1(\Pi_+)$ be the unique solution to \eqref{Eq.Omega-transport} with the initial data \eqref{Eq.Omega-initial}, where $\lambda=\lambda_0$. We then define
	\begin{equation}\label{S3eq24} \Omega(r,z):=\widehat\Omega\bigl(H_\psi(r,z)\bigr)=\widehat\Omega\bigl(r,(\mu+\psi(r,z))z\bigr),\qquad \forall\ (r,z)\in \Pi_+ ,\end{equation}
which together with  Lemma \ref{Lem.H_psi}
ensures that $\Omega\in C^1(\Pi_+)$.  While Lemma \ref{Lem.transport-eq-transform} implies that $\Omega$ is a solution to the original transport equation \eqref{Eq.Omega-rel-transport}. A direct computation shows that the prescribed datum on $\Gamma_0$ is exactly \eqref{Eq.transport-initial-2}. The uniqueness follows from Lemmas  \ref{Lem.transport-eq-transform} and \ref{Lem.hat-Omega-expression}.
	
	
	It remains to prove that $\Omega\in \mathcal B_{M_0'}$ for some constant $M_0'=M_0'(d,a)>1$, which is independent of $M$. We denote
	\[(R,Z):=H_\psi(r,z)=\bigl(r,(\mu+\psi(r,z))z\bigr).\]
	Since $0<\psi\le a_1$ and $\mu$ satisfies \eqref{Eq.mu-range}, there exists a constant $C=C(d,a)>1$ such that
	\[C^{-1}z\le Z\le Cz,\qquad R=r,
	\qquad \forall\ (r,z)\in\Pi_+ .\]
While it follows from  Corollary \ref{Cor.hat-Omega-upperbound1} and Lemma \ref{Lem.hat-Omega-upperbound2} that
	\[0<\widehat\Omega(R,Z)\le C\min\left\{ R^{-\frac{d-1}{d\mu}}Z^{-\frac1{d\mu}}, ZR^{-\gamma}+ZR^{-\gamma_1}\right\},\qquad \forall\ (R,Z)\in\Pi_+,\]
	where $C>1$ is a constant depending only on $d$ and $a$. 
	We thus deduce from \eqref{S3eq24} that
	\[0<\Omega(r,z)\le C \min\left\{r^{-\frac{d-1}{d\mu}}z^{-\frac1{d\mu}}, zr^{-\gamma}+zr^{-\gamma_1}\right\},\qquad \forall\ (r,z)\in\Pi_+.\]
	
	We next prove the lower bound in the core region. Let $(r,z)\in\Pi_+$ satisfy $2<z<r<2z$. By $(R,Z)=H_\psi(r,z)$, \eqref{Eq.mu-range} and \eqref{Eq.z-geq-2Z/3}, we have $R=r>2$, $z\geq 2Z/3$ and 
	\begin{equation}\label{Eq.z-leq-dZ}
		z=\frac{Z}{\mu+\psi(r,z)}\leq \frac1\mu Z\leq dZ.
	\end{equation}
	Hence $z<r$ implies that $\frac23Z<R$, which ensures  $Z<\frac32R<2R,$ and $r<2z$ implies that $R<2dZ$ so that  $Z<2R<4dZ$,  from which, Lemma \ref{Lem.hat-Omega-lowerbound} 
	and $R=r, Z\le Cz$, we obtain
	\[\Omega(r,z)\ge C^{-1}r^{-\frac{d-1}{d\mu}}z^{-\frac1{d\mu}},\qquad \forall\ 2<z<r<2z.\]
	
	By Lemma \ref{Lem.hat-Omega-derivate-est},  there holds \eqref{S3eq25}.
	We are going to transfer the estimate \eqref{S3eq25} back to the original variables. Indeed, 
	in view of \eqref{S3eq24}, we compute
	\[z\pa_z\Omega(r,z)=\left(1+\frac{z\pa_z\psi(r,z)}{\mu+\psi(r,z)}\right)Z\pa_Z\wh\Omega(R,Z).
	\]
	It follows from  \eqref{Eq.z-pa-z-psi-est} that  $|z\pa_z\psi|\leq \frac1{10}\psi$, we thus obtain	\[
	|z\pa_z\Omega(r,z)|\leq C\Omega(r,z),\qquad \forall\ (r,z)\in\Pi_+.
	\]
	For the $r$-derivative, the chain rule gives
	\[
	r\pa_r\Omega(r,z)=R\pa_R\wh\Omega(R,Z)+\frac{r\pa_r\psi(r,z)}{\mu+\psi(r,z)}Z\pa_Z\wh\Omega(R,Z),
	\]
from which and Lemma \ref{Lem.hat-Omega-derivate-est}, we infer
	\[
	|r\pa_r\Omega(r,z)|\leq C\Bigl(1+\frac{|r\pa_r\psi(r,z)|}{\mu+\psi(r,z)}\Bigr)\Omega(r,z).
	\]
As $\psi\in\cA^0$, we know that $\langle r,z\rangle|\pa_r\psi(r,z)|\lesssim \psi(r,z)\lesssim1$, hence  in particular   $|r\pa_r\psi(r,z)|\lesssim 1$. As a result, it follows that
	\[|z\pa_z\Omega(r,z)|+|r\pa_r\Omega(r,z)|\leq C\Omega(r,z).\]
	
	Finally, \eqref{Eq.Omega-decay} follows directly from Lemma \ref{Lem.hat-Omega-decay} and \eqref{Eq.z-leq-dZ}. This  completes the  proof of Proposition \ref{Prop.F}.
\end{proof}

\section{Analysis of the elliptic equation}\label{Sec.elliptic}
This section aims to prove Lemma \ref{Lem.Psi_0-bound} and Proposition \ref{Prop.G}. The key ingredient lies in the sharp estimates for the Newtonian potential associated with the singular density appearing in \eqref{S2eqPsi}. Rather than invoking standard elliptic regularity theory, we work with the integral representation and prove all estimates from the kernel bounds. Subsection \ref{Subsec.Psi0-C^1-est} provides the basic size, lower bounds, and first-order derivative estimates; Subsection \ref{Subsec.Psi0-C^2-est} is devoted to the Hessian estimate; Subsection \ref{Subsec.Psi0-radial} establishes an almost radial approximation for $\Psi_0$; and Subsection \ref{Subsec.Proof-Prop-G} combines these estimates to verify Proposition \ref{Prop.G}. We remark that all the implicit constants appearing in this section are independent of $\mu$.

\subsection{Zeroth- and first-order estimates for the Newtonian potential}\label{Subsec.Psi0-C^1-est}

We begin with the pointwise analysis of the Newtonian potential $\Psi_0$ defined in \eqref{Eq.Psi0-def}. The main difficulty lies in the fact that the density $|\xi|^{d-3}|\eta|^{-1}\Omega(|\xi|,|\eta|)$ is singular near both $\xi=0$ and $\eta=0$. Therefore, instead of invoking classical elliptic regularity at this stage, we estimate the integral representation directly. We prove the well-definedness of $\Psi_0$, the upper and lower bounds in Lemmas \ref{Lem.Psi0-upperbound} and \ref{Lem.Psi0-lowerbound}, and the pointwise bound for $|\nabla_X\Psi_0|$ in Lemma \ref{Lem.nablaPsi0-upperbound}, which will later be used to verify the gradient condition in the definition of $\mathcal A^0$.

Throughout this section, we work under the following assumptions:
\begin{equation}\label{S4eq0}
	M'>1,\qquad
	\Omega\in\mathcal B_{M'},\qquad
	\Psi_0(X)\ \text{is defined by \eqref{Eq.Psi0-def} for }
	X=(x,y)\in\mathbb R^{d+1}\times\mathbb R^3.
\end{equation}


\begin{lemma}[Upper bound for $\Psi_0$]\label{Lem.Psi0-upperbound}
{\sl  Under the assumptions in \eqref{S4eq0},  $\Psi_0$ is well-defined and there holds
	\begin{equation}\label{Eq.Psi0-upperbound}
		0<\Psi_0(X)\leq \frac{C}{1-(d-2)\mu}\langle X\rangle^{d-2-\frac1\mu},\quad\forall\ X\in \mathbb R^{d+4},
	\end{equation}
	where $C>1$ is a constant depending only on $d, a, M'$.}
\end{lemma}

\begin{proof}
	For $Y=(\xi,\eta)\in\mathbb R^{d+1}\times\mathbb R^3,$ we denote
	\begin{equation}\label{Eq.F_Omega-def}
		F_\Omega(Y):=|\xi|^{d-3}|\eta|^{-1}\Omega(|\xi|,|\eta|).
	\end{equation}
	Then in view of  \eqref{Eq.Psi0-def}, we write
		\begin{equation}\label{Eq.Psi0-identity}
		\Psi_0(X)=\beta_d\int_{\mathbb R^{d+4}}|X-Y|^{-d-2}F_\Omega(Y)\,\mathrm dY.
	\end{equation}
	As $\Omega\in \mathcal B_{M'}$, we deduce from \eqref{Eq.BM'-def} that
		\begin{align}\label{eq:F-bound-first}
		&0<F_\Omega(Y)\lesssim |\xi|^{d-3-\frac{d-1}{d\mu}}|\eta|^{-1-\frac1{d\mu}},\quad\text{for a.e.}\ \ Y=(\xi,\eta)\in \mathbb R^{d+1}\times\mathbb R^3,
	\\
	&		0<F_\Omega(Y)\lesssim |\xi|^{d-3-\gamma}+|\xi|^{d-3-\gamma_1},\quad\text{for a.e.}\ \  Y=(\xi,\eta)\in \mathbb R^{d+1}\times\mathbb R^3. \label{eq:F-bound-second}
	\end{align}
	Here and below, the implicit constants may depend only on $d,a,M'$.
	
	For each $\rho>0$, we define
	\begin{equation}\label{Eq.M(rho)-def}
		M(\rho):=\int_{|Y|\leq\rho}F_\Omega(Y)\,\mathrm dY,
	\end{equation}
	then it follows from  \eqref{eq:F-bound-first}  that
		\begin{equation}\label{Eq.M(rho)-est}	\begin{split}
		M(\rho)&\lesssim\int_{s^2+t^2\leq \rho^2} s^{2d-3-\frac{d-1}{d\mu}}t^{1-\frac1{d\mu}}\,\mathrm ds\,\mathrm dt\\
		&\lesssim \rho^{2d-\frac1{\mu}}\int_{s^2+t^2\leq 1} s^{2d-3-\frac{d-1}{d\mu}}t^{1-\frac1{d\mu}}\,\mathrm ds\,\mathrm dt\lesssim \rho^{2d-\frac1{\mu}}.
	\end{split} \end{equation}
		Here we emphasize that in \eqref{Eq.M(rho)-est} the implicit constant is independent of $\mu$, although the parameter $\mu$ appears in the integral. Indeed, \eqref{Eq.mu-range} implies that $d\mu>1$, so that $2d-3-\frac{d-1}{d\mu}>0$ and $1-\frac1{d\mu}>0$. 
	
	For each $X=(x,y)\in \mathbb R^{d+1}\times\mathbb R^3$, we denote $R_X:=1+|x|+|y|\sim \langle X\rangle$. We split $\Psi_0(X)$ into two integrals as follows:
	\begin{equation}\label{S4eq8}	\begin{split}
		\Psi_0(X)&=\beta_d\left(\Psi_{0, \text{near}}(X)+\Psi_{0, \text{far}}(X)\right) \with\\
			\Psi_{0, \text{near}}(X)&:=\int_{|Y-X|\leq R_X/4}|Y-X|^{-d-2}F_{\Omega}(Y)\,\mathrm dY,\\
		\Psi_{0, \text{far}}(X)&:=\int_{|Y-X|> R_X/4}|Y-X|^{-d-2}F_{\Omega}(Y)\,\mathrm dY.
	\end{split} \end{equation}
		
	By using \eqref{Eq.M(rho)-est} and $R_X\sim \langle X\rangle$, we obtain
	\begin{equation}\label{Eq.Psi0-far-est} \begin{split}
		0<\Psi_{0, \text{far}}(X)&=\sum_{k=0}^{\infty}\int_{2^k\frac{R_X}{4}<|Y-X|\leq 2^{k+1}\frac{R_X}{4}}|Y-X|^{-d-2}F_\Omega(Y)\,\mathrm dY\\
		&\lesssim \sum_{k=0}^{\infty}(2^k R_X)^{-d-2}M(2^{k+1}R_X)\lesssim \sum_{k=0}^{\infty}(2^k R_X)^{-d-2}(2^{k+1} R_X)^{2d-\frac1\mu}\\
		&\lesssim \sum_{k=0}^{\infty}(2^k R_X)^{d-2-\frac1\mu}\lesssim \frac1{1-(d-2)\mu}R_X^{d-2-\frac1\mu}\lesssim \frac{\langle X\rangle^{d-2-\frac1\mu}}{1-(d-2)\mu}.
	\end{split} \end{equation}
	
	As for $\Psi_{0, \text{near}}$, we consider the following three cases separately: $|x|\geq R_X/3$, $|y|\geq R_X/3$ and $\max\{|x|, |y|\}<R_X/3$.\smallskip
	
\noindent{\Large $\bullet$}	\underline{Case 1: $|x|\geq R_X/3$}. In this case, we have 
\begin{equation}\label{S4eq2}
\begin{split}
|x|\geq R_X/3 \andf |Y-X|\leq R_X/4 \Rightarrow |\xi|\geq& |x|-|x-\xi|\\
\geq & |x|-|X-Y|
\geq  R_X/12. \end{split}  \end{equation}
Whereas it follows from \eqref{Eq.mu-range} ($\mu<1/(d-2)$) that
\begin{equation}\label{Eq.d-3-(d-1)/(dmu)<0}
	d-3-\frac{d-1}{d\mu}<0,
\end{equation}
from which, \eqref{S4eq2} and  \eqref{eq:F-bound-first}, we infer
\begin{align}\label{Eq.F_Omega-est-case1}
	0<F_{\Omega}(Y)\lesssim R_X^{d-3-\frac{d-1}{d\mu}}|\eta|^{-1-\frac{1}{d\mu}},\quad\text{for a.e.}\ Y=(\xi,\eta)\in\mathbb R^{d+1}\times\mathbb R^3.
\end{align}
 By \eqref{Eq.F_Omega-est-case1}, we deduce that
\begin{align}\label{Eq.Psi0near-est-case1-1}
	0<\Psi_{0, \text{near}}(X)\lesssim R_X^{d-3-\frac{d-1}{d\mu}}\int_{|Y-X|\leq R_X/4}|Y-X|^{-d-2}|\eta|^{-1-\frac{1}{d\mu}}\,\mathrm d\xi\,\mathrm d\eta.
\end{align}
By using the change of variables $X-Y=(v,w)\in\mathbb R^{d+1}\times\mathbb R^3$,  we obtain
\begin{align*}
	\int_{|Y-X|\leq \frac{R_X}4}|Y-X|^{-d-2}|\eta|^{-1-\frac{1}{d\mu}}\,\mathrm d\xi\,\mathrm d\eta&=\int_{|(v,w)|\leq \frac{R_X}4}\frac{|y-w|^{-1-\frac{1}{d\mu}}}{(|v|^2+|w|^2)^{(d+2)/2}}\,\mathrm dv\,\mathrm dw,\end{align*}	
	which together with the fact:
\[\int_{\mathbb R^{d+1}}(|v|^2+A^2)^{-(d+2)/2}\,\mathrm dv\lesssim_d A^{-1},\quad\forall\ A>0,\]	
ensures that
\begin{equation}\int_{|Y-X|\leq \frac{R_X}4}|Y-X|^{-d-2}|\eta|^{-1-\frac{1}{d\mu}}\,\mathrm d\xi\,\mathrm d\eta	\lesssim \int_{|w|\leq R_X/4}|w|^{-1}|y-w|^{-1-\frac{1}{d\mu}}\,\mathrm dw.\label{Eq.Psi0near-est-case1-2} \end{equation}
We then get, by using change of variables $w=R_X\tilde w\in\mathbb R^3$ and $y=R_X\tilde y\in\mathbb R^3$, that
\begin{align}\label{Eq.Psi0near-est-case1-3}
	\int_{|w|\leq R_X/4}|w|^{-1}|y-w|^{-1-\frac{1}{d\mu}}\,\mathrm dw=R_X^{1-\frac{1}{d\mu}}\int_{|\tilde w|\leq 1/4}|\tilde w|^{-1}|\tilde y-\tilde w|^{-1-\frac{1}{d\mu}}\,\mathrm d\tilde w\lesssim R_X^{1-\frac{1}{d\mu}},
\end{align}
where in the last inequality we used Lemma \ref{Lem.A1convolution-wholespace} with  $m=3$, $p=1$, $q=1+\frac{1}{d\mu}$ and $C_0=1$. Indeed, by \eqref{Eq.mu-range} ($\mu>3/(3d-5)$), we have $m-(p+q)=1-\frac{1}{d\mu}>1-\frac{3d-5}{3d}=\frac{5}{3d}\gtrsim_d1,$ so that $p+q<m$.

 Combining the estimates \eqref{Eq.Psi0near-est-case1-1}, \eqref{Eq.Psi0near-est-case1-2} and \eqref{Eq.Psi0near-est-case1-3}, we obtain
\begin{align}\label{Eq.Psi0near-est-case1}
	0<\Psi_{0, \text{near}}(X)\lesssim R_X^{d-2-\frac{1}{\mu}}\lesssim \langle X\rangle^{d-2-\frac{1}{\mu}}\quad\text{if}\ \ |x|\geq R_X/3.
\end{align}

\noindent{\Large $\bullet$}	\underline{Case 2: $|y|\geq R_X/3$}. In this case, we have $|y|\geq R_X/3$ and $|Y-X|\leq R_X/4$, hence $|\eta|\geq |y|-|y-\eta|\geq |y|-|X-Y|\geq R_X/3-R_X/4=R_X/12$, then it follows from \eqref{eq:F-bound-first}  that
\begin{align}\label{Eq.F_Omega-est-case2}
	0<F_{\Omega}(Y)\lesssim |\xi|^{d-3-\frac{d-1}{d\mu}} R_X^{-1-\frac{1}{d\mu}},\quad\text{for a.e.}\ Y=(\xi,\eta)\in\mathbb R^{d+1}\times\mathbb R^3.
\end{align}
Therefore, mimicking the analysis in Case 1, we find
\begin{align}
	0<\Psi_{0,\text{near}}(X)&\lesssim R_X^{-1-\frac{1}{d\mu}}\int_{|Y-X|\leq \frac{R_X}4}|Y-X|^{-d-2}|\xi|^{d-3-\frac{d-1}{d\mu}}\,\mathrm dY\nonumber\\
	&\lesssim R_X^{-1-\frac{1}{d\mu}}\int_{|(v,w)|\leq \frac{R_X}4}\frac{|x-v|^{d-3-\frac{d-1}{d\mu}}}{(|v|^2+|w|^2)^{(d+2)/2}}\,\mathrm dv\,\mathrm dw\quad \mbox{by using}\ X-Y=(v,w) \nonumber\\
	&\lesssim R_X^{-1-\frac{1}{d\mu}}\int_{|v|\leq R_X/4}|v|^{-d+1}|x-v|^{d-3-\frac{d-1}{d\mu}}\,\mathrm dv\label{Eq.Psi0near-est-case2}\\
	&\lesssim R_X^{d-2-\frac{1}{\mu}}\int_{|\tilde v|\leq 1/4}|\tilde v|^{-d+1}|\tilde x-\tilde v|^{d-3-\frac{d-1}{d\mu}}\,\mathrm d\tilde v\quad \mbox{by using}\ (x,v)=\left(R_X\tilde x, R_X\tilde v\right)\nonumber\\
	&\lesssim  R_X^{d-2-\frac{1}{\mu}}\lesssim \langle X\rangle^{d-2-\frac{1}{\mu}}\quad\text{if}\ \ |y|\geq R_X/3,\nonumber
	\end{align}
 where we used
\[\int_{\mathbb R^3}(A^2+|w|^2)^{-(d+2)/2}\,\mathrm dw\lesssim_d A^{-d+1}\quad \forall\ A>0,\]
and Lemma \ref{Lem.A1convolution-wholespace} with  $m=d+1$, $p=d-1$, $q=-d+3+\frac{d-1}{d\mu}$ and $C_0=1.$  Such choices of $p$ and $q$ are admissible, as $\mu<1/(d-2)$ implies that $q>-d+3+\frac{(d-1)(d-2)}{d}>0$; and $\mu>3/(3d-5)$ implies that $m-(p+q)=(d-1)\left(1-\frac{1}{d\mu}\right)>\frac{5(d-1)}{3d}\gtrsim_d1$. 

\noindent{\Large $\bullet$}	\underline{Case 3: $|x|< R_X/3$ and $|y|< R_X/3$}. In this case, we have $R_X=1+|x|+|y|<1+2R_X/3$, so that $R_X<3$. It follows from \eqref{eq:F-bound-second} that
\begin{align}
	0<\Psi_{0,\text{near}}(X)&\lesssim\int_{|Y-X|\leq 1}|Y-X|^{-d-2}\Bigl(|\xi|^{d-3-\gamma}+|\xi|^{d-3-\gamma_1}\Bigr)\,\mathrm dY\nonumber\\
	&\lesssim \int_{|(v, w)|\leq 1}\frac{|x-v|^{d-3-\gamma}+|x-v|^{d-3-\gamma_1}}{(|v|^2+|w|^2)^{(d+2)/2}}\,\mathrm dv\,\mathrm dw\label{S4eq3}\\
	&\lesssim \int_{|v|\leq 1}|v|^{-d+1}\Bigl(|x-v|^{d-3-\gamma}+|x-v|^{d-3-\gamma_1}\Bigr)\,\mathrm dv\lesssim 1.\nonumber
	\end{align}
Here in the last inequality we used Lemma \ref{Lem.A1convolution-wholespace} with $m=d+1$, $p=d-1$, $q=-d+3+\gamma$ or $q=-d+3+\gamma_1$, and $C_0=1$. Such choices of $p,q$ are admissible. Indeed it follows from  \eqref{Eq.mu-range} and \eqref{Eq.gamma-range}  that
\[-d+3+\gamma>-d+3+\gamma_1>-d+3+\frac{d-1}{d\mu}>-d+3+\frac{(d-1)(d-2)}{d}>0,\]
whereas it follows from \eqref{Eq.a-range}, \eqref{Eq.mu-range} and \eqref{Eq.gamma-range}  that (here we prove $m-(p+q)\gtrsim_d1$) \footnote{\label{footnote.gamma>d-2}Here $d-1-\gamma\gtrsim_d1$ is sufficient for us to apply Lemma \ref{Lem.A1convolution-wholespace}. Nonetheless, we prove the stronger inequality $d-2-\gamma\gtrsim_d1$, which will be used in \eqref{Eq.I-near-est-case3} below.}
\begin{align*}
	(d&+1)-\big((d-1)+(-d+3+\gamma_1)\big)>(d+1)-\big((d-1)+(-d+3+\gamma)\big)\\
	>&d-2-\gamma\geq\frac{d(d-2)a-1}{\mu+a}>\, \frac{\frac{4d(d-2)}{(4d-3)(d-2)}-1}{\frac1{d-2}+\frac1{(d-1)(d-2)}}=\frac{3(d-1)(d-2)}{d(4d-3)}\gtrsim_d1.
\end{align*}
Here we used footnote \ref{Footnote.d-2-gamma>0}. As a consequence,  we deduce that
\begin{equation}\label{Eq.Psi0near-est-case3}
	0<\Psi_{0,\text{near}}(X)\lesssim 1\lesssim R_X^{d-2-\frac1\mu}\lesssim \langle X\rangle^{d-2-\frac1{\mu}}\quad\text{if}\ \ \max\{|x|, |y|\}< R_X/3.
\end{equation}

Combining the estimates \eqref{Eq.Psi0near-est-case1}, \eqref{Eq.Psi0near-est-case2} and \eqref{Eq.Psi0near-est-case3},  we obtain
 $$0<\Psi_{0,\text{near}}(X)\lesssim \langle X\rangle^{d-2-\frac1{\mu}} \ \mbox{ for any}\ X\in\mathbb R^{d+4}, $$
 which together with \eqref{Eq.Psi0-far-est} ensures  \eqref{Eq.Psi0-upperbound}. This completes the proof of Lemma \ref{Lem.Psi0-upperbound}. \end{proof}

\begin{lemma}[Lower bound for $\Psi_0$]\label{Lem.Psi0-lowerbound}
{\sl  Under the assumptions in \eqref{S4eq0}, 	one has	\begin{equation}\label{Eq.Psi0-lowerbound}
		\Psi_0(X)\geq \frac{1}{C\big(1-(d-2)\mu\big)}\langle X\rangle^{d-2-\frac{1}{\mu}},\quad\forall\ X\in \mathbb R^{d+4},
	\end{equation}
	where $C>1$ is a constant depending only on $d, a, M'$.}
\end{lemma}

\begin{proof}
	As $\Omega\in\mathcal B_{M'}$ for some $M'>1$,  for $Y=(\xi,\eta)\in\mathbb R^{d+1}\times\mathbb R^3$ with $2<|\eta|<|\xi|<2|\eta|$, we deduce from \eqref{Eq.BM'-def} that	\begin{equation*}
		|\xi|^{d-3}|\eta|^{-1}\Omega(|\xi|,|\eta|)\gtrsim |\xi|^{d-3-\frac{d-1}{d\mu}}|\eta|^{-1-\frac{1}{d\mu}}\gtrsim |\eta|^{d-4-\frac{1}{\mu}},
	\end{equation*}
from which and \eqref{Eq.Psi0-def}, we infer
	\begin{align*}
		\Psi_0(X)\gtrsim \int_{2\langle X\rangle<|\eta|<|\xi|<2|\eta|}|X-Y|^{-d-2}|\eta|^{d-4-\frac{1}{\mu}}\,\mathrm dY,\quad\forall\ X\in\mathbb R^{d+4}.
	\end{align*}
For $Y=(\xi,\eta)\in\mathbb R^{d+1}\times\mathbb R^3$ with $2\langle X\rangle<|\eta|<|\xi|<2|\eta|$, we have $$|Y|^2=|\xi|^2+|\eta|^2\geq 2|\eta|^2\geq 8\langle X\rangle^2\geq 8|X|^2, $$
which implies  $|Y|\geq 2|X|$ and then $|Y-X|\leq 3|Y|\lesssim |\eta|$. As a result, it follows that
	\begin{align*}
		\Psi_0(X)&\gtrsim \int_{2\langle X\rangle<|\eta|<|\xi|<2|\eta|}|\eta|^{-d-2}|\eta|^{d-4-\frac{1}{\mu}}\,\mathrm d\xi\,\mathrm d\eta\gtrsim \int_{|\eta|>2\langle X\rangle}|\eta|^{d+1}|\eta|^{-d-2}|\eta|^{d-4-\frac{1}{\mu}}\,\mathrm d\eta\\
		&\gtrsim \int_{|\eta|>2\langle X\rangle}|\eta|^{d-5-\frac{1}{\mu}}\,\mathrm d\eta\gtrsim \int_{2\langle X\rangle}^\infty \tilde r^{d-3-\frac{1}{\mu}}\,\mathrm d\tilde r\gtrsim \frac{1}{1-(d-2)\mu}\langle X\rangle^{d-2-\frac{1}{\mu}}
	\end{align*}
	for all $X\in\mathbb R^{d+4}$. This completes the proof of \eqref{Eq.Psi0-lowerbound}.
\end{proof}

\begin{lemma}[Pointwise estimate for $\nabla_X\Psi_0$]\label{Lem.nablaPsi0-upperbound}
{\sl 	
Under the assumptions in \eqref{S4eq0},  $\Psi_0\in C^1(\mathbb R^{d+4})$ and there holds
	\begin{equation}\label{Eq.nablaPsi0-upperbound}
		|\nabla_X\Psi_0(X)|\leq C \langle X\rangle^{d-3-\frac{d-1}{d\mu}}\langle y\rangle^{-\frac{1}{d\mu}},\quad\forall\ X=(x,y)\in \mathbb R^{d+1}\times\mathbb R^3,
	\end{equation}
	where $C>1$ is a constant depending only on $d, a, M'$.}
\end{lemma}
\begin{proof}
    
	We first derive the pointwise bound for the formal gradient. As $|\nabla_X |X-Y|^{-d-2}|\lesssim |X-Y|^{-d-3}$,  in view of \eqref{Eq.Psi0-identity},  it suffices to estimate
	\[\mathrm I(X):=\int_{\mathbb R^{d+4}}|X-Y|^{-d-3}F_\Omega(Y)\,\mathrm dY\with
F_\Omega(Y) \ \mbox{being given by}\ 	\eqref{Eq.F_Omega-def}.\]
	For each $X=(x,y)\in\mathbb R^{d+1}\times\mathbb R^3$, set $R_X:=1+|x|+|y|\sim \langle X\rangle$. We split
	\begin{equation}\label{S4eq1}
	\begin{split}	
	 \mathrm I(X)&=\mathrm I_{\rm near}(X)+\mathrm I_{\rm far}(X)\with \\
	 		\mathrm I_{\rm near}(X):&=\int_{|Y-X|\le R_X/4}|X-Y|^{-d-3}F_\Omega(Y)\,\mathrm dY,\\
		\mathrm I_{\rm far}(X):&=
		\int_{|Y-X|>R_X/4}|X-Y|^{-d-3}F_\Omega(Y)\,\mathrm dY.
	\end{split} \end{equation}
	
		For $M(\rho)$ being defined by \eqref{Eq.M(rho)-def}, 
		we get, by a similar derivation of \eqref{Eq.Psi0-far-est}, that 
	\begin{align}\label{Eq.I-far-est}
		\mathrm I_{\rm far}(X)\lesssim\sum_{k=0}^\infty(2^kR_X)^{-d-3}M(2^kR_X)\lesssim\sum_{k=0}^\infty(2^kR_X)^{d-3-\frac{1}{\mu}}\lesssim R_X^{d-3-\frac{1}{\mu}}\lesssim \langle X\rangle^{d-3-\frac{1}{\mu}}.
	\end{align}
	Here we used \eqref{Eq.mu-range} so that $d-3-\frac{1}{\mu}<d-3-(d-2)=-1$. It remains to  handle  the estimate of $I_{\rm near}$. As in the proof of Lemma \ref{Lem.Psi0-upperbound}, we divide the proof into three cases.
	
\noindent{\Large $\bullet$}		\underline{Case 1: $|x|\ge R_X/3$}. 
We first observe from \eqref{Eq.gamma-range} and \eqref{Eq.d-3-(d-1)/(dmu)<0} that 	\begin{equation}\label{Eq.gamma>gamma1>d-3}
		\gamma>\gamma_1>\frac{d-1}{d\mu}>d-3.
	\end{equation}
	Then  by virtue of \eqref{S4eq2} and \eqref{Eq.d-3-(d-1)/(dmu)<0}, we deduce from \eqref{eq:F-bound-first} and \eqref{eq:F-bound-second} that
	\begin{align*}
		F_\Omega(Y)&\lesssim\min\Bigl\{R_X^{d-3-\frac{d-1}{d\mu}}|\eta|^{-1-\frac{1}{d\mu}}, R_X^{d-3-\gamma}+R_X^{d-3-\gamma_1}\Bigr\}\\
		&\lesssim \min\Bigl\{R_X^{d-3-\frac{d-1}{d\mu}}|\eta|^{-1-\frac{1}{d\mu}}, R_X^{d-3-\gamma_1}\Bigr\},\quad \text{for a.e.}\ Y=(\xi,\eta)\in\mathbb R^{d+1}\times\mathbb R^3.
	\end{align*}
	By using  the change of variables $X-Y=(v,w)\in \mathbb R^{d+1}\times \mathbb R^3$, we find
	\begin{align*}
		\mathrm I_{\rm near}(X)&\lesssim\int_{|(v,w)|\le \frac {R_X}4}(|v|^2+|w|^2)^{-\frac{d+3}{2}}\min\Bigl\{R_X^{d-3-\frac{d-1}{d\mu}}|y-w|^{-1-\frac{1}{d\mu}}, R_X^{d-3-\gamma_1}\Bigr\}\,\mathrm dv\,\mathrm dw,
	\end{align*}
	from which and 
$$\int_{\mathbb R^{d+1}}(|v|^2+A^2)^{-\frac{d+3}{2}}\,\mathrm dv\lesssim_d A^{-2},\quad \forall \ A>0,		$$
we infer
\begin{equation}			\mathrm I_{\rm near}(X)	\lesssim \int_{|w|\le R_X/4}|w|^{-2}\min
\Bigl\{R_X^{d-3-\frac{d-1}{d\mu}}|y-w|^{-1-\frac{1}{d\mu}}, R_X^{d-3-\gamma_1}\Bigr\}\,\mathrm dw.\label{Eq.I-near-case1-est1} \end{equation}	
	Let
\begin{equation}	\ell_1:=R_X^{\frac{\gamma_1-\frac{d-1}{d\mu}}{1+\frac{1}{d\mu}}},\quad\text{then}\quad R_X^{d-3-\frac{d-1}{d\mu}}\ell_1^{-1-\frac{1}{d\mu}}=R_X^{d-3-\gamma_1}.
\label{S4eq6} \end{equation}	It follows from  \eqref{Eq.gamma-range} and $R_X\geq1$ that $\ell_1\geq 1$. We split the $w$-integral in \eqref{Eq.I-near-case1-est1} into the two regions: $|y-w|\le \ell_1$ and $|y-w|>\ell_1$. Then we get, by applying  Lemmas \ref{Lem.A2convolution-near} and \ref{Lem.A3convolution-far}, that
	\begin{align*}
		\mathrm I_{\rm near}(X)&\lesssim R_X^{d-3-\gamma_1}\int_{|y-w|\leq \ell_1}|w|^{-2}\,\mathrm dw+R_X^{d-3-\frac{d-1}{d\mu}}\int_{|y-w|> \ell_1}|w|^{-2}|y-w|^{-1-\frac{1}{d\mu}}\,\mathrm dw\\
		&\lesssim R_X^{d-3-\gamma_1}\ell_1^{1+\frac{1}{d\mu}}\langle y\rangle^{-\frac{1}{d\mu}}+R_X^{d-3-\frac{d-1}{d\mu}}\langle y\rangle^{-\frac{1}{d\mu}}\lesssim R_X^{d-3-\frac{d-1}{d\mu}}\langle y\rangle^{-\frac{1}{d\mu}},
	\end{align*}
	which ensures 
	\begin{equation}\label{Eq.I-near-est-case1}
		\mathrm I_{\rm near}(X)\lesssim \langle X\rangle^{d-3-\frac{d-1}{d\mu}}\langle y\rangle^{-\frac{1}{d\mu}},	\qquad \text{if}\ \ |x|\ge R_X/3.
	\end{equation}	
	
\noindent{\Large $\bullet$}	\underline{Case 2: $|y|\ge R_X/3$.} In this case,
we get, by a similar derivation of \eqref{Eq.Psi0near-est-case2}, that
\begin{align}
	\mathrm I_{\rm near}(X)&\lesssim R_X^{-1-\frac{1}{d\mu}}\int_{|Y-X|\leq \frac{R_X}4}|Y-X|^{-d-3}|\xi|^{d-3-\frac{d-1}{d\mu}}\,\mathrm dY\nonumber\\
	&\lesssim R_X^{-1-\frac{1}{d\mu}}\int_{|(v,w)|\leq \frac{R_X}4}\frac{|x-v|^{d-3-\frac{d-1}{d\mu}}}{(|v|^2+|w|^2)^{(d+3)/2}}\,\mathrm dv\,\mathrm dw\quad \mbox{by using}\ X-Y=(v,w)\nonumber\\
	&\lesssim R_X^{-1-\frac{1}{d\mu}}\int_{|v|\leq R_X/4}|v|^{-d}|x-v|^{d-3-\frac{d-1}{d\mu}}\,\mathrm dv\label{Eq.I-near-est-case2}\\
	&\lesssim R_X^{d-3-\frac{1}{\mu}}\int_{|\tilde v|\leq 1/4}|\tilde v|^{-d}|\tilde x-\tilde v|^{d-3-\frac{d-1}{d\mu}}\,\mathrm d\tilde v\quad \mbox{by using}\ (x,v)=\left(R_X\tilde x, R_X\tilde v\right)\nonumber\\
	&\lesssim  R_X^{d-3-\frac{1}{\mu}}\lesssim \langle X\rangle^{d-3-\frac{1}{\mu}}\quad\text{if}\ \ |y|\geq R_X/3.\nonumber
	\end{align}
Here we used the fact
\[\int_{\mathbb R^3}(A^2+|w|^2)^{-(d+3)/2}\,\mathrm dw\lesssim_d A^{-d}\quad \forall\ A>0,\]
and Lemma \ref{Lem.A1convolution-wholespace} with $m=d+1$, $p=d$, $q=-d+3+\frac{d-1}{d\mu}$ and $C_0=1.$  Such choices of $p$ and $q$ are admissible as
\begin{align*}
 &\mu<1/(d-2)\Rightarrow q>-d+3+\frac{(d-1)(d-2)}{d}>0,\andf \\
 &\mu>3/(3d-5)\Rightarrow m-(p+q)=(d-1)\left(1-\frac{1}{d\mu}\right)-1>\frac{2d-5}{3d}\gtrsim_d1. 
 \end{align*}

\noindent{\Large $\bullet$}	\underline{Case 3: $\max\{|x|,|y|\}<R_X/3$.} In this case, 
we get, by a similar derivation of \eqref{S4eq3}, that
\begin{equation}\label{Eq.I-near-est-case3}
	\mathrm I_{\rm near}(X)\lesssim 1\lesssim \langle X\rangle^{d-3-\frac{1}{\mu}},
\end{equation}
due to $\langle X\rangle \sim R_X\lesssim 1$. See also footnote \ref{footnote.gamma>d-2}. 

In view of \eqref{S4eq1}, combining the estimates \eqref{Eq.I-far-est}, \eqref{Eq.I-near-est-case1}, \eqref{Eq.I-near-est-case2} and \eqref{Eq.I-near-est-case3}, we obtain
\[\mathrm I(X)\lesssim	\langle X\rangle^{d-3-\frac{d-1}{d\mu}}\langle y\rangle^{-\frac{1}{d\mu}}+\langle X\rangle^{d-3-\frac{1}{\mu}}\lesssim \langle X\rangle^{d-3-\frac{d-1}{d\mu}}\langle y\rangle^{-\frac{1}{d\mu}},\]
from which, we deduce that  $\Psi_0\in C^1(\mathbb R^{d+4})$,   \eqref{Eq.nablaPsi0-upperbound} holds, and 
\begin{equation*}
	\nabla_X\Psi_0(X)=-\beta_d(d+2)\int_{\mathbb R^{d+4}}\frac{X-Y}{|X-Y|^{d+4}}F_\Omega(Y)\,\mathrm dY,\quad\forall\ X\in\mathbb R^{d+4}.
\end{equation*}
This completes the proof of Lemma \ref{Lem.nablaPsi0-upperbound}.\end{proof}

The next lemma is the key anisotropic estimate, which will be used in the verification of the elliptic map below. Although Lemma \ref{Lem.nablaPsi0-upperbound} provides a bound for the full gradient $\nabla_X\Psi_0$, this bound is not sufficiently small, after rescaling, to close all the differential inequalities in the definition of $\mathcal A^0$ (see \eqref{S2eq8}). Our key observation here is that the $x$-derivatives enjoy a better estimate than the full gradient. This improvement comes from the extra factor $|\xi|^{-1}$ obtained after moving the derivative from the Newtonian kernel to the density $F_\Omega$, and then using the estimate $|r\partial_r\Omega|\lesssim \Omega$ in the definition of $\mathcal B_{M'}$. In particular, the estimate below implies that, for the rescaled solution $\psi=\frac{a}{\mathfrak M(\Omega)}\psi_0$, the quantity $\langle r,z\rangle\,|\partial_r\psi|$ carries the small factor $1-\mu(d-2)$ as $\mu\uparrow (d-2)^{-1}$. This is precisely what will be used in the proof of Proposition \ref{Prop.G} to verify the condition $|\langle r,z\rangle\partial_r\psi|\le \frac1{10}\psi$ in the definition of $\mathcal A^0$.

\begin{lemma}[Pointwise estimate for $\nabla_x\Psi_0$]\label{Lem.nabla-x-Psi0-bound}
{\sl	
Under the assumptions in \eqref{S4eq0}, we have	\begin{equation}\label{Eq.nabla-x-Psi0-upperbound}
		|\nabla_x\Psi_0(X)|\leq C \langle X\rangle^{d-3-\frac{1}{\mu}},\quad\forall\ X=(x,y)\in \mathbb R^{d+1}\times\mathbb R^3,
	\end{equation}
	where $C>1$ is a constant depending only on $d, a, M'$.}
\end{lemma}

\begin{proof}
As  $\Omega\in \mathcal B_{M'}$,	for a.e. $Y=(\xi,\eta)\in \mathbb R^{d+1}\times\mathbb R^3$, we deduce from \eqref{Eq.BM'-def} and \eqref{Eq.F_Omega-def} that
	\begin{equation}\label{Eq.nabla-xi-F-Lem45}		|\nabla_\xi F_\Omega(Y)|\lesssim |\xi|^{-1}F_\Omega(Y).
	\end{equation}
For each $1\le i\le d+1$, using $\partial_{x_i}|X-Y|^{-d-2}=-\partial_{\xi_i}|X-Y|^{-d-2}$ and integration by parts in \eqref{Eq.Psi0-identity} for $\xi_i$ variables, we get 
	\begin{equation}\label{S4eq5}
	\partial_{x_i}\Psi_0(X)=\beta_d\int_{\mathbb R^{d+4}}|X-Y|^{-d-2}\partial_{\xi_i}F_\Omega(Y)\,\mathrm dY.
	\end{equation}
	Then to prove \eqref{Eq.nabla-x-Psi0-upperbound},
 it suffices to show that 
	\begin{equation}\label{Eq.nabla-x-Psi0-est}
	 \bar{\mathrm I}(X):=	\int_{\mathbb R^{d+4}} |X-Y|^{-d-2}|\xi|^{-1}F_\Omega(Y)\,dY
		\lesssim \langle X\rangle^{d-3-\frac{1}{\mu}}.
	\end{equation}
	In order to do so, as in \eqref{S4eq1}, we split $ \bar{\mathrm I}(X)$ as	\begin{equation}\label{S4eq4}
	\begin{split}	
	 \bar{\mathrm I}(X)&=\bar{\mathrm I}_{\rm near}(X)+\bar{\mathrm I}_{\rm far}(X)\with \\
	 		\bar{\mathrm I}_{\rm near}(X):&=\int_{|Y-X|\le R_X/4}|X-Y|^{-d-2}|\xi|^{-1} F_\Omega(Y)\,\mathrm dY,\\
		\bar{\mathrm I}_{\rm far}(X):&=
		\int_{|Y-X|>R_X/4}|X-Y|^{-d-2}|\xi|^{-1} F_\Omega(Y)\,\mathrm dY.
	\end{split} \end{equation}	
	We first deal with the far part. By \eqref{eq:F-bound-first} and a similar derivation of  \eqref{Eq.M(rho)-est}, for every $\rho>0$, we have
	\[
	\int_{|Y|\le \rho}|\xi|^{-1}F_\Omega(Y)\,\mathrm dY\lesssim\int_{s^2+t^2\le \rho^2}s^{2d-4-\frac{d-1}{d\mu}}t^{1-\frac1{d\mu}}\,\mathrm ds\,\mathrm dt\lesssim \rho^{2d-1-\frac1\mu}.
	\]
	Then along the same line as  \eqref{Eq.Psi0-far-est}, we find
\begin{equation}\label{Eq.I-far-est-Lem4.4}	\begin{split}	
		{\rm I}_{\rm far}(X)&\lesssim \sum_{k=0}^{\infty}(2^kR_X)^{-d-2}\int_{|Y-X|\le 2^{k+1}R_X}|\xi|^{-1}F_\Omega(Y)\,\mathrm dY\\
		&\lesssim \sum_{k=0}^{\infty}(2^kR_X)^{d-3-\frac1\mu}\lesssim R_X^{d-3-\frac1\mu}.
\end{split} \end{equation}		Here we used $d-3-\frac{1}{\mu}<-1$, which follows from $\mu<1/(d-2)$.
	
	It remains to handle  $\bar{\rm I}_{\rm near}$. As in the proof of Lemmas \ref{Lem.Psi0-upperbound} and  \ref{Lem.nablaPsi0-upperbound}, we divide the proof into three cases.
	
\noindent{\Large $\bullet$}	\underline{Case 1: $|x|\ge R_X/3$.} In this case,  there holds \eqref{S4eq2}.
Then
we get, by a similar derivation of 
\eqref{Eq.Psi0near-est-case1}, that
\begin{equation}\label{Eq.I-near-est1-Lem4.4}	\begin{split}	
		\bar{\rm I}_{\rm near}(X)
		&\lesssim R_X^{-1}R_X^{d-3-\frac{d-1}{d\mu}}
		\int_{|Y-X|\le R_X/4}|Y-X|^{-d-2}|\eta|^{-1-\frac1{d\mu}}\,\mathrm dY  \\
		&\lesssim R_X^{d-4-\frac{d-1}{d\mu}}
		\int_{|w|\le R_X/4}|w|^{-1}|y-w|^{-1-\frac1{d\mu}}\,dw
		\lesssim R_X^{d-3-\frac1\mu}.
\end{split} \end{equation}		Here we used Lemma \ref{Lem.A1convolution-wholespace} exactly as in \eqref{Eq.Psi0near-est-case1-3}.
	
\noindent{\Large $\bullet$}	\underline{Case 2: $|y|\ge R_X/3$.} In this case, if $|Y-X|\le R_X/4$, one has  $|\eta|\ge R_X/12$. Then it follows from  \eqref{eq:F-bound-first} that
	$$|\xi|^{-1}F_\Omega(Y)\lesssim R_X^{-1-\frac1{d\mu}}|\xi|^{d-4-\frac{d-1}{d\mu}}.$$ By using the change of variables $X-Y=(v,w)\in \mathbb R^{d+1}\times\mathbb R^3$ and
$$\int_{\mathbb R^3}(|v|^2+|w|^2)^{-(d+2)/2}\,\mathrm dw\lesssim |v|^{-d+1}, $$ we obtain
	\begin{equation}\label{Eq.I-near-est2-Lem4.4}
		\begin{split}	
\bar{\rm I}_{\rm near}(X)
		&\lesssim R_X^{-1-\frac1{d\mu}}
		\int_{|(v,w)|\le R_X/4}
		\frac{|x-v|^{d-4-\frac{d-1}{d\mu}}}
		{(|v|^2+|w|^2)^{\frac{d+2}{2}}}\,\mathrm dv\,\mathrm dw  \\
		&\lesssim R_X^{-1-\frac1{d\mu}}
		\int_{|v|\le R_X/4}|v|^{-d+1}|x-v|^{d-4-\frac{d-1}{d\mu}}\,\mathrm dv \\
		&\lesssim R_X^{d-3-\frac1\mu}
		\int_{|\widetilde v|\le 1/4}
		|\widetilde v|^{-d+1}|\widetilde x-\widetilde v|^{d-4-\frac{d-1}{d\mu}}\,\mathrm d\widetilde v
		\lesssim R_X^{d-3-\frac1\mu}. 
	\end{split} \end{equation}		Here we used the change of variables $(x,v)=\left(R_X\tilde x, R_X\tilde v\right)$ and Lemma \ref{Lem.A1convolution-wholespace} with $m=d+1$, $p=d-1$, $q=-d+4+\frac{d-1}{d\mu}$ and $C_0=1$. Such choices are admissible by \eqref{Eq.mu-range}. See also the discussion in Case 2 of Lemma \ref{Lem.nablaPsi0-upperbound}.
	
\noindent{\Large $\bullet$}	\underline{Case 3: $\max\{|x|,|y|\}<R_X/3$.} Then $R_X<3$. By \eqref{eq:F-bound-second}, we have $$|\xi|^{-1}F_\Omega(Y)\lesssim |\xi|^{d-4-\gamma}+|\xi|^{d-4-\gamma_1}. $$
Then along the same line to the derivation of \eqref{S4eq3}, we infer
\begin{equation}\label{Eq.I-near-est3-Lem4.4}
	\begin{split}	
	\bar{\rm I}_{\rm near}(X)&\lesssim \int_{|(v,w)|\le 1}\frac{|x-v|^{d-4-\gamma}+|x-v|^{d-4-\gamma_1}}{(|v|^2+|w|^2)^{\frac{d+2}{2}}}\,\mathrm dv\,\mathrm dw  \\
		&\lesssim \int_{|v|\le 1}|v|^{-d+1}\big(|x-v|^{d-4-\gamma}+|x-v|^{d-4-\gamma_1}\big)\,\mathrm dv\lesssim1.
	\end{split} \end{equation}	
		In the last step we used Lemma \ref{Lem.A1convolution-wholespace} with $m=d+1$, $p=d-1$ and $q=-d+4+\gamma$ or $q=-d+4+\gamma_1$. This is admissible because 
		\eqref{Eq.gamma>gamma1>d-3} holds and $m-(p+q)\geq d-2-\gamma\gtrsim_d1$, recalling footnote \ref{footnote.gamma>d-2}. Since $R_X<3$ and $d-3-\frac1\mu<0$, we also have $1\lesssim R_X^{d-3-\frac1\mu}$.
	
Combining the estimates \eqref{Eq.I-far-est-Lem4.4}, \eqref{Eq.I-near-est1-Lem4.4}, \eqref{Eq.I-near-est2-Lem4.4} and \eqref{Eq.I-near-est3-Lem4.4}, we arrive at  \eqref{Eq.nabla-x-Psi0-est}.
 This completes the proof of Lemma \ref{Lem.nabla-x-Psi0-bound}.
\end{proof}

\subsection{Hessian estimates for the Newtonian potential}\label{Subsec.Psi0-C^2-est}
In this subsection, we shall derive the second-order derivative estimates for the Newtonian potential $\Psi_0$. These estimates are the last analytic input needed to verify the Hessian condition in the definition of $\mathcal A^0$ for the normalized elliptic solution. As  the density function $F_\Omega(Y)=|\xi|^{d-3}|\eta|^{-1}\Omega(|\xi|,|\eta|)$ is singular near the coordinate axes, a direct estimate of all second derivatives by differentiating the Newtonian kernel twice is not convenient. We shall proceed in two steps. First, we handle the estimate of all second derivatives containing at least one $x$-derivative by moving one $x$-derivative from the kernel to the density function and using the definition of $\mathcal B_{M'}$. This is the content of Lemma \ref{Lem.nabla-Xx-Psi0-bound}. In Lemma \ref{Lem.nabla2Psi0-bound}, we recover the pure $y$-derivatives from the elliptic equation and the radial symmetry in the $y$-variables.

\begin{lemma}[Pointwise estimate for $\nabla_X\nabla_x\Psi_0$]\label{Lem.nabla-Xx-Psi0-bound}
{\sl Under the assumptions in \eqref{S4eq0},  	for all $X=(x,y)\in \R^{d+1}\times\R^3$ with $|x|\neq 0$,  one has	\begin{equation}\label{Eq.nabla-Xx-Psi0-upperbound}
		|\nabla_X\nabla_x\Psi_0(X)|\leq C\left(|x|^{d-3-\gamma}\,\mathbf1_{\{|X|\leq 1\}}+|x|^{d-3-\frac{d-1}{d\mu}}\langle y\rangle^{-1-\frac1{d\mu}}\right),
	\end{equation}
 where $C>1$ is a constant depending only on $d, a, M'$.}
\end{lemma}

\begin{proof}
	
By	differentiating \eqref{S4eq5} in $X$ and using \eqref{Eq.nabla-xi-F-Lem45}, to prove \eqref{Eq.nabla-Xx-Psi0-upperbound}, it suffices to show  that
	\begin{equation}\label{Eq.key-integral-Lem45}
	{\rm J}(X):=	\int_{\R^{d+4}}|X-Y|^{-d-3}|\xi|^{-1}F_\Omega(Y)\,\mathrm dY
		\lesssim |x|^{d-3-\gamma}\,\mathbf1_{\{|X|\le1\}}+|x|^{d-3-\frac{d-1}{d\mu}}\langle y\rangle^{-1-\frac1{d\mu}} .
	\end{equation}
	Indeed, the estimates below also justify the differentiation under the integral sign.
	
For each $X=(x,y)$, set $R_X:=1+|x|+|y|$, then we split 
${\rm J}(X)$ as 
\begin{equation}\label{S4eq5a}\begin{split}
{\rm J}(X)&=		{\rm J}_{\rm near}(X)+{\rm J}_{\rm far}(X)\with\\		
{\rm J}_{\rm near}(X):&=\int_{|Y-X|\le R_X/4}|X-Y|^{-d-3}|\xi|^{-1}F_\Omega(Y)\,\mathrm dY,\\
		{\rm J}_{\rm far}(X):&=\int_{|Y-X|> R_X/4}|X-Y|^{-d-3}|\xi|^{-1}F_\Omega(Y)\,\mathrm dY .
	\end{split}\end{equation}
		
It follows 	from \eqref{eq:F-bound-first} and \eqref{eq:F-bound-second} that
	\begin{align}\label{Eq.G-bound-first-Lem45}
		&|\xi|^{-1}F_\Omega(Y)\lesssim |\xi|^{d-4-\frac{d-1}{d\mu}}|\eta|^{-1-\frac1{d\mu}},\\
	&	|\xi|^{-1}F_\Omega(Y)\lesssim |\xi|^{d-4-\gamma}+|\xi|^{d-4-\gamma_1}.\label{Eq.G-bound-second-Lem45}	\end{align}

	We first deal with ${\rm J}_{\rm far}$. 
	By \eqref{Eq.G-bound-first-Lem45}, for every $\rho>0$, we have
	\[
	\int_{|Y|\le \rho}|\xi|^{-1}F_\Omega(Y)\,\mathrm dY
	\lesssim \int_{s^2+t^2\le \rho^2}s^{2d-4-\frac{d-1}{d\mu}}t^{1-\frac1{d\mu}}\,\mathrm ds\,\mathrm dt
	\lesssim \rho^{2d-1-\frac1\mu}.
	\]
	Combining this estimate with 
	$d-4-1/\mu<d-4-(d-2)=-2$, we obtain
	\begin{align}
		{\rm J}_{\rm far}(X)
		&\lesssim \sum_{k=0}^\infty (2^kR_X)^{-d-3}\int_{|Y-X|\le 2^{k+1}R_X}|\xi|^{-1}F_\Omega(Y)\,\mathrm dY \nonumber\\
		&\lesssim \sum_{k=0}^\infty (2^kR_X)^{d-4-\frac1\mu}
		\lesssim R_X^{d-4-\frac1\mu}\lesssim |x|^{d-3-\gamma}\,\mathbf1_{\{|X|\le1\}}+|x|^{d-3-\frac{d-1}{d\mu}}\langle y\rangle^{-1-\frac1{d\mu}}. \label{Eq.I-far-Lem45}
	\end{align}
	Here we  used
	$$
d-3-\frac{d-1}{d\mu}<0 \andf |x|+\langle y\rangle\le R_X \Rightarrow	R_X^{d-4-\frac1\mu}
\lesssim |x|^{d-3-\frac{d-1}{d\mu}}\langle y\rangle^{-1-\frac1{d\mu}}.$$
		
	It remains to handle the term ${\rm J}_{\rm near}$. We divide the proof into the following three cases.
	
\noindent{\Large $\bullet$}	\underline{Case 1: $|x|\ge R_X/3$.} In this case, if $|Y-X|\le R_X/4$, then $|\xi|\ge R_X/12$. Notice that $d-4-\frac{d-1}{d\mu}<0$ and $\gamma>\gamma_1$,  we deduce from \eqref{Eq.G-bound-first-Lem45} and \eqref{Eq.G-bound-second-Lem45} that
	\[
	|\xi|^{-1}F_\Omega(Y)\lesssim \min\Bigl\{R_X^{d-4-\frac{d-1}{d\mu}}|\eta|^{-1-\frac1{d\mu}},\,R_X^{d-4-\gamma_1}\Bigr\}.
	\]
Then along the same line as \eqref{Eq.I-near-case1-est1}, we find
	\begin{align}
		{\rm J}_{\rm near}(X)		&\lesssim \int_{|w|\le R_X/4}|w|^{-2}		\min\Bigl\{R_X^{d-4-\frac{d-1}{d\mu}}|y-w|^{-1-\frac1{d\mu}},\,R_X^{d-4-\gamma_1}\Bigr\}\,\mathrm dw . \label{Eq.I-near-case1-start-Lem45}
	\end{align}
	For $\ell_1$ given by \eqref{S4eq6},  we have
	\[
	 R_X^{d-4-\frac{d-1}{d\mu}}\ell_1^{-1-\frac1{d\mu}}=R_X^{d-4-\gamma_1}.
	\]
	By \eqref{Eq.gamma-range}, $\ell_1\ge1$.  Observing that
	 $$\gamma_1<d-2 \andf \mu<1/(d-2)
	\Rightarrow
	0<\gamma_1-\frac{d-1}{d\mu}<
	\frac{d-2}{d}<1+\frac1{d\mu},
	$$ we also have $\ell_1\le R_X$. 	
	
	Next we decompose  the integral domain  in \eqref{Eq.I-near-case1-start-Lem45} into 
	\begin{align*}
	&{ D}_1:=\bigl\{ w\in \R^{3}:\ |w|\le R_X/4, |y-w|\le \ell_1\bigr\}\andf \\
	&{ D}_2:=\bigl\{ w\in \R^{3}:\ |w|\le R_X/4, |y-w|> \ell_1\bigr\}.  \end{align*}

If  $|y|\le2\ell_1$,  we have
\[\int_{{ D}_1}|w|^{-2}\,\mathrm dw
\leq CR_X, \]
which together with  $\ell_1\ge1$ and $\langle y\rangle\lesssim \ell_1$, ensures that
\[\int_{{ D}_1}|w|^{-2}\,\mathrm dw
\lesssim R_X\ell_1^{1+\frac1{d\mu}}\langle y\rangle^{-1-\frac1{d\mu}}.\]
 If $|y|>2\ell_1$ and $w\in { D}_1,$ we have $|w|\sim |y|$ and 
 \[\int_{{D}_1}|w|^{-2}\,\mathrm dw
\leq C|y|^{-2}\ell_1^3. \]
Note that $1+\frac1{d\mu}<2,$ for $w\in { D}_1,$  we have  $|y|\le R_X/4+\ell_1\lesssim R_X,$ so there holds
\begin{equation*}
		\int_{ {D}_1}|w|^{-2}\,\mathrm dw
		\lesssim R_X\ell_1^{1+\frac1{d\mu}}\langle y\rangle^{-1-\frac1{d\mu}},
	\end{equation*}	
from which and  $R_X^{d-4-\gamma_1}\ell_1^{1+\frac1{d\mu}}=R_X^{d-4-\frac{d-1}{d\mu}}$, we infer
	\begin{align}
		R_X^{d-4-\gamma_1}\int_{|w|\le R_X/4, |y-w|\le \ell_1}|w|^{-2}\,\mathrm dw \lesssim R_X^{d-3-\frac{d-1}{d\mu}}\langle y\rangle^{-1-\frac1{d\mu}}. \label{Eq.case1-w-near-Lem45}
	\end{align}
	
	Next we handle the contribution of ${D}_2.$ If $|y|\le2$, due to $|y-w|>\ell_1\ge1,$	we have
	\[ 
	\int_{{D}_2}|w|^{-2}|y-w|^{-1-\frac1{d\mu}}\,\mathrm dw		\leq CR_X.\]	
	If $|y|>2$, we split the integral domain ${D}_2$ into three parts.  In the case where $w\in {D}_2$ and 	
	 $|w|\le |y|/2$, we have $|y-w|\sim |y|$,  and
	 \[ 
	\int_{{D}_2\cap\{|w|\leq|y|/2\} }|w|^{-2}|y-w|^{-1-\frac1{d\mu}}\,\mathrm dw	 
	 \leq C|y|^{-1-\frac1{d\mu}}\int_{|w|\le R_X/4}|w|^{-2}\,\mathrm dw\lesssim R_X|y|^{-1-\frac1{d\mu}}. \]
In case where $w\in { D}_2$, 	
	  $|w|>|y|/2$ and $|y-w|\le |y|/2$, we have $|w|\sim |y|$ and $|y|\lesssim R_X$, so  
	  \begin{align*}
	&\int_{{ D}_2\cap\{w: \ |w|>|y|/2,\ |y-w|\leq|y|/2\} }|w|^{-2}|y-w|^{-1-\frac1{d\mu}}\,\mathrm dw  	\\
	&\leq C|y|^{-2}\int_{|y-w|\le |y|/2}|y-w|^{-1-\frac1{d\mu}}\,\mathrm dw
	\lesssim |y|^{-2}|y|^{2-\frac1{d\mu}}
	=|y|^{-\frac1{d\mu}}
	\lesssim R_X|y|^{-1-\frac1{d\mu}}.
	\end{align*}
	Finally, it is easy to observe that if  $|y-w|>|y|/2$,  one has  $|y-w|^{-1-\frac1{d\mu}}\lesssim |y|^{-1-\frac1{d\mu}}$  and	\[
	\int_{{ D}_2\cap\{w: \ |w|>|y|/2,\ |y-w|>|y|/2\} }|w|^{-2}|y-w|^{-1-\frac1{d\mu}}\,\mathrm dw  
	\leq CR_X|y|^{-1-\frac1{d\mu}}.\]
	As a consequence, we obtain	\begin{equation*}
		\int_{{ D}_2}|w|^{-2}|y-w|^{-1-\frac1{d\mu}}\,\mathrm dw		\lesssim R_X\langle y\rangle^{-1-\frac1{d\mu}},
	\end{equation*}
	which implies that
	\begin{align}
		R_X^{d-4-\frac{d-1}{d\mu}}\int_{{D}_2}|w|^{-2}|y-w|^{-1-\frac1{d\mu}}\,\mathrm dw\lesssim R_X^{d-3-\frac{d-1}{d\mu}}\langle y\rangle^{-1-\frac1{d\mu}}. \label{Eq.case1-w-far-Lem45}
	\end{align}
	
Combining the estimates \eqref{Eq.I-near-case1-start-Lem45}, \eqref{Eq.case1-w-near-Lem45} and \eqref{Eq.case1-w-far-Lem45}, and using $R_X\sim |x|$, we achieve
	\begin{equation}\label{Eq.I-near-case1-Lem45}
		{\rm J}_{\rm near}(X)\lesssim |x|^{d-3-\frac{d-1}{d\mu}}\langle y\rangle^{-1-\frac1{d\mu}},\qquad |x|\ge R_X/3 .
	\end{equation}
	
\noindent{\Large $\bullet$}	\underline{Case 2: $|y|\ge R_X/3$.} In this case, if $|Y-X|\le R_X/4$, then $|\eta|\ge R_X/12,$ which together with  \eqref{Eq.G-bound-first-Lem45} ensures  $|\xi|^{-1}F_\Omega(Y)\lesssim R_X^{-1-\frac1{d\mu}}|\xi|^{d-4-\frac{d-1}{d\mu}}$. So we get, by a similar derivation  of \eqref{Eq.I-near-est-case2}, that
	\begin{align}
		{\rm J}_{\rm near}(X)
		&\lesssim R_X^{-1-\frac1{d\mu}}\int_{|(v,w)|\le R_X/4}\frac{|x-v|^{d-4-\frac{d-1}{d\mu}}}{(|v|^2+|w|^2)^{\frac{d+3}{2}}}\,\mathrm dv\,\mathrm dw \nonumber\\
		&\lesssim R_X^{-1-\frac1{d\mu}}\int_{|v|\le R_X/4}|v|^{-d}|x-v|^{d-4-\frac{d-1}{d\mu}}\,\mathrm dv . \label{Eq.I-near-case2-start-Lem45}
	\end{align}
	We split the integral  domain above into three parts. In case $|v|\le |x|/2$, we have $|x-v|\sim |x|$, and 
	\[\int_{|v|\le R_X/4, |v|\le |x|/2 }|v|^{-d}|x-v|^{d-4-\frac{d-1}{d\mu}}\,\mathrm dv	
	\lesssim |x|^{d-4-\frac{d-1}{d\mu}}\int_{|v|\le |x|/2}|v|^{-d}\,\mathrm dv\lesssim |x|^{d-3-\frac{d-1}{d\mu}}. \]
	In the case where $|x-v|\le |x|/2$, we have $|v|\sim |x|$, and 
	\begin{align*}
	\int_{|v|\le R_X/4, |x-v|\le |x|/2 }|v|^{-d}|x-v|^{d-4-\frac{d-1}{d\mu}}\,\mathrm dv\lesssim&\, |x|^{-d}\int_{|x-v|\le |x|/2}|x-v|^{d-4-\frac{d-1}{d\mu}}\,\mathrm dv\\
	\lesssim&\, |x|^{-d}|x|^{2d-3-\frac{d-1}{d\mu}}
	\lesssim|x|^{d-3-\frac{d-1}{d\mu}}.
	\end{align*}
	On the remaining region $\mathcal R:=\{|v|\le R_X/4,\ |v|>|x|/2,\ |x-v|>|x|/2\}$, we further split
	\[
	\mathcal R=\big(\mathcal R\cap\{|v|\le 2|x|\}\big)\cup\big(\mathcal R\cap\{|v|>2|x|\}\big).
	\]
	On $\mathcal R\cap\{|v|\le 2|x|\}$, we have $|v|\sim |x|$ and $|x-v|>|x|/2$. Due to $d-4-(d-1)/(d\mu)<0$, we have $|x-v|^{d-4-\frac{d-1}{d\mu}}\lesssim |x|^{d-4-\frac{d-1}{d\mu}},$ so that
	\[
	\int_{\mathcal R\cap\{|v|\le 2|x|\}}|v|^{-d}|x-v|^{d-4-\frac{d-1}{d\mu}}\,\mathrm dv
	\lesssim |x|^{-d}|x|^{d-4-\frac{d-1}{d\mu}}|\{|v|\le2|x|\}|
	\lesssim |x|^{d-3-\frac{d-1}{d\mu}}.
	\]
	On $\mathcal R\cap\{|v|>2|x|\}$, we have $|x-v|\sim |v|$ and 
	\begin{align*}
		\int_{\mathcal R\cap\{|v|>2|x|\}}|v|^{-d}|x-v|^{d-4-\frac{d-1}{d\mu}}\,\mathrm dv
		&\lesssim \int_{2|x|<|v|\le R_X/4}|v|^{-4-\frac{d-1}{d\mu}}\,\mathrm dv \\
		&\lesssim \int_{2|x|}^{+\infty}s^{d-4-\frac{d-1}{d\mu}}\,\mathrm ds\lesssim |x|^{d-3-\frac{d-1}{d\mu}}.
	\end{align*}
	Here we have used $d-3-(d-1)/(d\mu)<0$.
	
	Therefore, we obtain
	\begin{equation}\label{Eq.v-convolution-case2-Lem45}
		\int_{|v|\le R_X/4}|v|^{-d}|x-v|^{d-4-\frac{d-1}{d\mu}}\,\mathrm dv
		\lesssim |x|^{d-3-\frac{d-1}{d\mu}} .
	\end{equation}
	As $|y|\ge R_X/3$, we have $R_X\sim\langle y\rangle$. Then by substituting \eqref{Eq.v-convolution-case2-Lem45} into 	 \eqref{Eq.I-near-case2-start-Lem45}, we get
	\begin{equation}\label{Eq.I-near-case2-Lem45}
		{\rm J}_{\rm near}(X)\lesssim |x|^{d-3-\frac{d-1}{d\mu}}\langle y\rangle^{-1-\frac1{d\mu}},\qquad |y|\ge R_X/3 .
	\end{equation}

\noindent{\Large $\bullet$}	\underline{Case 3: $\max\{|x|,|y|\}<R_X/3$.} In this case, $R_X<3$. Then we deduce from \eqref{Eq.G-bound-second-Lem45} that
	\begin{equation}\label{Eq.I-near-case3-start-Lem45}	\begin{split}
		{\rm J}_{\rm near}(X)
		&\lesssim \int_{|(v,w)|\le1}\frac{|x-v|^{d-4-\gamma}+|x-v|^{d-4-\gamma_1}}{(|v|^2+|w|^2)^{\frac{d+3}{2}}}\,\mathrm dv\,\mathrm dw \\
		&\lesssim \int_{|v|\le1}|v|^{-d}\big(|x-v|^{d-4-\gamma}+|x-v|^{d-4-\gamma_1}\big)\,\mathrm dv . 
	\end{split}\end{equation}
Due to $\frac{d-1}{d\mu}<\gamma_1<\gamma<d-2$ and $\mu<1/(d-2),$	we get, by a similar derivation   of \eqref{Eq.v-convolution-case2-Lem45}, but with $\frac{d-1}{d\mu}$ being replaced by  $\gamma$ or $\gamma_1$, that
	\begin{align*}
	\int_{|v|\le1}|v|^{-d}|x-v|^{d-4-\gamma}\,\mathrm dv\lesssim |x|^{d-3-\gamma},
	\qquad\int_{|v|\le1}|v|^{-d}|x-v|^{d-4-\gamma_1}\,\mathrm dv\lesssim |x|^{d-3-\gamma_1}.
	\end{align*}
	  As in the present case $|x|<R_X/3<1$, and due to $\gamma>\gamma_1$, we have $|x|^{d-3-\gamma_1}\lesssim |x|^{d-3-\gamma}$. Therefore,
	  $${\rm J}_{\rm near}(X)\lesssim |x|^{d-3-\gamma}.	$$
	  If $|X|\le1$, this is exactly the first term on the right-hand side of \eqref{Eq.key-integral-Lem45}. If $|X|>1$, then, as above, the assumption $\max\{|x|,|y|\}<R_X/3$ implies that $|x|\gtrsim1$, and then
	  $$|x|^{d-3-\gamma}\lesssim1\lesssim |x|^{d-3-\frac{d-1}{d\mu}}\langle y\rangle^{-1-\frac1{d\mu}}. $$
	We  thus obtain
	\begin{equation}\label{Eq.I-near-case3-Lem45}
		{\rm J}_{\rm near}(X)\lesssim |x|^{d-3-\gamma}\,\mathbf1_{\{|X|\le1\}}+|x|^{d-3-\frac{d-1}{d\mu}}\langle y\rangle^{-1-\frac1{d\mu}} .
	\end{equation}
	
Combining  the estimates \eqref{Eq.I-far-Lem45}, \eqref{Eq.I-near-case1-Lem45}, \eqref{Eq.I-near-case2-Lem45} and \eqref{Eq.I-near-case3-Lem45}, we obtain \eqref{Eq.key-integral-Lem45}.  This completes the proof of Lemma \ref{Lem.nabla-Xx-Psi0-bound}.
\end{proof}

We now upgrade the mixed derivatives  estimate in Lemma \ref{Lem.nabla-Xx-Psi0-bound} to the full
Hessian estimate. The only derivatives which are not covered by \eqref{Eq.nabla-Xx-Psi0-upperbound}  are the pure
$y$-derivatives. Instead of estimating them by differentiating the singular integral
with respect to the $y$ variables, we use the equation
$-\Delta_X\Psi_0=F_\Omega$. As Lemma \ref{Lem.nabla-Xx-Psi0-bound} already controls
$\Delta_x\Psi_0$, and the source term $F_\Omega$ satisfies the same pointwise
bound, we obtain the corresponding estimate for $\Delta_y\Psi_0$. Finally, the radial
symmetry of $\Psi_0$ in the $y$ variable converts the bound for $\Delta_y\Psi_0$ into
the estimates for $z^{-1}\partial_z\psi_0$ and $\partial_z^2\psi_0$, which leads to  the
control of the pure $y$-Hessian.

\begin{lemma}[Pointwise estimate for $\nabla_X^2\Psi_0$]\label{Lem.nabla2Psi0-bound}
{\sl	 Under the assumptions in \eqref{S4eq0},  	we have  $\Psi_0\in C^1(\R^{d+4})\cap C^2(\R^{d+4}\setminus\{|x|=0\}),$ and for all $X=(x,y)\in \R^{d+1}\times\R^3$ with $|x|\neq 0$,  there holds	\begin{equation}\label{Eq.nabla^2Psi0-upperbound}
		|\nabla_X^2\Psi_0(X)|\leq C\left(|x|^{d-3-\gamma}\,\mathbf1_{\{|X|\leq 1\}}+|x|^{d-3-\frac{d-1}{d\mu}}\langle y\rangle^{-1-\frac1{d\mu}}\right),
	\end{equation}
	where $C>1$ is a constant depending only on $d, a, M'$.}
\end{lemma}
\begin{proof} We first deduce from 
	 Lemma \ref{Lem.nabla-Xx-Psi0-bound} that
	\begin{equation}\label{Eq.Delta-x-bound-Lem46}
		|\Delta_x\Psi_0(X)|\lesssim |x|^{d-3-\gamma}\,\mathbf1_{\{|X|\le1\}}+|x|^{d-3-\frac{d-1}{d\mu}}\langle y\rangle^{-1-\frac1{d\mu}} .
	\end{equation}

	We observe that $F_\Omega$ shares the same type of pointwise bound as \eqref{Eq.Delta-x-bound-Lem46}. Indeed, if $|X|\le1$, then \eqref{eq:F-bound-second} implies $F_\Omega(X)\lesssim |x|^{d-3-\gamma}$. If  $|X|>1$,  the bound $F_\Omega(X)\lesssim |x|^{d-3-\frac{d-1}{d\mu}}|y|^{-1-\frac1{d\mu}}$ ensures that  $F_\Omega(X)$ can also be controlled by the right-hand side of \eqref{Eq.Delta-x-bound-Lem46}	
	in the case when $|y|\gtrsim1.$ For $|y|\lesssim1,$ the bound $$F_\Omega(X)\lesssim |x|^{d-3-\gamma_1}\lesssim |x|^{d-3-\frac{d-1}{d\mu}}$$
	 follows from $\gamma_1>\frac{d-1}{d\mu}$ and $|x|\gtrsim1$. Therefore,
	\begin{equation}\label{Eq.F-bound-Lem46}
		F_\Omega(X)\lesssim |x|^{d-3-\gamma}\,\mathbf1_{\{|X|\le1\}}+|x|^{d-3-\frac{d-1}{d\mu}}\langle y\rangle^{-1-\frac1{d\mu}} .
	\end{equation}

As $\Delta_y\Psi_0=-F_\Omega-\Delta_x\Psi_0,$ 	by combining \eqref{Eq.Delta-x-bound-Lem46} with \eqref{Eq.F-bound-Lem46}, we find
	\begin{equation}\label{Eq.Delta-y-bound-Lem46}
		|\Delta_y\Psi_0(X)|\lesssim |x|^{d-3-\gamma}\,\mathbf1_{\{|X|\le1\}}+|x|^{d-3-\frac{d-1}{d\mu}}\langle y\rangle^{-1-\frac1{d\mu}} .
	\end{equation}
	
	Next we use the radial symmetry of $\Psi_0(x,y)$ in the $y$ variables. Fix $x\neq0$ and write $\Psi_0(x,y)=\psi_0(|x|,z)$ with $z=|y|$. If $z>0$, then
	\begin{equation}\label{S4eq7}	
	\Delta_y\Psi_0(x,y)=\pa_z^2\psi_0(|x|,z)+\frac2z\pa_z\psi_0(|x|,z).
	\end{equation}
		We first bound $z^{-1}\pa_z\psi_0$.
	Due to $d-3-\frac{d-1}{d\mu}<0$ and $\langle X\rangle\ge |x|,$	 in the case where $z\ge1/4$, we deduce from Lemma \ref{Lem.nablaPsi0-upperbound} that
	\[
	\left|\frac1z\pa_z\psi_0(|x|,z)\right|
	\lesssim \langle X\rangle^{d-3-\frac{d-1}{d\mu}}\langle z\rangle^{-1-\frac1{d\mu}}
	\lesssim |x|^{d-3-\frac{d-1}{d\mu}}\langle z\rangle^{-1-\frac1{d\mu}}.
	\]
	 When $0<z<1/4$, we get, by multiplying \eqref{S4eq7} by $z^2$ and then integrating the resulting equality over $[0,z],$   that
	\[
	z^2\pa_z\psi_0(|x|,z)=\int_0^z \tau^2\Delta_y\Psi_0(x,\tau)\,\mathrm d\tau,
	\]
from which and  \eqref{Eq.Delta-y-bound-Lem46}, we infer
	\[
	\left|\frac1z\pa_z\psi_0(|x|,z)\right|
	\lesssim |x|^{d-3-\gamma}\,\mathbf1_{\{|x|\le1\}}+|x|^{d-3-\frac{d-1}{d\mu}}\langle z\rangle^{-1-\frac1{d\mu}}.
	\]
	Here, if $|X|>1$ but $0<z<1/4$, then either $|x|$ is bounded from below, in which case the local term is harmless and is absorbed by the second term, or the local term is absent. Therefore,
	\begin{equation}\label{Eq.zinv-psiz-bound-Lem46}
		\left|\frac1z\pa_z\psi_0(|x|,z)\right|
		\lesssim |x|^{d-3-\gamma}\,\mathbf1_{\{|X|\le1\}}+|x|^{d-3-\frac{d-1}{d\mu}}\langle z\rangle^{-1-\frac1{d\mu}} .
	\end{equation}
	
	In view of \eqref{S4eq7}, by
	combining \eqref{Eq.Delta-y-bound-Lem46} and \eqref{Eq.zinv-psiz-bound-Lem46}, we achieve 
	\[
	|\pa_z^2\psi_0(|x|,z)|\lesssim |x|^{d-3-\gamma}\,\mathbf1_{\{|X|\le1\}}+|x|^{d-3-\frac{d-1}{d\mu}}\langle z\rangle^{-1-\frac1{d\mu}} .
	\]
	Finally, the Hessian of a radial function in the $y$ variable is given by
	\[
	\pa_{y_k}\pa_{y_\ell}\Psi_0(x,y)=\Bigl(\pa_z^2\psi_0-\frac1z\pa_z\psi_0\Bigr)\frac{y_ky_\ell}{z^2}+\frac1z\pa_z\psi_0\,\delta_{k\ell},\qquad 1\le k,\ell\le3,
	\]
	for $z>0$, and the estimate at $z=0$ follows by taking the limit. This proves \eqref{Eq.nabla^2Psi0-upperbound} for the pure $y$-derivatives. Together with Lemma \ref{Lem.nabla-Xx-Psi0-bound}, we complete the proof of Lemma \ref{Lem.nabla2Psi0-bound}.
\end{proof}

\subsection{\texorpdfstring{A radial approximation of $\Psi_0$}{A radial approximation of Psi0}}\label{Subsec.Psi0-radial}



For each $\rho\ge 0$, we define $\Phi_0(\rho)$ by truncating the Newtonian potential at spatial infinity (see \eqref{Eq.Phi0-def}). The function $\Phi_0$ is radial and decreasing in $\rho$, and it captures the only contribution which may grow like $1/(1-(d-2)\mu)$ as $\mu\uparrow 1/(d-2)$. The purpose of this subsection is to prove the estimate \eqref{Eq.Psi0-Phi0-est} concerning  the difference between $\Psi_0(X)$ and $\Phi_0(1+2|X|).$ 
This almost radial monotonicity estimate will be used in the proof of Proposition \ref{Prop.G} to show that the re-scaled solution $\psi$ remains below $a_1$, provided that $\mu$ is sufficiently close to $1/(d-2)$.

\begin{lemma}\label{Lem.Psi0-Phi0-difference-est}
{\sl	 For each $\rho\geq 0$, we define
	\begin{equation}\label{Eq.Phi0-def}
		\Phi_0(\rho):=\beta_d\int_{Y\in\R^{d+4}, |Y|\geq\rho}|Y|^{-d-2}|\xi|^{d-3}|\eta|^{-1}\Omega(|\xi|,|\eta|)\,\mathrm dY,
	\end{equation}
	where the constant $\beta_d$ is the same as that in \eqref{Eq.Psi0-def},	
	 $Y=(\xi,\eta)\in\R^{d+1}\times\R^3.$ Then  under the assumptions in \eqref{S4eq0},   we have	\begin{equation}\label{Eq.Psi0-Phi0-est}
		\left|\Psi_0(X)-\Phi_0(2|X|+1)\right|\leq C\langle X\rangle^{d-2-\frac1\mu},\quad\forall\ X\in\R^{d+4},
	\end{equation}
	where $C>1$ is a constant depending only on $d,a,M'$.}
\end{lemma}
\begin{proof}
	For each $X\in \R^{d+4}$, we denote $\text{D}(X):=|\Psi_0(X)-\Phi_0(2|X|+1)|$, then by virtue of 
	\eqref{Eq.Psi0-identity} and \eqref{Eq.Phi0-def},  we have 
	\begin{align*}
	\text{D}(X)\leq & \beta_d\big(\text{D}_1(X)+\text{D}_2(X)\big)\with\\
			\text{D}_1(X):&=\int_{|Y|<2|X|+1}|X-Y|^{-d-2}F_\Omega(Y)\,\mathrm dY,\\ \text{D}_2(X):&=\int_{|Y|\geq2|X|+1}\left||X-Y|^{-d-2}-|Y|^{-d-2}\right|F_\Omega(Y)\,\mathrm dY.
	\end{align*}
	
We first handle the estimate of $\text{D}_1(X)$. We split $\text{D}_{1}(X)=\text{D}_{1,\text{near}}(X)+\text{D}_{1,\text{far}}(X)$, where
	\begin{align*}
		\text{D}_{1,\text{near}}(X):&=\int_{|Y|<2|X|+1, |Y-X|\leq R_X/4}|X-Y|^{-d-2}F_\Omega(Y)\,\mathrm dY,\\
		\text{D}_{1,\text{far}}(X):&=\int_{|Y|<2|X|+1, |Y-X|> R_X/4}|X-Y|^{-d-2}F_\Omega(Y)\,\mathrm dY.
	\end{align*}
	Recall that $R_X=1+|x|+|y|\leq 1+2|X|.$ If $|Y-X|\leq R_X/4 $,  then $|Y|\leq |X|+R_X/4<2|X|+1$. Then in view of \eqref{S4eq8}, we get, by applying \eqref{Eq.Psi0near-est-case1}, \eqref{Eq.Psi0near-est-case2} and \eqref{Eq.Psi0near-est-case3}, that
	 $$\text{D}_{1,\text{near}}(X)=\Psi_{0, \text{near}}(X)\lesssim \langle X\rangle^{d-2-\frac1{\mu}}.$$   Whereas, it follows from \eqref{Eq.M(rho)-def} and \eqref{Eq.M(rho)-est} that
	\begin{equation}\label{Eq.D1far-est}
		\text{D}_{1,\text{far}}(X)\lesssim R_X^{-d-2}M(2|X|+1)\lesssim \langle X\rangle^{-d-2}(2|X|+1)^{2d-\frac1\mu}\lesssim \langle X\rangle^{d-2-\frac1\mu}.
	\end{equation}
	We thus obtain 
\begin{equation*} \text{D}_1(X)\lesssim \langle X\rangle^{d-2-\frac1{\mu}}.\end{equation*}

	Next we handle ${\rm D}_2(X)$. Since 
	\begin{align*} 
	|Y|\geq 2|X|+1\Rightarrow |Y-X|\geq|Y|/2 \andf \left||X-Y|^{-d-2}-|Y|^{-d-2}\right|\lesssim |X|\cdot|Y|^{-d-3}, \end{align*}
	we find
	\begin{align*}
		\text{D}_{2}(X)\lesssim |X|\int_{|Y|\geq 2|X|+1}|Y|^{-d-3}F_{\Omega}(Y)\,\mathrm dY,
	\end{align*}
from which, together with \eqref{Eq.M(rho)-def} and \eqref{Eq.M(rho)-est}, we infer
	\begin{align*}
		\text{D}_{2}(X)\lesssim &\,|X|\sum_{k=0}^\infty \int_{2^k(2|X|+1)\leq |Y|< 2^{k+1}(2|X|+1)}|Y|^{-d-3}F_{\Omega}(Y)\,\mathrm dY\\ \lesssim&\, |X|\sum_{k=0}^\infty (2^k(2|X|+1))^{-d-3}M(2^{k+1}(2|X|+1))\lesssim |X|\sum_{k=0}^\infty (2^k(2|X|+1))^{d-3-\frac1\mu}\\
\lesssim&\, |X|(2|X|+1)^{d-3-\frac1\mu}\lesssim\langle X\rangle^{d-2-\frac1\mu},
	\end{align*}
	as $d-3-\frac1\mu<-1$, which follows from  \eqref{Eq.mu-range} with $\mu<1/(d-2)$. 
	
	Combining the above estimates, we obtain \eqref{Eq.Psi0-Phi0-est}. This proves Lemma \ref{Lem.Psi0-Phi0-difference-est}.
\end{proof}

\subsection{Proof of Proposition \ref{Prop.G}}\label{Subsec.Proof-Prop-G}

In this subsection, we  present the proof of  Proposition \ref{Prop.G}, which basically combines all the pointwise estimates obtained in  Subsections \ref{Subsec.Psi0-C^1-est}, \ref{Subsec.Psi0-C^2-est} and \ref{Subsec.Psi0-radial}. 

\begin{proof}[Proof of Proposition \ref{Prop.G}]
	Let $M'>1$ and $\Omega\in \mathcal B_{M'}$, and let $\Psi_0$ be defined by
	\eqref{Eq.Psi0-def}. By Lemma \ref{Lem.Psi_0-bound}, there exists $\psi_0=\psi_0(r,z)\in C^1(\overline{\Pi_+})\cap C^2(\Pi_+)$	such that $\Psi_0(x,y)=\psi_0(|x|,|y|)$	for all $(x,y)\in\mathbb R^{d+1}\times\mathbb R^3,$ and $\psi_0$ solves	\eqref{Eq.psi-elliptic} in $\Pi_+$. Recall that $\mathfrak M(\Omega)=\Psi_0(0)=\psi_0(0,0)>0$.	We define
		\begin{equation}\label{S4eq9}
\psi(r,z):=\frac a{\mathfrak M(\Omega)}\psi_0(r,z),\quad  \forall \ (r,z)\in\overline{\Pi_+}. \end{equation}
Then $\psi\in C^1(\overline{\Pi_+})\cap C^2(\Pi_+)$, $\psi$ solves \eqref{Eq.elliptic-scaling} in $\Pi_+$ and $\psi(0,0)=a$. Furthermore, due to $d-2-\frac1\mu<0$, it follows from Lemma \ref{Lem.Psi_0-bound} that  $\lim_{r+z\to+\infty}\psi(r,z)=0$. 
	
We first consider the two-sided estimate for \(\psi\). 
By \eqref{Eq.Psi0-lower-upper-bound}, \eqref{S2eq2} and \eqref{S4eq9}, we have
	\[\frac{a}{(M'')^2}\langle r,z\rangle^{d-2-\frac1\mu}\le	\psi(r,z)\le a(M'')^2 \langle r,z\rangle^{d-2-\frac1\mu},	\qquad\forall\  (r,z)\in\overline{\Pi_+}.\]
Since \eqref{Eq.a-range}, $1/d^2<a<1<d^2$, we thus obtain
	\begin{equation}\label{Eq.psi-twosided-bound}
		\frac{1}{(dM'')^2}\langle r,z\rangle^{d-2-\frac1\mu}\le	\psi(r,z)\le (dM'')^2\langle r,z\rangle^{d-2-\frac1\mu},	\qquad \forall\  (r,z)\in\overline{\Pi_+}.
	\end{equation}
	This gives the two-sided bound required in the definition of
	\(\mathcal A_{(dM'')^2}\) (see \eqref{defam}).
	
	It remains to verify that $\psi\in\mathcal A^0$. Positivity follows directly from the positivity of $\Psi_0$. We next prove the upper bound $\psi\le a_1$. Let $\Phi_0=\Phi_0(\rho)$ be the radial function defined in \eqref{Eq.Phi0-def}. 
	As $\rho\mapsto\Phi_0(\rho)$ is non-increasing, we have $$\Phi_0(2|X|+1)\le \Phi_0(\rho=0)=\Psi_0(X=0)=\mathfrak M(\Omega),\quad \forall\ X\in\R^{d+4}.$$
Note that $d-2-\frac1\mu<0$ and	$\langle X\rangle^{d-2-\frac1\mu}\leq 1.$ Then we deduce from \eqref{Eq.Psi0-Phi0-est} that 
	$$\Psi_0(X)\leq \mathfrak M(\Omega)+C,\quad \forall \ X\in\R^{d+4},  $$
	from which and \eqref{S2eq2}, we infer
	\[\psi(r,z)=a\frac{\Psi_0(X)}{\mathfrak M(\Omega)} \le a\left(1+C M''\bigl(1-\mu(d-2)\bigr)\right),\quad\text{for}\ \  |x|=r,\ |y|=z.\]
Recalling from \eqref{Eq.a-range} that $a<a_1$, we may choose
	\[\mu_0=\mu_0(d,a,M')\in\left(
	\frac{(d-1)a_1}{2}+\frac{1}{2(d-2)},\frac{1}{d-2}\right)\]
	sufficiently close to $1/(d-2)$ so that for all $\mu\in(\mu_0,1/(d-2))$ we have $$a\left(1+C M''\bigl(1-\mu(d-2)\bigr)\right)\le a_1.$$
	We thus obtain
		\begin{equation}\label{S4eq10}	0<\psi(r,z)\le a_1, \quad  \forall\ (r,z)\in\overline{\Pi_+}.\end{equation}
			
	We next prove the full gradient bound. 
As $\Psi_0(x,y)=\psi_0(|x|,|y|),$  we deduce from \eqref{Eq.nablaPsi0-upperbound} and \eqref{S4eq9} that
	\[|\nabla_{r,z}\psi(r,z)|\le\frac{aC}{\mathfrak M(\Omega)}\langle r,z\rangle^{d-3-\frac{d-1}{d\mu}}\langle z\rangle^{-\frac1{d\mu}},\quad\forall\ (r,z)\in\overline{\Pi_+}.\]
On the other hand, it follows from Lemma \ref{Lem.Psi_0-bound}  that
	\begin{equation}\label{S4eq11}	\psi(r,z)=\frac{a}{\mathfrak M(\Omega)}\psi_0(r,z)\ge	\frac{a}{\mathfrak M(\Omega)}	\frac{1}{M''\bigl(1-\mu(d-2)\bigr)}	\langle r,z\rangle^{d-2-\frac1\mu},\quad\forall\ (r,z)\in\overline{\Pi_+}.\end{equation}
		Combining the above estimates yields
	\[|\nabla_{r,z}\psi(r,z)|\le C M''\bigl(1-\mu(d-2)\bigr)\langle r,z\rangle^{-1+\frac1{d\mu}}\langle z\rangle^{-\frac1{d\mu}}\psi(r,z),\quad\forall\ (r,z)\in\overline{\Pi_+}.\]
	Taking $\mu_0$ even closer to $1/(d-2)$ if necessary, we may further assume that $C M''\bigl(1-\mu(d-2)\bigr)\le \frac1{10}$	for all $\mu\in(\mu_0,1/(d-2))$. Therefore,
	\begin{equation}\label{S4eq12}		|\nabla_{r,z}\psi(r,z)|\le\frac1{10}\langle r,z\rangle^{-1+\frac1{d\mu}}\langle z\rangle^{-\frac1{d\mu}}\psi(r,z),\qquad (r,z)\in\overline{\Pi_+}.\end{equation}
		
	Next we verify the improved $r$-derivative bound. It follows from \eqref{Eq.nabla-x-Psi0-upperbound} that
for $(r,z)\in\Pi_+$,
	\[|\partial_r\psi(r,z)|\leq\frac{aC}{\mathfrak M(\Omega)}\langle r,z\rangle^{d-3-\frac1\mu},\]
which together with \eqref{S4eq11} ensures that 
	\begin{equation}\label{Eq.r-pa_r-psi0est}
		\langle r,z\rangle|\partial_r\psi(r,z)|\leq C M''(1-\mu(d-2))\psi(r,z).
	\end{equation}
	After increasing $\mu_0$ toward $1/(d-2)$ once more if necessary, this gives 
\begin{equation}\label{S4eq13}	 |\langle r,z\rangle\partial_r\psi(r,z)|\leq \frac1{10}\psi(r,z),\qquad (r,z)\in\overline{\Pi_+}.\end{equation}
	
	Finally, we prove the second-order derivative bound. By Lemma	\ref{Lem.nabla2Psi0-bound}, for $X=(x,y)$ with $|x|=r>0$ and $|y|=z$, one has 
	\[|\nabla_{r,z}^2\psi_0(r,z)|\lesssim |\nabla_X^2\Psi_0(X)|\leq C\left(r^{d-3-\gamma}\mathbf 1_{r^2+z^2\leq1}+r^{d-3-\frac{d-1}{d\mu}}\langle z\rangle^{-1-\frac1{d\mu}}\right),\]
	which implies 
	\[|\nabla_{r,z}^2\psi(r,z)|\leq\frac{aC}{\mathfrak M(\Omega)}\left(r^{d-3-\gamma}\mathbf 1_{r^2+z^2\leq1}+r^{d-3-\frac{d-1}{d\mu}}\langle z\rangle^{-1-\frac1{d\mu}}\right).\]
By using \eqref{S4eq11} and  $\langle r,z\rangle^{d-2-\frac1\mu}\gtrsim 1$ for $r^2+z^2\leq1$,	we obtain
	\[|\nabla_{r,z}^2\psi(r,z)|\leq C M''(1-\mu(d-2))\Bigl(\langle r,z\rangle^{2-d+\frac1\mu}r^{d-3-\frac{d-1}{d\mu}}\langle z\rangle^{-1-\frac1{d\mu}}+r^{d-3-\gamma}\mathbf 1_{r^2+z^2\leq1}\Bigr)\psi(r,z).\]
	Again taking $\mu_0$ sufficiently close to $1/(d-2)$, we achieve
	\[|\nabla_{r,z}^2\psi(r,z)|\leq\frac1{10}\left(\langle r,z\rangle^{2-d+\frac1\mu}r^{d-3-\frac{d-1}{d\mu}}\langle z\rangle^{-1-\frac1{d\mu}}+	r^{d-3-\gamma}\mathbf 1_{r^2+z^2\leq1}\right)\psi(r,z),\]
	which together with \eqref{S4eq10}, \eqref{S4eq12} and \eqref{S4eq13} ensures that  $\psi\in\mathcal A^0$. Together with \eqref{Eq.psi-twosided-bound}, we
	conclude that $\psi\in \mathcal A_{(dM'')^2}$. This completes the  proof of Proposition \ref{Prop.G}.
\end{proof}

\section{Continuity and compactness of the fixed-point map}\label{Sec.existence-fixedpoint}

In this section, we present the proof of Theorem \ref{Thm.fixed-point}. As explained in Subsection \ref{Subsec.Proof-fixed-point}, with  Definition \ref{Def.nonlinear-map},  it remains to justify the  continuity properties of the two constitutive maps $\mathcal F_{\mu,M}:\mathcal A_M\to\mathcal B_{M_0'}$ and $\mathcal G_{\mu,M'}:\mathcal B_{M'}\to\mathcal A_{(dM'')^2}$. The first one is a stability statement for the transport equation under perturbations of the stream function. The second one is a stability and compactness statement for the normalized Newtonian potential. 
Throughout this section, the parameters $d,a,\mu$ are fixed and satisfy \eqref{Eq.a-range} and \eqref{Eq.mu-range}. All implicit constants are allowed to depend on these parameters.

\subsection{Continuity of the transport map}
\label{subsec:continuity-F}

The purpose of this subsection is to prove Lemma \ref{Lem.F-continuity}, which relates to the map from $\psi\in \cA_M$ to $\Omega\in\cB_{M_0'}$, where $\Omega$ solves the linear transport equation \eqref{Eq.Omega-rel-transport} with  the initial data \eqref{Eq.transport-initial-2}. 
Recall from Section \ref{Sec.transport} that  the first key ingredient used to solve \eqref{Eq.Omega-rel-transport} with \eqref{Eq.transport-initial-2} is to use the change of variables \eqref{Eq.H_psi},
which maps the initial curve $\Gamma_0$ (see \eqref{Eq.Gamma0-def}) to the fixed unit arc $\Gamma=\{(R,Z):R^2+Z^2=1,\ R>0,\ Z>0\}$, and transforms  the original  transport equation  to \eqref{Eq.Omega-transport}.
Then the continuity of $\mathcal F_{\mu,M}$ follows from the stability of $H_\psi^{-1}$, the stability of $h_\psi$ (see \eqref{Eq.h-def}), and the continuous dependence of characteristic curves on their vector fields.

We first establish the stability of the coordinate transformation.

\begin{lemma}[Stability of the coordinate transform]\label{lem:stability-H}
{\sl	
Let $\psi_n,\psi\in\mathcal A^0$ be such that $\psi_n\to\psi$ in $C^2_{\rm loc}(\Pi_+)\cap C_{\rm loc}\left(\overline{\Pi_+}\right)$. Let $H_n=H_{\psi_n}$ and $H=H_\psi$ be given by \eqref{Eq.H_psi}. 
Then $H_n^{-1}\to H^{-1}$ in $C^2_{\rm loc}(\Pi_+)$. Consequently,  let $h_n=h_{\psi_n}$ and $h=h_\psi$ be defined by \eqref{Eq.h-def}, then $h_n\to h$ in $C^2_{\rm loc}(\Pi_+)$.}
\end{lemma}

\begin{proof}
We first	fix a compact set $K\Subset\Pi_+$. Recall from the proof of Lemma \ref{Lem.H_psi} that for $(R,Z)\in K$,  we denote by $\zeta_n(R,Z)$ the unique solution of $Z=(\mu+\psi_n(R,z))z$. Since $0<\psi_n\le a_1$ uniformly in $n$, we have $Z/(\mu+a_1)\le \zeta_n(R,Z)\le Z/\mu$. Therefore all points $(R,\zeta_n(R,Z))$, with $(R,Z)\in K$, remain in a fixed compact subset $K'\Subset\Pi_+$, which is independent of $n$.
	
	As in the proof of Lemma \ref{Lem.H_psi}, for $z>0$, $R>0$ and $n\in\Z_+,$ we define 
	$$F_{n,R}(z):=\big(\mu+\psi_n(R, z)\big)z \andf F_{R}(z):=\big(\mu+\psi(R, z)\big)z.$$  It follows from   \eqref{Eq.Fr'>0} that $z\mapsto F_{n,R}(z)$ is uniformly strictly increasing, with $F_{n, R}'(z)>\mu$. Moreover, using the fact that $F_{n,R}\big(\zeta_n(R, Z)\big)=Z=F_{R}\big(\zeta(R, Z)\big)$, we obtain
	\begin{align*}
		\mu|\zeta_n(R, Z)-\zeta(R, Z)|&\leq \left|F_{n,R}\big(\zeta_n(R, Z)\big)-F_{n,R}\big(\zeta(R, Z)\big)\right|\\
		&=\left|F_{R}\big(\zeta(R, Z)\big)-F_{n,R}\big(\zeta(R, Z)\big)\right|\\
		&=\zeta(R, Z)\left|\psi_n\big(R, \zeta(R, Z) \big)-\psi\big(R, \zeta(R, Z) \big)\right|.
	\end{align*}
	In particular, we obtain $\sup_K|\zeta_n-\zeta|\lesssim \sup_{K'}|\psi_n-\psi|\to0$ as $n\to\infty$.
	
	For the derivatives, we get, by differentiating $(\mu+\psi_n(R,\zeta_n(R,Z)))\zeta_n(R,Z)=Z,$ that 
	\begin{align}\label{S5eq1}
	\partial_R\zeta_n=-\frac{\zeta_n\partial_r\psi_n(R,\zeta_n)}{\mu+\psi_n(R,\zeta_n)+\zeta_n\partial_z\psi_n(R,\zeta_n)},
	\quad
	\partial_Z\zeta_n=\frac1{\mu+\psi_n(R,\zeta_n)+\zeta_n\partial_z\psi_n(R,\zeta_n)}.
	\end{align}
	By the gradient bound in $\mathcal A^0$, as in \eqref{Eq.Fr'>0}, the denominator above is bounded from below by $\mu$. As $\psi_n\to\psi$ in $C^1(K')$ and $\zeta_n\to\zeta$ uniformly on $K$, we deduce from \eqref{S5eq1} that  $\zeta_n\to\zeta$ in $C^1(K)$. Since $K\Subset\Pi_+$ is arbitrary, $H_n^{-1}\to H^{-1}$ in $C^1_{\rm loc}(\Pi_+)$.
	
	The convergence of the second derivatives follows by differentiating once more the identity $(\mu+\psi_n(R,\zeta_n(R,Z)))\zeta_n(R,Z)=Z$. All resulting terms contain only derivatives of $\psi_n$ up to order two, $\zeta_n$, and	first derivatives of $\zeta_n$. The denominator is always $\mu+\psi_n(R,\zeta_n)+\zeta_n\partial_z\psi_n(R,\zeta_n)\ge \mu$,	by the same argument as in \eqref{Eq.Fr'>0}. Since $\psi_n\to\psi$ in \(C^2\) on compact subsets of \(\Pi_+\), we deduce that  \(\zeta_n\to\zeta\) in \(C^2(K)\). Hence $H_n^{-1}\to H^{-1}$ in $C_{\rm loc}^2(\Pi_+)$.
	
	Finally, due to $$h_n(R,Z)=\f{\mu-(d-1)\psi_n(R,\zeta_n(R,Z))}{\mu+\psi_n(R,\zeta_n(R,Z))}, $$
	the denominator is bounded below by $\mu$, and the numerator and denominator above converge in $C^2(K)$. Hence $h_n\to h$ in $C^2(K)$. This finishes the proof of the lemma.
\end{proof}

Next we prove the stability of the characteristic coordinates and of the exponential correction factor in the formula for the transport solution.

\begin{lemma}[Stability of the characteristic representation]\label{lem:stability-characteristics}
	{\sl Under the assumptions of Lemma \ref{lem:stability-H}, for each $n$, let $(s_n(R,Z),\sigma_n(R,Z))$ be the characteristic coordinates associated with $h_n$, and  $J_n(R,Z)$ be the corresponding function defined by \eqref{Eq.J_psi-def}. Let $(s,\sigma)$ and $J$ be the analogous objects associated with $h$. Then for every compact set $K\Subset\Pi_+$, one has 
	$$s_n\to s, \  \sigma_n\to\sigma\ \ \mbox{ in}\ \ C^2(K), \andf J_n\to J\ \ \mbox{ in}\ \ C^1(K). $$}
\end{lemma}

\begin{proof}
	By Lemma \ref{lem:stability-H}, $h_n\to h$ in $C^2_{\rm loc}(\Pi_+)$. It follows from  \eqref{Eq.h-bound} that all $h_n$ and $h$ share  the same uniform bounds $0<h_*\le h_n,h\le1$ on $\Pi_+$, where $h_*>0$ depends only on $d$ and $a$. See \eqref{Eq.h_*-def} for the explicit formula for $h_*$.
	
	Let $\Phi_n(t;P)$ be the flow generated by $V_n(R,Z)=(R,h_n(R,Z)Z)$, and let $\Phi(t;P)$ be the flow generated by $V(R,Z)=(R,h(R,Z)Z),$ that is, for any $t\in\R$ and $P\in\Pi_+$,
	\begin{align*}
		\frac{\mathrm d}{\mathrm dt}\Phi_n(t;P)=V_n\big(\Phi_n(t;P)\big),\quad \frac{\mathrm d}{\mathrm dt}\Phi(t;P)=V\big(\Phi(t;P)\big),\quad \Phi_n(0;P)=\Phi(0;P)=P.
	\end{align*} 
	Along each flow, one has
	\begin{equation}\label{S5eq2}
	\begin{split}
		\frac{\mathrm d}{\mathrm dt}|\Phi_n(t;P)|^2&=2R_n(t;P)^2+2h_n\big(R_n(t;P),Z_n(t;P)\big)Z_n(t;P)^2\\
		&\in\left[2h_*|\Phi_n(t;P)|^2, 2|\Phi_n(t;P)|^2\right].
	\end{split}\end{equation}	
	Then for $P$ ranging  over a fixed compact set $K\Subset\Pi_+$, the time needed for the trajectory through $P$ to hit the unit arc $\Gamma$ is uniformly bounded. Precisely, there exists $S_K>0$ such that the hitting time $\tau_n(P)$ defined by $|\Phi_n(\tau_n(P);P)|=1$ satisfies $$|\tau_n(P)|\le S_K, \ \ \forall \ P\in K\andf n\in\N. $$ The same bound holds for the limiting flow.
	
	Since $K\Subset\Pi_+$ and the hitting times are uniformly bounded, all trajectories $\Phi_n(t;P)$ with $|t|\le S_K$ and $P\in K$ remain in a common compact subset $K_1\Subset\Pi_+$. On $K_1$, we have $V_n\to V$ in $C^2$. The standard continuous dependence theorem for ODE flows ensures that $\Phi_n\to\Phi$ in $C^2([-S_K,S_K]\times K)$.
	
Let  $F_n(t,P):=|\Phi_n(t;P)|^2-1$. At the hitting time $\tau_n(P)$, we deduce from \eqref{S5eq2} that  $\partial_tF_n(\tau_n(P),P)\ge2h_*$. Thus,  the intersection with $\Gamma$ is uniformly transverse. The implicit function theorem and the $C^2$ convergence of $F_n$ to $F$ imply  $\tau_n\to\tau$ in $C^2(K)$. As $Q_n(P):=\Phi_n(\tau_n(P);P)$ converges to $Q(P):=\Phi(\tau(P);P)$ in $C^2(K)$, and  $Q_n(P)=(\sin\sigma_n(P),\cos\sigma_n(P))$ and $Q(P)=(\sin\sigma(P),\cos\sigma(P))$, we find
 $$\sigma_n\to\sigma\ \mbox{ in}\ C^2(K) \andf  s_n=-\tau_n\to s=-\tau\ \mbox{ in}\  C^2(K).$$
	
	Finally, by  \eqref{Eq.J_psi-def},
	\[J_n(P)=-\frac1{d\mu}\int_0^{s_n(P)}(Z\partial_Zh_n)(\Phi_n(\ell;Q_n(P)))\,\mathrm d\ell.\]
	The intervals of integration are uniformly bounded, $Q_n\to Q$ in $C^2(K)$, the flows converge in $C^2$, and $Z\partial_Zh_n\to Z\partial_Zh$ in $C^1_{\rm loc}(\Pi_+)$ (since $\psi_n\to\psi$ in $C^2_{\text{loc}}(\Pi_+)$). We conclude that $J_n\to J$ in $C^1(K)$. This finishes  the proof of Lemma \ref{lem:stability-characteristics}.
\end{proof}

We  now prove the continuity of the transport map.

\if0\begin{lemma}[Continuity of the transport map]\label{Lem.F-continuity-proof}
	Let $M>1$ be fixed, and let $\lambda_0=\lambda_0(d,a,\mu,M)$ be the parameter in Proposition~\ref{prop:transport-bound}. Suppose that $\psi_n,\psi\in\mathcal A_M$ and $\psi_n\to\psi$ in $C^1_{\rm loc}(\Pi_+)\cap C_{\rm loc}(\overline{\Pi_+})$. Then $\mathcal F_{\mu,M}(\psi_n)\to\mathcal F_{\mu,M}(\psi)$ in $C_{\rm loc}(\Pi_+)$.
\end{lemma}\fi

\begin{proof}[Proof of Lemma \ref{Lem.F-continuity}]
Assume that $d\in\N_{\geq 3}$, and that $a,\mu$ satisfy \eqref{Eq.a-range} and \eqref{Eq.mu-range}. Let $M>1$, and $M_0'=M_0'(d,a)>1$ and $\lambda_0=\lambda_0(d,a,\mu,M)>1$ be the constants determined by  Proposition \ref{Prop.F}. Thanks to Proposition \ref{Prop.F}, for any $\psi\in\cA_M$, there exists a unique solution $\Omega\in \cB_{M_0'}$ to \eqref{Eq.Omega-rel-transport} with the initial data \eqref{Eq.transport-initial-2}. We denote this map from $\psi$ to $\Omega$  by $\cF_{\mu, M}$, that is,  $\Omega=\cF_{\mu, M}(\psi)$. 
	
	We take a sequence $\bigl\{\psi_n\bigr\} \in \cA_M$ such that $\psi_n\to \psi\in \cA_M$ in $C_{\text{loc}}^2(\Pi_+)\cap C_{\text{loc}}\left(\overline{\Pi_+}\right)$. Let $\Omega_n=\mathcal F_{\mu,M}(\psi_n)$ and $\Omega=\mathcal F_{\mu,M}(\psi)$. We define $\widehat\Omega_n:=\Omega_n\circ H_n^{-1}$ and $\widehat\Omega:=\Omega\circ H^{-1}$. Then it follows from  Lemma \ref{Lem.hat-Omega-expression} that
	\[\widehat\Omega_n(R,Z)=R^{-\frac{d-1}{d\mu}}Z^{-\frac1{d\mu}}\Theta_{\lambda_0}(\sigma_n(R,Z))\chi(\sigma_n(R,Z))\ \mathrm e^{J_n(R,Z)},\]
	and the same formula holds for $\widehat\Omega$ with $\sigma_n,J_n$ above being replaced by $\sigma,J$. Here $\lambda_0$ is fixed, as $d,a,\mu,M$ are fixed and $\lambda_0$ does not depend on the particular choice of $\psi\in\mathcal A_M$.
	
	Fix $K\Subset\Pi_+$. By Lemma \ref{lem:stability-characteristics}, we have
	$$\sigma_n\to\sigma \andf J_n\to J\ \ \mbox{in}\ C^1(K). $$
	Since $K\Subset\Pi_+$, the functions $R^{-\frac{d-1}{d\mu}}$ and $Z^{-\frac{1}{d\mu}}$ are smooth and bounded on $K$, and the hitting angles stay in a compact sub-interval of $(0,\pi/2)$. Hence $\widehat\Omega_n\to\widehat\Omega$ in $C^1(K)$ follows from the smoothness of $\Theta_{\lambda_0}$ and $\chi.$ 
	
	Now let $K_0\Subset\Pi_+$ be compact in the original $(r,z)$ variables. Since $H_n(r,z)=(r,(\mu+\psi_n(r,z))z)$ and $\psi_n\to\psi$ in $C^2(K_0)$, we have $H_n\to H$ in $C^2(K_0)$. Moreover, the sets $H_n(K_0)$ and $H(K_0)$ are contained in a common compact subset $K_1\Subset\Pi_+$, because $r$ and $z$ are bounded away from zero on $K_0$ and $\mu+\psi_n\ge\mu$. Due to $\Omega_n=\widehat\Omega_n\circ H_n$ and
	$\Omega=\widehat\Omega\circ H$,	the chain rule and the $C^1$-convergences above yield $\Omega_n\to\Omega$ in $C^1(K_0)$. Since $K_0\Subset\Pi_+$ is arbitrary, we complete the proof of Lemma \ref{Lem.F-continuity}.
\end{proof}

\subsection{Continuity and compactness of the elliptic map}
\label{subsec:continuity-compactness-G}

This subsection aims to prove Lemma \ref{Lem.G-continuity}. In contrast with the transport map, the elliptic map is nonlocal and contains the normalization factor $\mathfrak M(\Omega)^{-1}$. We therefore first prove the stability of the Newtonian potential and of the normalization. Compactness then follows from local elliptic regularity and the uniform bounds obtained in Proposition \ref{Prop.G}.

Throughout this subsection, $M'>1$ is fixed. For $\Omega\in\mathcal B_{M'}$, let $F_\Omega(Y)$ and $\Psi_0(X)$ be defined by \eqref{Eq.F_Omega-def} and \eqref{Eq.Psi0-identity} respectively, and we also define
$\Psi_\Omega(X):=\Psi_0(X),$ in order to emphasize the dependence on $\Omega$.
Then the normalization constant $\mathfrak M(\Omega)$ equals $\Psi_\Omega(0)$, and the elliptic map becomes
\[\mathcal G_{\mu,M'}(\Omega)(r,z)=\frac a{\mathfrak M(\Omega)}\Psi_\Omega(X),\quad \text{with}\ \ |x|=r,\  |y|=z,\  X=(x,y)\in\R^{d+1}\times\R^3.\]

We begin with a dominated convergence lemma for the singular Newtonian potentials.

\begin{lemma}[Stability of the Newtonian potential in $C^0$]
	\label{lem:stability-newtonian-C0}
{\sl	
Let $M'>1$ and  $\Omega_n,\Omega\in\mathcal B_{M'}.$ Let $\Psi_n:=\Psi_{\Omega_n}$ and $\Psi:=\Psi_\Omega$.  We assume that $\Omega_n\to\Omega$ in
	$C_{\rm loc}^1(\Pi_+)$. 
	Then $\Psi_n\to\Psi$ in $C^0_{\rm loc}(\R^{d+4})$. In particular,
	$\mathfrak M(\Omega_n)\to \mathfrak M(\Omega)$.}
\end{lemma}

\begin{proof}
	For conciseness, we  denote $F_n:=F_{\Omega_n}$ and $F:=F_{\Omega}$. 
	Then in view of \eqref{Eq.BM'-def} and \eqref{Eq.F_Omega-def},	
	 we have,
	\begin{equation}\label{S5eq3}
	0<F_n(Y)\le C\min\left\{|\xi|^{d-3-\frac{d-1}{d\mu}}|\eta|^{-1-\frac1{d\mu}}, |\xi|^{d-3-\gamma}+|\xi|^{d-3-\gamma_1} \right\} \end{equation}
	for a.e. $Y=(\xi,\eta)\in\R^{d+1}\times\R^3$, 	where $C>1$ depends only on
	$d,a,M'$. The same estimate holds for $F(Y)$. 
	
	
	Let $K\Subset\R^{d+4}$ be a fixed compact subset of $\R^{d+4}$. Choose  $A>2$ large enough such that $K\subset B_{A}(0)\subset
	\mathbb R^{d+4}$. We first handle  the far-field part. Let $R>2A$. If $X\in K$ and $|Y|>R$, then $|X-Y|\ge |Y|-|X|\ge |Y|/2$. Hence,
	\begin{align}
		\int_{|Y|>R}|X-Y|^{-d-2}F_n(Y)\,\mathrm dY&\lesssim \int_{|Y|>R}|Y|^{-d-2}F_n(Y)\,\mathrm dY\nonumber\\
		&\lesssim \sum_{k=0}^\infty\int_{2^kR<|Y|\leq 2^{k+1}R}|Y|^{-d-2}F_n(Y)\,\mathrm dY\nonumber\\
		&\lesssim \sum_{k=0}^\infty (2^kR)^{-d-2}\int_{|Y|\leq 2^{k+1}R}F_n(Y)\,\mathrm dY\label{Eq.stability-elliptic-farfield} \\
		&\lesssim \sum_{k=0}^\infty(2^kR)^{-d-2}(2^{k+1}R)^{2d-\frac1\mu}\quad\mbox{by using}\ \eqref{Eq.M(rho)-est}\nonumber\\
		&\lesssim \frac1{1-(d-2)\mu}R^{d-2-\frac1\mu}\to 0\ \mbox{as}\ R\to+\infty. \nonumber	\end{align}
	The same estimate also holds with $F_n$ being replaced by $F$.
	
	Next we estimate the contribution near the kernel singularity. We choose 	$q$ such that $(d+4)/2<q<d+1$. 
	Let $q'$ be 	the conjugate exponent of $q$. Since $\gamma,\gamma_1<d-2$ (recalling footnote \ref{footnote.gamma>d-2}), we have 
	$$q(d-3-\gamma)+d>-1 \andf q(d-3-\gamma_1)+d>-1.$$  Thus, for every fixed	$R>0$, one has
	\begin{align*}
		\sup_n\int_{|Y|\le R}F_n(Y)^q\,\mathrm dY
		&\le C\int_{|Y|\le R}
		\bigl(|\xi|^{q(d-3-\gamma)}
		+|\xi|^{q(d-3-\gamma_1)}\bigr)\,\mathrm dY  \\
		&\le C_R\int_0^R
		\bigl(\tilde r^{q(d-3-\gamma)}+\tilde r^{q(d-3-\gamma_1)}\bigr)\tilde r^d\,\mathrm d\tilde r,
	\end{align*}
	which implies
	$$\sup_n\|F_n\|_{L^q(B_R)}\le C_R.$$
	 The same bound holds for $F$.
	 
	For $X\in K$,  we get, by applying H\"older's inequality, that
	\begin{align}
		&\int_{|Y|\le R,\ |X-Y|<\delta}|X-Y|^{-d-2}F_n(Y)\,\mathrm dY\nonumber \\
				&\le \|F_n\|_{L^q(B_R)}
		\Bigl(\int_{|X-Y|<\delta}|X-Y|^{-(d+2)q'}\,\mathrm dY\Bigr)^{1/q'}\label{Eq.stability-elliptic-nearsingularity}  \\
		&\le C_R
		\left(\int_0^\delta \rho^{d+3-(d+2)q'}\,\mathrm d\rho\right)^{1/q'}.\nonumber 	\end{align}
	Since $q>(d+4)/2$, one has $(d+2)q'<d+4$, and 
	$d+3-(d+2)q'>-1$. Consequently, we have
	\[\sup_n\sup_{X\in K}\int_{|Y|\le R,\ |X-Y|<\delta}|X-Y|^{-d-2}F_n(Y)\,\mathrm dY\le C_R\delta^{2-(d+4)/q}\to0\]
	as $\delta\to0^+$. The same estimate holds with $F_n$  being replaced by $F$.
	
	Finally, we  estimate the regions close to the coordinate axes. Fix $R>0$ and
	$\delta>0$. On the set $|X-Y|\ge\delta$, the kernel is bounded by
	$\delta^{-d-2}$. Then we get, by applying \eqref{S5eq3}, that 
		\begin{align}
		&\int_{|Y|\le R,\ |X-Y|\ge\delta,\ |\xi|<\kappa}
		|X-Y|^{-d-2}F_n(Y)\,\mathrm dY  \nonumber\\
		&\le C\delta^{-d-2}
		\int_{|Y|\le R,\ |\xi|<\kappa}
		\bigl(|\xi|^{d-3-\gamma}
		+|\xi|^{d-3-\gamma_1}\bigr)\,\mathrm dY \label{Eq.stability-elliptic-nearaxis1}\ \\
		&\le C_{R,\delta}
		\int_0^\kappa
		\left(\tilde r^{d-3-\gamma}
		+\tilde r^{d-3-\gamma_1}\right)\tilde r^d\,\mathrm d\tilde r  \le C_{R,\delta}
		\left(\kappa^{2d-2-\gamma}
		+\kappa^{2d-2-\gamma_1}\right).\nonumber	\end{align}
As $\gamma,\gamma_1<d-2$, both exponents $2d-2-\gamma$ and
	$2d-2-\gamma_1$ are positive,  the above term tends to zero as
	$\kappa\to0^+$, uniformly in $n$ and $X\in K$.
	
Along the same line as in \eqref{Eq.stability-elliptic-nearaxis1}, we  deduce that
	\begin{align}
		&\int_{|Y|\le R,\ |X-Y|\ge\delta,\ |\eta|<\kappa}		|X-Y|^{-d-2}F_n(Y)\,dY  \nonumber\\
		&\quad\le C\delta^{-d-2}		\int_{|Y|\le R,\ |\eta|<\kappa}
		|\xi|^{d-3-\frac{d-1}{d\mu}}|\eta|^{-1-\frac1{d\mu}}\,dY\label{Eq.stability-elliptic-nearaxis2} \\
		&\quad\le C_{R,\delta}		\left(\int_0^R\tilde r^{d-3-\frac{d-1}{d\mu}}\tilde r^d\,d\tilde r\right)		\left(\int_0^\kappa\tilde z^{-1-\frac1{d\mu}}\tilde z^2\,d\tilde z\right) \le C_{R,\delta}\kappa^{2-\frac1{d\mu}}.\nonumber 	\end{align}
	Here the first integral is finite due to $2d-3-(d-1)/(d\mu)>-1$, and the last exponent is positive because $1/(d\mu)<1$. Thus, $\kappa^{2-\frac1{d\mu}}$ tends to zero	as $\kappa\to0^+$, uniformly in $n$ and $X\in K$. The  estimates \eqref{Eq.stability-elliptic-nearaxis1} and \eqref{Eq.stability-elliptic-nearaxis2} hold with $F_n$ being replaced by $F$.
	
	We are now in a position to complete the proof. Fix $\varepsilon>0$. In view of  the far-field estimate \eqref{Eq.stability-elliptic-farfield}, we can choose $R>2A$ so large that, for every $X\in K$,
	\begin{equation}\label{S5eq4}
\int_{|Y|>R}|X-Y|^{-d-2}\bigl(F_n(Y)+F(Y)\bigr)\,\mathrm dY\le \varepsilon,\ \forall \ n.	\end{equation}
 While by the kernel-singularity estimate \eqref{Eq.stability-elliptic-nearsingularity}, we can choose $\delta>0$ such
	that
	\begin{equation}\label{S5eq5}
\int_{|Y|\le R,\ |X-Y|<\delta}|X-Y|^{-d-2}\bigl(F_n(Y)+F(Y)\bigr)\,\mathrm dY\le \varepsilon, \  	\forall \ n \andf \forall \ X\in K.\end{equation} 
Whereas by virtue of  the two axis estimates \eqref{Eq.stability-elliptic-nearaxis1} and \eqref{Eq.stability-elliptic-nearaxis2}, we can choose $\kappa>0$ such that
	\begin{equation}\label{S5eq6}		\int_{\substack{|Y|\le R,\ |X-Y|\ge\delta,\\
	 |\xi|<\kappa\ \text{or}\  |\eta|<\kappa}}
		|X-Y|^{-d-2}\bigl(F_n(Y)+F(Y)\bigr)\,\mathrm dY \le \varepsilon, \  	\forall \ n \andf \forall \ X\in K.\	\end{equation} 	 It remains to deal with the compact region
	\[
	D_{R,\delta,\kappa}(X)
	:=\{|Y|\le R,\ |X-Y|\ge\delta,\ |\xi|\ge\kappa,\ |\eta|\ge\kappa\}.
	\]
	On this region, $|X-Y|^{-d-2}\le\delta^{-d-2},$  $|\xi|$ and $|\eta|$ are bounded away from zero and $|Y|\le R$, so that 
	\[
	\sup_{|Y|\le R,\ |\xi|\ge\kappa,\ |\eta|\ge\kappa}
	|F_n(Y)-F(Y)|
	\le C_{R,\kappa}
	\sup_{\substack{0\le r,z\le R\\ r,z\ge\kappa}}
	|\Omega_n(r,z)-\Omega(r,z)|.
	\]
	The right-hand side tends to zero, since $\Omega_n\to\Omega$ in
	$C_{\rm loc}(\Pi_+)$. Therefore,
	\begin{equation}\label{S5eq7}	\begin{split}
	&		\sup_{X\in K}
		\int_{D_{R,\delta,\kappa}(X)}
		|X-Y|^{-d-2}|F_n(Y)-F(Y)|\,\mathrm dY\\
		& \le
		C_{R,\delta,\kappa}
		\sup_{\substack{0\le r,z\le R\\ r,z\ge\kappa}}
		|\Omega_n(r,z)-\Omega(r,z)|
		\to0 . \end{split} \end{equation}	
	
Combining the estimates \eqref{S5eq4}$\sim$\eqref{S5eq7},	 we obtain, for $n$ sufficiently large,
	\[\sup_{X\in K}
	|\Psi_n(X)-\Psi(X)|\le 4\varepsilon .\]
	Since $\varepsilon>0$ is arbitrary, this proves
	$\Psi_n\to\Psi$ uniformly on $K$. Since $K\Subset\R^{d+4}$ is arbitrary, we achieve $$\Psi_n\to\Psi\ \mbox{ in} \ C^0_{\rm loc}(\R^{d+4}). $$
	Furthermore, as $\mathfrak M(\Omega_n)=\Psi_n(X=0)$ and $\mathfrak M(\Omega)=\Psi(X=0),$ taking a compact set $K\Subset\R^{d+4}$ containing the origin gives $\mathfrak M(\Omega_n)\to \mathfrak M(\Omega)$. This completes the proof of Lemma \ref{lem:stability-newtonian-C0}.
\end{proof}

Similar to the discussions following Lemma \ref{Lem.Psi_0-bound}, there is a sequence $\psi_{0, n}=\psi_{0, n}(r,z)\in C^1\left(\overline{\Pi_+}\right)\cap C^2(\Pi_+)$ and $\psi_0=\psi_0(r,z)\in C^1\left(\overline{\Pi_+}\right)\cap C^2(\Pi_+)$ such that for any $X=(x,y)\in\R^{d+1}\times\R^3,$
\begin{align*}
\Psi_n(X)=\Psi_n(x,y)=\psi_{0, n}(|x|,|y|) \andf \Psi(X)=\Psi(x,y)=\psi_{0}(|x|,|y|). 
\end{align*}
Moreover, it follows from \eqref{Eq.psi-elliptic} that
\begin{align}\label{S5eq8}
	\Bigl(\partial_r^2+\frac dr\pa_r+\pa_z^2+\frac2z\pa_z\Bigr)\psi_{n,0}(r,z)=-r^{d-3}\frac{\Omega_n(r,z)}{z},\ \forall \ (r,z)\in\Pi_+.
\end{align}

\begin{lemma}[Stability of the Newtonian potential in $C^2$]
	\label{lem:stability-newtonian}
{\sl	Under the same hypothesis as Lemma \ref{lem:stability-newtonian-C0}, we have $\psi_{0,n}\to\psi_0$ in $C^2_{\rm loc}(\Pi_+)$.}
\end{lemma}

\begin{proof}
	We first deduce from  Lemma \ref{lem:stability-newtonian-C0} that $\psi_{0,n}\to\psi_0$ in $C^0_{\rm loc}\left(\overline{\Pi_+}\right)$. It remains to prove the local convergence of the first and second derivatives  of $\psi_{0,n}$ in the interior of $\Pi_+$.
	
	Fix a compact set $K\Subset\Pi_+$.  We choose an open set $U\Subset\Pi_+$ such that $K\Subset U$. As $U$ is compactly contained in $\Pi_+$, both $r$ and $z$ are	bounded away from zero on $U$. Hence the operator $$\mathscr L:=\partial_r^2+\frac dr\partial_r+\partial_z^2+\frac2z\partial_z$$ is uniformly elliptic on $U$, with smooth coefficients.  Then in view of \eqref{Eq.psi-elliptic} and \eqref{S5eq8}, $\psi_{0,n}$ and $\psi_0$ solve
	\[\mathscr L\psi_{0,n}=-r^{d-3}\frac{\Omega_n(r,z)}{z},\quad \mathscr L\psi_{0}=-r^{d-3}\frac{\Omega(r,z)}{z}  \quad \text{in } U,\]
which implies
	\[\mathscr L\left(\psi_{0,n}-\psi_0\right)=-r^{d-3}\frac{\Omega_n(r,z)-\Omega(r,z)}{z}:=f_n(r,z) \quad \text{in } U.\]
	Since $\Omega_n\to\Omega$ in $C_{\rm loc}^1(\Pi_+)$ and the factor $r^{d-3}/z$
	is smooth and bounded on $U$, we have $f_n\to0$ in $C^1(U)$.
	
	By the standard interior Schauder estimate for uniformly elliptic equations with	smooth coefficients, for each $\alpha'\in(0,1)$, there exists a constant $C_K>0$, depending only on $K,U,d,a,M'$ and $\alpha'\in(0,1)$, such that
	\[
	\|\psi_{0,n}-\psi_0\|_{C^{2,\alpha'}(K)}
	\le C_K\bigl(\|\psi_{0,n}-\psi_0\|_{C^0(U)}+\|f_n\|_{C^1(U)}\bigr).
	\]
	The right-hand side tends to zero, as $\|\psi_{0,n}-\psi_0\|_{C^0(U)}\to0$ by
	Lemma \ref{lem:stability-newtonian-C0}, and $\|f_n\|_{C^1(U)}\to0$ by the local
	uniform convergence of $\Omega_n$ to $\Omega$. Therefore
	$\psi_{0,n}\to\psi_0$ in $C^2(K)$. As $K\Subset\Pi_+$ is arbitrary, this ensures	$\psi_{0,n}\to\psi_0$ in $C^2_{\rm loc}(\Pi_+)$. This finishes the proof of Lemma \ref{lem:stability-newtonian}.
\end{proof}

\begin{lemma}[Continuity of the elliptic map]
	\label{lem:continuity-G-proof}
{\sl  Let $M'>1$, and $\mu_0=\mu_0(d,a,M')>0$ be determined by Proposition \ref{Prop.G}.  Let $\Omega_n, \Omega\in\cB_{M'}$ be such that $\Omega_n\to\Omega$ in $C_{\rm loc}^1(\Pi_+).$ Then
for $\mu\in(\mu_0,1/(d-2)),$ one has \begin{align}\label{S5eq9}\mathcal G_{\mu,M'}(\Omega_n)\to\mathcal G_{\mu,M'}(\Omega)\ \mbox{ in}\ C^2_{\rm loc}(\Pi_+)\cap C_{\rm loc}(\overline{\Pi_+}).\end{align}}
\end{lemma}

\begin{proof}
	This follows directly from
	\[\mathcal G_{\mu,M'}(\Omega_n)=\frac a{\mathfrak M(\Omega_n)}\psi_{0,n},\quad \mathcal G_{\mu,M'}(\Omega)=\frac a{\mathfrak M(\Omega)}\psi_{0},\]
	Lemmas \ref{lem:stability-newtonian-C0} and  \ref{lem:stability-newtonian}. 
\end{proof}

We next prove the compactness of the operator  $\mathcal G_{\mu,M'}$. The proof uses only uniform estimates, but not convergence of the input sequence.

\begin{lemma}[Uniform local regularity]
	\label{lem:uniform-regularity-G}
{\sl Let $M'>1$ and $\mu_0=\mu_0(d,a,M')>0$ be determined  by Proposition \ref{Prop.G}. For every $j\in\mathbb N_+$ and every $\alpha'\in(0,1)$, there exists a constant $C_{j,\alpha'}>0$, depending only on $j,\alpha',d,a,M'$, such that for any $\Omega\in\mathcal B_{M'}$ and $\mu\in(\mu_0,1/(d-2))$, one has
	\begin{align}\label{S5eq10}
	\|\mathcal G_{\mu,M'}(\Omega)\|_{C^{0,1}(L_j)}+
	\|\mathcal G_{\mu,M'}(\Omega)\|_{C^{2,\alpha'}(K_j)}\le C_{j,\alpha'},
	\end{align}
	where $L_j=[0,j]^2$ and $K_j=[1/j,j]^2$.}
\end{lemma}

\begin{proof}
	Let $\psi:=\mathcal G_{\mu,M'}(\Omega)$. Then it follows from Proposition \ref{Prop.G} that $\psi\in\mathcal A_{(dM'')^2}$. In particular, the pointwise bounds in the definition of $\mathcal A_{(dM'')^2}$ (see \eqref{defam}) and the gradient estimate in the definition of $\mathcal A^0$ (see \eqref{S2eq8})  ensure that $\|\psi\|_{C^{0,1}(L_j)}\le C_j$. 
	
	It remains to derive the $C^{2,\alpha'}$ estimate of $\psi$ on $K_j$.  In order to do so, we lift $\psi$ to $\Psi(X):=\psi(|x|,|y|)$, where $X=(x,y)\in\mathbb R^{d+1}\times\mathbb R^3$. Then $\Psi$ weakly solves
	\[
	-\Delta_X\Psi(X)=\frac{a}{\mathfrak M(\Omega)}|x|^{d-3}|y|^{-1}\Omega(|x|,|y|).
	\]
	 By Lemma \ref{Lem.Psi_0-bound}, $\mathfrak M(\Omega)$ is bounded below by a positive constant depending only on $d,a,M'$. On the region $U_j:=\{X=(x,y): |x|\in[1/(2j),2j],\ |y|\in[1/(2j),2j]\}$, the factor $|x|^{d-3}|y|^{-1}$ is smooth and bounded.  The  derivative estimate in \(\cB_{M'}\) leads to  a uniform bound for	\(\nabla_{x,y}\bigl(\Omega(|x|,|y|)\bigr)\) on \(U_j\). Together with the lower bound for \(\mathfrak M(\Omega)\),  we thus obtain 
	 $$\|\Delta_X\Psi\|_{W^{1,\infty}(U_j)}\le C_j.$$  Combining this bound with the interior \(W^{3,q}\) estimates,
	  we infer
	   $$\|\Psi\|_{W^{3,q}(U_j')}\le C_{j,q}\ 	\mbox{for any}\  1<q<\infty, $$
	   where  \(U_j'\Subset U_j\) still contains the lift of \(K_j\). Choosing \(q>d+4\) and using Sobolev embedding yield $\|\Psi\|_{C^{2,\alpha'}(U_j')}\le C_{j,\alpha'}$.	Restricting  \(X=(r,0,\ldots,0,z,0,0)\) gives rise to $\|\psi\|_{C^{2,\alpha'}(K_j)}\le C_{j,\alpha'}$. This completes the proof of  Lemma \ref{lem:uniform-regularity-G}.
\end{proof}

\begin{lemma}[Compactness of the elliptic map]
	\label{lem:compactness-G-proof}
{\sl  Let $M'>1$ and  $\mu_0=\mu_0(d,a,M')>0$ be determined  by Proposition \ref{Prop.G}. Then for $\mu\in(\mu_0,1/(d-2)),$
	the set $\mathcal G_{\mu,M'}(\mathcal B_{M'})$ is relatively compact in $C^2_{\rm loc}(\Pi_+)\cap C_{\rm loc}(\overline{\Pi_+})$.}
\end{lemma}

\begin{proof}
	Let $\Omega_n\in\mathcal B_{M'}$ be arbitrary and set $\psi_n:=\mathcal G_{\mu,M'}(\Omega_n)$. By Lemma \ref{lem:uniform-regularity-G}, for each fixed $j$, the sequence $\{\psi_n\}$ is uniformly bounded and equicontinuous in $C(L_j)$, and uniformly bounded in $C^{2,\alpha'}(K_j)$ for some fixed $\alpha'\in(0,1)$. The Arzelà-Ascoli theorem then implies that, after passing to a subsequence depending on $j$, $\psi_n$ converges in $C(L_j)$ and in $C^2(K_j)$.
	
	Applying this argument successively for $j=1,2,3,\dots$ and taking a diagonal subsequence, we obtain a function $\psi_\infty$ such that $\psi_n\to\psi_\infty$ in $C(L_j)$ and in $C^2(K_j)$ for every $j\in\mathbb N_+$. This is exactly convergence in the Fr\'echet topology of $C^2_{\rm loc}(\Pi_+)\cap C_{\rm loc}\left(\overline{\Pi_+}\right)$. Therefore every sequence in $\mathcal G_{\mu,M'}(\mathcal B_{M'})$ has a convergent subsequence in this topology, which leads to Lemma \ref{lem:compactness-G-proof}.
\end{proof}

We now present the proof of Lemma \ref{Lem.G-continuity}.

\begin{proof}[Proof of Lemma \ref{Lem.G-continuity}]
	The continuity statement follows from Lemma \ref{lem:continuity-G-proof}. The compactness statement follows from Lemma \ref{lem:compactness-G-proof}. This completes the proof of Lemma \ref{Lem.G-continuity}.
\end{proof}

\section{Regularity of the fixed point}
\label{sec:smoothness}

In this section, we upgrade the fixed point obtained in Theorem \ref{Thm.fixed-point} to the regularity class required by Theorem \ref{Thm.profile}. The fixed-point construction provides only the a priori membership $\psi_*\in\mathcal A_{M_0}$ and $\Omega_*\in\mathcal B_{M_0'}$, which suffices for the existence argument but does not immediately imply smoothness across the boundaries of the quadrant. We first prove that the elliptic--transport structure propagates the precise endpoint vanishing imposed on the initial data, which yields smoothness of both $\psi_*$ and $\Omega_*$ away from the origin. We then establish the global H\"older bounds for the vorticity profile and for the corresponding Cartesian components of the vorticity matrix.


  Let $\mu_0=\mu_0(d,a)$, $M_0=M_0(d,a)>1$ and $M_0'=M_0'(d,a)>1$ be determined by Definition \ref{Def.nonlinear-map}. For each $\mu\in(\mu_0,1/(d-2))$, we denote by $\psi_*\in\cA_{M_0}$ a fixed point of $\mathcal T_{\mu}:\cA_{M_0}\to\cA_{M_0}$ constructed in Theorem \ref{Thm.fixed-point}, and we also set $\Omega_*:=\cF_{\mu, M_0}(\psi_*)\in \cB_{M'_0}$. Then, with $\Pi_+$ being defined by \eqref{S2eq1}, $(\Omega_*, \psi_*)$ solve
\begin{align}
	\label{eq:omega-transport-fixed}
	&(\mu+\psi_*+z\partial_z\psi_*)r\partial_r\Omega_*
	+
	(\mu-(d-1)\psi_*-r\partial_r\psi_*)z\partial_z\Omega_*
	=
	-\Omega_*\quad\mbox{in} \ \Pi_+,\\
		&\qquad\qquad\qquad\Bigl(
	\partial_r^2+\frac d r\partial_r+\partial_z^2+\frac2z\partial_z
	\Bigr)\psi_*
	=
	-c_* r^{d-3}\frac{\Omega_*}{z},\label{eq:psi-elliptic-fixed}
\end{align}
where    $c_*:=\frac{a}{\mathfrak M(\Omega_*)}>0$. 

Below the parameters $d$, $a$, and $\mu$ are fixed. All implicit constants are allowed to depend on these parameters.

\subsection{Smoothness of the fixed point}\label{Subsec.smoothness}

We prove Proposition \ref{prop:smoothness-fixed-point} by combining the characteristic structure of the transport equation with elliptic regularity for the lifted Poisson problem. The first step is to identify the finite-order endpoint structures of the transported initial data on the fixed arc. These structures are then propagated along characteristics to yield the precise forms of $\Omega_*$ near the two boundary components: an even structure in $r$ near the symmetry axis and an odd structure in $z$ near the plane $z=0$. Once these boundary structures are available, the lifted elliptic equation provides a first regularity gain for $\psi_*$, and the transport and elliptic equations can then be bootstrapped to obtain smoothness.

Let $H_*(r,z):=(r,(\mu+\psi_*(r,z))z)$, and define $\widehat \Omega(R,Z):=\Omega_*(H_*^{-1}(R,Z))$. The image of the initial curve $\Gamma_0$ under $H_*$ is $\Gamma:=\{(R,Z):R^2+Z^2=1,\ R>0,\ Z>0\}$. We parametrize $\Gamma$ by $R=\sin\sigma$, $Z=\cos\sigma$, where $\sigma\in(0,\pi/2)$. The initial value of $\widehat\Omega$ on $\Gamma$ is then given by
\begin{equation}
	\label{eq:initial-data-hatomega}
	\widehat\Omega(\sin\sigma,\cos\sigma)
	=
	\Theta_{\lambda_0}(\sigma)
	(\sin\sigma)^{-(d-1)/(d\mu)}
	(\cos\sigma)^{-1/(d\mu)}
	\chi(\sigma),
\end{equation}
where $\lambda_0=\lambda_0(d,a,\mu,M_0)>1$ is the constant given by Proposition \ref{Prop.F}, with $M=M_0(d,a)$.

\begin{lemma}
	\label{lem:endpoint-initial-data}
{\sl	The initial data \eqref{eq:initial-data-hatomega} have the following endpoint structures: near $(R,Z)=(0,1)$, there exists a smooth function $A_0$ with $A_0(0)>0$ such that $\widehat\Omega(R,Z)=R^{\delta_d} A_0(R^2)$ on $\Gamma$; near $(R,Z)=(1,0)$, there exists a smooth function $B_0$ with $B_0(0)>0$ such that $\widehat\Omega(R,Z)=ZB_0(Z^2)$ on $\Gamma$.}
\end{lemma}

\begin{proof}
	We first consider the endpoint $(R,Z)=(0,1)$, which corresponds to
	$\sigma=0$. By the construction of $\chi$, one has $\chi(\sigma)=1$ for
	$\sigma$ sufficiently close to $0$. For $\sigma>0$ sufficiently small,  by \eqref{Eq.Theta-vanishing-condition},  one has
	\[
	\Theta_{\lambda_0}(\sigma)=\Theta(\lambda_0\sigma)=(\lambda_0\sigma)^{\delta_d+(d-1)/(d\mu)},
	\]
which implies
	\[
	\Theta_{\lambda_0}(\sigma)(\sin\sigma)^{-(d-1)/(d\mu)}
	=\lambda_0^{\delta_d+(d-1)/(d\mu)}(\sin\sigma)^{\delta_d}
	\left(\frac{\sigma}{\sin\sigma}\right)^{\delta_d+(d-1)/(d\mu)}.
	\]
	Since $\sigma/\sin\sigma$ is a smooth positive even function of $\sigma$ near $0$, and $R=\sin\sigma$, the right-hand side  above is of the form $R^{\delta_d} A_1(R^2)$ for some smooth $A_1$. Moreover, $Z^{-1/(d\mu)}=(1-R^2)^{-1/(2d\mu)}$ is also a smooth function of $R^2$ near $R=0$. This shows that $\wh\Omega(R,Z)=R^{\delta_d}A_0(R^2)$ on $\Gamma$ near $(R,Z)=(0,1)$ for some smooth function $A_0$ with $A_0(0)=\lambda_0^{\delta_d+(d-1)/(d\mu)}>0$.
	
	The smoothness structure near $(R,Z)=(1,0)$ can be proved similarly, by using the finite-order vanishing condition \eqref{Eq.chi-vanishing-condition}. We omit the details. 
\end{proof}

We next prove a preliminary boundedness statement. This is the point at which the finite-order vanishing conditions on $\Theta$ and $\chi$ are first used.

\begin{lemma}
	\label{lem:first-boundedness}
{\sl 	The following two statements hold.
	\begin{enumerate}
		\item[(1)]  For each $z_0>0$, there exist $\delta>0$ and $C>0$ such that
		$$|\Omega_*(r,z)|\le C\quad \mbox{ whenever}\ 0<r<\delta \andf |z-z_0|<\delta.$$
		
		\item[(2)] For each $r_0>0$, there exist $\delta>0$ and $C>0$ such that $$\Bigl|\frac{\Omega_*(r,z)}{z}\Bigr|\le C\quad\mbox{ whenever}\ |r-r_0|<\delta \andf 0<z<\delta.$$
	\end{enumerate}}
\end{lemma}

\begin{proof}
	We first deduce from Lemma \ref{Lem.hat-Omega-expression} that
	\[\widehat\Omega(R,Z)=R^{-(d-1)/(d\mu)}Z^{-1/(d\mu)} \Theta_{\lambda_0}(\sigma(R,Z))\chi(\sigma(R,Z))\mathrm e^{J(R,Z)},\]
	where $\sigma(R,Z)$ is the angular coordinate of the point at which the characteristic through $(R,Z)$ meets $\Gamma$, and $J(R, Z)$ is defined in \eqref{Eq.J_psi-def}. By Lemma \ref{Lem.J_psi-bound}, we have $|J(R,Z)|\le C$. 
	
	We first consider the region near $R=0$ and away from $Z=0$. Fix a small neighborhood $\mathcal U=\{(R,Z):0<R<\varepsilon,\ c_0<Z<C_0\}$ of a point on $\{R=0,\ Z>0\}$, where $0<c_0<C_0<\infty$. For $(R,Z)\in \mathcal U$, let $(R_0,Z_0)=(\sin\sigma,\cos\sigma)\in \Gamma$ be the intersection point of the backward characteristic through $(R,Z)$ with the initial curve $\Gamma$. 
	In other words, if the characteristic is parametrized by
	\[\frac{\mathrm dR}{\mathrm ds}=R,\qquad \frac{\mathrm dZ}{\mathrm ds}=h(R,Z)Z,
	\qquad (R,Z)|_{s=0}=(R_0,Z_0),\]
	then $(R,Z)=(R(s),Z(s))$ for some $s\in\mathbb R$. 
Recalling \eqref{Eq.h-bound} and 	 $0<h_*\le h(R,Z)<1$, we find
	\[R=\mathrm{e}^sR_0,
	\qquad Z=Z_0\exp\left(\int_0^s h(R(\tau),Z(\tau))\,\mathrm d\tau\right).\]
	The assumption $Z\ge c_0$ implies that $s$ cannot be arbitrarily negative. Indeed, if $s<0$, then using $h\ge h_*$ and $Z_0\le 1$ gives
	\[Z=Z_0\exp\left(-\int_s^0 h(R(\tau),Z(\tau))\,d\tau\right)\le \exp(h_*s),\]
	which implies 
	 $$c_0\le Z\le \mathrm{e}^{h_*s} \Rightarrow s\ge h_*^{-1}\log c_0. $$ Consequently, $R_0=\mathrm{e}^{-s}R\le C_{\mathcal U}R$. Hence $(R_0, Z_0)$ lies near $(0,1)$. The characteristic representation may now be written in the form
	\[
	\widehat\Omega(R,Z)=	\left(\frac{R_0}{R}\right)^{(d-1)/(d\mu)}\left(\frac{Z_0}{Z}\right)^{1/(d\mu)}
	\widehat\Omega(R_0,Z_0)\,\mathrm{e}^{J(R,Z)}.\]
	By Lemma \ref{lem:endpoint-initial-data}, for $(R_0, Z_0)\in\Gamma$ near $(0,1)$, we have $\widehat\Omega(R_0,Z_0)=R_0^{\delta_d}A_0(R_0^2)$, where $A_0$ is smooth. Thus, $|\widehat\Omega(R_0,Z_0)|\le C$. Moreover, on $\mathcal U$ we have $Z\ge c_0$, $Z_0\le 1$, the ratio $R_0/R$ is uniformly bounded by the estimate above, and $|J(R,Z)|\le C$ by Lemma \ref{Lem.J_psi-bound}. Therefore, $|\widehat\Omega(R,Z)|\le C_{\mathcal U}$ in a neighborhood $\mathcal U$ of a point on $\{R=0,\ Z>0\}$. Finally,  as $R=r$ and $Z=(\mu+\psi_*(r,z))z$, and $0<\psi_*\le a_1$, any neighborhood of a point $(0,z_0)$ with $z_0>0$ is mapped into such a set $\mathcal U$ after possibly shrinking it. This proves the first assertion.
	
	The second assertion follows directly from the defining property $\Omega_*(r,z)\lesssim z(r^{-\gamma}+r^{-\gamma_1})$ in the definition of $\cB_{M'}$, see \eqref{Eq.BM'-def}. This finishes the proof of Lemma \ref{lem:first-boundedness}.
	\end{proof}

We now turn to the first elliptic gain.

\begin{lemma}
	\label{lem:first-elliptic-gain}
{\sl Let $D$ be defined by \eqref{S2eq1}.	For each $0<\alpha'<1$, we have $\psi_*\in C^{1,\alpha'}_{\rm loc}(D)$. 
	More precisely, if $X=(x,y)\in\mathbb R^{d+1}\times\mathbb R^3$ with $r=|x|$ and $z=|y|$, and $\Psi_*(X):=\psi_*(|x|,|y|)$, then $\Psi_*\in C^{1,\alpha'}_{\rm loc}(\mathbb R^{d+4}\setminus\{0\})$ for every $0<\alpha'<1$. }
\end{lemma}

\begin{proof}
	By the elliptic construction of $\mathcal G_{\mu,M_0'}$, the function
	$\Psi_*$ weakly solves
	\begin{equation}
		\label{eq:lifted-poisson-fixed}
		-\Delta_X\Psi_*(X)
		=
		 c_*|x|^{d-3}|y|^{-1}\Omega_*(|x|,|y|)=:c_*F_*(X)	\quad\mbox{ in}\ \mathbb R^{d+4}\setminus\{0\}.\end{equation}
	 By Lemma \ref{lem:first-boundedness}, $F_*(X)$ is locally bounded in $\mathbb R^{d+4}\setminus\{0\}$. Indeed, near $\{|y|=0,\ |x|>0\}$ this follows from the boundedness of $\Omega_*/z$, and near $\{|x|=0,\ |y|>0\}$ it follows from the boundedness of $\Omega_*$.
	
	Therefore, by the local Calderon-Zygmund estimate, for every $1<q<+\infty$, we have $\Psi_*\in W^{2,q}_{\rm loc}(\R^{d+4}\setminus\{0\})$. Taking $q>d+4$ and applying the Sobolev embedding theorem, we obtain $\Psi_*\in C^{1,\alpha'}_{\rm loc}(\mathbb R^{d+4}\setminus\{0\})$ for every $0<\alpha'<1$. Restricting to $x=(r,0,\ldots,0)\in\R^{d+1}$ and $y=(z,0,0)\in\R^3$, we obtain $\psi_*\in C^{1,\alpha'}_{\rm loc}(D)$ for every $0<\alpha'<1$. 
\end{proof}

The next lemma provides the transport regularity input needed for the elliptic bootstrap. The point is to prove regularity of the lifted density appearing in the Poisson equation directly in Cartesian variables.

\begin{lemma}\label{lem:lifted-density-regularity}
{\sl	Let $k\in\N_+$ and $0<\alpha'<1$. Assume that $\Psi_*(X)=\psi_*(|x|,|y|)$ belongs to $C^{k,\alpha'}_{\rm loc}(\mathbb R^{d+4}\setminus\{0\})$. Then for $F_*(X)$ defined by \eqref{eq:lifted-poisson-fixed}, one has 	\begin{equation} \label{S6eq1} F_*(X)\in C^{k-1,\alpha'}_{\rm loc}(\mathbb R^{d+4}\setminus\{0\}).\end{equation}}
\end{lemma}

\begin{proof}
As the estimate \eqref{S6eq1} on the region $|x|>0$, $|y|>0$ follows from the standard  local regularity of the transport equation, below we discuss only the two boundary cases. Recall that $\widehat\Omega(R,Z)=\Omega_*(H_*^{-1}(R,Z))$ solves \eqref{Eq.Omega-transport}
	with $h(R,Z)=\frac{\mu-(d-1)\psi_*(H_*^{-1}(R,Z))}{\mu+\psi_*(H_*^{-1}(R,Z))}$.
	
	We first consider a point with $|x|=0$ and $|y|>0$. After a rotation in the
	$y$-variables, it suffices to work near a point $(0,z_0)$ with $z_0>0$. Since \(\Psi_*\in C^{k,\alpha'}\) and	\(\mu+\psi_*+z\partial_z\psi_*>0\), the map	\((x,z)\mapsto (x,(\mu+\psi_*(|x|,z))z)\) has a \(C^{k,\alpha'}\) local inverse. Hence \(h(|x|,Z)\), viewed as a function of \((x,Z)\), is \(C^{k,\alpha'}\).
	
 For \(x\ne0\), set	\(q(x,Z):=|x|^{d-3}\widehat\Omega(|x|,Z)\). Then in view of \eqref{Eq.Omega-transport}, \(q\) solves
	\[x\cdot\nabla_x q+h(|x|,Z)Z\partial_Zq=\Bigl(d-3-\frac1{d\mu}\bigl(Z\partial_Zh+h+d-1\bigr)\Bigr)q.\]
	On the initial surface \(|x|^2+Z^2=1\), Lemma \ref{lem:endpoint-initial-data} ensures that	\(q(x,Z)=|x|^{d-3+\delta_d}A_0(|x|^2)\) near $(x,Z)=(0,1)$, which is smooth	as \(d-3+\delta_d\) is a nonnegative even integer. The characteristic	map generated by the vector field \((x,h(|x|,Z)Z)\), with initial points on	\(|x|^2+Z^2=1\), is a \(C^{k,\alpha'}\) diffeomorphism near the characteristic	from \((0,1)\) to \((0,Z_*)\), where \(Z_*=(\mu+\psi_*(0,z_0))z_0\). The	coefficient in the scalar ODE for \(q\) is \(C^{k-1,\alpha'}\). Therefore	\(q\in C^{k-1,\alpha'}\) near \((0,Z_*)\). Returning to the original variables,	\(Z=(\mu+\Psi_*(x,y))|y|\), and hence	\(F_*(x,y)=|y|^{-1}q(x,Z)\) near \(|x|=0\), \(|y|=z_0\). This proves the	 claim in this case.
	
	We next consider a point with \(|x|>0\) and \(|y|=0\). Since \(R=|x|\) is
	smooth near such a point, we may use \((R,y)\) as local variables. The map
	\((R,y)\mapsto (R,(\mu+\Psi_*)y)\) has a \(C^{k,\alpha'}\) local inverse near
	\(y=0\). Thus \(h(R,|\eta|)\), viewed as a function of \((R,\eta)\), is
	\(C^{k,\alpha'}\). 
	
	For \(\eta\ne0\), set
	\(p(R,\eta):=R^{d-3}|\eta|^{-1}\widehat\Omega(R,|\eta|)\). Then in view of \eqref{Eq.Omega-transport},  \(p\) solves
	\[
	R\partial_Rp+h(R,|\eta|)\eta\cdot\nabla_\eta p
	=
	\Bigl(d-3-h-\frac1{d\mu}\bigl(\eta\cdot\nabla_\eta h+h+d-1\bigr)\Bigr)p .
	\]
	On the initial surface \(R^2+|\eta|^2=1\), Lemma
	\ref{lem:endpoint-initial-data} gives \(p(R,\eta)=R^{d-3}B_0(|\eta|^2)\)
	near \((R,\eta)=(1,0)\), which is smooth. The characteristic map generated by
	\((R,h(R,|\eta|)\eta)\) is a \(C^{k,\alpha'}\) diffeomorphism near the
	characteristic from \((1,0)\) to \((R_*,0)\), where \(R_*=|x_0|>0\). Since the
	coefficient in the scalar ODE for \(p\) is \(C^{k-1,\alpha'}\), we obtain
	\(p\in C^{k-1,\alpha'}\). Returning to the original variables leads to
	\[F_*(x,y)=(\mu+\Psi_*(x,y))p(|x|,(\mu+\Psi_*(x,y))y) \in 
	C^{k-1,\alpha'}.\]  This completes the proof  of Lemma \ref{lem:lifted-density-regularity}.
\end{proof}

Once the bootstrap has produced smoothness of \(\Psi_*\), we recover the endpoint structures of \(\Omega_*\).

\begin{lemma}
	\label{lem:smooth-endpoint-structures}
{\sl	Let \(\Psi_*$ belong to $ C^\infty_{\rm loc}(\mathbb R^{d+4}\setminus\{0\})\).
	Then  there hold
	\begin{itemize}
\item[(1)]	Near each point \((0,z_0)\) with \(z_0>0\), there exists a smooth
	function \(A\), with \(A(0,z_0)>0\), such that
	\[
	\Omega_*(r,z)=r^{\delta_d}A(r^2,z).
	\]
	
\item[(2)] Near each point \((r_0,0)\) with \(r_0>0\), there exists a smooth function
	\(B\), with \(B(r_0,0)>0\), such that
	\[
	\Omega_*(r,z)=zB(r,z^2).
	\] \end{itemize} }
\end{lemma}

\begin{proof}
	We first deal with the case  near $(0,z_0)$, with $z_0>0$. By the classical smooth radial
	representation theorem, after possibly shrinking the neighborhood, there exists a
	smooth function \(\psi_{\rm P}\) such that
	\(\psi_*(r,z)=\psi_{\rm P}(r^2,z)\). Write \(S=R^2\). Since
	\(\mu+\psi_*+z\partial_z\psi_*>0\), the map
	\((S,z)\mapsto (S,(\mu+\psi_{\rm P}(S,z))z)\) has a smooth local inverse
	 \(z=\zeta_{\rm P}(S,Z)\). So, 
	\[h(R,Z)=h_{\rm P}(R^2,Z)\with
	h_{\rm P}(S,Z):=\frac{\mu-(d-1)\psi_{\rm P}(S,\zeta_{\rm P}(S,Z))}
	{\mu+\psi_{\rm P}(S,\zeta_{\rm P}(S,Z))},
	\]
	and $h_{\rm P}(S,Z)$ is smooth. 
	
	For \(R>0\), define
	\(\widetilde\Omega(S,Z):=R^{-\delta_d}\widehat\Omega(R,Z)\). Then by virtue of \eqref{Eq.Omega-transport}, 	\(\widetilde\Omega\) solves
	\[
	2S\partial_S\widetilde\Omega+h_{\rm P}(S,Z)Z\partial_Z\widetilde\Omega
	=-c_{\rm P}(S,Z)\widetilde\Omega,
	\]
	where \(c_{\rm P}:=\delta_d+\frac1{d\mu}(Z\partial_Zh_{\rm P}+h_{\rm P}+d-1)\)
	is smooth. On the initial curve \(S+Z^2=1\), Lemma
	\ref{lem:endpoint-initial-data} gives the smooth initial value
	\(\widetilde\Omega(s,\sqrt{1-s})=A_0(s)\). The characteristic map in
	\((S,Z)\) is a smooth local diffeomorphism near the characteristic from
	\((0,1)\) to \((0,Z_*)\), where \(Z_*=(\mu+\psi_*(0,z_0))z_0\). Therefore
	\(\widetilde\Omega\) is smooth there, and after returning to \((r,z)\) variables, we find
	\(\Omega_*(r,z)=r^{\delta_d}A(r^2,z)\) with \(A\) smooth. Since the initial
	value \(A_0(0)\) is positive and the characteristic formula only multiplies
	by a positive exponential factor, \(A(0,z_0)>0\).
	
	The proof near $(r_0,0)$ with $r_0>0$ is analogous, and is omitted. 
\end{proof}

We now prove the elliptic Schauder upgrade.

\begin{lemma}
	\label{lem:elliptic-upgrade}
{\sl	Let $k\in\N_+$ and $0<\alpha'<1$. Let $\Psi_*(X)=\psi_*(|x|,|y|)$ belong to $C^{k,\alpha'}_{\rm loc}(\mathbb R^{d+4}\setminus\{0\})$. Then $\Psi_*\in C^{k+1,\alpha'}_{\rm loc}(\mathbb R^{d+4}\setminus\{0\})$. Consequently, $\psi_*\in C^{k+1,\alpha'}_{\rm loc}(D)$, where $D$ is defined by \eqref{S2eq1}.}
\end{lemma}

\begin{proof}
	By Lemma \ref{lem:lifted-density-regularity}, $F_\ast$  belongs to
	$C^{k-1,\alpha'}_{\rm loc}(\mathbb R^{d+4}\setminus\{0\})$. Combining this with \eqref{eq:lifted-poisson-fixed} and the local Schauder estimate, 
	we deduce  that $\Psi_*\in C^{k+1,\alpha'}_{\rm loc}(\mathbb R^{d+4}\setminus\{0\})$. Restricting to $x=(r,0,\ldots,0)\in\R^{d+1}$ and $y=(z,0,0)\in\R^3$, we obtain $\psi_*\in C^{k+1,\alpha'}_{\rm loc}(D)$.
\end{proof}

Let us present the  proof of  Proposition \ref{prop:smoothness-fixed-point}.

\begin{proof}[Proof of Proposition \ref{prop:smoothness-fixed-point}]
We first get, 	by applying Lemma \ref{lem:first-elliptic-gain}, that for every $0<\alpha'<1$,  $\Psi_*\in C^{1,\alpha'}_{\rm loc}(\mathbb R^{d+4}\setminus\{0\})$. By applying Lemma \ref{lem:elliptic-upgrade} with $k=1$, we obtain $\Psi_*\in C^{2,\alpha'}_{\rm loc}(\mathbb R^{d+4}\setminus\{0\})$. Repeating the same argument inductively yields $\Psi_*\in C^{k,\alpha'}_{\rm loc}(\mathbb R^{d+4}\setminus\{0\})$ for every $k\in\N_+$ and every $0<\alpha'<1$. Hence $\Psi_*\in C^\infty_{\rm loc}(\mathbb R^{d+4}\setminus\{0\})$. Restricting to $x=(r,0,\ldots,0)$ and $y=(z,0,0)$, we conclude that $\psi_*\in C^\infty(D)$.
	
	By Lemma \ref{lem:smooth-endpoint-structures}, near each point $(0,z_0)$ with $z_0>0$, we have $\Omega_*(r,z)=r^{\delta_d}A(r^2,z)$ for some smooth $A$ with $A(0,z_0)>0$. Near each point $(r_0,0)$ with $r_0>0$, we have $\Omega_*(r,z)=zB(r,z^2)$ for some smooth $B$ with $B(r_0,0)>0$. In particular, $\Omega_*\in C^\infty(D)$. This finishes the proof of  Proposition \ref{prop:smoothness-fixed-point}.
\end{proof}

\subsection{H\"older regularity of the fixed point}
The purpose of this subsection is to prove the global H\"older regularity asserted in Proposition \ref{Prop.Holder-fixed-point}. Unlike the smoothness statement in the previous subsection, this part concerns the behavior at the origin and therefore relies on the sharp pointwise bounds built into $\mathcal B_{M_0'}$, together with the far-field decay obtained in Proposition \ref{Prop.F}. The first lemma establishes the H\"older regularity of the scalar vorticity profile $r^{d-2}\Omega_*$ on the half-plane after the odd extension in $z$. The second lemma then transfers this estimate to the Cartesian components $x_i r^{d-3}\Omega_*$ of the vorticity matrix, including continuity and H\"older control across the symmetry axis.

\begin{lemma}[Global H\"older regularity of $r^{d-2}\Omega_*$]\label{Lem.Holder-rd-2-omega}
{\sl	
Let $\mu\in(\mu_0,1/(d-2))$ with $\mu_0$ being chosen sufficiently close to $1/(d-2)$ so that
	\[
	\alpha_*:=d-2-\frac{1}{\mu+a}\in(0,1).
	\]
	Let $\psi_*\in \cA_{M_0}$ be a fixed point of the map $\cT_\mu$ (see \eqref{S2eq5}) and  $\Omega_*:=\cF_{\mu,M_0}(\psi_*)\in \cB_{M_0'}$.  Then by extending $\Omega_\ast$ oddly in $z$, there exists a constant $C=C(d,a,\mu)>1$ such that
	\begin{equation}\label{S6eq2} 
	r^{d-2}\Omega_*\in C^{\alpha_*}\bigl(\{(r,z):r\geq 0,\ z\in\mathbb R\}\bigr)\andf \bigl\|r^{d-2}\Omega_*\bigr\|_{C^{\alpha_*}(\{r\geq 0\})}\leq C.\end{equation}	}
\end{lemma}

\begin{proof}
	We write the proof for the $z$-odd extension of $\Omega_*$. All estimates below are first obtained for $r>0$ and $z\neq 0$, and then extended to the axes by continuity. 
	Indeed in view of \eqref{Eq.BM'-def},
	by replacing \(z\) there  by \(|z|\), we find
	\begin{equation}\label{S6eq3} 	|\Omega_*(r,z)|\leq C\min\left\{r^{-\frac{d-1}{d\mu}}|z|^{-\frac1{d\mu}},\ |z|r^{-\gamma}+|z|r^{-\gamma_1}\right\},
	\end{equation}
		with \(\gamma\) and \(\gamma_1\) being determined by \eqref{Eq.gamma-gamma1-def}. 
		
		As \(a_1>a\) and \(\gamma_1<\gamma<d-2\), for \(0<r\leq 1\) we deduce from \eqref{S6eq3} that
	\[
	|r^{d-2}\Omega_*(r,z)|\leq C\min\left\{r^{d-2-\frac{d-1}{d\mu}}|z|^{-\frac1{d\mu}},\ |z|r^{d-2-\gamma}\right\}.
	\]
	Using \(\min\{A,B\}\leq A^{\frac{d\mu}{d\mu+1}}B^{\frac1{d\mu+1}}\), we obtain
	\[
	|r^{d-2}\Omega_*(r,z)|\leq C r^{\frac{d\mu}{d\mu+1}\left(d-2-\frac{d-1}{d\mu}\right)+\frac1{d\mu+1}(d-2-\gamma)},
	\]
	from which and 
	\[
	\frac{d\mu}{d\mu+1}\Bigl(d-2-\frac{d-1}{d\mu}\Bigr)+\frac1{d\mu+1}(d-2-\gamma)
	=d-2-\frac{d-1+\gamma}{d\mu+1}
	=d-2-\frac1{\mu+a}=\alpha_*,
	\]
	we infer
	\begin{equation}\label{S6eq4} 	|r^{d-2}\Omega_*(r,z)|\leq C r^{\alpha_*},\qquad 0<r\leq 1,\ z\in\mathbb R.
	\end{equation}
	In particular, \(r^{d-2}\Omega_*\) extends continuously to \(r=0\) by setting its value equal to zero.
	
	Next we prove the global boundedness of \(r^{d-2}\Omega_*\). 
As $\mu<1/(d-2),$	we get, by using the far-field estimate \eqref{Eq.Omega-decay} that for \(|(r,z)|\geq 1\),
	\[
	|r^{d-2}\Omega_*(r,z)|\leq C r^{d-2}|z| |(r,z)|^{-1-\frac1\mu}
	\leq C |(r,z)|^{d-2-\frac1\mu}\leq C,
	\]
	which together with \eqref{S6eq4} ensures that 
	\begin{equation}\label{Eq.r^d-2-Omega-est1}
	|r^{d-2}\Omega_*(r,z)|\leq C \andf	|r^{d-2}\Omega_*(r,z)|\leq C r^{\alpha_*},\qquad r\geq 0,\ z\in\mathbb R.
	\end{equation}
	Whereas if \(0<r\leq 1\) and \(|z|\leq 1\),  we deduce from \eqref{S6eq3}  that
		\[
	|r^{d-2}\Omega_*(r,z)|\leq C|z|\left(r^{d-2-\gamma}+r^{d-2-\gamma_1}\right)\leq C|z|.
	\]
	If \(|(r,z)|\geq 1\), due to $d-3-1/\mu<0,$ we get, by  applying \eqref{Eq.Omega-decay} once again, that
	\[|r^{d-2}\Omega_*(r,z)|\leq C r^{d-2}|z| |(r,z)|^{-1-\frac1\mu}\leq C |z| |(r,z)|^{d-3-\frac1\mu}\leq C|z|,\]
As a result, it follows that 
 \[|r^{d-2}\Omega_*(r,z)|\leq C|z|\ \ \mbox{ whenever}\  \ |z|\leq 1. \]
 For \(|z|\geq 1\), the above estimate follows from \eqref{Eq.r^d-2-Omega-est1}. 
   Therefore
	\begin{equation}\label{Eq.r^d-2-Omega-est2}
		|r^{d-2}\Omega_*(r,z)|\leq C |z|^{\alpha_*},\qquad r\geq 0,\ z\in\mathbb R.
	\end{equation}
	
	On the other hand, it follows from \eqref{Eq.BM'-def} that	\[
	|r\partial_r\Omega_*|+|z\partial_z\Omega_*|\leq C|\Omega_*|,
	\]
	which implies that for \(r>0\) and \(z\neq 0\),
	\[
	|r\partial_r(r^{d-2}\Omega_*)|+|z\partial_z(r^{d-2}\Omega_*)|\leq C|r^{d-2}\Omega_*|.
	\]
	By using the two global pointwise bounds \eqref{Eq.r^d-2-Omega-est1} and \eqref{Eq.r^d-2-Omega-est2}, we obtain
	\begin{equation}\label{Eq.S70716}
		|\partial_r(r^{d-2}\Omega_*)|\leq C r^{\alpha_*-1},\qquad
		|\partial_z(r^{d-2}\Omega_*)|\leq C |z|^{\alpha_*-1}.
	\end{equation}
	As \(0<\alpha_*<1\), 
	for any fixed \(z\in\mathbb R\) and any \(r,r'\geq 0\), we find
	\begin{equation}\label{Eq.r^d-2-Omega-Holde-r}
		\bigl|r^{d-2}\Omega_*(r,z)-(r')^{d-2}\Omega_*(r',z)\bigr|\leq
		\Bigl|\int_{r'}^r\partial_\tau(\tau^{d-2}\Omega_*)	\,{\rm d}\tau\Bigr|\leq	 C|r-r'|^{\alpha_*}.
	\end{equation}
	Along the same line, for any fixed \(r\geq 0\) and any \(z,z'\in\mathbb R\), integration in the \(z\)-variable yields the same estimate when \(z,z'\) have the same sign. If \(z\) and \(z'\) have opposite signs, we instead get, by  using the bound \(|r^{d-2}\Omega_*(r,z)|\leq C|z|^{\alpha_*}\) and the fact \(r^{d-2}\Omega_*(r,0)=0\), that
	\begin{equation}\label{Eq.r^d-2-Omega-Holde-z}
		|r^{d-2}\Omega_*(r,z)-r^{d-2}\Omega_*(r,z')|\leq C|z-z'|^{\alpha_*}.
	\end{equation}
	
	Combining \eqref{Eq.r^d-2-Omega-Holde-r} with \eqref{Eq.r^d-2-Omega-Holde-z}, for arbitrary \((r,z),(r',z')\in\{r\geq 0\}\), we obtain
	\[|r^{d-2}\Omega_*(r,z)-(r')^{d-2}\Omega_*(r',z')|\leq C|r-r'|^{\alpha_*}+C|z-z'|^{\alpha_*}\leq C|(r,z)-(r',z')|^{\alpha_*},\]
	which leads to \eqref{S6eq2}. This completes the proof of Lemma \ref{Lem.Holder-rd-2-omega}.
	\end{proof}

\begin{lemma}[Global H\"older regularity of the Cartesian components]
	\label{lem:global-holder-cartesian-omega}
	{\sl  For each \(1\leq i\leq d-1\), define \(x_i r^{d-3}\Omega_*(r,z)\) to be zero on the axis \(\{r=0\}\), where \(r=(x_1^2+\cdots+x_{d-1}^2)^{1/2}\) and \(z=x_d\). Then under the assumption  of Lemma \ref{Lem.Holder-rd-2-omega},  there exists a constant \(C=C(d,a,\mu)>1\) such that	
		\begin{equation}\label{S6eq5} 	x\mapsto x_i r^{d-3}\Omega_*(r,z)\in C^{\alpha_*}(\mathbb R^d)\andf
	\bigl\|x_i r^{d-3}\Omega_*(r,z)\bigr\|_{C^{\alpha_*}(\mathbb R^d)}\leq C.\end{equation} }\end{lemma}

\begin{proof}
As \(x_i r^{d-3}\Omega_*=(x_i/r)r^{d-2}\Omega_*\) for \(r>0\), we deduce from
\eqref{Eq.r^d-2-Omega-est1} that
	\[
	|x_i r^{d-3}\Omega_*(r,z)|\leq C r^{\alpha_*},
	\]
	which ensures that $x_i r^{d-3}\Omega_*(r,z)$ on the axis \(\{r=0\}\) is continuous.
	
	We first prove the H\"older estimate in the transverse variables. Fix \(z\in\mathbb R\). If \(z=0\), then the oddness of \(\Omega_*\) in \(z\) leads to \(x_i r^{d-3}\Omega_*(r,0)=0\), so there is nothing to prove. Let \(z\neq0\). For \(r>0\),  by virtue  of \eqref{Eq.r^d-2-Omega-est1} and \eqref{Eq.S70716}, we get, by differentiating in \(x'=(x_1,\ldots,x_{d-1})\), that 
	\[\bigl|\nabla_{x'}\bigl(x_i r^{d-3}\Omega_*(r,z)\bigr)\bigr|\leq C\Bigl(\frac{|r^{d-2}\Omega_*(r,z)|}{r}+|\partial_r(r^{d-2}\Omega_*)(r,z)|\Bigr)\leq C r^{\alpha_*-1}.\]
	Take \(x',y'\in\mathbb R^{d-1}\). If \(|x'-y'|\leq |x'|/2\), then the line segment joining \(x'\) and \(y'\) lies in the region where \(r\geq |x'|/2\). Hence integration along this segment gives
	\[\bigl|x_i r^{d-3}\Omega_*(|x'|,z)-y_i r^{d-3}\Omega_*(|y'|,z)\bigr|
	\leq C |x'|^{\alpha_*-1}|x'-y'|\leq C|x'-y'|^{\alpha_*}.
	\]
	If instead \(|x'-y'|>|x'|/2\), then \(|x'|\leq 2|x'-y'|\) and \(|y'|\leq |x'|+|x'-y'|\leq 3|x'-y'|\). By using the pointwise bound just proved, we find
	\[\bigl|x_i (r^{d-3}\Omega_*)(|x'|,z)-y_i (r^{d-3}\Omega_*)(|y'|,z)\bigr|\leq C(|x'|^{\alpha_*}+|y'|^{\alpha_*})\leq C|x'-y'|^{\alpha_*}.\]
	Therefore, for every fixed \(z\in\mathbb R\), we arrive at
	\begin{equation}\label{Eq.Omega-Holder-r}
		\bigl|x_i (r^{d-3}\Omega_*)(|x'|,z)-y_i (r^{d-3}\Omega_*)(|y'|,z)\bigr|\leq C|x'-y'|^{\alpha_*}.
	\end{equation}
	
	Next we estimate in the \(z\)-variable. For fixed \(x'\), if \(r=0\), both values are zero. If \(r>0\), then \(|x_i/r|\leq 1\), and it follows from \eqref{Eq.r^d-2-Omega-Holde-z} that 	\begin{equation}\label{Eq.Omega-Holder-z}
		\bigl|x_i r^{d-3}\Omega_*(r,z)-x_i r^{d-3}\Omega_*(r,z')\bigr|\leq \bigl|r^{d-2}\Omega_*(r,z)-r^{d-2}\Omega_*(r,z')\bigr|\leq C|z-z'|^{\alpha_*}.
	\end{equation}
	
	Finally, let \(x=(x',z)\) and \(y=(y',z')\) be arbitrary points in \(\mathbb R^d\). By combining the transverse estimate \eqref{Eq.Omega-Holder-r} at height \(z\) with the \(z\)-estimate \eqref{Eq.Omega-Holder-z} at transverse position \(y'\), we obtain
	\[\bigl|x_i r^{d-3}\Omega_*(|x'|,z)-y_i r^{d-3}\Omega_*(|y'|,z')\bigr|\leq C\bigl(|x'-y'|^{\alpha_*}+|z-z'|^{\alpha_*}\bigr)\leq C|x-y|^{\alpha_*},\]
which leads to \eqref{S6eq5}.	We thus complete the proof of Lemma \ref{lem:global-holder-cartesian-omega}. \end{proof}

\section{Finite co-dimensional stability}\label{Sec.stability-proof}
In this section, we establish the finite-codimensional stability of the self-similar profiles constructed in Theorem \ref{Thm.profile}. In particular, we supply complete proofs for Lemmas \ref{Lem.outgoing}–\ref{Lem.K-compact}, Lemma \ref{Lem.correction-directions}, Propositions \ref{Prop.N[G]-est}–\ref{Prop.bootstrap}, and Theorem \ref{Thm.global-stability}.


\subsection{Estimates on the transport semigroup}\label{Subsec.transport-semigroup}
In this subsection, we present the proof of  Lemmas \ref{Lem.outgoing} and \ref{Lem.L_tr-semigroup-est}. 

\begin{proof}[Proof of Lemma \ref{Lem.outgoing}] 
Let $(\psi_*, \Omega_*)$ be the functions obtained at the beginning of Section \ref{Sec.stability}. We set 
	\begin{equation}\label{A*B*}  A_*:=\mu+\psi_*+z\pa_z\psi_*,\qquad B_*:=\mu-(d-1)\psi_*-r\pa_r\psi_*.\end{equation}
	Then, by virtue of \eqref{Eq.U-operator}, \eqref{Eq.Q_*-def} 	
		and the fixed-point relation $\psi_*=\frac{a}{\mathfrak M(\Omega_*)}\psi_{0,*}$, we find
\begin{equation}\label{S8eq1} Q_*=(A_*r,B_*z). \end{equation} 
Combining \eqref{S8eq1} 
	and \eqref{Eq.w*-def}, we compute that
		\begin{align*}
		&Q_*\cdot\nabla_{r,z}[(\mu+\psi_*)z]=A_*r\pa_r\psi_*z+B_*z(\mu+\psi_*+z\pa_z\psi_*)\\&
\qquad\qquad\qquad\quad\quad\ \ =A_*z(r\pa_r\psi_*+B_*)
		=A_*z(\mu-(d-1)\psi_*),\\
		&\frac{Q_*\cdot\nabla_{r,z}w_*}{w_*}
		=\frac{A_*r^2+A_*z^2(\mu+\psi_*)(\mu-(d-1)\psi_*)}{r^2+(\mu+\psi_*)^2z^2}.
	\end{align*}
	Since $\psi_*\in \cA_{M_0}\subset\cA^0$, it follows from  \eqref{Eq.z-pa-z-psi-est} and \eqref{Eq.Fr'>0} that
	 \begin{align*} |z\pa_z\psi_*|\le\frac1{10}\psi_*, \quad A_*=\mu+\psi_*+z\pa_z\psi_*>\mu>0 \andf 0<\psi_*\leq a_1<\mu/(d-1).\end{align*}
	We thus	obtain 
	\begin{equation*}
		\frac{Q_*\cdot\nabla_{r,z}w_*}{w_*}\geq A_*	\min\Bigl\{1,\frac{\mu-(d-1)\psi_*}{\mu+\psi_*}\Bigr\}\geq \mu	\min\Bigl\{1,\frac{\mu-(d-1)a_1}{\mu+a_1}\Bigr\},
	\end{equation*} 
	which leads to the outgoing estimate \eqref{S0eqY}.	This finishes the proof of Lemma \ref{Lem.outgoing}.
\end{proof}

\begin{proof}[Proof of Lemma \ref{Lem.L_tr-semigroup-est}]
	Let $\mathcal Y_s(P)$ be the flow generated by $Q_*$, namely
	\[\frac{\mathrm d}{\mathrm ds}\mathcal Y_s(P)=Q_*(\mathcal Y_s(P)),\quad \mathcal Y_0(P)=P,\quad\forall\ P\in\Pi_+.\]
	Here $Q_*=(A_*r,B_*z)$ with $A_*,B_*$  being defined in \eqref{A*B*}. We first choose $\mu_0'\in(\mu_0,1/(d-2))$ sufficiently close to $1/(d-2)$. This choice will be repeated  several times below.  By virtue of the full gradient bound and \eqref{Eq.r-pa_r-psi0est}, which are derived in the proof of Proposition \ref{Prop.G},  and using $0<\psi_*\leq a_1$, $w_*^2=r^2+(\mu+\psi_*)^2z^2$, we find
	\begin{align}\label{Eq.small-first-log-psi-star}
		&|r\pa_r\psi_*|+|z\pa_z\psi_*|\leq C\left(1-\mu(d-2)\right),\\
		\label{Eq.w*-asymptotic}
		&w_*(r,z)\sim_{d,a} \sqrt{r^2+z^2},\quad\forall\ (r,z)\in\Pi_+,
	\end{align}
	Here $C>1$ is independent of $\mu$.	In particular, by increasing $\mu_0'$ if necessary, we obtain
	\begin{equation}\label{Eq.AB-positive-section8}
		0<c_1\leq A_*(r,z),B_*(r,z)\leq c_1^{-1},\qquad \forall\ (r,z)\in\Pi_+,
	\end{equation}
	for some $c_1=c_1(d,a)>0$ (see also \eqref{Eq.B(r,z)-lowerbound}).
	
	Thus,  in view of \eqref{S8eq1} and \eqref{Eq.AB-positive-section8},  the vector field $Q_*$ has at most linear growth and points strictly into $\Pi_+$ along both coordinate directions. Hence,  $\mathcal Y_s$ is globally defined for all $s\in\mathbb R$, preserves $\Pi_+$, and $\mathcal Y_s:\Pi_+\to\Pi_+$ is a bijection with inverse $\mathcal Y_{-s}$.
	
	We first construct the solution of 
\begin{equation}\label{S8eq2} 	\pa_sG+\mathscr L_{\rm tr}G=0. \end{equation} 
 As $\mathscr L_{\rm tr}G=Q_*\cdot\nabla_{r,z}G+G$,  we may write the equation \eqref{S8eq2}  as
	 $$\pa_sG+Q_*\cdot\nabla_{r,z}G=-G.$$ Then along the characteristic starting from $P$, the solution solves
\begin{equation}\label{S8eq3} 
\frac{\mathrm d}{\mathrm ds}G(s,\mathcal Y_s(P))=-G(s,\mathcal Y_s(P))\Longrightarrow G(s,\mathcal Y_s(P))=\mathrm e^{-s}G_0(P),\end{equation}
	or equivalently
	\[
	G(s,P)=\mathrm e^{-s}G_0(\mathcal Y_{-s}(P)),\qquad \forall \ s\geq0,\quad \forall\ P\in\Pi_+.
	\]
	This formula yields a continuous solution of \eqref{S8eq2}  whenever $G_0\in\mathscr X_{0,\delta}$, and it also proves uniqueness for \eqref{S8eq2} among solutions in $\mathscr X_{0,\delta}$, since any two solutions with the same initial data have zero difference along every characteristic.
		
It remains to prove the weighted estimate \eqref{S0eqcL}. By  \eqref{Eq.Omega-rel-transport} and
\eqref{Eq.Q_*-def}, the fixed profile $\Omega_*$ satisfies
	\[Q_*\cdot\nabla_{r,z}\Omega_*=-\Omega_*\Longrightarrow	
	\frac{\mathrm d}{\mathrm ds}\Omega_*(\mathcal Y_s(P))=-\Omega_*(\mathcal Y_s(P)),\]
	from which and \eqref{S8eq3},  we infer
	\[\frac{G(s,\mathcal Y_s(P))}{\Omega_*(\mathcal Y_s(P))}=\frac{G_0(P)}{\Omega_*(P)},\qquad \forall\ s\geq0,\quad \forall\ P\in\Pi_+.\]
It follows from  Lemma \ref{Lem.outgoing}  that
	\[\frac{\mathrm d}{\mathrm ds}w_*(\mathcal Y_s(P))=Q_*\cdot\nabla_{r,z}w_*(\mathcal Y_s(P))\geq \varkappa w_*(\mathcal Y_s(P)).\]
	Then $w_*(\mathcal Y_s(P))\geq \mathrm e^{\varkappa s}w_*(P)$ for all $s\geq0$. As a result, for every $P\in\Pi_+$ and $s\geq0$, we obtain 
	\[w_*(\mathcal Y_s(P))^{-\delta}\Bigl|\frac{G(s,\mathcal Y_s(P))}{\Omega_*(\mathcal Y_s(P))}\Bigr|\leq \mathrm e^{-\delta\varkappa s}w_*(P)^{-\delta}\Bigl|\frac{G_0(P)}{\Omega_*(P)}\Bigr|.\]
	Since $\mathcal Y_s:\Pi_+\to\Pi_+$ is onto, taking the supremum over $P\in\Pi_+$ leads to \eqref{S0eqcL}.	
	We thus complete the proof of Lemma \ref{Lem.L_tr-semigroup-est}.\end{proof}


\subsection{\texorpdfstring{Compactness of $\mathscr K$}{Compactness of mathscr K}}\label{Subsec.K-compactness}
In this subsection, we prove the compactness of $\mathscr K$ defined by \eqref{Eq.mathscrL-def}, i.e., Lemma \ref{Lem.K-compact}. 
The crucial point is the elliptic estimates established in Lemma \ref{Lem.psi_G-est-zero-order} and Lemma \ref{Lem.psi_G-est-first-order}. 
In order to do so, 
for any $G\in \mathscr X_{0,\delta}$, we denote by $\psi_G=\psi_G(r,z)$ the unique solution to the elliptic equation
\begin{equation}\label{Eq.psi_G-elliptic-eq}
	\Bigl(\pa_r^2+\frac dr\pa_r+\pa_z^2+\frac2z\pa_z\Bigr)\psi_G=-\frac{r^{d-3}}{z}G, \quad \lim_{|(r,z)|\to+\infty}\psi_G(r,z)=0.
\end{equation}
Hence, $\mathfrak M(G)=\psi_G(0,0)$. 

\begin{lemma}\label{Lem.psi_G-est-zero-order}
	{\sl 
	Let $\mu_0=\mu_0(d,a)$ be given by Theorem \ref{Thm.fixed-point}. Let
	\begin{equation}\label{Eq.delta-range}
		\mu\in(\mu_0, 1/(d-2))\andf  0<\delta<\min\Bigl\{\frac1\mu-(d-2),\frac12\Bigr\}.
	\end{equation}
	Then there exists a constant $C=C(d,a,\mu,\delta)>1$ such that for all $G\in \mathscr X_{0,\delta}$, we have
	\begin{equation}\label{S8eq4}
		\left|\psi_G(r,z)\right|\leq C\langle r,z\rangle^{d-2-\frac1\mu+\delta}\|G\|_{\mathscr{X}_{0,\delta}}.
	\end{equation}
}
\end{lemma}
\begin{proof}
	As in the proof of Lemma \ref{Lem.Psi0-upperbound}, we lift $\psi_G$ to $\Psi_G(X)=\psi_G(|x|,|y|)$ in $\mathbb R^{d+1}\times\mathbb R^3$.  Then for $X=(x,y), Y=(\xi,\eta)\in \R^{d+1}\times\R^3$,
we may define $F_G(Y)$ and $\Psi_G(X)$	via \eqref{Eq.F_Omega-def} and \eqref{Eq.Psi0-identity}.
	By \eqref{Eq.w*-asymptotic}, $w_*(|\xi|,|\eta|)\sim |\xi|+|\eta|$. Then by virtue of \eqref{Eq.mathscrX0-def} and \eqref{Eq.F_Omega-def}, we find	\begin{align} 
		|F_G(Y)|&\lesssim |Y|^\delta |\xi|^{d-3-\frac{d-1}{d\mu}}|\eta|^{-1-\frac1{d\mu}}\|G\|_{\mathscr X_{0,\delta}}, \label{Eq.FG-bound-first}\\
		|F_G(Y)|&\lesssim  |Y|^\delta\left(|\xi|^{d-3-\gamma}+|\xi|^{d-3-\gamma_1}\right)\|G\|_{\mathscr X_{0,\delta}},\label{Eq.FG-bound-second}
	\end{align} 
	which correspond to \eqref{eq:F-bound-first} and \eqref{eq:F-bound-second}. The only difference lies in  the harmless extra factor $|Y|^\delta$. In particular, if $M_G(\rho):=\int_{|Y|\leq\rho}|F_G(Y)|\,\mathrm dY$, then similar to \eqref{Eq.M(rho)-est} we have
	\begin{equation}\label{Eq.MG-rho-est}
		M_G(\rho)\lesssim \rho^{2d-\frac1\mu+\delta}\|G\|_{\mathscr X_{0,\delta}},\qquad \forall\ \rho>0. 
	\end{equation} 
	Here we use the same integrability conditions as in Lemma \ref{Lem.Psi0-upperbound}; the additional factor $\rho^\delta$ only improves the behavior near the origin. 
	
	For $X=(x,y)$, set $R_X:=1+|x|+|y|\sim\langle X\rangle$. As in \eqref{S4eq8}, we decompose
	\begin{align*}
	\Psi_G(X)&=\beta_d\left(\Psi_{G, \text{near}}(X)+\Psi_{G, \text{far}}(X)\right).
	\end{align*}	
	Then by \eqref{Eq.MG-rho-est} and using the dyadic argument as in \eqref{Eq.Psi0-far-est}, we get
	\[|\Psi_{G,\mathrm{far}}(X)|\lesssim \sum_{k=0}^{\infty}(2^kR_X)^{d-2-\frac1\mu+\delta}\|G\|_{\mathscr X_{0,\delta}}\lesssim R_X^{d-2-\frac1\mu+\delta}\|G\|_{\mathscr X_{0,\delta}},\] 
	where the series converges due to $\delta<1/\mu-(d-2)$. 
	
	It remains to estimate the near part, in which case we have $|Y|\leq |X|+|X-Y|\lesssim R_X$, hence $|Y|^\delta\lesssim R_X^\delta$. 
	Then it follows from the proof of Lemma \ref{Lem.Psi0-upperbound} that 
	 $$|\Psi_{G,\mathrm{near}}(X)|\lesssim R_X^{d-2-1/\mu+\delta}\|G\|_{\mathscr X_{0,\delta}}.$$ 
	 
	 Combining the near and far estimates yields 
	\[|\psi_G(r,z)|=|\Psi_G(X)|\lesssim \langle X\rangle^{d-2-\frac1\mu+\delta}\|G\|_{\mathscr X_{0,\delta}},\] 
	which leads to \eqref{S8eq4}.  This completes the proof of Lemma \ref{Lem.psi_G-est-zero-order}. \end{proof}

\begin{lemma}\label{Lem.psi_G-est-first-order}
	{\sl Under the same assumptions as in  Lemma \ref{Lem.psi_G-est-zero-order}, there exists a constant $C=C(d,a,\mu,\delta)>1$ such that for all $G\in\mathscr X_{0,\delta},$ 
	\begin{equation}\label{S8eq5}
		\left|r\pa_r\psi_G(r,z)\right|+\left|z\pa_z\psi_G(r,z)\right|\leq C|(r,z)|\langle r,z\rangle^{d-3-\frac1\mu+\delta}\|G\|_{\mathscr{X}_{0,\delta}}.
	\end{equation}}
\end{lemma}
\begin{proof} 
We use the same lift-up argument as in the proof of Lemma \ref{Lem.psi_G-est-zero-order}. By 
\eqref{Eq.mathscrX0-def} and	 \eqref{Eq.w*-asymptotic}, we have
	\begin{equation}\label{Eq.FG-by-Omega-star}
		|F_G(Y)|\lesssim \|G\|_{\mathscr X_{0,\delta}}|Y|^\delta F_{\Omega_*}(Y),\qquad F_{\Omega_*}(Y):=|\xi|^{d-3}|\eta|^{-1}\Omega_*(|\xi|,|\eta|).
	\end{equation}
	In particular, there hold \eqref{Eq.FG-bound-first} and \eqref{Eq.FG-bound-second}.
	 Moreover, it follows from the far-field estimate \eqref{Eq.Omega-decay} that for $|Y|\geq1$, there holds 
	\begin{equation}\label{Eq.FG-far-decay}
		|F_G(Y)|\lesssim \|G\|_{\mathscr X_{0,\delta}}|Y|^{-1-\frac1\mu+\delta}|\xi|^{d-3}\lesssim \|G\|_{\mathscr X_{0,\delta}}|Y|^{d-4-\frac1\mu+\delta}.
	\end{equation}
	
	We claim that
	\begin{equation}\label{Eq.nablaX-psiG-est}
		|\nabla_X\Psi_G(X)|\lesssim \langle X\rangle^{d-3-\frac1\mu+\delta}\|G\|_{\mathscr X_{0,\delta}},\qquad \forall\ X\in\mathbb R^{d+4}.
	\end{equation}
	Indeed, as $|\nabla_X|X-Y|^{-d-2}|\lesssim |X-Y|^{-d-3}$, it suffices to deal with estimate
	\[\mathrm I(X):=\int_{\mathbb R^{d+4}}|X-Y|^{-d-3}|F_G(Y)|\,\mathrm dY.\]
As in \eqref{S4eq1}, we decompose 	$\mathrm I(X)$ as	
	 $$\mathrm I(X)=\mathrm I_{\rm near}(X)+\mathrm I_{\rm far}(X). $$
	We first get, by  using \eqref{Eq.MG-rho-est} and the same dyadic argument as in \eqref{Eq.I-far-est}, that
	\[
	\mathrm I_{\rm far}(X)\lesssim \sum_{k=0}^{\infty}(2^kR_X)^{-d-3}(2^kR_X)^{2d-\frac1\mu+\delta}\|G\|_{\mathscr X_{0,\delta}}\lesssim R_X^{d-3-\frac1\mu+\delta}\|G\|_{\mathscr X_{0,\delta}}.
	\]
	Here the series converges due to $d-3-\frac1\mu+\delta<-1$, which follows from $\delta<1/\mu-(d-2)$.
	
	It remains to handle $\mathrm I_{\rm near}(X)$. If $|X|\leq10$, then $R_X\lesssim1$. In this bounded region, we get, by using \eqref{Eq.FG-bound-first}, \eqref{Eq.FG-bound-second},  and similar local estimates in Case 3 of Lemma \ref{Lem.nablaPsi0-upperbound}, that
	$\mathrm I_{\rm near}(X)\lesssim \|G\|_{\mathscr X_{0,\delta}}$. Since $R_X\sim1$ in this case, this implies
	\[\mathrm I_{\rm near}(X)\lesssim R_X^{d-3-\frac1\mu+\delta}\|G\|_{\mathscr X_{0,\delta}}.\]
Whereas for  $|X|\geq10,$ if $|X-Y|\leq R_X/4$, then $|Y|\sim R_X$ and $|Y|\geq1$. Then it follows from \eqref{Eq.FG-far-decay} that $$|F_G(Y)|\lesssim R_X^{d-4-\frac1\mu+\delta}\|G\|_{\mathscr X_{0,\delta}},$$ from which, we infer
	\[
	\mathrm I_{\rm near}(X)\lesssim R_X^{d-4-\frac1\mu+\delta}\|G\|_{\mathscr X_{0,\delta}}\int_{|X-Y|\leq R_X/4}|X-Y|^{-d-3}\,\mathrm dY\lesssim R_X^{d-3-\frac1\mu+\delta}\|G\|_{\mathscr X_{0,\delta}}.
	\]
	
	Combining the estimates of $\mathrm I_{\rm far}$ and $\mathrm I_{\rm near}$ gives rise to \eqref{Eq.nablaX-psiG-est}. 
	Finally,  \eqref{S8eq5} follows  from \eqref{Eq.nablaX-psiG-est}.
\end{proof}


Using Lemma \ref{Lem.psi_G-est-zero-order} and Lemma \ref{Lem.psi_G-est-first-order}, we can prove the boundedness of $\mathscr K:\mathscr X_{0,\delta}\to\mathscr X_{0,\delta}$. 

\begin{lemma}\label{Lem.K-bounded}
	{\sl Under the assumptions of  Lemma \ref{Lem.psi_G-est-zero-order},  $\mathscr K:\mathscr X_{0,\delta}\to\mathscr X_{0,\delta}$ is a bounded operator.}
\end{lemma}
\begin{proof}
	Let $\psi_{0,*}$ be the un-normalized stream function generated by $\Omega_*,$ i.e., $\psi_*=\frac a{\mathfrak{M}(\Omega_*)}\psi_{0,*}$. For $G\in\mathscr X_{0,\delta}$, set
	\begin{equation}\label{Eq.phi_G-def}
		\phi_G:=\psi_G-\frac{\mathfrak M(G)}{\mathfrak{M}(\Omega_*)}\psi_{0,*}=\psi_G-\frac{\mathfrak M(G)}{a}\psi_{*}.
	\end{equation} 
	Then $\phi_G(0,0)=0$, and by the linearity of $U[\cdot]$ in the stream function 
	(see \eqref{Eq.U-operator}), we have 
	$$U[G]\mathfrak{M}(\Omega_*)-\mathfrak M(G)U[\Omega_*]=\mathfrak{M}(\Omega_*)U[\phi_G]. $$
	Here we abuse the notation slightly, i.e., 
	\begin{equation}\label{Eq.U[phi_G]-def}
		U[\phi_G]=\left(U^r[\phi_G], U^z[\phi_G]\right)=\Big(r\left(\phi_G+z\pa_z\phi_G\right), -z\left((d-1)\phi_G+r\pa_r\phi_G\right)\Big).
	\end{equation}
	Then in view of  \eqref{Eq.mathscrL-def}, we have	\begin{equation}\label{Eq.K-phiG-form}
		\mathscr K[G]=\frac a{\mathfrak{M}(\Omega_*)}U[\phi_G]\cdot\nabla_{r,z}\Omega_*.
	\end{equation} 
	By \eqref{Eq.U[phi_G]-def}, \eqref{Eq.K-phiG-form}, \eqref{Eq.U-operator},  the derivative bound $|r\pa_r\Omega_*|+|z\pa_z\Omega_*|\lesssim \Omega_*$ and $\mathfrak{M}(\Omega_*)\sim_{d,a\mu} 1$, we have
	\begin{equation}\label{Eq.K-Omega-bound-by-phi} 
		\frac{|\mathscr K[G](r,z)|}{\Omega_*(r,z)}\lesssim |\phi_G(r,z)|+|r\pa_r\phi_G(r,z)|+|z\pa_z\phi_G(r,z)|. 
	\end{equation} 
	It follows from Lemmas \ref{Lem.psi_G-est-zero-order} and \ref{Lem.psi_G-est-first-order}, and the corresponding estimates for the fixed function $\psi_{0,*}$ that,  for $r+z\geq 1$,
	\begin{align}\label{Eq.phi_G-est1}
		|\phi_G|+|r\pa_r\phi_G|+|z\pa_z\phi_G|\lesssim \langle r,z\rangle^{d-2-\frac1\mu+\delta}\|G\|_{\mathscr X_{0,\delta}}.
	\end{align} 
	 As for $r+z\leq1$, 
	\eqref{Eq.nablaX-psiG-est} gives a uniform $C^1$ bound for $\psi_G$, while $\psi_{0,*}$ is fixed and $C^1$ near the origin. Since $\phi_G(0,0)=0$, we find
	\begin{equation}\label{Eq.phi_G-est2}
		|\phi_G|+|r\pa_r\phi_G|+|z\pa_z\phi_G|\lesssim |(r,z)|\|G\|_{\mathscr X_{0,\delta}},\qquad\text{for}\ \  r+z\leq1.
	\end{equation} 
	
Thanks to \eqref{Eq.delta-range},	by combining \eqref{Eq.phi_G-est1} with \eqref{Eq.phi_G-est2},   and using $w_*(r,z)\sim r+z$ near the origin and $w_*(r,z)\sim\langle r,z\rangle$ away from the origin, 
we obtain 
	\[w_*^{-\delta}\frac{|\mathscr K[G]|}{\Omega_*}\lesssim \|G\|_{\mathscr X_{0,\delta}}.\] 
	This shows that  $\mathscr K:\mathscr X_{0,\delta}\to\mathscr X_{0,\delta}$ is bounded.  We complete the proof of Lemma \ref{Lem.K-bounded}.\end{proof}

We are now ready to prove Lemma \ref{Lem.K-compact}.

\begin{proof}[Proof of Lemma \ref{Lem.K-compact}] 
	We prove compactness of $\mathscr K$ for fixed $\mu, \delta$ satisfying \eqref{Eq.delta-range}.	
	Let $\{G_n\}_{n\geq1}$ be a bounded sequence in $\mathscr X_{0,\delta}$, say $\|G_n\|_{\mathscr X_{0,\delta}}\leq1$. We shall prove that $\{\mathscr K[G_n]\}_{n\geq1}$ has a convergent subsequence in $\mathscr X_{0,\delta}$. 
	We define $\phi_n:=\psi_{G_n}-\frac{\mathfrak M(G_n)}{\mathfrak M(\Omega_*)}\psi_{0,*}$.  
	
It follows from 	\eqref{Eq.K-Omega-bound-by-phi} and \eqref{Eq.phi_G-est2} that for $r+z\leq1$ 
	\begin{align}\label{w1}w_*^{-\delta}\frac{|\mathscr K[G_n](r,z)|}{\Omega_*(r,z)}\lesssim (r+z)^{1-\delta},\end{align} 
	uniformly in $n$. Thus, the contribution of the region $r+z\leq\varepsilon$ is uniformly small if $\varepsilon>0$ is small. For $r+z\geq1$, 
	we deduce from \eqref{Eq.K-Omega-bound-by-phi} and \eqref{Eq.phi_G-est1} that
	\begin{align}\label{w2}w_*^{-\delta}\frac{|\mathscr K[G_n](r,z)|}{\Omega_*(r,z)}\lesssim \langle r,z\rangle^{d-2-\frac1\mu},\end{align}
	which tends to zero as $r+z\to+\infty$, uniformly in $n$. Therefore, it suffices to prove compactness of $\mathscr K$ on each closed annulus $\mathscr A_{\varepsilon,R}:=\{(r,z)\in\overline{\Pi_+}: \varepsilon\leq r+z\leq R\}$.
	
	 On $\mathscr A_{\varepsilon,R}$, the lifted right-hand sides $|\xi|^{d-3}|\eta|^{-1}G_n(|\xi|,|\eta|)$ are uniformly bounded in $L^\infty_{\rm loc}$ after the lift to $\mathbb R^{d+4}$. This uses the endpoint structures of $\Omega_*$ from Proposition \ref{prop:smoothness-fixed-point}: near $z=0$, $\Omega_*(r,z)=zB(r,z^2)$, and near $r=0$, $\Omega_*(r,z)=r^{\delta_d}A(r^2,z)$. Then the possible factors $|\eta|^{-1}$ and $|\xi|^{d-3}$ are harmless on annuli away from the origin. Standard interior elliptic estimates for the lifted equation, together with Lemma \ref{Lem.psi_G-est-zero-order}, ensures  that $\{\psi_{G_n}\}$ is precompact in $C^1$ on the lift of $\mathscr A_{\varepsilon,R}$. Lemma \ref{Lem.psi_G-est-zero-order} also implies that $\mathfrak{M}(G_n)$ is uniformly bounded, hence has a convergent subsequence. After passing to a subsequence, we may therefore assume that $\phi_n$ converges in $C^1(\mathscr A_{\varepsilon,R})$. Using \eqref{Eq.K-phiG-form}, we  write 
	\[\frac{\mathscr K[G_n]}{\Omega_*}=\frac a{\mathfrak M(\Omega_*)}\left(r(\phi_n+z\pa_z\phi_n)\frac{\pa_r\Omega_*}{\Omega_*}-z((d-1)\phi_n+r\pa_r\phi_n)\frac{\pa_z\Omega_*}{\Omega_*}\right).\] 
Thanks to  Proposition \ref{prop:smoothness-fixed-point} and the first-order derivative bound in $\mathcal B_{M_0'},$	the coefficients $r\pa_r\Omega_*/\Omega_*$ and $z\pa_z\Omega_*/\Omega_*$ are bounded and continuous on $\mathscr A_{\varepsilon,R}$. Since $w_*^{-\delta}$ is also continuous and bounded on $\mathscr A_{\varepsilon,R}$, it follows that $w_*^{-\delta}\mathscr K[G_n]/\Omega_*$ is precompact in $C(\mathscr A_{\varepsilon,R})$.

 We now use a diagonal argument over $\varepsilon=1/m$ and $R=m$. The uniform smallness near the origin and at infinity proved above allows us to pass from local convergence on annuli to convergence in the full $\mathscr X_{0,\delta}$ norm. Hence, every bounded sequence in $\mathscr X_{0,\delta}$ has a subsequence whose image under $\mathscr K$ converges in $\mathscr X_{0,\delta}$. This proves that $\mathscr K$ is compact. Together with Lemma \ref{Lem.K-bounded},  we complete the proof of Lemma \ref{Lem.K-compact}. 
\end{proof}

\subsection{Second-order estimates for the Newtonian potential}
\begin{lemma}\label{Lem.2-D-interior}
{\sl Let $m\in\Z_{\geq3}$, $X\in \R^m$,   and $\Delta_Xf=F$. Then for any $R>0$ 
	\begin{equation}\label{Eq.2-D-interior0}
		|\nabla f(X)|\leq CR^{-1}\|f\|_{L^{\infty}(B(X, R))}+C\int_{B(X,R)}\frac{|F(Y)|}{|Y-X|^{m-1}}\,\mathrm dY.
	\end{equation}
	Here $C>1$ is a constant depending only on $m$.}
\end{lemma}
\begin{proof}
	The proof is similar to that of \cite[Lemma 3.1]{SWZ2026}. 
	By translation and dilation, we may assume without loss of generality that $X=0$ and $R=1$. Hence, it suffices to prove that
	\begin{equation}\label{Eq.|f(0)|est}
		|\nabla f(0)|\leq C\|f\|_{L^{\infty}(B(0,1))}+C\int_{B(0,1)}\frac{|F(Y)|}{|Y|^{m-1}}\,\mathrm dY.
	\end{equation}
	
	Let $\zeta\in C_c^\infty(\R^m; [0, 1])$ be such that $\zeta\big|_{B(0,1/2)}\equiv1$ and $\operatorname{supp}\zeta\subset B(0,1)$. Then due to $\Delta_Xf=F,$ we have
	\begin{align*}
		\Delta_X (\zeta f)=2\dive (f\nabla \zeta)-f\Delta \zeta+\zeta F,
	\end{align*}
	which implies that for  $K_m(X)=-\frac{1}{m(m-2){\rm v}_m|X|^{m-2}}$,
	\begin{align*}
		&(\zeta f)(X)=\int_{\R^m}K_m(X-Y)\bigl[
		2\dive (f\nabla \zeta)-f\Delta \zeta+\zeta F\bigr](Y)\,\mathrm dY, \\
		&  \nabla f(0)=\nabla(\zeta f)(0) =\int_{\R^m}\nabla K_m(-Y)\bigl[
		2\dive (f\nabla \zeta)-f\Delta \zeta+\zeta F\bigr](Y)\,\mathrm dY.
	\end{align*}
	Then our desired \eqref{Eq.|f(0)|est} follows from the following two estimates:
	
\noindent{\Large$\bullet$}
	As $\zeta\big|_{B(0,1/2)}\equiv1$, $\operatorname{supp}\zeta\subset B(0,1)$ and $|\nabla^2K_m|+|\nabla K_m|\in L^\infty(B(0,1)\setminus B(0,1/2))$, we have
		\begin{align*}
			&\Bigl|\int_{\R^m}\nabla K_m(-Y)\dive (\nabla \zeta f)(Y)\,\mathrm dY\Bigr|+\Bigl|\int_{\R^m}\nabla K_m(-Y)(\Delta \zeta f)(Y)\,\mathrm dY\Bigr|\\
			&\leq C\int_{1/2\leq|Y|<1}|\nabla^2 K_m(-Y)||f(Y)|\,\mathrm dY+C\int_{1/2\leq|Y|<1} |\nabla K_m(-Y)||f(Y)|\,\mathrm dY\\
			&\leq  \ C\|f\|_{L^1(B(0,1))}\leq C\|f\|_{L^{\infty}(B(0,1))}.
		\end{align*}
		
\noindent{\Large$\bullet$}	Due to $|\nabla K_m(Y)|\lesssim |Y|^{1-m}$, we find
		\begin{align*}
			&\Bigl|\int_{\R^m}\nabla K_m(-Y)(\zeta F)(Y)\,\mathrm dY\Bigr|
			\leq \int_{B(0,1)}|\nabla K_m(-Y)||F(Y)|\,\mathrm dY
			\leq\ C\int_{B(0,1)}\frac{|F(Y)|}{|Y|^{m-1}}\,\mathrm dY.
		\end{align*}

	This completes the proof of Lemma \ref{Lem.2-D-interior}.
\end{proof}

\begin{lemma}\label{Lem1}
{	\sl Let $X=(x,y)$, $Y=(\xi,\eta)\in \R^{d+1}\times\R^3$, $F_{\Omega_*}(Y):=|\xi|^{d-3}|\eta|^{-1}\Omega_*(|\xi|,|\eta|)$. Then we have
	\begin{enumerate}[(i)]
		\item[(1)]  if $0<R\leq |x|/2$, 
		\begin{align}\label{I1}I_1(X):=\int_{B(X,R)}|X-Y|^{-d-3}|\xi|^{-1}F_{\Omega_*}(Y)\,\mathrm dY\lesssim
			\langle X\rangle^{d-3-\frac1\mu}|x|^{-2}R, \end{align}
		\item[(2)] if $0<R\leq |y|/2$ and $|X|\geq 10$, 
		\begin{align}\label{I2}I_2(X):=\int_{B(X,R)}|X-Y|^{-d-3}|\eta|^{-1}F_{\Omega_*}(Y)\,\mathrm dY\lesssim
			|X|^{d-4-\frac1\mu}|y|^{-1}R,\end{align}
		\item[(3)]  if $0<R\leq |y|/2$ and $|X|\leq 10$, 
		\begin{equation}\label{S8eq6}
		I_2(X)\lesssim (|x|+R)^{d-3-\gamma}|y|^{-1}R. \end{equation}
	\end{enumerate}}
\end{lemma}
\begin{proof}
(1)	Under the assumption $0<R\leq |x|/2,$ if $|X|\leq 10$, we get, by  using \eqref{Eq.G-bound-second-Lem45} and $d-2>\gamma>\gamma_1$, that
	\begin{equation}\label{X1}
	\begin{split}
		I_{1}(X)
		&\lesssim \int_{|(v,w)|\le R}\frac{|x-v|^{d-4-\gamma}+|x-v|^{d-4-\gamma_1}}{(|v|^2+|w|^2)^{\frac{d+3}{2}}}\,\mathrm dv\,\mathrm dw\lesssim \int_{|v|\le R}|v|^{-d}|x-v|^{d-4-\gamma}\,\mathrm dv\\
		&   \lesssim |x|^{d-4-\gamma}\int_{|v|\le R}|v|^{-d}\,\mathrm dv\lesssim R|x|^{d-4-\gamma}\lesssim R|x|^{-2}.
	\end{split}\end{equation}
	If $|X|\geq 10$, by the far-field estimate \eqref{Eq.Omega-decay}, we have
	$$
	|\xi|^{-1}F_{\Omega_*}(Y)\lesssim |Y|^{-1-\frac1\mu}|\xi|^{d-4},
	$$
	so \begin{equation}\label{X2}
	\begin{split}		I_{1}(X)
		&\lesssim \int_{|(v,w)|\le R}\frac{|X-(v,w)|^{-1-\frac1\mu}|x-v|^{d-4}}{(|v|^2+|w|^2)^{\frac{d+3}{2}}}\,\mathrm dv\,\mathrm dw\\&\lesssim |X|^{-1-\frac1\mu}|x|^{d-4}\int_{|(v,w)|\le R}|(v,w)|^{-d-3}\,\mathrm dv\,\mathrm dw\\
		&\lesssim |X|^{-1-\frac1\mu}|x|^{d-4}R\lesssim |X|^{d-3-\frac1\mu}|x|^{-2} R. 
\end{split}\end{equation}
	 Then \eqref{I1} follows from \eqref{X1} (for $|X|\leq 10$) and \eqref{X2} (for $|X|\geq 10$). 
	
\noindent (2)	If $0<R\leq |y|/2$ and $|X|\geq 10$, it follows from  the far-field estimate \eqref{Eq.Omega-decay} that
	\begin{align*}
		|\eta|^{-1}F_{\Omega_*}(Y)\lesssim |Y|^{-1-\frac1\mu}|\xi|^{d-3}|\eta|^{-1}\lesssim |Y|^{d-4-\frac1\mu}|\eta|^{-1},	
	\end{align*} from which and $|X|\geq |y|\geq 2R,$ we infer
	\begin{equation*}\label{X3}
	\begin{split}		I_{2}(X)
		&\lesssim \int_{|(v,w)|\le R}\frac{|X-(v,w)|^{d-4-\frac1\mu}|y-w|^{-1}}{(|v|^2+|w|^2)^{\frac{d+3}{2}}}\,\mathrm dv\,\mathrm dw\\ \nonumber&\lesssim |X|^{d-4-\frac1\mu}|y|^{-1}\int_{|(v,w)|\le R}|(v,w)|^{-d-3}\,\mathrm dv\,\mathrm dw \lesssim |X|^{d-4-\frac1\mu}|y|^{-1}R,
\end{split}\end{equation*}	which leads to \eqref{I2}. 

\noindent (3) If $0<R\leq |y|/2$ and $|X|\leq 10$, we deduce from \eqref{Eq.G-bound-second-Lem45}, $\gamma>\gamma_1$ 
	and $|X|\geq |y|\geq 2R$ that
	\begin{align*}
		I_{2}(X)
		&\lesssim \int_{|(v,w)|\le R}\frac{|x-v|^{d-3-\gamma}+|x-v|^{d-3-\gamma_1}}{|y-w|(|v|^2+|w|^2)^{\frac{d+3}{2}}}\,\mathrm dv\,\mathrm dw\\ &\lesssim \int_{|(v,w)|\le R}\frac{|x-v|^{d-3-\gamma}}{|y|(|v|^2+|w|^2)^{\frac{d+3}{2}}}\,\mathrm dv\,\mathrm dw\lesssim \int_{|v|\le R}\frac{|x-v|^{d-3-\gamma}}{|y||v|^d}\,\mathrm dv. 
	\end{align*}
	If $|x|\geq 2R$, then we have $|x-v|\sim |x|$ for $|v|\leq R$ and \begin{align*}
		I_{2}(X)&\lesssim \int_{|v|\le R}\frac{|x|^{d-3-\gamma}}{|y||v|^d}\,\mathrm dv\lesssim |x|^{d-3-\gamma}|y|^{-1}R.
	\end{align*}
	If $|x|\leq 2R$, due to  $d-3<\gamma<d-2$, we find
	\begin{align*}
		I_{2}(X)&\lesssim \int_{|v|\le R}\frac{|x-v|^{d-3-\gamma}}{|y||v|^d}\,\mathrm dv\lesssim |y|^{-1}
		\int_{|v|\le R}(|x-v|^{-3-\gamma}+|v|^{-3-\gamma})\,\mathrm dv\lesssim |y|^{-1}R^{d-2-\gamma}. 
	\end{align*} Combining these two inequalities leads to \eqref{S8eq6}. This completes the proof of Lemma \ref{Lem1}. \end{proof}

\begin{lemma}\label{Lem.E12-est}
	{\sl There exists a constant $C=C(d,a,\mu)>1$ such that
	\begin{equation}\label{Eq.nabla^2psi_*-bound}
		|r^2\pa_r^2\psi_*(r,z)|+|rz\pa_r\pa_z\psi_*(r,z)|+|z^2\pa_z^2\psi_*(r,z)|\leq C|(r,z)|\langle r,z\rangle^{d-3-\frac1\mu},\quad\forall\ (r,z)\in\Pi_+.
	\end{equation}	}
\end{lemma}
\begin{proof}
	Similar to the proof of Lemma \ref{Lem.nabla-Xx-Psi0-bound}, we introduce the lift-up $\Psi_{*}(x,y)=\psi_*(|x|,|y|)$, which solves \eqref{eq:lifted-poisson-fixed} 
	 on $\R^{d+4}$, with $x\in\R^{d+1}$ and $y\in\R^{3}$. Then  $\Delta_{X}\partial_i \Psi_*=-c_*\partial_iF_{\Omega_*},$ from which and Lemma \ref{Lem.2-D-interior}, we deduce that for any $R>0,$
	\begin{align}\label{G2}
		&\left|\nabla\partial_i\Psi_*(X)\right|\leq CR^{-1}\|\partial_i\Psi_*\|_{L^{\infty}(B(X,R))}+C\int_{B(X,R)}\frac{|\partial_iF_{\Omega_*}(Y)|}{|X-Y|^{d+3}}\,\mathrm dY.
	\end{align}
	We first estimate the second derivatives containing at least one $x$-derivative. As $\psi_*\in \mathcal A_{M_0}\subset \mathcal A^0$, we have $$|\nabla_r\psi_*(r,z)|\leq \langle r,z\rangle^{-1}\psi_*(r,z)\lesssim\langle r,z\rangle^{d-3-\frac{1}{\mu}} \andf 
	 |\nabla_x\Psi_*(X)|\lesssim\langle X\rangle^{d-3-\frac{1}{\mu}}, $$
from which, \eqref{Eq.nabla-xi-F-Lem45}, \eqref{I1}, 	and taking $R=|x|/2$ in \eqref{G2}, we infer
	\begin{align*}
		\left|\nabla_X\nabla_x\Psi_*(X)\right|
		&\lesssim R^{-1}\langle X\rangle^{d-3-\frac1\mu}+I_1(X)\\ \lesssim &|x|^{-1}\langle X\rangle^{d-3-\frac1\mu}+|x|^{-2}R\langle X\rangle^{d-3-\frac1\mu}\lesssim |x|^{-1}\langle X\rangle^{d-3-\frac1\mu}.
	\end{align*}
	This gives rise to
	\begin{equation}\label{r1}
		|r^2\pa_r^2\psi_*(r,z)|+|rz\pa_r\pa_z\psi_*(r,z)|\lesssim |(r,z)|\langle r,z\rangle^{d-3-\frac1\mu}. 
	\end{equation}
	
	It remains to handle $z^2\pa_z^2\psi_*$. As $\psi_*\in \mathcal A_{M_0}\subset \mathcal A^0$, we have
	\begin{align}
		&|\pa_z\psi_*(r,z)|\leq \langle r,z\rangle^{-1+\frac1{d\mu}}\langle z\rangle^{-\frac1{d\mu}}\psi_*(r,z)
		\lesssim \langle r,z\rangle^{d-3-\frac{d-1}{d\mu}}\langle z\rangle^{-\frac1{d\mu}},\nonumber\\
		&|\nabla_y\Psi_*(X)|\lesssim\langle X\rangle^{d-3-\frac{d-1}{d\mu}}\langle y\rangle^{-\frac1{d\mu}},\label{G3}\\
				& \|\nabla_y\Psi_*\|_{L^{\infty}(B(X,R))}\lesssim\langle X\rangle^{d-3-\frac{d-1}{d\mu}}\langle y\rangle^{-\frac1{d\mu}}.\nonumber
	\end{align}
	If $|X|\geq 10$, we get, by taking $R=|y|/2$  \eqref{I2}, that
	$$I_2(X)\lesssim|X|^{d-4-\frac1\mu}|y|^{-1}R\leq|X|^{d-4-\frac1\mu}.$$
	If $|X|\leq 10$, we get, by using Lemma \ref{Lem1} and $d-2>\gamma>d-3,$ that
	\begin{align*}
		I_{2}(X)&\lesssim (|x|+R)^{d-3-\gamma}|y|^{-1}R=(|x|+|y|/2)^{d-3-\gamma}/2\lesssim |X|^{d-3-\gamma}\lesssim |X|^{-1}. 
	\end{align*}
	Combining the above two inequalities  yields
	$$I_{2}(X)\lesssim \langle X\rangle^{d-3-\frac1\mu}|X|^{-1}.$$
	 Then by  taking $R=|y|/2$  in 	\eqref{G2}, and using  \eqref{G3}, \eqref{I2} and $|\nabla_\eta F_{\Omega_*}|\lesssim |\eta|^{-1}F_{\Omega_*},$ we find
	\begin{align*}
		\left|\nabla_X\nabla_y\Psi_*(X)\right|
		\lesssim&\, R^{-1}\langle X\rangle^{d-3-\frac{d-1}{d\mu}}\langle y\rangle^{-\frac1{d\mu}}+I_2(X)\\ \lesssim &\,|y|^{-1}\langle X\rangle^{d-3-\frac{d-1}{d\mu}}\langle y\rangle^{-\frac1{d\mu}}+|X|^{-1}\langle X\rangle^{d-3-\frac1\mu}\\
		\lesssim&\, |y|^{-1}\langle X\rangle^{d-3-\frac{d-1}{d\mu}}\langle y\rangle^{-\frac1{d\mu}}.
	\end{align*}
As $z=|y|$, this implies $|\pa_z^2\psi_*(r,z)|\lesssim |z|^{-1}\langle r,z\rangle^{d-3-\frac{d-1}{d\mu}}\langle z\rangle^{-\frac1{d\mu}}$ and
	\begin{align}\label{z1}
		|z^2\pa_z^2\psi_*(r,z)|\lesssim |z|\langle r,z\rangle^{d-3-\frac{d-1}{d\mu}}\langle z\rangle^{-\frac1{d\mu}}\lesssim
		|(r,z)|\langle r,z\rangle^{d-3-\frac{1}{\mu}}.
	\end{align}  
	Here we used $|(r,z)|/|z|=\sqrt{1+r^2/|z|^2}\geq \sqrt{1+r^2/\langle z\rangle^2}=\langle r,z\rangle/\langle z\rangle\geq \langle r, z\rangle^{\frac1{d\mu}}\langle z\rangle^{-\frac1{d\mu}}$ (since $d\mu>1$). 
	
Combining the estimates \eqref{r1} and \eqref{z1}, we obtain \eqref{Eq.nabla^2psi_*-bound}.
This  completes the proof of Lemma \ref{Lem.E12-est}.
\end{proof}\begin{lemma}\label{Lem2}
{	\sl  Let $A_*, B_*$ be given by \eqref{A*B*}  and $w_1:=-K(1+w_*)^{d-2-\frac1\mu}.$ There exist constants $C=C(d,a,\mu)>1$, $K=K(d,a,\mu)>0$ such that
	\begin{align}\label{A1}
		&|r\pa_rA_*|+|r\pa_rB_*|+|z\pa_zA_*|+|z\pa_zB_*|\leq C|(r,z)|\langle r,z\rangle^{d-3-\frac1\mu}\leq C,\\
	& |r\pa_rA_*|+|r\pa_rB_*|+|z\pa_zA_*|+|z\pa_zB_*|\leq Q_*\cdot\nabla w_1.\label{S8eq7}
\end{align}}
\end{lemma}
\begin{proof}
	As $\psi_*\in \mathcal A_{M_0}\subset \mathcal A^0$ and $|(r,z)|/|z|\geq  \langle r, z\rangle^{\frac1{d\mu}}\langle z\rangle^{-\frac1{d\mu}}$, we have \begin{align}\label{r2}
	|r\partial_r\psi_*(r,z)|+|z\partial_z\psi_*(r,z)|\leq |(r,z)|\langle r,z\rangle^{-1}\psi_*(r,z)\leq C|(r,z)|\langle r,z\rangle^{d-3-\frac1\mu}.\end{align}
While in view of  \eqref{A*B*}, one has
	\begin{align*}
		r\pa_r A_*=r\pa_r\psi_*+r\pa_r(z\pa_z)\psi_*, \quad r\pa_r B_*=-(d-1)r\pa_r\psi_*-(r\pa_r)^2\psi_*,\\
		z\pa_z A_*=z\pa_z\psi_*+(z\pa_z)^2\psi_*,\quad z\pa_zB_*=-(d-1)z\pa_z\psi_*-r\pa_r(z\pa_z)\psi_*,
	\end{align*}
	from which, \eqref{Eq.nabla^2psi_*-bound}, \eqref{r2} and $\frac1\mu>d-2,$ we infer
	\begin{align*}
		&|r\pa_rA_*|+|r\pa_rB_*|+|z\pa_zA_*|+|z\pa_zB_*|\\ 
		&\lesssim\,|r\pa_r\psi_*|+|z\pa_z\psi_*|+
		|r^2\pa_r^2\psi_*|+|rz\pa_r\pa_z\psi_*|+|z^2\pa_z^2\psi_*|\\
		&\lesssim|(r,z)|\langle r,z\rangle^{d-3-\frac1\mu}\lesssim 1,
	\end{align*} which implies \eqref{A1}.
	
Notice that  $w_*\sim|(r,z)|$. It follows from  \eqref{A1} that \begin{align*}
		&|r\pa_rA_*|+|r\pa_rB_*|+|z\pa_zA_*|+|z\pa_zB_*|\leq C_1w_*(1+w_*)^{d-3-\frac1\mu}.
	\end{align*} As  $w_1=-K(1+w_*)^{d-2-\frac1\mu}$ ($K>0$),  we get, by using Lemma \ref{Lem.outgoing} and $\frac1\mu>d-2,$  that 
	\begin{align*}
		&Q_*\cdot\nabla w_1=K(1/\mu-d+2)(1+w_*)^{d-3-\frac1\mu}Q_*\cdot\nabla w_*\geq K(1/\mu-d+2)(1+w_*)^{d-3-\frac1\mu}\varkappa\, w_*.
	\end{align*} Then \eqref{S8eq7} follows by taking $K=(1/\mu-d+2)^{-1}C_1/\varkappa$.
\end{proof}

\begin{lemma}\label{Lem.psi_G-est-second-order}\def\a{\delta_1}
	{\sl Under the same assumptions as in Lemma \ref{Lem.psi_G-est-zero-order}, there exists a constant $C=C(d,a,\mu,\delta)>1$ such that for all $G\in\mathscr X_{0,\delta}$ with $\|G\|_{\mathscr X_{1,\delta}}<+\infty$ and $\a\in(0,1/2],$ 
	\begin{equation}\label{S8eq8}
		\begin{split}
		&\left|r^2\pa_r^2\psi_G(r,z)\right|+\left|rz\pa_r\pa_z\psi_G(r,z)\right|+\left|z^2\pa_z^2\psi_G(r,z)\right|\\
		&\leq C|(r,z)|\langle r,z\rangle^{d-3-\frac1\mu+\delta}\mathcal N_G \with \mathcal N_G:=\a^{-1}\|G\|_{\mathscr X_{0,\delta}}+\a^{d-2-\gamma}\|G\|_{\mathscr X_{1,\delta}}.		
	\end{split}\end{equation}}
\end{lemma}

\begin{proof}\def\a{\delta_1}
As in the proof of Lemma \ref{Lem.psi_G-est-zero-order},
	we lift $\psi_G$ to $\Psi_G(X)=\psi_G(|x|,|y|)$ in $\mathbb R^{d+1}\times\mathbb R^3,$ so that 
	$$\Delta_{X} \Psi_G=-F_G\with F_G(X):=|x|^{d-3}|y|^{-1}G(|x|,|y|).$$ Then $\Delta_{X}\partial_i \Psi_G=-\partial_iF_G$, we get, by applying  Lemma \ref{Lem.2-D-interior}, that  for any $R>0,$\begin{align}\label{G1}
		&\left|\nabla\partial_i\Psi_G(X)\right|\leq CR^{-1}\|\nabla_X\Psi_G\|_{L^{\infty}(B(X,R))}+C\int_{B(X,R)}\frac{|\partial_iF_G(Y)|}{|X-Y|^{d+3}}\,\mathrm dY.
	\end{align}
Due to  $d-2>\gamma>d-3$ and $\a\in(0,1/2],$
 by virtue of \eqref{Eq.mathscrX0-def} 	 and \eqref{S0eqG1}, together with $w_*(|\xi|,|\eta|)\sim |\xi|+|\eta|$,  for any $Y=(\xi,\eta)\in \R^{d+1}\times\R^3$, we find 
	\begin{equation}\label{Eq.FG-derivative-bounds} 
		|\nabla_\xi F_G(Y)|\lesssim \a^{-1}\mathcal N_G |Y|^\delta|\xi|^{-1}F_{\Omega_*}(Y),\qquad |\nabla_\eta F_G(Y)|\lesssim \a^{\gamma-d+2}\mathcal N_G|Y|^\delta|\eta|^{-1}F_{\Omega_*}(Y). 
	\end{equation}	This is the analogue of the derivative bounds used in Lemma \ref{Lem.E12-est}, with the only new feature given by the factor $|Y|^\delta$. 
	
	We first estimate the second derivatives containing at least one $x$-derivative. 
	Observe that $|Y|^\delta\lesssim |X|^\delta$ on $B(X,R),$ where $R=\a|x|$.  Using
	\eqref{G1}, \eqref{Eq.nablaX-psiG-est}, \eqref{Eq.FG-derivative-bounds} and \eqref{I1}, we obtain
	\begin{align*}
		\left|\nabla_X\nabla_x\Psi_G(X)\right|
		&\lesssim R^{-1}\langle X\rangle^{d-3-\frac1\mu+\delta}\|G\|_{\mathscr X_{0,\delta}}+I_1(X)|X|^\delta\a^{-1}\mathcal N_G\\ 
		&\lesssim \a^{-1}|x|^{-1}\langle X\rangle^{d-3-\frac1\mu+\delta}\|G\|_{\mathscr X_{0,\delta}}+|x|^{-2}R\langle X\rangle^{d-3-\frac1\mu+\delta}\a^{-1}\mathcal N_G\\
		&\lesssim |x|^{-1}\langle X\rangle^{d-3-\frac1\mu+\delta}\mathcal N_G,
	\end{align*}
	which implies 
	\begin{equation}\label{Eq.rr-rz-psiG-est}
		|r^2\pa_r^2\psi_G(r,z)|+|rz\pa_r\pa_z\psi_G(r,z)|\lesssim \mathcal N_G|(r,z)|\langle r,z\rangle^{d-3-\frac1\mu+\delta}. 
	\end{equation}
	
	It remains to estimate $z^2\pa_z^2\psi_G$.  If $|X|\geq 10,$  we get, by taking $R=\a|y|$ in \eqref{I2}, that
	$$ I_2(X)\lesssim|X|^{d-4-\frac1\mu}|y|^{-1}R=\a|X|^{d-4-\frac1\mu}.$$
While if $|X|\leq 10,$  applying Lemma \ref{Lem1} gives 
\begin{align*}
		I_{2}(X)&\lesssim (|x|+R)^{d-3-\gamma}|y|^{-1}R=(|x|+\a|y|)^{d-3-\gamma}\a
		\lesssim |X|^{-1}\a^{d-2-\gamma}.
	\end{align*}
	Combining the above two inequalities gives rise to
	 $$I_{2}(X)\lesssim \a^{d-2-\gamma}\langle X\rangle^{d-3-\frac1\mu}|X|^{-1}. $$
	 Then we get,  by using
	\eqref{G1}, \eqref{Eq.nablaX-psiG-est}, \eqref{Eq.FG-derivative-bounds}, \eqref{I2} and $|Y|^\delta\lesssim |X|^\delta$ on $B(X,R)$ that
	\begin{align*}
		\left|\nabla_X\nabla_y\Psi_G(X)\right|
		&\lesssim R^{-1}\langle X\rangle^{d-3-\frac1\mu+\delta}\|G\|_{\mathscr X_{0,\delta}}+I_2(X)|X|^{\delta}\a^{\gamma-d+2}\mathcal N_G\\ 
		&\lesssim \a^{-1}|y|^{-1}\langle X\rangle^{d-3-\frac1\mu+\delta}\|G\|_{\mathscr X_{0,\delta}}+|X|^{-1}\langle X\rangle^{d-3-\frac1\mu+\delta}\mathcal N_G\\
	&\lesssim |y|^{-1}\langle X\rangle^{d-3-\frac1\mu+\delta}\mathcal N_G.
	\end{align*}
	As $z=|y|$,  we thus obtain
	\begin{equation}\label{Eq.zz-psiG-est}
		|z^2\pa_z^2\psi_G(r,z)|\lesssim \mathcal N_G|(r,z)|\langle r,z\rangle^{d-3-\frac1\mu+\delta}.
	\end{equation} 
	
	Combining \eqref{Eq.rr-rz-psiG-est} and \eqref{Eq.zz-psiG-est} gives rise to \eqref{S8eq8}. This completes the proof of Lemma \ref{Lem.psi_G-est-second-order}. 
\end{proof}

\subsection{Nonlinear estimates}\label{Subsec.nonlinear}
In this subsection, we aim at proving Proposition \ref{Prop.N[G]-est}. 

\begin{proof}[Proof of Proposition \ref{Prop.N[G]-est}]
To simplify notation, we denote
\begin{equation}\label{Eq.mathscrB-def}
	\mathscr B[G]:=U[G]\mathfrak M(\Omega_*)-\mathfrak M(G)U[\Omega_*]=\mathfrak{M}(\Omega_*)U[\phi_G],
\end{equation}
where $\phi_G$ and $U[\phi_G]$ are defined respectively by \eqref{Eq.phi_G-def}  and \eqref{Eq.U[phi_G]-def}. 

Then in view of  \eqref{Eq.mathscrN-def}, we write
\begin{equation}\label{Eq.mathscrN-expression}
\begin{split}
	\mathscr N[G]&=-a\frac{\mathscr B[G]}{\mathfrak{M}(\Omega_*)\cdot\mathfrak M(\Omega)}\cdot\nabla_{r,z}G+a\frac{\mathfrak M(G)\mathscr B[G]}{\left(\mathfrak{M}(\Omega_*)\right)^2\mathfrak M(\Omega)}\cdot\nabla_{r,z}\Omega_*\\
	&=-a\frac{U[\phi_G]}{\mathfrak M(\Omega)}\cdot\nabla_{r,z}G+a\frac{\mathfrak M(G)U[\phi_G]}{\mathfrak{M}(\Omega_*)\cdot\mathfrak M(\Omega)}\cdot\nabla_{r,z}\Omega_*.\end{split}\end{equation}

Under the assumptions in \eqref{Eq.delta-range}, 
we get,
	by using \eqref{Eq.U[phi_G]-def}, \eqref{Eq.phi_G-est1}, \eqref{Eq.phi_G-est2} and \eqref{Eq.w*-asymptotic},  that
	\begin{equation*}
		\Bigl|\frac{U^r[\phi_G]}{r}\Bigr|+\Bigl|\frac{U^z[\phi_G]}{z}\Bigr|\lesssim |(r,z)|\langle r,z\rangle^{d-3-\frac1\mu+\delta}\|G\|_{\mathscr X_{0,\delta}}\lesssim \min\{1, w_*(r,z)^\delta\}\|G\|_{\mathscr X_{0,\delta}},
	\end{equation*}
from which,  we infer
	\begin{equation}\label{Eq.N[G]-est-1}
		\begin{aligned}
			|U[\phi_G]\cdot\nabla_{r,z}G|&\lesssim \Bigl|\frac{U^r[\phi_G]}{r}\Bigr||r\pa_rG|+\Bigl|\frac{U^z[\phi_G]}{z}\Bigr||z\pa_zG|\lesssim \|G\|_{\mathscr X_{0,\delta}}w_*(r,z)^{\delta}\Omega_*\|G\|_{\mathscr X_{1,\delta}},\\
			|U[\phi_G]\cdot\nabla_{r,z}\Omega_*|&\lesssim \Bigl|\frac{U^r[\phi_G]}{r}\Bigr||r\pa_r\Omega_*|+\Bigl|\frac{U^z[\phi_G]}{z}\Bigr||z\pa_z\Omega_*|\lesssim w_*(r,z)^\delta\|G\|_{\mathscr X_{0,\delta}}\Omega_*.
		\end{aligned}
	\end{equation}
	Here we also used the definition of $\|\cdot\|_{\mathscr X_{1,\delta}}$ (see \eqref{S0eqG1}) and $|r\pa_r\Omega_*|+|z\pa_z\Omega_*|\lesssim\Omega_*$. 
	
On the other hand, 	it follows from Lemma \ref{Lem.psi_G-est-zero-order} and $\mathfrak M(G)=\psi_G(0,0)$ that
	\begin{equation}\label{Eq.M(G)-est}
		|\mathfrak M(G)|\lesssim \|G\|_{\mathscr X_{0,\delta}}.
	\end{equation}
As $\mathfrak M(\Omega)=\mathfrak M(\Omega_*)+\mathfrak M(G)$ and $\mathfrak M(\Omega_*)\sim 1$, there exists a sufficiently small constant $\varepsilon_0=\varepsilon_0(d,a,\mu,\delta)\in(0,1)$ 
such that
	\begin{align}\label{Eq.M(Omega)-est}
		\text{if $\|G\|_{\mathscr X_{0,\delta}}<\varepsilon_0$ then $\mathfrak M(\Omega)\sim 1$}.
	\end{align}
	
By plugging \eqref{Eq.N[G]-est-1}, \eqref{Eq.M(G)-est}, $\mathfrak{M}(\Omega_*)\sim 1$ and \eqref{Eq.M(Omega)-est} into \eqref{Eq.mathscrN-expression}, we obtain
	\begin{equation*}
		\left|\mathscr N[G]\right|\lesssim w_*(r,z)^\delta \Omega_*\|G\|_{\mathscr X_{0,\delta}}\left(\|G\|_{\mathscr X_{0,\delta}}+\|G\|_{\mathscr X_{1,\delta}}\right)
	\end{equation*}
	if $\|G\|_{\mathscr X_{0,\delta}}\leq \varepsilon_0$.	This completes the proof of Proposition \ref{Prop.N[G]-est}. 
\end{proof}

\subsection{Propagation of the first-order estimates}\label{Subsec.propagation-first-order-est}
In this subsection, we present the proof of Proposition \ref{Prop.G-X1-est}.  Throughout this subsection, let $\mu_0''\in(\mu_0, 1/(d-2))$ be determined  by Lemma \ref{Lem.K-compact}. We always  assume that $\mu$
and  $\delta$ satisfy \eqref{S0eq1},  and that $G=G(s,\cdot)$ is a $C^1$ solution to \eqref{Eq.G-eq-detailed} on the time interval $s\in[0, s_0)$. 

We first derive the evolution equation for $r\pa_rG$ and $z\pa_zG$.
To simplify notation, we introduce the vector field
\begin{equation}\label{Eq.Q_Omega-def}
	Q_\Omega(r,z):=(\mu r,\mu z)+a\frac{U[\Omega]}{\mathfrak M(\Omega)}\in\R^2,\quad \forall\ (r,z)\in\Pi_+. 
\end{equation}
 Then in view of \eqref{Eq.Q_*-def}, \eqref{Eq.G-eq-detailed} becomes 
\begin{equation*}
	\pa_sG+Q_{\Omega}\cdot\nabla_{r,z}G+G=-\left(Q_\Omega-Q_*\right)\cdot\nabla_{r,z}\Omega_*,
\end{equation*}
from which, we infer 
\begin{equation}\label{Eq.G-eq-short2}	\begin{aligned}
		\left(\pa_s+Q_\Omega\cdot\nabla_{r,z}+1\right)(r\pa_rG)=-r\pa_r\left(\left(Q_\Omega-Q_*\right)\cdot\nabla_{r,z}\Omega_*\right)-\left[r\pa_r; Q_\Omega\cdot\nabla_{r,z}\right]G,\\
		\left(\pa_s+Q_\Omega\cdot\nabla_{r,z}+1\right)(z\pa_zG)=-z\pa_z\left(\left(Q_\Omega-Q_*\right)\cdot\nabla_{r,z}\Omega_*\right)-\left[z\pa_z; Q_\Omega\cdot\nabla_{r,z}\right]G,
	\end{aligned}
\end{equation}
where the commutators $\left[r\pa_r; Q_\Omega\cdot\nabla_{r,z}\right]$ and $\left[z\pa_z; Q_\Omega\cdot\nabla_{r,z}\right]$ are given explicitly by
\begin{equation}\label{Eq.communtators}
	\begin{aligned}
		\left[r\pa_r; Q_\Omega\cdot\nabla_{r,z}\right]G&=\left(r\pa_rQ_\Omega^r-Q_\Omega^r\right)\pa_rG+r\pa_rQ_\Omega^z\pa_zG,\\
		\left[z\pa_z; Q_\Omega\cdot\nabla_{r,z}\right]G&=\left(z\pa_zQ_\Omega^r\right)\pa_rG+\left(z\pa_zQ_\Omega^z-Q_\Omega^z\right)\pa_zG.
	\end{aligned}
\end{equation}

Our aim is to estimate $\|G\|_{\mathscr X_{1,\delta}}$. It is natural to introduce 
\begin{equation}\label{Eq.g_j-def}
	g_r:=w_*^{-\delta}\frac{r\pa_rG}{\Omega_*}\andf g_z:=w_*^{-\delta}\frac{z\pa_zG}{\Omega_*}.
\end{equation}
By \eqref{Eq.G-eq-short2}, we find that  $g_\nu$ with  $\nu\in\{r,z\}$ satisfies
\begin{equation}\label{Eq.g_j-eq}
	\begin{aligned}
		&\pa_sg_\nu+Q_\Omega\cdot\nabla_{r,z}g_\nu+\frac{\left[\mathcal D_\nu; Q_*\cdot\nabla_{r,z}\right]G}{\Omega_*w_*^\delta}
		=-\frac{\left(Q_\Omega\cdot\nabla_{r,z}+1\right)\left(\Omega_*w_*^\delta\right)}{\Omega_*w_*^\delta}g_\nu\\
		&\quad -\frac{\mathcal D_\nu\left(\left(Q_\Omega-Q_*\right)\cdot\nabla_{r,z}\Omega_*\right)}{\Omega_*w_*^\delta}-\frac{\left[\mathcal D_\nu; \left(Q_\Omega-Q_*\right)\cdot\nabla_{r,z}\right]G}{\Omega_*w_*^\delta}.
	\end{aligned}
\end{equation}
Here we use the convention
\begin{equation}\label{Eq.Dr-Dz-def}
	\mathcal D_r=r\pa_r,\qquad \mathcal D_z=z\pa_z \andf g_*:=|g_r|+|g_z|.\end{equation}
Then we have\begin{equation}\label{g1}
\begin{split}
	\pa_sg_*&+Q_\Omega\cdot\nabla_{r,z}g_*\leq \sum_{\nu=r,z}\frac{|\left[\mathcal D_\nu; Q_*\cdot\nabla_{r,z}\right]G|}{\Omega_*w_*^\delta}
	-\frac{\left(Q_\Omega\cdot\nabla_{r,z}+1\right)\left(\Omega_*w_*^\delta\right)}{\Omega_*w_*^\delta}g_*\\
	& +\sum_{\nu=r,z}\frac{|\mathcal D_\nu\left(\left(Q_\Omega-Q_*\right)\cdot\nabla_{r,z}\Omega_*\right)|}{\Omega_*w_*^\delta}+\sum_{\nu=r,z}\frac{|\left[\mathcal D_\nu; \left(Q_\Omega-Q_*\right)\cdot\nabla_{r,z}\right]G|}{\Omega_*w_*^\delta}.
\end{split}\end{equation}
Using $Q_*=(A_*r,B_*z)$, $g_*=|g_r|+|g_z|,$  \eqref{Eq.communtators} and \eqref{Eq.g_j-def}, we find
\begin{align*}
	&\frac{\left[\mathcal D_r; Q_*\cdot\nabla_{r,z}\right]G}{\Omega_*w_*^\delta}=\frac{r\pa_rA_* r\pa_rG+r\pa_rB_* z\pa_zG}{\Omega_*w_*^\delta}=g_rr\pa_rA_*+g_zr\pa_rB_*,\\
	&\frac{\left[\mathcal D_z; Q_*\cdot\nabla_{r,z}\right]G}{\Omega_*w_*^\delta}=\frac{z\pa_zA_* r\pa_rG+z\pa_zB_* z\pa_zG}{\Omega_*w_*^\delta}=g_rz\pa_zA_*+g_zz\pa_zB_*,\\
	&\frac{|\left[\mathcal D_r; Q_*\cdot\nabla_{r,z}\right]G|}{\Omega_*w_*^\delta}+\frac{|\left[\mathcal D_z; Q_*\cdot\nabla_{r,z}\right]G|}{\Omega_*w_*^\delta}\leq (|r\pa_rA_*|+|r\pa_rB_*|+|z\pa_zA_*|+|z\pa_zB_*|)g_*.
\end{align*}
We thus obtain
\begin{equation}\label{g2}
\begin{split}
	&\pa_sg_*+Q_\Omega\cdot\nabla_{r,z}g_*\\ 
	&\leq \bigl(|r\pa_rA_*|+|r\pa_rB_*|+|z\pa_zA_*|+|z\pa_zB_*|\bigr)g_*
	-\frac{\left(Q_\Omega\cdot\nabla_{r,z}+1\right)\left(\Omega_*w_*^\delta\right)}{\Omega_*w_*^\delta}g_*\\ 
	&\qquad +\sum_{\nu=r,z}\frac{|\mathcal D_\nu\left(\left(Q_\Omega-Q_*\right)\cdot\nabla_{r,z}\Omega_*\right)|}{\Omega_*w_*^\delta}+\sum_{\nu=r,z}\frac{|\left[\mathcal D_\nu; \left(Q_\Omega-Q_*\right)\cdot\nabla_{r,z}\right]G|}{\Omega_*w_*^\delta}.
\end{split}\end{equation}

The key ingredient in the proof of Proposition \ref{Prop.G-X1-est} is the following outgoing property.

\begin{lemma}\label{Lem.outgoing-Omega}
	{\sl Let $\mu_0'=\mu_0'(d,a)\in(\mu_0,1/(d-2))$ and $w_1$ be determined respectively by Lemmas \ref{Lem.outgoing} and  \ref{Lem2}. 
	Then for $\mu,\delta$ satisfying \eqref{S0eq1},  there exists $\varepsilon_0'\in(0,1/2)$ such that for all $G\in \mathscr X_{0,\delta}$ with $\|G\|_{\mathscr X_{0,\delta}}<\varepsilon_0'$ and
	$(r,z)\in\Pi_+,$  there holds
	\begin{equation}\label{S8eq9}
		\left(Q_\Omega\cdot\nabla_{r,z}+1\right)(\Omega_*w_*^\delta \mathrm{e}^{w_1})
		\geq \bigl(\varkappa'+|r\pa_rA_*|+|r\pa_rB_*|+|z\pa_zA_*|+|z\pa_zB_*|\bigr)\Omega_*w_*^\delta\mathrm{e}^{w_1},
	\end{equation}
	for some $\varkappa'=\varkappa'(d,a,\mu,\delta)\in(0,1)$. }
\end{lemma}
\begin{proof}
	Let $\psi_G$ be the elliptic potential generated by $G$, and let $\phi_G$ be defined by \eqref{Eq.phi_G-def}. As $\Omega=\Omega_*+G$, it follows from  \eqref{Eq.Q_*-def} and 
	\eqref{Eq.Q_Omega-def} that 
\begin{equation}\label{S8eq11}	Q_\Omega-Q_*=a\left(\frac{U[\Omega]}{\mathfrak M(\Omega)}-\frac{U[\Omega_*]}{\mathfrak M(\Omega_*)}\right)=\frac{a}{\mathfrak M(\Omega)}U[\phi_G],\end{equation}
	from which, \eqref{Eq.M(Omega)-est},
		 \eqref{Eq.U[phi_G]-def}, \eqref{Eq.phi_G-est1}, \eqref{Eq.phi_G-est2} and \eqref{Eq.w*-asymptotic}, we deduce  by a similar  proof of Proposition \ref{Prop.N[G]-est} that
	\[\left|\frac{(Q_\Omega-Q_*)^r}{r}\right|+\left|\frac{(Q_\Omega-Q_*)^z}{z}\right|\lesssim \|G\|_{\mathscr X_{0,\delta}}.\]
	
	On the other hand,  noticing that $w_*^2=r^2+(\mu+\psi_*)^2z^2,$  we deduce  from \eqref{Eq.small-first-log-psi-star} that 
	for  $A_*:=\mu+\psi_*+z\pa_z\psi_*$, there hold
	 $$A_*\geq\mu \andf 
	\left|r\pa_rw_*\right|+\left|z\pa_zw_*\right|\lesssim w_*,$$
	from which, 
	 the first-order derivative bound for $\Omega_*$ and $ w_1=-K(1+w_*)^{d-2-\frac1\mu}$, we infer
	\begin{equation}\label{Eq.QOmega-Qstar-perturb-outgoing}
	\begin{split}
		&\left|(Q_\Omega-Q_*)\cdot\nabla_{r,z}\Omega_*\right|\lesssim \|G\|_{\mathscr X_{0,\delta}}\Omega_*,\\
		&\left|(Q_\Omega-Q_*)\cdot\nabla_{r,z}w_*\right|\lesssim \|G\|_{\mathscr X_{0,\delta}}w_*,\\
		&\left|(Q_\Omega-Q_*)\cdot\nabla_{r,z}w_1\right|\lesssim \|G\|_{\mathscr X_{0,\delta}}.
	\end{split}\end{equation}
	
	 Since $\Omega_*$ is the steady profile, we have $Q_*\cdot\nabla_{r,z}\Omega_*=-\Omega_*$. Then  we  get, by using Lemmas \ref{Lem.outgoing} and \ref{Lem2}, that
	\begin{align*}
		\left(Q_*\cdot\nabla_{r,z}+1\right)(\Omega_*w_*^\delta\mathrm{e}^{w_1})
		=&w_*^\delta\mathrm{e}^{w_1}\left(Q_*\cdot\nabla_{r,z}\Omega_*+\Omega_*\right)
		+\delta\Omega_*w_*^{\delta-1}\mathrm{e}^{w_1}Q_*\cdot\nabla_{r,z}w_*\\
		&+\Omega_*w_*^\delta\mathrm{e}^{w_1}Q_*\cdot\nabla_{r,z}w_1\\
			\geq & \bigl(\delta\varkappa+|r\pa_rA_*|+|r\pa_rB_*|+|z\pa_zA_*|+|z\pa_zB_*|\bigr)\,\Omega_*w_*^\delta\mathrm e^{w_1}.
	\end{align*}
	which along with \eqref{Eq.QOmega-Qstar-perturb-outgoing} ensures that
	\[
	\begin{aligned}
		\left(Q_\Omega\cdot\nabla_{r,z}+1\right)(\Omega_*w_*^\delta\mathrm{e}^{w_1})
		=&\left(Q_*\cdot\nabla_{r,z}+1\right)(\Omega_*w_*^\delta\mathrm{e}^{w_1})+(Q_\Omega-Q_*)\cdot\nabla_{r,z}(\Omega_*w_*^\delta\mathrm{e}^{w_1})\\
		\geq& \bigl(\delta\varkappa+|r\pa_rA_*|+|r\pa_rB_*|+|z\pa_zA_*|+|z\pa_zB_*|\bigr)\Omega_*w_*^\delta\mathrm e^{w_1}\\
		&-
		C\|G\|_{\mathscr X_{0,\delta}}\Omega_*w_*^\delta \mathrm{e}^{w_1}.
	\end{aligned}
	\]
By decreasing $\varepsilon_0'$ if necessary so that $C\varepsilon_0'\leq \delta\varkappa/2,$ we conclude the proof of \eqref{S8eq9} with $\varkappa':=\delta\varkappa/2\in(0,1)$.
\end{proof}

Let $g_\ast$ be given by \eqref{Eq.Dr-Dz-def}, and we denote
\begin{equation} \label{S8eq15}
g_1:=\mathrm{e}^{-w_1}g_* \with  w_1:=-K(1+w_*)^{d-2-\frac1\mu}.
\end{equation}
 Then it follows from \eqref{g2} that
\begin{equation}\label{g3}
\begin{split}
	&\pa_sg_1+Q_\Omega\cdot\nabla_{r,z}g_1
	\leq \bigl(|r\pa_rA_*|+|r\pa_rB_*|+|z\pa_zA_*|+|z\pa_zB_*|\bigr)g_1\\
	&\qquad\qquad\qquad-\frac{\left(Q_\Omega\cdot\nabla_{r,z}+1\right)\left(\Omega_*w_*^\delta\mathrm{e}^{w_1}\right)}{\Omega_*w_*^\delta\mathrm{e}^{w_1}}g_1\\ 
	& \qquad\qquad\qquad+\sum_{\nu=r,z}\biggl(\frac{|\mathcal D_\nu\left(\left(Q_\Omega-Q_*\right)\cdot\nabla_{r,z}\Omega_*\right)|}{\Omega_*w_*^\delta\mathrm{e}^{w_1}}+\frac{|\left[\mathcal D_\nu; \left(Q_\Omega-Q_*\right)\cdot\nabla_{r,z}\right]G|}{\Omega_*w_*^\delta\mathrm{e}^{w_1}}\biggr),
\end{split}
\end{equation}
from which and \eqref{S8eq9}, we infer
\begin{equation}\label{g4}
\begin{split}
	&\pa_sg_1+Q_\Omega\cdot\nabla_{r,z}g_1\leq -\varkappa'g_1\\ 
	& +\sum_{\nu=r,z}\biggl(\frac{|\mathcal D_\nu\left(\left(Q_\Omega-Q_*\right)\cdot\nabla_{r,z}\Omega_*\right)|}{\Omega_*w_*^\delta\mathrm{e}^{w_1}}+\frac{|\left[\mathcal D_\nu; \left(Q_\Omega-Q_*\right)\cdot\nabla_{r,z}\right]G|}{\Omega_*w_*^\delta\mathrm{e}^{w_1}}\biggr).
\end{split}\end{equation}

The following  lemma deals with  two source terms in the second line of \eqref{g4}.

\begin{lemma}\label{Eq.g_j-source1}
{\sl  Let $\mathcal N_G$ be given by \eqref{S8eq8}. Then under the assumptions of Lemma \ref{Lem.outgoing-Omega}, by adjusting $\varepsilon_0'\in(0,1/2)$ to be smaller if necessary,  for all $G\in C^1$ with $\|G\|_{\mathscr X_{0,\delta}}<\varepsilon_0'$ and $\|G\|_{\mathscr X_{1,\delta}}<+\infty$,  and $\nu\in\{r,z\},$ there holds
	\begin{align}\label{S8eq10}
		&\Bigl|\frac{\left[\mathcal D_\nu; \left(Q_\Omega-Q_*\right)\cdot\nabla_{r,z}\right]G}{\Omega_*w_*^\delta}\Bigr|\lesssim  \|G\|_{\mathscr X_{\delta}} \|G\|_{\mathscr X_{1,\delta}}\with \|G\|_{\mathscr X_{\delta}}:= \|G\|_{\mathscr X_{0,\delta}}+\|G\|_{\mathscr X_{1,\delta}},\\
		&\def\a{\delta_1}
		\Bigl|\frac{\mathcal D_\nu\left(\left(Q_\Omega-Q_*\right)\cdot\nabla_{r,z}\Omega_*\right)}{\Omega_*w_*^\delta}\Bigr|\lesssim \mathcal N_G,
		\quad \forall\ \nu\in\{r,z\},\ \a\in(0,1/2]. \label{S8eq12}
	\end{align}
	}
\end{lemma}

\begin{proof} By decreasing $\varepsilon_0'$ if necessary, we deduce from Lemmas \ref{Lem.psi_G-est-zero-order},  \ref{Lem.psi_G-est-first-order},  \ref{Lem.psi_G-est-second-order} (for $ \delta_1=1/2$) and  \ref{Lem.E12-est}, \eqref{Eq.phi_G-def}, \eqref{r2}, \eqref{Eq.M(G)-est} and \eqref{Eq.M(Omega)-est} that
	\begin{equation}\label{Eq.phiG-log-est-source}
		|\phi_G|+|r\pa_r\phi_G|+|z\pa_z\phi_G|+|r^2\pa_r^2\phi_G|+|rz\pa_r\pa_z\phi_G|+|z^2\pa_z^2\phi_G|\lesssim \|G\|_{\mathscr X_{\delta}}.
	\end{equation}
	
By \eqref{S8eq11}, we write
	\begin{equation}\label{S8eq0717} \frac{Q_\Omega^r-Q_*^r}{r}=\frac a{\mathfrak M(\Omega)}(\phi_G+z\pa_z\phi_G),\qquad \frac{Q_\Omega^z-Q_*^z}{z}=-\frac a{\mathfrak M(\Omega)}((d-1)\phi_G+r\pa_r\phi_G).\end{equation}
	Applying \eqref{Eq.phiG-log-est-source} gives
	\begin{equation}\label{Eq.Pr-Pz-log-bound}
		\left|\mathcal D_\nu\left(\frac{Q_\Omega^r-Q_*^r}{r}\right)\right|+\left|\mathcal D_\nu\left(\frac{Q_\Omega^z-Q_*^z}{z}\right)\right|\lesssim \|G\|_{\mathscr X_{\delta}},\quad \forall\ \nu\in\{r,z\}.
	\end{equation}
	
In view of 	 \eqref{Eq.communtators}, we write 
	\[\left[\mathcal D_\nu; \left(Q_\Omega-Q_*\right)\cdot\nabla_{r,z}\right]G=\mathcal D_\nu\left(\frac{Q_\Omega^r-Q_*^r}{r}\right)\cdot r\pa_rG+\mathcal D_\nu\left(\frac{Q_\Omega^z-Q_*^z}{z}\right)\cdot z\pa_zG,\quad\forall\ \nu\in\{r,z\},\]
	from which, \eqref{Eq.Pr-Pz-log-bound} and the definition of $\|G\|_{\mathscr X_{1,\delta}}$, we infer
	\[\left|\left[\mathcal D_j,\left(Q_\Omega-Q_*\right)\cdot\nabla_{r,z}\right]G\right|
	\lesssim\|G\|_{\mathscr X_{\delta}}\left(|r\pa_rG|+|z\pa_zG|\right)\lesssim \|G\|_{\mathscr X_{\delta}}\Omega_*w_*^\delta\|G\|_{\mathscr X_{1,\delta}},\]
	which leads to \eqref{S8eq10}.


To prove \eqref{S8eq12}, we first get, by	using Lemmas \ref{Lem.psi_G-est-zero-order},  \ref{Lem.psi_G-est-first-order},  \ref{Lem.psi_G-est-second-order} and \ref{Lem.E12-est}, \eqref{Eq.phi_G-def}, \eqref{r2}, 
	the bound $|\mathfrak M(G)|\lesssim\|G\|_{\mathscr X_{0,\delta}}$,
	and $\phi_G(0,0)=0$,  that
	\begin{align*}
		&|\phi_G|+|r\pa_r\phi_G|+|z\pa_z\phi_G|+|r^2\pa_r^2\phi_G|+|rz\pa_r\pa_z\phi_G|+|z^2\pa_z^2\phi_G|\lesssim \def\a{\delta_1}|(r,z)|\langle r,z\rangle^{d-3-\frac1\mu+\delta}\mathcal N_G.		
	\end{align*}
	Then due to $w_*\sim|(r,z)|$, $\delta>0 $ and $d-2-\frac1\mu+\delta<0$, we deduce from \eqref{S8eq0717} that  for $\nu\in\{r,z\}$,
	\begin{equation}\label{Eq.Qdiff-log-bound-source}
		\begin{aligned}
			&\Bigl|\frac{Q_\Omega^r-Q_*^r}{r}\Bigr|+\Bigl|\frac{Q_\Omega^z-Q_*^z}{z}\Bigr|+\Bigl|\mathcal D_\nu\left(\frac{Q_\Omega^r-Q_*^r}{r}\right)\Bigr|+\Bigl|\mathcal D_\nu\left(\frac{Q_\Omega^z-Q_*^z}{z}\right)\Bigr|\lesssim w_*^\delta\mathcal N_G.
		\end{aligned}
	\end{equation}
	
	On the other hand, it follows  from $\Omega_*\in\mathcal B_{M_0'}$ and Proposition \ref{prop:second-log-derivative-Omega} that
	\begin{equation}\label{Eq.Omegastar-log2-source}
		|\mathcal D_\nu\Omega_*|+|\mathcal D_\nu\mathcal D_{\nu'}\Omega_*|\lesssim \Omega_*,\qquad \nu, \nu'\in\{r,z\}.
	\end{equation}
	
	We now write
	\[(Q_\Omega-Q_*)\cdot\nabla_{r,z}\Omega_*=\frac{Q_\Omega^r-Q_*^r}{r}\mathcal D_r\Omega_*+\frac{Q_\Omega^z-Q_*^z}{z}\mathcal D_z\Omega_*.\]
By	applying $\mathcal D_\nu$ to the above identity and using \eqref{Eq.Qdiff-log-bound-source} and \eqref{Eq.Omegastar-log2-source}, we find
	\begin{align*}
		\left|\mathcal D_\nu\left((Q_\Omega-Q_*)\cdot\nabla_{r,z}\Omega_*\right)\right|
		\leq &\,\Bigl|\mathcal D_\nu\left(\frac{Q_\Omega^r-Q_*^r}{r}\right)\Bigr||\mathcal D_r\Omega_*|+\Bigl|\frac{Q_\Omega^r-Q_*^r}{r}\Bigr|\mathcal D_\nu\mathcal D_r\Omega_*|\\
		&+\Bigl|\mathcal D_\nu\left(\frac{Q_\Omega^z-Q_*^z}{z}\right)\Bigr||\mathcal D_z\Omega_*|+\Bigl|\frac{Q_\Omega^z-Q_*^z}{z}\Bigr||\mathcal D_\nu\mathcal D_z\Omega_*|\\
		&\lesssim\def\a{\delta_1}	 \Omega_*w_*^\delta\left(\a^{-1}\|G\|_{\mathscr{X}_{0,\delta}}+\a^{d-2-\gamma}\|G\|_{\mathscr{X}_{1,\delta}}\right).
	\end{align*}
	Dividing the above inequality by $\Omega_*w_*^\delta$ leads to \eqref{S8eq12}. This proves Lemma \ref{Eq.g_j-source1}. \end{proof}

Now, we are ready to present the proof of Proposition \ref{Prop.G-X1-est}.

\begin{proof}[Proof of Proposition \ref{Prop.G-X1-est}]
	We choose $\varepsilon_0'\in(0,1/2)$ to be  sufficiently small so that Lemmas \ref{Lem.outgoing-Omega} and  \ref{Eq.g_j-source1} 
	can be applied on the time interval under consideration. 
	
	For $0\leq \tau\leq s<s_0$, let $\mathcal Y_{s,\tau}(P)$ be the flow generated by the time-dependent vector field $Q_\Omega$ defined by \eqref{Eq.Q_Omega-def}, namely,
	\begin{equation}\label{S8eq14}
	\begin{cases}
		\frac{\mathrm d}{\mathrm ds}\mathcal Y_{s,\tau}(P)=Q_\Omega(s,\mathcal Y_{s,\tau}(P)),\\ \mathcal Y_{\tau,\tau}(P)=P,\qquad P\in\Pi_+.
		\end{cases}
	\end{equation}
	Since $Q_\Omega^r$ has a factor $r$ and $Q_\Omega^z$ has a factor $z$, the flow preserves $\Pi_+$. Moreover, the bounds proved above imply that $Q_\Omega$ has at most linear growth, so that $\mathcal Y_{s,\tau}:\Pi_+\to\Pi_+$ is a bijection with inverse $\mathcal Y_{\tau,s}$.
	
	Let 
	$g_1$ be given by \eqref{S8eq15}. Then in view of \eqref{g4}  and \eqref{S8eq14}, we write
	\begin{equation*}
		\begin{aligned}
			&\frac{\mathrm d}{\mathrm d\sigma}g_1(\sigma,\mathcal Y_{\sigma,0}(P))\leq-\varkappa'g_1(\sigma,\mathcal Y_{\sigma,0}(P))\\
			&+\sum_{\nu=r,z}\Bigl(\frac{|\mathcal D_\nu\left(\left(Q_\Omega-Q_*\right)\cdot\nabla_{r,z}\Omega_*\right)|}{\Omega_*w_*^\delta\mathrm{e}^{w_1}}(\sigma,\mathcal Y_{\sigma,0}(P))+\frac{|\left[\mathcal D_\nu; \left(Q_\Omega-Q_*\right)\cdot\nabla_{r,z}\right]G|}{\Omega_*w_*^\delta\mathrm{e}^{w_1}}(\sigma,\mathcal Y_{\sigma,0}(P))\Bigr).
		\end{aligned}
	\end{equation*}\def\a{\delta_1}
It follows from Lemma \ref{Eq.g_j-source1} that
	\begin{align*}
		\frac{\left|\mathcal D_\nu\left(\left(Q_\Omega-Q_*\right)\cdot\nabla_{r,z}\Omega_*\right)\right|}{\Omega_*w_*^\delta}&+\frac{\left|\left[\mathcal D_\nu; \left(Q_\Omega-Q_*\right)\cdot\nabla_{r,z}\right]G\right|}{\Omega_*w_*^\delta}\\
		&\leq C \bigl(\|G\|_{\mathscr X_{\delta}} \|G\|_{\mathscr X_{1,\delta}}+\mathcal N_G\bigr),\quad \forall\ \a\in(0,1/2],
	\end{align*} 
	where $\|G\|_{\mathscr X_{\delta}}$ and $\mathcal N_G$ are given respectively by \eqref{S8eq12} and \eqref{S8eq8}.	
		
	As $ w_1=-K(1+w_*)^{d-2-\frac1\mu}$, we have $\mathrm{e}^{w_1}\sim1 $, $g_1\sim g_*$. Therefore, for every $P\in\Pi_+$ and every $s\in[0,s_0)$, the usual comparison principle gives rise to 
	\begin{align*}
		|g_1(s,\mathcal Y_{s,0}(P))|\leq &\,\mathrm e^{-\varkappa's}|g_1(0,P)|\\
		&+C\int_0^s\mathrm e^{-\varkappa'(s-\tau)} \bigl(\|G(\tau)\|_{\mathscr X_{\delta}} \|G(\tau)\|_{\mathscr X_{1,\delta}}+\mathcal N_G(\tau)\bigr)		
		 d\tau, \ \forall\ \a\in(0,1/2].
	\end{align*}
	Since $\mathcal Y_{s,0}$ is onto $\Pi_+$, taking the supremum in $P$ in the above inequality leads to  the same estimate for $\|g_1(s)\|_{L^\infty(\Pi_+)}$. 
	By using $g_1\sim g_*$, $g_*=|g_r|+|g_z|$, and recalling  \eqref{Eq.g_j-def} for the definition of $g_r,g_z$, we obtain
	\begin{align*}
		\|G(s)\|_{\mathscr X_{1,\delta}}\leq &\,C\mathrm e^{-\varkappa's}\|G(0)\|_{\mathscr X_{1,\delta}}\\
		&+C_2\int_0^s\mathrm e^{-\varkappa'(s-\tau)}\big[\|G(\tau)\|_{\mathscr X_{1,\delta}}\left(\|G(\tau)\|_{\mathscr X_{0,\delta}}+\|G(\tau)\|_{\mathscr X_{1,\delta}}\right)\\
		&\qquad\qquad+\a^{-1}\|G(\tau)\|_{\mathscr X_{0,\delta}}+\a^{d-2-\gamma}\|G(\tau)\|_{\mathscr X_{1,\delta}}\big]\,\mathrm d\tau, \ \forall\ \a\in(0,1/2].
	\end{align*}As $ d-2>\gamma$ and $ \varkappa'>0$, we can take $ \a\in(0,1/2]$ small enough (but fixed) such that $C_2\a^{d-2-\gamma}\leq \varkappa'/2$. Then by Gr\"onwall's inequality, we obtain
	\eqref{SGs}.
	This completes the proof of Proposition \ref{Prop.G-X1-est}.
\end{proof}

\subsection{Proof of global nonlinear stability}\label{Subsec.global-stability-proof}

In this subsection, we aim at proving Proposition \ref{Prop.bootstrap}, Lemma \ref{Lem.correction-directions} and Theorem \ref{Thm.global-stability}.

\begin{proof}[Proof of Proposition \ref{Prop.bootstrap}] 
	Throughout the proof, $C>1$ denotes a constant depending only on $d,a,\mu,\delta$. By decreasing $\varepsilon_0''$ if necessary, we may assume that $\varepsilon_0''<\varepsilon_0'$, so that Propositions \ref{Prop.N[G]-est} and \ref{Prop.G-X1-est} can be applied under the bootstrap assumptions.
	
	We first get, by applying  Proposition \ref{Prop.N[G]-est}, that
	\begin{align}\label{N1} \|\mathscr N[G(s)]\|_{\mathscr X_{0,\delta}}\leq C(\varepsilon_0'')^2\mathrm e^{-2\varkappa_0s},\qquad \forall\ 0\leq s\leq s_0. \end{align}
By \eqref{Eq.G-eq-short}, Duhamel's formula and  using the semigroup estimate \eqref{Eq.L-semigroup-est}, we have
	\begin{align*}
		\|(I-P_{\rm u})G(s)\|_{\mathscr X_{0,\delta}}&\leq C\mathrm e^{-\varkappa_0s}\|(I-P_{\rm u})G(0)\|_{\mathscr X_{0,\delta}}+C\int_0^s\mathrm e^{-\varkappa_0(s-\tau)}\|\mathscr N[G(\tau)]\|_{\mathscr X_{0,\delta}}\,\mathrm d\tau\\
		&\leq C\eta_0\varepsilon_0''\mathrm e^{-\varkappa_0s}+C(\varepsilon_0'')^2\mathrm e^{-\varkappa_0s}\\
		&\leq C\bigl(\eta_0+\varepsilon_0''\bigr)\varepsilon_0''\mathrm e^{-\varkappa_0s}.
	\end{align*} 
	Since $\mathscr X_{{\rm u},\delta}$ is finite dimensional, the norms $|\cdot|_{\mathscr H}$ and $\|\cdot\|_{\mathscr X_{0,\delta}}$ are equivalent in $\mathscr X_{{\rm u},\delta}$. The bootstrap assumption \eqref{b2} on $P_{\rm u}G$ therefore yields 
	\begin{align} \label{b0}\|G(s)\|_{\mathscr X_{0,\delta}}\leq C(\eta_0+\eta_{\rm u}+\varepsilon_0'')\varepsilon_0''\mathrm e^{-\varkappa_0s},\qquad \forall\ 0\leq s\leq s_0. \end{align} 
	
	We next deal with the semi-norm concerning the first-order derivatives of $G.$ Up to decreasing $\varkappa_0$ in \eqref{Eq.L-semigroup-est}, which causes no loss, we may assume that $\varkappa_0<\varkappa_1$. Then we get, by using Proposition \ref{Prop.G-X1-est}, the bootstrap assumption \eqref{b1} and  \eqref{b0},  that
	\begin{align*}
		\|G(s)\|_{\mathscr X_{1,\delta}}\leq &\,C\mathrm e^{-\varkappa_1s}\|G(0)\|_{\mathscr X_{1,\delta}}\\
		&+C\int_0^s\mathrm e^{-\varkappa_1(s-\tau)}\big[\|G(\tau)\|_{\mathscr X_{1,\delta}}\left(\|G(\tau)\|_{\mathscr X_{0,\delta}}+\|G(\tau)\|_{\mathscr X_{1,\delta}}\right)+\|G(\tau)\|_{\mathscr X_{0,\delta}}\big]\,\mathrm d\tau\\
		 \leq&\, C\eta_0\varepsilon_0''\mathrm e^{-\varkappa_1s}+C(\varepsilon_0'')^2\int_0^s\mathrm e^{-\varkappa_1(s-\tau)}\mathrm e^{-\varkappa_0\tau}\,\mathrm d\tau\\
		 &+C\bigl(\eta_0+\eta_{\rm u}+\varepsilon_0''\bigr)\varepsilon_0''\int_0^s\mathrm e^{-\varkappa_1(s-\tau)}\mathrm e^{-\varkappa_0\tau}\,\mathrm d\tau\\ 
		 \leq&\, C\bigl(\eta_0+\eta_{\rm u}+\varepsilon_0''\bigr)\varepsilon_0''\mathrm e^{-\varkappa_0s}.
	\end{align*}
	Choosing $\eta_0,\eta_{\rm u}$ and $\varepsilon_0''$ sufficiently small gives rise to \eqref{SGXd}.
	This completes the proof of  Proposition \ref{Prop.bootstrap}. 
\end{proof}

\begin{proof}[Proof of Lemma \ref{Lem.correction-directions}]
	If $N=0$, then there is nothing to prove. In the rest of the proof, we always assume that $N\geq1$.
	
	We first prove that $P_{\rm u}\mathscr D$ is dense in $\mathscr X_{{\rm u},\delta}$. We claim that 
	\begin{equation}\label{S8eq16}
	\mbox{every} \ h\in\mathscr X_{{\rm u},\delta}\ \mbox{ can be approximated in}\ \mathscr X_{0,\delta}\ \mbox{ by elements of}\ \mathscr D. 
	\end{equation}
	In order to do so, we recall the following consequence of the proof of Lemma \ref{Lem.K-compact} (see \eqref{w1} and \eqref{w2}): if $\mathcal B$ is bounded in $\mathscr X_{0,\delta}$, then $\mathscr K\mathcal B$ has uniformly vanishing weighted tails, namely
	\[\lim_{\rho\to0+}\sup_{G\in\mathcal B}\sup_{0<r+z\leq\rho}w_*^{-\delta}\frac{|\mathscr K G|}{\Omega_*}=0\andf  \lim_{R\to+\infty}\sup_{G\in\mathcal B}\sup_{r+z\geq R}w_*^{-\delta}\frac{|\mathscr K G|}{\Omega_*}=0.\]
	
	Let $h$ belong to a generalized eigenspace associated with an eigenvalue $\lambda$ of $-\mathscr L$ with $\operatorname{Re}\lambda>-1.1\varkappa_0$. Since $\sigma(-\mathscr L_{\rm tr})\subset\{\operatorname{Re}\lambda\leq-\delta\varkappa\}$, we have $\lambda\in\rho(-\mathscr L_{\rm tr})$. We shall prove, by induction on the length of the Jordan chain, that $h$ has vanishing weighted tails. If $(\lambda+\mathscr L)h=0$, then $(\lambda+\mathscr L_{\rm tr})h=-\mathscr K h$. In general, if $(\lambda+\mathscr L)^m h=0$, and the claim \eqref{S8eq16}  has been proved for shorter chains, then  we have $$
	(\lambda+\mathscr L_{\rm tr})h=h_1-\mathscr K h\with h_1:=(\lambda+\mathscr L)h,$$
	where $h_1$ has vanishing weighted tails by the induction hypothesis, while $\mathscr K h$ has vanishing weighted tails by the property above. 
	It remains only to note that the resolvent $(\lambda+\mathscr L_{\rm tr})^{-1}$ preserves this vanishing-tail property. Indeed, if $u=(\lambda+\mathscr L_{\rm tr})^{-1}F$ and $q:=w_*^{-\delta}u/\Omega_*$, then, using $Q_*\cdot\nabla_{r,z}\Omega_*=-\Omega_*$, we find
	\[Q_*\cdot\nabla_{r,z}q+\left(\lambda+\delta\frac{Q_*\cdot\nabla_{r,z}w_*}{w_*}\right)q=w_*^{-\delta}\frac{F}{\Omega_*}.\]
	By Lemma \ref{Lem.outgoing}, the real part of the coefficient in front of $q$ is bounded from below by a positive constant depending on $\lambda$ and $\delta$. The characteristic formula for the last equation then shows that $q$ has vanishing weighted tails whenever the right-hand side has vanishing weighted tails. This follows by splitting the characteristic integral into a bounded time part and a large time part: the bounded time part remains in the corresponding small-tail region, while the large time part is made small by the exponential factor. Therefore, every $h\in\mathscr X_{{\rm u},\delta}$ satisfies
	\begin{equation}\label{S8eq17} 
	\lim_{\rho\to0+}\sup_{0<r+z\leq\rho}w_*^{-\delta}\frac{|h|}{\Omega_*}=0,\qquad \lim_{R\to+\infty}\sup_{r+z\geq R}w_*^{-\delta}\frac{|h|}{\Omega_*}=0.\end{equation}
	
	Let us choose $\zeta_{\rho,R}\in C_c^\infty([0,+\infty))$ such that $\zeta_{\rho,R}=0$ on $[0,\rho]\cup[2R,+\infty)$ and $\zeta_{\rho,R}=1$ on $[2\rho,R]$, and set $\Phi_{\rho,R}(r,z):=\zeta_{\rho,R}\left(\sqrt{r^2+z^2}\right)h(r,z)$. Then 
	it follows from \eqref{S8eq17} that
	\[\|\Phi_{\rho,R}-h\|_{\mathscr X_{0,\delta}}\to0,\qquad \rho\to0+,\quad R\to+\infty.\]
	Moreover, the lifted odd extension of $\Phi_{\rho,R}$ has compact support in $\mathbb R^d\setminus\{0\}$, and $\Phi_{\rho,R}$ belongs to the $\mathscr X_{0,\delta}$-closure of $\mathscr D$.
	Thus, $h$ belongs to the $\mathscr X_{0,\delta}$-closure of $\mathscr D$. Since $P_{\rm u}h=h$ and $P_{\rm u}$ is bounded on $\mathscr X_{0,\delta}$, we obtain
\begin{equation}\label{S8eq18} \overline{P_{\rm u}\mathscr D}^{\,\mathscr X_{0,\delta}}=\mathscr X_{{\rm u},\delta}.\end{equation}
	
	We now choose the correction directions. Let $e_1,\ldots,e_N$ be a basis of $\mathscr X_{{\rm u},\delta}$. By virtue of \eqref{S8eq18}, for each $j$ and $ \delta_0>0$, we may choose $\Phi_j\in\mathscr D$ so that $\|\Phi_j-e_j\|_{\mathscr X_{0,\delta}}<\delta_0$. As $P_{\rm u}e_j=e_j$ and $P_{\rm u}$ is bounded on $\mathscr X_{0,\delta},$ we find
	\[\|P_{\rm u}\Phi_j-e_j\|_{\mathscr X_{0,\delta}}=\|P_{\rm u}(\Phi_j-e_j)\|_{\mathscr X_{0,\delta}}\leq C\|\Phi_j-e_j\|_{\mathscr X_{0,\delta}}\leq C\delta_0.\]
	Since the set of bases of the finite-dimensional space $\mathscr X_{{\rm u},\delta}$ is open, choosing $\delta_0 $ small enough
	guarantees that $P_{\rm u}\Phi_1,\ldots,P_{\rm u}\Phi_N$ form a basis of $\mathscr X_{{\rm u},\delta}$.
	
	Finally, for $h\in\mathscr X_{{\rm u},\delta}$, we write $h=\sum_{j=1}^N c_j(h)P_{\rm u}\Phi_j$ and define 
	\begin{equation}\label{S8eq20} 
		\mathcal I h:=\sum_{j=1}^N c_j(h)\Phi_j \Longrightarrow P_{\rm u}\mathcal I h=h.\end{equation} 
		Since $\mathscr X_{{\rm u},\delta}$ is finite dimensional, the coordinate functionals $c_j$ are bounded with respect to $|\cdot|_{\mathscr H}$. Furthermore, as each fixed $\Phi_j$ belongs to $\mathscr D$, we also have $\|\Phi_j\|_{\mathscr X_{0,\delta}}+\|\Phi_j\|_{\mathscr X_{1,\delta}}<+\infty$. Hence, after increasing the constant if necessary,
	\[\|\mathcal I h\|_{\mathscr X_{0,\delta}}+\|\mathcal I h\|_{\mathscr X_{1,\delta}}\leq C_{\mathcal I}|h|_{\mathscr H},\qquad \forall\ h\in\mathscr X_{{\rm u},\delta},\]
	which together with \eqref{S8eq20} ensures \eqref{S8eq19}.
	This completes the proof of Lemma \ref{Lem.correction-directions}.
\end{proof}

We are now in a position to complete the proof of Theorem \ref{Thm.global-stability}.

\begin{proof}[Proof of Theorem \ref{Thm.global-stability}] 
	We first choose the constants in a convenient order. Recall that $\mathcal I$ is bounded from the finite-dimensional space $\mathscr X_{{\rm u},\delta}$ to $\mathscr D$, where $\mathscr D$ is the class of functions in $\mathscr X_{0,\delta}$ whose lifted odd extensions are smooth and compactly supported away from the origin. Let $C_{\mathcal I}$ be the constant determined by \eqref{S8eq19}, we first decrease $\eta_{\rm u}$ so that 
	\[C_{\mathcal I}\eta_{\rm u}\leq \frac1{100}\eta_0.\] 
	Then we choose $\varepsilon_{\rm tn}>0$ sufficiently small so that, 
	\begin{align*}
	&\|G_{\rm tn}\|_{\mathscr X_{0,\delta}}+\|G_{\rm tn}\|_{\mathscr X_{1,\delta}}\leq\varepsilon_{\rm tn}\Longrightarrow |P_{\rm u}G_{\rm tn}|_{\mathscr H}\leq \frac1{100}\eta_{\rm u}\varepsilon_0'' \andf\\ 
	&\qquad\|G_{\rm tn}\|_{\mathscr X_{0,\delta}}+\|G_{\rm tn}\|_{\mathscr X_{1,\delta}}+C_{\mathcal I}|P_{\rm u}G_{\rm tn}|_{\mathscr H}\leq \frac1{100}\eta_0\varepsilon_0''.\end{align*}
	For $h\in\mathscr X_{{\rm u},\delta}$ with $|h|_{\mathscr H}\leq \eta_{\rm u}\varepsilon_0''$, we set 
	\[ G_h(0):=G_{\rm tn}+\mathcal I(h-P_{\rm u}G_{\rm tn}).\] 
	Then $P_{\rm u}G_h(0)=h$, and the above choices give rise to
	\[\|G_h(0)\|_{\mathscr X_{0,\delta}}+\|G_h(0)\|_{\mathscr X_{1,\delta}}\leq \eta_0\varepsilon_0''.\] 
	Let $G_h(s)$ be the corresponding solution to \eqref{Eq.G-eq-detailed}, defined on its maximal time interval $[0,T^\star).$ We introduce the closed ball $\mathbb B_{\rm u}:=\{h\in\mathscr X_{{\rm u},\delta}: |h|_{\mathscr H}\leq \eta_{\rm u}\varepsilon_0''\}$. 
	
	We claim that there exists $h\in\mathbb B_{\rm u}$ such that the solution $G_h$ never exits the bootstrap regime 
\begin{equation}\label{S8eq21}  \|G_h(s)\|_{\mathscr X_{0,\delta}}+\|G_h(s)\|_{\mathscr X_{1,\delta}}\leq \varepsilon_0''\mathrm e^{-\varkappa_0s}\andf |P_{\rm u}G_h(s)|_{\mathscr H}\leq \eta_{\rm u}\varepsilon_0''\mathrm e^{-2\varkappa_0s}.\end{equation}
Otherwise, for every $h\in\mathbb B_{\rm u}$, let $s_*(h)$ be the first exit time. It follows from  Proposition \ref{Prop.bootstrap} that the first inequality of \eqref{S8eq21} is still strict at any positive exit time. Therefore, if $s_*(h)>0$, the exit can only occur through the unstable boundary: 
	\[ |P_{\rm u}G_h(s_*(h))|_{\mathscr H}=\eta_{\rm u}\varepsilon_0''\mathrm e^{-2\varkappa_0s_*(h)}.\] 
Notice that  $|P_{\rm u}G_h(0)|_{\mathscr H}=|h|_{\mathscr H}=\eta_{\rm u}\varepsilon_0'',$ if $h\in\partial\mathbb B_{\rm u}$, then $s_*(h)=0.$

Next we prove that the map $h\mapsto s_*(h)$ is continuous on $\mathbb B_{\rm u}$. Below we fix $h\in \mathbb B_{\rm u}$. We get, by applying $ P_{\rm u}$ to \eqref{Eq.G-eq-short}, that
	\[\pa_sP_{\rm u}G_h+\mathscr LP_{\rm u}[G_h]=P_{\rm u}\mathscr N[G_h].\]
By \eqref{L1}, \eqref{N1}, the boundedness of $P_{\rm u} $ on $\mathscr X_{0,\delta} $, the equivalence of the norms $|\cdot|_{\mathscr H}$ and $\|\cdot\|_{\mathscr X_{0,\delta}}$ on $\mathscr X_{{\rm u},\delta}$, we find
	$$ |P_{\rm u}\mathscr N[G_h]|_{\mathscr H}\leq C\|\mathscr N[G_h]\|_{\mathscr X_{0,\delta}}\leq C(\varepsilon_0'')^2\mathrm e^{-2\varkappa_0s}, $$ and
	\begin{align*}
		\frac{\mathrm{d}}{\mathrm{d}s}|P_{\rm u}G_h|_{\mathscr H}^2&=2\langle-\mathscr LP_{\rm u}[G_h],P_{\rm u}[G_h]\rangle_{\mathscr H}+2\langle P_{\rm u}\mathscr N[G_h],P_{\rm u}[G_h]\rangle_{\mathscr H}\\
		&\geq -2.4\varkappa_0|P_{\rm u}G_h|_{\mathscr H}^2-2|P_{\rm u}\mathscr N[G_h]|_{\mathscr H}|P_{\rm u}G_h|_{\mathscr H}\\
		&\geq -2.4\varkappa_0|P_{\rm u}G_h|_{\mathscr H}^2-C(\varepsilon_0'')^2\mathrm e^{-2\varkappa_0s}|P_{\rm u}G_h|_{\mathscr H},\quad \forall\ 0\leq s\leq s_*(h).
	\end{align*}
	Evaluating at $s=s_*(h)$, we have (by taking $\varepsilon_0'' $ smaller with $ \eta_{\rm u}$ fixed if necessary)\begin{align}\notag
		\frac{\mathrm{d}}{\mathrm{d}s}|P_{\rm u}G_h|_{\mathscr H}^2&\geq -2.4\varkappa_0\big|\eta_{\rm u}\varepsilon_0''\mathrm e^{-2\varkappa_0s_*(h)}\big|^2-C(\varepsilon_0'')^2\mathrm e^{-2\varkappa_0s_*(h)}\eta_{\rm u}\varepsilon_0''\mathrm e^{-2\varkappa_0s_*(h)}\\
		\notag&=-(2.4\varkappa_0\eta_{\rm u}+C\varepsilon_0'')\eta_{\rm u}\left(\varepsilon_0''\right)^2\mathrm e^{-4\varkappa_0s_*(h)},
		\end{align}
		which implies
		\begin{equation}
		\label{Gd1}
		\begin{split}
		\frac{\mathrm{d}}{\mathrm{d}s}\big[\mathrm e^{4\varkappa_0s}|P_{\rm u}G_h|_{\mathscr H}^2\big]&=
		\mathrm e^{4\varkappa_0s_*(h)}\frac{\mathrm{d}}{\mathrm{d}s}|P_{\rm u}G_h|_{\mathscr H}^2+4\varkappa_0\mathrm e^{4\varkappa_0s_*(h)}|P_{\rm u}G_h|_{\mathscr H}^2\\
		&\geq -(2.4\varkappa_0\eta_{\rm u}+C\varepsilon_0'')\eta_{\rm u}\left(\varepsilon_0''\right)^2+4\varkappa_0\left(\eta_{\rm u}\varepsilon_0''\right)^2
		\\ 
		&=(1.6\varkappa_0\eta_{\rm u}-C\varepsilon_0'')\eta_{\rm u}\varepsilon_0''^2>0.
	\end{split}\end{equation}
	
	For every sequence $\{h_k\}$ and $h$ in $\mathbb B_{\rm u}$ satisfying $ h_k\to h$  in $\mathbb B_{\rm u}$, if $\limsup_{k\to\infty}s_*(h_k)>s_*(h)$, up to a subsequence,  we may assume that $s_*(h_k)>s_1>s_*(h) $. By \eqref{Gd1}, i.e., $\frac{\mathrm{d}}{\mathrm{d}s}\big[\mathrm e^{4\varkappa_0s}|P_{\rm u}G_h|_{\mathscr H}^2\big]\big|_{s=s_*(h)}>0$, there exists $s_2\in (s_*(h),s_1)$ such that \begin{align*}
		\mathrm e^{4\varkappa_0s}|P_{\rm u}G_h|_{\mathscr H}^2\big|_{s=s_2}	>\mathrm e^{4\varkappa_0s}|P_{\rm u}G_h|_{\mathscr H}^2\big|_{s=s_*(h)}=(\eta_{\rm u}\varepsilon_0'')^2.
	\end{align*}
	By the local well-posedness and continuous dependence of the characteristic flow associated with \eqref{Eq.G-eq-detailed}, 
	for $k$ large enough, we have $
	\mathrm e^{4\varkappa_0s}|P_{\rm u}G_{h_k}|_{\mathscr H}^2\big|_{s=s_2}>(\eta_{\rm u}\varepsilon_0'')^2$.
	Thus, \begin{align*}
		|P_{\rm u}G_{h_k}(s_2)|_{\mathscr H}^2>\mathrm e^{-4\varkappa_0s_2}(\eta_{\rm u}\varepsilon_0'')^2,\quad |P_{\rm u}G_{h_k}(s_2)|_{\mathscr H}>\eta_{\rm u}\varepsilon_0''\mathrm e^{-2\varkappa_0s_2},\quad s_*(h_k)<s_2<s_1,
	\end{align*}
	for $k$ large enough, which contradicts $ s_*(h_k)>s_1$. Similarly, if $\liminf_{k\to\infty}s_*(h_k)<s_*(h)$, up to a subsequence, we may assume that $s_*(h_k)\to s_3<s_*(h) $. Then we have 
	$$|P_{\rm u}G_{h_k}(s_*(h_k))|_{\mathscr H}=\eta_{\rm u}\varepsilon_0''\mathrm e^{-2\varkappa_0s_*(h_k)} ,  \ \ h_k\to h \andf s_*(h_k)\to s_3. $$
	By the local well-posedness and continuous dependence of  \eqref{Eq.G-eq-detailed}, we have $|P_{\rm u}G_{h}(s_3)|_{\mathscr H}=\eta_{\rm u}\varepsilon_0''\mathrm e^{-2\varkappa_0s_3}  $, thus $s_*(h)\leq s_3$, which contradicts $s_3<s_*(h)$.
	
	Therefore, we must have $\lim_{k\to\infty}s_*(h_k)=s_*(h)$, for every sequence $ h_k\to h$, $h_k\in \mathbb B_{\rm u}$, which  implies the continuity of $h\mapsto s_*(h)$ on $\partial\mathbb B_{\rm u}$.
	Hence, the map 
	\[\mathcal R(h):=\mathrm e^{2\varkappa_0s_*(h)}P_{\rm u}G_h(s_*(h))\] 
	is a continuous map from $\mathbb B_{\rm u}$ to $\partial\mathbb B_{\rm u}$, and $\mathcal R(h)=h$ for every $h\in\partial\mathbb B_{\rm u}$. Here, the continuity follows from the local well-posedness and continuous dependence of \eqref{Eq.G-eq-detailed}, together with the continuity of $h\mapsto s_*(h)$ on $\mathbb B_{\rm u}$.
	This gives a continuous retraction from the closed ball $\mathbb B_{\rm u}$ onto its boundary, contradicting with Brouwer's fixed-point theorem. The claim is proved. 
	
	For the value of $h=h_{\rm tn}$ given by the claim, the solution exists globally and satisfies the two bootstrap bounds for all $s\geq0$. Applying Proposition \ref{Prop.bootstrap} on every interval $[0,s_0]$ gives the improved estimate \eqref{SGXd}.
	The unstable bound is exactly the bootstrap bound along the non-exiting solution. This completes the proof of Theorem \ref{Thm.global-stability}. 
\end{proof}

\appendix
\section{Some convolution inequalities}\label{AppendixA}

This appendix gathers several elementary convolution inequalities that are repeatedly invoked throughout Section \ref{Sec.elliptic} for estimating the Newtonian potential and its derivatives. These inequalities serve primarily to bound integrals featuring two competing singular weights: one inherited from the convolution kernel and the other arising from the density function $F_\Omega(Y)=|\xi|^{d-3}|\eta|^{-1}\Omega(|\xi|,|\eta|)$.  Lemma \ref{Lem.A1convolution-wholespace} furnishes a local convolution bound valid in the subcritical regime $p+q<m$, with explicit dependence on the gap quantity $m-(p+q)$. Lemmas \ref{Lem.A2convolution-near} and \ref{Lem.A3convolution-far} decompose integrals into near and far domains, enabling us to trade a power of the localization scale for decay in the external variable.

\begin{lemma}\label{Lem.A1convolution-wholespace}
	{\sl Let $m\in\Z_{+}$ and $C_0\geq1$. There exists a constant $C>1$ depending only on $m$ and $C_0$ such that for all $p, q\in(0, m)$ satisfying $p+q<m$, we have
	\begin{equation}\label{SAeq1}
		\sup_{z\in\R^m}\int_{w\in\R^m, |w|\leq C_0}|w|^{-p}|z-w|^{-q}\,\mathrm dw\leq \frac C{m-(p+q)}.
	\end{equation}}
\end{lemma}

\begin{proof}
	Fix $z\in\R^m$. We first consider the case $|z|\geq 2C_0$. Then for $|w|\leq C_0$, we have $|z-w|\geq |z|/2$, and 
	\[
	\int_{|w|\leq C_0}|w|^{-p}|z-w|^{-q}\,\mathrm dw
	\lesssim |z|^{-q}\int_{|w|\leq C_0}|w|^{-p}\,\mathrm dw
	\lesssim \frac1{m-p}.
	\]
As $m-(p+q)<m-p$, this leads to \eqref{SAeq1} in this case.
	
	To handle the case $|z|<2C_0,$ we split the integral into the two regions $\bigl\{w\in\R^m: |w|\leq |z-w|\bigr\}$ and $\bigl\{w\in\R^m: |z-w|<|w|\bigr\}$. In the first region, we have $|z-w|^{-q}\leq |w|^{-q}$, and 
	\[
	\int_{\substack{|w|\leq C_0\\ \ |w|\leq |z-w|}}|w|^{-p}|z-w|^{-q}\,\mathrm dw
	\leq \int_{|w|\leq C_0}|w|^{-p-q}\,\mathrm dw
	\lesssim \frac1{m-(p+q)}.
	\]
	In the second region, setting $u=z-w$, we have $|u|=|z-w|\leq 3C_0$, and 
	\[
	\int_{\substack{|w|\leq C_0\\ |z-w|< |w|}}|w|^{-p}|z-w|^{-q}\,\mathrm dw
	\leq \int_{|u|\leq 3C_0}|u|^{-p-q}\,\mathrm du
	\lesssim \frac1{m-(p+q)}.
	\]
	
Combining the above estimates, we obtain \eqref{SAeq1}.
\end{proof}

\begin{lemma}\label{Lem.A2convolution-near}
	{\sl Let $m\in\Z_{+}$. There exists a constant $C>1$ depending only on $m$ such that for all $p\in(0, m)$, $\alpha\in(0,p)$ and $\ell\geq1$, we have
	\begin{equation}\label{SAeq2}
		\int_{w\in\R^m, |z-w|\leq \ell}|w|^{-p}\,\mathrm dw\leq \frac{C}{m-p}\ell^{m-p+\alpha}\langle z\rangle^{-\alpha},\quad\forall\ z\in\R^m.
	\end{equation}}
\end{lemma}

\begin{proof}
	Fix $z\in\R^m$. We distinguish two cases: $|z|\leq 2\ell$ and $|z|> 2\ell$. If $|z|\leq 2\ell$, since $\ell\geq1,$ we have $\langle z\rangle\lesssim \ell.$ Moreover, the region where $|z-w|\leq\ell$ is contained in 
	the one where $|w|\leq 3\ell$. Therefore,
	\[
	\int_{|z-w|\leq \ell}|w|^{-p}\,\mathrm dw
	\leq \int_{|w|\leq 3\ell}|w|^{-p}\,\mathrm dw
	\lesssim \frac1{m-p}\ell^{m-p}
	\lesssim \frac1{m-p}\ell^{m-p+\alpha}\langle z\rangle^{-\alpha}.
	\]
	
	If $|z|>2\ell$, then for $|z-w|\leq \ell$, we have $|w|\geq |z|-\ell>|z|/2,$ and
	\[
	\int_{|z-w|\leq \ell}|w|^{-p}\,\mathrm dw
	\lesssim |z|^{-p}\ell^m.
	\]
	Since $|z|>2\ell$ and $\alpha<p$, we have $\ell^{p-\alpha}\lesssim |z|^{p-\alpha}$, so 
	 $$|z|^{-p}\ell^m\lesssim \ell^{m-p+\alpha}|z|^{-\alpha}\lesssim \ell^{m-p+\alpha}\langle z\rangle^{-\alpha}. $$
	 This completes the proof of \eqref{SAeq2}.
\end{proof}

\begin{lemma}\label{Lem.A3convolution-far}
{	\sl Let $m\in\Z_+$. There exists a constant $C>1$ depending only on $m$ such that for all $p,q\in(0, m)$ satisfying $p+q>m$, any $\alpha\in(0,p+q-m]$ and  $\ell\geq 1$, we have
	\begin{equation}\label{SAeq3}
		\int_{w\in\R^m, |z-w|> \ell}|w|^{-p}|z-w|^{-q}\,\mathrm dw
		\leq C\Bigl(\frac1{m-p}+\frac1{m-q}+\frac1{p+q-m}\Bigr)\langle z\rangle^{-\alpha},\quad\forall\ z\in\R^m.
	\end{equation}}
\end{lemma}

\begin{proof}
	Fix $z\in\R^m$. We first consider the case $|z|\leq2$. As $\ell\geq1$, the singularity at $w=z$ is excluded.  In the case where $|w|\leq4$, we get, by  using $|z-w|^{-q}\leq \ell^{-q}\leq1,$  that
	\[
	\int_{\substack{|z-w|>\ell\\ |w|\leq4}} |w|^{-p}|z-w|^{-q}\,\mathrm dw
	\lesssim \int_{|w|\leq4}|w|^{-p}\,\mathrm dw
	\lesssim \frac1{m-p}.
	\]
In the case where $|w|>4$, we have $|z-w|\geq |w|/2$, and 
	\[
	\int_{|w|>4}|w|^{-p}|z-w|^{-q}\,\mathrm dw
	\lesssim \int_{|w|>4}|w|^{-p-q}\,\mathrm dw
	\lesssim \frac1{p+q-m}.
	\]
	Since $\langle z\rangle^{-\alpha}\gtrsim_m1,$ we thus obtain \eqref{SAeq3} when $|z|\leq2$.
	
	We now consider the case $|z|>2$. We split the integral domain into three regions: $\bigl\{ w\in\R^m: |w|\leq |z|/2\bigr\}$, $\bigl\{ w\in\R^m: |z-w|\leq |z|/2 \bigr\}$, and the remaining region. In the first region, $|z-w|\geq |z|/2$, so
	\[
	\int_{|w|\leq |z|/2}|w|^{-p}|z-w|^{-q}\,\mathrm dw
	\lesssim |z|^{-q}\int_{|w|\leq |z|/2}|w|^{-p}\,\mathrm dw
	\lesssim \frac1{m-p}|z|^{m-p-q}.
	\]
	In the second region,  using the change of variables $u=z-w$, we obtain
	\[
	\int_{\ell<|z-w|\leq |z|/2}|w|^{-p}|z-w|^{-q}\,\mathrm dw
	\lesssim |z|^{-p}\int_{|u|\leq |z|/2}|u|^{-q}\,\mathrm du
	\lesssim \frac1{m-q}|z|^{m-p-q}.
	\]
	It remains to handle the region where $|w|>|z|/2$ and $|z-w|>|z|/2$. If in addition $|w|\leq2|z|$, then the volume of this region is bounded by $C|z|^m$, and hence its contribution is bounded by $C|z|^{m-p-q}$. If $|w|>2|z|$, then $|z-w|\geq |w|/2$, and therefore,
	\[
	\int_{|w|>2|z|}|w|^{-p}|z-w|^{-q}\,\mathrm dw
	\lesssim \int_{|w|>2|z|}|w|^{-p-q}\,\mathrm dw
	\lesssim \frac1{p+q-m}|z|^{m-p-q}.
	\]
	Combining the above estimates, we find	
\begin{equation}\label{SAeq4}	\int_{|z-w|>\ell}|w|^{-p}|z-w|^{-q}\,\mathrm dw
	\lesssim \Bigl(\frac1{m-p}+\frac1{m-q}+\frac1{p+q-m}\Bigr)|z|^{-(p+q-m)},\quad \forall\ |z|>2.
\end{equation}
	For $|z|>2$, we have $\langle z\rangle\sim |z|$. Then due to $\alpha\leq p+q-m$,  we have $|z|^{-(p+q-m)}\leq C\langle z\rangle^{-\alpha}$, from which and \eqref{SAeq4}, we deduce \eqref{SAeq3}.
\end{proof}

\section{Second-order derivative estimates for the fixed point}
\label{app:second-log-derivative}

In this appendix,  we shall derive pointwise $(r,z)$-weighted  second-order
derivative estimates for the fixed point constructed in Theorem \ref{Thm.fixed-point}. Throughout this appendix, we fix 
$\mu\in(\mu_0,1/(d-2))$ with $\mu_0$ being determined by Theorem \ref{Thm.fixed-point}.  Let $\psi_* \in \mathcal A_{M_0}$ be a fixed point of the map $\mathcal T_\mu$ defined by \eqref{S2eq5},  let $\Omega_*:=\mathcal F_{\mu,M_0}(\psi_*)\in \mathcal B_{M_0'}$. Recall that $(\Omega_*, \psi_*)$ solves
\begin{align}
	&(\mu+\psi_*+z\partial_z\psi_*)r\partial_r\Omega_*+(\mu-(d-1)\psi_*-r\partial_r\psi_*)z\partial_z\Omega_*=-\Omega_*,\label{app-eq:transport}\\
	&\qquad\qquad\left(\partial_r^2+\frac d r\partial_r+\partial_z^2+\frac 2 z\partial_z\right)\psi_*=-c_*r^{d-3}\frac{\Omega_*}{z},\label{app-eq:elliptic}
\end{align}
where $c_*:=a/\mathfrak M(\Omega_*)>0$.

The main result of this appendix states as follows.

\begin{proposition}\label{prop:second-log-derivative-Omega}
	{\sl There exists $\mu_0'\in\left(\mu_0, 1/(d-2)\right)$ such that for all $\mu\in\left(\mu_0', 1/(d-2)\right)$, there holds 
	\begin{equation}\label{app-eq:second-derivative-final}
		|r^2\partial_r^2\Omega_*(r,z)|+|rz\partial_r\partial_z\Omega_*(r,z)|+|z^2\partial_z^2\Omega_*(r,z)|\leq C\Omega_*(r,z),\quad \forall\ (r,z)\in\Pi_+,
	\end{equation}
	for some constant $C=C(d,a,\mu)>1$.}
\end{proposition}

Let $A_\ast, B_\ast$ be defined by \eqref{A*B*}. Then we write  \eqref{app-eq:transport} as
\begin{equation}
	\label{app-eq:T-Omega}
	\mathbf T\Omega_*=-\Omega_* \with  \mathbf T:=A_\ast(r,z)r\pa_r+B_*(r,z)z\pa_z.
\end{equation}
Recalling \eqref{Eq.r-pa_r-psi0est}, we have $$r|\pa_r\psi_*(r,z)|\leq C(1-\mu(d-2))\psi_*(r,z),\quad\forall\ (r,z)\in\Pi_+,$$
where $C=C(d,a)>1$ is a constant independent of $\mu$. By taking $\mu_0'\in (\mu_0, 1/(d-2))$ sufficiently close to $1/(d-2)$ such that $$C(1-\mu(d-2))<\frac{\mu-(d-1)a_1}{10a_1}<1,\quad\forall\ \mu\in\left(\mu_0', 1/(d-2)\right),$$
we deduce from  \eqref{A*B*}
that for all $(r,z)\in\Pi_+$,
\begin{equation}\label{Eq.B(r,z)-lowerbound}
	B_*(r,z)\geq \mu-\Bigl(d-1+\frac{\mu-(d-1)a_1}{10a_1}\Bigr)\psi_*\geq \mu-\Bigl(d-1+\frac{\mu-(d-1)a_1}{10a_1}\Bigr)a_1>0.
\end{equation}
 In what follows,  we fix such $\mu_0',$ $d,a$ and $\mu\in(\mu_0', 1/(d-2))$. All constants in this appendix are allowed to be dependent on $d$, $a$ and $\mu$, unless otherwise specified.

By \eqref{Eq.B(r,z)-lowerbound}, the definition of $\mathcal A^0$ (see \eqref{S2eq8}) and the choice of parameters, there exists $c_0=c_0(d,a)>0$ such that
\begin{equation}
	\label{app-eq:A-B-positive}
	0<c_0\leq A_*(r,z),B_*(r,z)\leq c_0^{-1},\ \  A_*(r,z)-B_*(r,z)\geq c_0\psi_*(r,z)>0,\quad \forall\ (r,z)\in\Pi_+.\end{equation}

We also introduce the second-order operator
\begin{equation}\label{app-eq:L-def}
	\mathcal L:=z^2\bigl((r\pa_r)^2+(d-1)r\pa_r\bigr)+r^2\bigl((z\pa_z)^2+z\pa_z\bigr)= r^2z^2\Bigl(\partial_r^2+\frac d r\partial_r+\partial_z^2+\frac 2 z\partial_z\Bigr).
\end{equation}
Then we may write \eqref{app-eq:elliptic} as
\begin{equation}
	\label{app-eq:Lpsi}
	\mathcal L\psi_*=-c_*r^{d-1}z\Omega_*.
\end{equation}

\begin{lemma}[Equation for $\mathcal L\Omega_*$]
	\label{lem:L-Omega-equation}
{	\sl 
	Let
	\begin{align}
		F_0:&=\mathcal L\Omega_*-(d-1)z^2(r\pa_r)\Omega_*-r^2(z\pa_z)\Omega_*,\nonumber\\
		E_1:&=-(1+r\pa_r A_*)r\pa_r\Omega_*-(r\pa_r B_*)z\pa_z\Omega_*,\label{SBeq1}\\
		E_2:&=-(z\pa_z A_*)r\pa_r\Omega_*-(1+z\pa_zB_*)z\pa_z\Omega_*.\nonumber	\end{align}
	Then the second-order derivatives
	satisfy
	\begin{equation}\label{app-eq:XYZ-algebra}
		\begin{aligned}
			(r\pa_r)^2\Omega_*&=\frac{A_*r^2E_1-B_*r^2E_2+B_*^2F_0}{A_*^2r^2+B_*^2z^2},\\ r\pa_r(z\pa_z)\Omega_*&=\frac{B_*z^2E_1+A_*r^2E_2-A_*B_*F_0}{A_*^2r^2+B_*^2z^2},\\ (z\pa_z)^2\Omega_*&=\frac{B_*z^2E_2-A_*z^2E_1+A_*^2F_0}{A_*^2r^2+B_*^2z^2}.
		\end{aligned}
	\end{equation}
	Furthermore, $\mathcal L\Omega_*$ satisfies
	\begin{equation}\label{app-eq:F-equation}
		\mathbf T(\mathcal L\Omega_*)+\mathcal L\Omega_*=\mathscr S,
	\end{equation}
	where
	\begin{align}\label{app-eq:S-def}
		\mathscr S={}&2B_*z^2\bigl((r\pa_r)^2\Omega_*+(d-1)r\pa_r\Omega_*\bigr)+2A_*r^2\bigl((z\pa_z)^2\Omega_*+z\pa_z\Omega_*\bigr)\notag\\
		&-2z^2(r\pa_r A_*)(r\pa_r)^2\Omega_* -2z^2(r\pa_r B_*)r\pa_r(z\pa_z)\Omega_*-2r^2(z\pa_z A_*)r\pa_r(z\pa_z)\Omega_*\\
		&-2r^2(z\pa_zB_*)(z\pa_z)^2\Omega_* -(\mathcal L A_*)r\pa_r\Omega_*-(\mathcal L B_*)z\pa_z\Omega_*,\notag
	\end{align}
	and
	\begin{align}
		\mathcal L A_*&=-c_*r^{d-1}z(z\pa_z\Omega_*+2\Omega_*)-2z^2\bigl((r\pa_r)^2+(d-1)r\pa_r\bigr)\psi_*,\label{app-eq:LA-LB}\\
		\mathcal L B_*&=c_*r^{d-1}z(r\pa_r\Omega_*+2(d-1)\Omega_*)+2r^2\bigl((z\pa_z)^2+z\pa_z\bigr)\psi_*.\label{app-eq:LB}
	\end{align}}
\end{lemma}

\begin{proof}
	The proof  is a brute force computation. Applying $r\pa_r$ and $z\pa_z$ to \eqref{app-eq:T-Omega}, and using $[r\pa_r; z\pa_z]=0$, gives
	\begin{equation}\label{app-eq:E1E2}
		A_*(r\pa_r)^2\Omega_*+B_*r\pa_r(z\pa_z)\Omega_*=E_1,\quad A_*r\pa_r(z\pa_z)\Omega_*+B_*(z\pa_z)^2\Omega_*=E_2.
	\end{equation}
It follows from  the definition of $F_0$ (see \eqref{SBeq1}) that
\begin{equation}\label{SBeq2}
z^2(r\pa_r)^2\Omega_*+r^2(z\pa_z)^2\Omega_*=\mathcal L\Omega_*-(d-1)z^2(r\pa_r)\Omega_*-r^2(z\pa_z)\Omega_*=F_0.\end{equation}
Then  $\left((r\pa_r)^2\Omega_*, r\pa_r(z\pa_z)\Omega_*, (z\pa_z)^2\Omega_*\right)$ solves the linear system \eqref{app-eq:E1E2}--\eqref{SBeq2}  with the determinant $D_*:=A_*^2r^2+B_*^2z^2>0$. Solving this system 
gives  \eqref{app-eq:XYZ-algebra}.
	
	We next derive the equation for $\mathcal L\Omega_*$. By applying the operator $\mathcal L$ to the equation \eqref{app-eq:T-Omega}, we obtain
	\[\mathbf T(\mathcal L\Omega_*)=\mathcal L\mathbf T\Omega_*+[\mathbf T;\mathcal L]\Omega_*=-\mathcal L\Omega_*+[\mathbf T; \mathcal L]\Omega_*,\]
	hence,
	 \begin{equation}
	 \label{SBeq3} \mathbf T(\mathcal L\Omega_*)+\mathcal L\Omega_*=[\mathbf T; \mathcal L]\Omega_*. \end{equation} We compute the commutator explicitly.  It is easy to observe that
	\[[\mathbf T; r\pa_r]=-(r\pa_r A_*)r\pa_r-(r\pa_r B_*)z\pa_z,\quad [\mathbf T; z\pa_z]=-(z\pa_z A_*)r\pa_r-(z\pa_zB_*)z\pa_z,\]
from which, we infer
	\begin{align*}
		[\mathbf T; (r\pa_r)^2+(d-1)r\pa_r]=&-2(r\pa_r A_*)(r\pa_r)^2-2(r\pa_r B_*)r\pa_r(z\pa_z)\\
		&-\big((r\pa_r)^2A_*+(d-1)r\pa_rA_*\big)r\pa_r-\big((r\pa_r)^2B_*+(d-1)r\pa_rB_*\big)z\pa_z,\end{align*}
		and
		\begin{align*}
		[\mathbf T; (z\pa_z)^2+z\pa_z]=&-2(z\pa_z A_*)r\pa_r(z\pa_z)-2(z\pa_zB_*)(z\pa_z)^2\\
		&-\big((z\pa_z)^2A_*+(z\pa_z)A_*\big)r\pa_r-\big((z\pa_z)^2B_*+z\pa_zB_*\big)z\pa_z.
	\end{align*}
As $\mathbf Tz^2=2B_*z^2$ and $\mathbf Tr^2=2A_*r^2$, we obtain 
$$[\mathbf T; \mathcal L]\Omega_*=\mathscr S\with \mathscr S\ \mbox{ being given by}\  \eqref{app-eq:S-def},$$
which together with \eqref{SBeq3} ensures \eqref{app-eq:F-equation}.
	
	It remains to prove \eqref{app-eq:LA-LB} and \eqref{app-eq:LB}.  Again due to $[r\pa_r; z\pa_z]=0$, $z\pa_z (z^2)=2z^2$ and $r\pa_r(r^2)=2r^2$, we find
	\begin{align*}
	&\mathcal L(z\pa_z\psi_*)=z\pa_z(\mathcal L\psi_*)-2z^2\bigl((r\pa_r)^2+(d-1)r\pa_r\bigr)\psi_*, \andf\\
	&\mathcal L(r\pa_r\psi_*)=r\pa_r(\mathcal L\psi_*)-2r^2\bigl((z\pa_z)^2+z\pa_z\bigr)\psi_*,
	\end{align*}
		from which, \eqref{A*B*}  and \eqref{app-eq:Lpsi}, we infer
	\[\mathcal L A_*=\mathcal L\psi_*+\mathcal L(z\pa_z\psi_*)=\mathcal L\psi_*+z\pa_z(\mathcal L\psi_*)-2z^2\bigl((r\pa_r)^2+(d-1)r\pa_r\bigr)\psi_*.\]
	Since $z\pa_z(\mathcal L\psi_*)=-c_*r^{d-1}z(z\pa_z\Omega_*+\Omega_*)$,  \eqref{app-eq:LA-LB} follows. Along the same line,  we can derive  \eqref{app-eq:LB}. This finishes the proof of  Lemma \ref{lem:L-Omega-equation}.
\end{proof}


\begin{lemma}[Source computations]
	\label{lem:source-computation-L-Omega}
{\sl	Let $D_*:=A_*^2r^2+B_*^2z^2$ and $\mathcal H:=(\mathcal L\Omega_*)/(D_*\cdot\Omega_*)$. Then $\mathcal H$ satisfies
	\begin{equation}\label{app-eq:H-equation}
		\mathbf T\mathcal H=-2(r\pa_r A_*+z\pa_z B_*)\mathcal H+\mathcal E,
	\end{equation}
	where
	\begin{equation}\label{app-eq:E-def}
		\mathcal E:=\frac{\mathscr S-\frac{\mathbf TD_*}{D_*}\mathcal L\Omega_*+2(r\pa_r A_*+z\pa_z B_*)\mathcal L\Omega_*}{D_*\Omega_*}.
	\end{equation}
Furthermore, there holds
	\begin{equation}\label{app-eq:E-expression}
		\begin{aligned}
			D_*^2\Omega_*\mathcal E
			=&2r^2z^2\bigl(B_*-A_*+z\pa_z B_*-r\pa_r A_*\bigr)\bigl(A_*E_1-B_*E_2+A_*^2(d-1)r\pa_r\Omega_*\\
			&-B_*^2z\pa_z\Omega_*\bigr)+2z^2(d-1)D_*r\pa_r A_*r\pa_r\Omega_*+2D_*r^2z\pa_z B_*z\pa_z\Omega_*\\
			&-D_*(\mathcal L A_*)r\pa_r\Omega_*-D_*(\mathcal L B_*)z\pa_z\Omega_*-2rz\bigl(r\pa_z A_*+z\pa_rB_*\bigr)\times\\
		&\quad\times\bigl(B_*z^2E_1+A_*r^2E_2+A_*B_*z^2(d-1)r\pa_r\Omega_*
		+A_*B_*r^2z\pa_z\Omega_*\bigr).
		\end{aligned}
	\end{equation}}
\end{lemma}

\begin{remark}
	It is critical to note that although the second-order derivatives ($(r\pa_r)^2\Omega_*$, $r\pa_r(z\pa_z)\Omega_*$, and $(z\pa_z)^2\Omega_*$) of  $\Omega_*$ enter the expression \eqref{app-eq:S-def}, they cancel out entirely in \eqref{app-eq:E-expression}.	
	\end{remark}
	
\begin{proof}
A direct computation yields \eqref{app-eq:H-equation}. To simplify notation, we denote \begin{align*}
&D_{\ast,1}:=\mathbf TD_\ast-2(r\pa_r A_*+z\pa_zB_*)D_*, \quad A_{\ast,1}:=A_*-z\pa_z B_*, \\
&B_{\ast,1}:=B_*-r\pa_r A_*, \andf C_{\ast,1}:=r\pa_z A_*+z\pa_rB_*. 
\end{align*}
Then
	\begin{align*}
		&\mathbf TD_*=2A_*^2r^2(r\pa_r A_*+A_*)+2A_*B_*r^2z\pa_z A_*+2B_*^2z^2(z\pa_z B_*+B_*)+2A_*B_*rz^2\pa_rB_*,\\
		&D_{\ast,1}=2A_*^2r^2(A_*-z\pa_z B_*)+2B_*^2z^2(B_*-r\pa_r A_*)+2A_*B_*rz(r\pa_z A_*+z\pa_rB_*).
	\end{align*}
Hence,  in view of \eqref{app-eq:L-def} and \eqref{app-eq:F-equation}, we find	
\begin{align*}
		D_*^2\Omega_*\mathcal E=&D_*\mathscr S-D_{\ast,1}\mathcal L\Omega_*\\
		=&D_*\mathscr S-D_{\ast,1}z^2\bigl((r\pa_r)^2+(d-1)r\pa_r\bigr)\Omega_* -D_{\ast,1}r^2\bigl((z\pa_z)^2+z\pa_z\bigr)\Omega_*\\
		=&(2B_{\ast,1}D_*-D_{\ast,1})z^2(r\pa_r)^2\Omega_*+(2B_*D_*-D_{\ast,1})z^2(d-1)r\pa_r\Omega_*\\
		& +(2A_{\ast,1}D_*-D_{\ast,1})r^2(z\pa_z)^2\Omega_*+(2A_*D_*-D_{\ast,1})r^2z\pa_z\Omega_*\\
		& -2rzD_*(r\pa_z A_*+z\pa_rB_*)r\pa_r(z\pa_z)\Omega_*-D_*(\mathcal L A_*)r\pa_r\Omega_*-D_*(\mathcal L B_*)z\pa_z\Omega_*.
	\end{align*}
	We compute that
	\begin{align*}
		2B_{\ast,1}D_*-D_{\ast,1}&=2B_{\ast,1}(A_*^2r^2+B_*^2z^2)-2A_*^2r^2A_{\ast,1}-2B_*^2z^2B_{\ast,1}-2A_*B_*rzC_{\ast,1}\\
		&=2A_*^2r^2(B_{\ast,1}-A_{\ast,1})-2A_*B_*rzC_{\ast,1},\\
		2A_{\ast,1}D_*-D_{\ast,1}&=2A_{\ast,1}(A_*^2r^2+B_*^2z^2)-2A_*^2r^2A_{\ast,1}-2B_*^2z^2B_{\ast,1}-2A_*B_*rzC_{\ast,1}\\
		&=2B_*^2z^2(A_{\ast,1}-B_{\ast,1})-2A_*B_*rzC_{\ast,1},
	\end{align*}
	hence,
	\begin{align*}
		D_*^2\Omega_*\mathcal E
		=&2r^2z^2(B_{\ast,1}-A_{\ast,1})(A_*^2(r\pa_r)^2\Omega_*-B_*^2(z\pa_z)^2\Omega_*)\\
		&+(2B_*D_*-D_{\ast,1})z^2(d-1)r\pa_r\Omega_*+(2A_*D_*-D_{\ast,1})r^2z\pa_z\Omega_*\\
		&-2rzC_{\ast,1}\bigl[A_*B_*(z^2(r\pa_r)^2+r^2(z\pa_z)^2)\Omega_*+D_*r\pa_r(z\pa_z)\Omega_*\bigr]\\
		& -D_*(\mathcal L A_*)r\pa_r\Omega_*-D_*(\mathcal L B_*)z\pa_z\Omega_*.
	\end{align*}
	Thanks to \eqref{app-eq:E1E2}, we have
	\begin{align*}
		&A_*^2(r\pa_r)^2\Omega_*-B_*^2(z\pa_z)^2\Omega_*=A_*E_1-B_*E_2,\\
		&A_*B_*(z^2(r\pa_r)^2+r^2(z\pa_z)^2)\Omega_*+D_*r\pa_r(z\pa_z)\Omega_*=B_*z^2E_1+A_*r^2E_2,
	\end{align*}
	thus,
	\begin{align*}
		D_*^2\Omega_*\mathcal E
		=&2r^2z^2(B_{\ast,1}-A_{\ast,1})(A_*E_1-B_*E_2)+(2B_*D_*-D_{\ast,1})z^2(d-1)r\pa_r\Omega_*\\
		&+(2A_*D_*-D_{\ast,1})r^2z\pa_z\Omega_*-2rzC_{\ast,1}(B_*z^2E_1+A_*r^2E_2)\\
		&-D_*(\mathcal L A_*)r\pa_r\Omega_*-D(\mathcal L B_*)z\pa_z\Omega_*,
	\end{align*}
which together with
	\begin{align*}
		&B_{\ast,1}-A_{\ast,1}=B_*-A_*+z\pa_z B_*-r\pa_r A_*,\quad C_{\ast,1}=r\pa_z A_*+z\pa_rB_*,\\
		&2B_*D_*-D_{\ast,1}=2A_*^2r^2(B_{\ast,1}-A_{\ast,1})-2A_*B_*rzC_{\ast,1}+2D_*r\pa_r A_*,\\
		&2A_*D_*-D_{\ast,1}=2B_*^2z^2(A_{\ast,1}-B_{\ast,1})-2A_*B_*rzC_{\ast,1}+2D_*z\pa_z B_*,
	\end{align*}
	ensures \eqref{app-eq:E-expression}. This completes the proof of Lemma \ref{lem:source-computation-L-Omega}.
	\end{proof}

\begin{lemma}[Source estimates]
	\label{lem:source-estimates-L-Omega}
{\sl	There exists a non-negative function $\mathfrak q=\mathfrak q(r,z)$ such that
	\begin{equation}
		\label{app-eq:E-q-bound}
		|r\pa_r A_*+z\pa_zB_*|+|\mathcal E|\leq C\mathfrak q,
	\end{equation}
	and along every characteristic curve $(r(s),z(s))$ of $\mathbf T$, namely
	\[\frac{\mathrm d}{\mathrm ds}r(s)=A_*(r(s),z(s))r(s),\qquad \frac{\mathrm d}{\mathrm ds}z(s)=B_*(r(s),z(s))z(s),\]
	one has
	\begin{equation}
		\label{app-eq:q-integrable}
		\int_{-\infty}^{+\infty}\mathfrak q(r(s),z(s))\,ds\leq C.
	\end{equation}
	The constants depend only on $d,a,\mu$.}
\end{lemma}

\begin{proof}
Due to $\Omega_*\in\mathcal B_{M_0'}$,	 $|r\pa_r\Omega_*|+|z\pa_z\Omega_*|\leq C\Omega_*,$  and by \eqref{A1},	$|E_1|+|E_2|\leq C\Omega_*$ , we deduce from \eqref{app-eq:E-expression} that
	\begin{align*}
		|D_*^2\Omega_*\mathcal E|
		\leq& Cr^2z^2\bigl(|B_*-A_*|+|z\pa_z B_*|+|r\pa_r A_*|\bigr)\Omega_*+Cz^2D_*|r\pa_r A_*|\Omega_*\\
		& +CD_*r^2|z\pa_z B_*|\Omega_*+D_*(|\mathcal L A_*|+|\mathcal L B_*|)\Omega_*\\
		& +C(r^2|z\pa_z A_*|+z^2|r\pa_rB_*|)(z^2+r^2)\Omega_*.
	\end{align*}
	Here we used \eqref{app-eq:A-B-positive} to deduce that $|A_*|+|B_*|\leq C$. \eqref{app-eq:A-B-positive} also implies $D_*\sim r^2+z^2$. Then
	\begin{align}
		\notag		|\mathcal E|
		\leq&\, \frac{Cr^2z^2}{(r^2+z^2)^2}(|B_*-A_*|+|z\pa_z B_*|+|r\pa_r A_*|)+\frac{Cz^2}{r^2+z^2}|r\pa_r A_*|+\frac{Cr^2}{r^2+z^2}|z\pa_z B_*|\\
		&
		\label{app-eq:E-schematic-bound}+C\frac{|\mathcal L A_*|+|\mathcal L B_*|}{r^2+z^2}+C\frac{r^2|z\pa_z A_*|+z^2|r\pa_rB_*|}{r^2+z^2}\\
		\notag
		\leq&\, \frac{Cr^2z^2|B_*-A_*|}{(r^2+z^2)^2}+\frac{Cz^2\bigl(|r\pa_r A_*|+|r\pa_r B_*|\bigr)}{r^2+z^2}+\frac{Cr^2\bigl(|z\pa_z A_*|+|z\pa_zB_*|\bigr)}{r^2+z^2}\\
		\notag
		&+C\frac{|\mathcal L A_*|+|\mathcal L B_*|}{r^2+z^2}.
	\end{align}
	Now we define
	\begin{align}
		\label{app-eq:q-def}
		\mathfrak q(r,z):={}&
		\frac{r^2z^2(A_*-B_*)}{(r^2+z^2)^2} 
		+|r\pa_r A_*|+|r\pa_r B_*|+|z\pa_z A_*|+|z\pa_zB_*|+\frac{|\mathcal L A_*|+|\mathcal L B_*|}{r^2+z^2}.
	\end{align}
	Then \eqref{app-eq:E-q-bound} follows from \eqref{app-eq:E-schematic-bound} (note that $A_*-B_*>0$).
	
	It remains to prove \eqref{app-eq:q-integrable}. Let $(r(s),z(s))$ be a characteristic curve of $\mathbf T$. By \eqref{app-eq:A-B-positive}, $r(s)$ and $z(s)$ are strictly increasing functions of $s$, and every characteristic intersects the initial curve $\Gamma_0$ exactly once. We choose the parametrization so that $(r(0),z(0))\in\Gamma_0$. First, we set $\vartheta(s):=r(s)/z(s)$, then $$\vartheta'(s)=(A_*-B_*)(r(s),z(s))\vartheta(s)>0, $$ hence,
	\begin{equation}\label{Eq.A-B-integral-est}
		\int_{-\infty}^{+\infty}(A_*-B_*)(r(s),z(s))\frac{r(s)^2z(s)^2}{(r(s)^2+z(s)^2)^2}\,\mathrm ds\lesssim \int_0^\infty \frac{\vartheta}{(1+\vartheta^2)^2}\,\mathrm d\vartheta\lesssim 1.
	\end{equation}
	
	Next we estimate the remaining terms. It follows from  \eqref{A1} that
	\begin{equation}
		\label{app-eq:weighted-coeff-bound}
		|r\pa_r A_*|+|r\pa_r B_*|+|z\pa_z A_*|+|z\pa_zB_*|\leq C\mathfrak p(r,z)\ \text{with}\  \mathfrak p(r,z):=|(r,z)|\langle r,z\rangle^{d-3-1/\mu}.
	\end{equation}
	Since $\gamma_1<d-2$, we have $d-1-\gamma_1>1$. We  claim that
	\begin{equation}\label{app-eq:LA-LB-bound}
		\frac{|\mathcal L A_*|+|\mathcal L B_*|}{r^2+z^2}\leq C\mathfrak p(r,z).
	\end{equation}
	By \eqref{app-eq:LA-LB} and \eqref{app-eq:LB}, it suffices to estimate
	\[\frac{r^{d-1}z\Omega_*}{r^2+z^2},\qquad \frac{z^2|((r\pa_r)^2+(d-1)r\pa_r)\psi_*|}{r^2+z^2},\qquad \frac{r^2|((z\pa_z)^2+z\pa_z)\psi_*|}{r^2+z^2}.\]
	The last two terms are bounded by $C\mathfrak p(r,z)$ by the decay estimate in $\mathcal A_{M_0}$.
	For the first term, if $r^2+z^2\leq 4$, then the bound $\Omega_*\leq C(zr^{-\gamma}+zr^{-\gamma_1})$ gives
	\[\frac{r^{d-1}z\Omega_*}{r^2+z^2}\leq C\frac{r^{d-1}z^2(r^{-\gamma}+r^{-\gamma_1})}{r^2+z^2}\leq C\bigl(r^{d-1-\gamma}+r^{d-1-\gamma_1}\bigr)\leq Cr,\]
	because $\gamma_1<\gamma<d-2$. If $r^2+z^2\geq 1$, then the far-field estimate \eqref{Eq.Omega-decay} yields
	\[\frac{r^{d-1}z\Omega_*}{r^2+z^2}\leq C\frac{r^{d-1}z^2 |(r,z)|^{-1-1/\mu}}{r^2+z^2}\leq C |(r,z)|^{d-2-1/\mu}\leq C\langle r,z\rangle^{d-2-1/\mu}.\]
	Thus, \eqref{app-eq:LA-LB-bound} follows.
	
	It remains to integrate $\mathfrak p$ along characteristics. By the definition, we have 
	$$r(s)+z(s)\leq C e^{c_0s}\  \forall\ s\leq 0 \andf r(s)+z(s)\geq C^{-1} e^{c_0s}\  \forall\ s\geq 0. $$ Therefore, using $d-2-1/\mu<0$,
	\begin{align*}
		\int_{-\infty}^{+\infty}\mathfrak p(r(s),z(s))\,\mathrm ds\leq C+C\int_{-\infty}^0 e^{c_0s}\,\mathrm ds+C\int_0^\infty e^{c_0(d-2-1/\mu)s}\,\mathrm ds\leq C.
	\end{align*}
	Together with \eqref{Eq.A-B-integral-est}, \eqref{app-eq:weighted-coeff-bound}, and \eqref{app-eq:LA-LB-bound}, this proves \eqref{app-eq:q-integrable}. 
\end{proof}

\begin{lemma}[Estimate for $\mathcal L\Omega_*$]
	\label{lem:L-Omega-bound}
{\sl 	There exists a constant $C=C(d,a,\mu)>1$ such that
	\begin{equation}
		\label{app-eq:L-Omega-bound}
		|\mathcal L\Omega_*(r,z)|\leq C(r^2+z^2)\Omega_*(r,z),\qquad\forall\  (r,z)\in\Pi_+.
	\end{equation}}
\end{lemma}

\begin{proof}
	Let $s\mapsto (r(s),z(s))$ be a characteristic curve of $\mathbf T$, parametrized so that $(r(0),z(0))\in\Gamma_0$. Since $\Omega_*>0$ and $D_*\sim r^2+z^2$, it suffices to prove that
	\[\mathcal H(s):=\frac{\mathcal L\Omega_*(r(s),z(s))}{D_*(r(s),z(s))\Omega_*(r(s),z(s))}\]
	is uniformly bounded for all $s\in\mathbb R$ and all characteristics.  By Lemma \ref{lem:source-estimates-L-Omega}, $\mathcal H$ satisfies
	\[\frac{\mathrm d}{\mathrm ds}\mathcal H(s)=-2(r\pa_r A_*+z\pa_zB_*)(r(s),z(s))\mathcal H(s)+\mathcal E(r(s),z(s)),\]
	hence,
	\[\frac{\mathrm d}{\mathrm ds}|\mathcal H(s)|\leq C\mathfrak q(r(s),z(s))|\mathcal H(s)|+C\mathfrak q(r(s),z(s)).\]
	By applying Grönwall's inequality and using \eqref{app-eq:q-integrable}, we obtain
	\begin{equation}\label{app-eq:H-gronwall}
		|\mathcal H(s)|\leq C\bigl(|\mathcal H(0)|+1\bigr).
	\end{equation}
	
	It remains to show that $|\mathcal H(0)|$ is uniformly bounded on the initial curve $\Gamma_0$. This follows from Proposition \ref{prop:smoothness-fixed-point}. Indeed, $\Gamma_0$ is compact after adding its two endpoints. Near the endpoint on $\{r=0,z>0\}$, Proposition \ref{prop:smoothness-fixed-point} ensures that $\Omega_*(r,z)=r^{\delta_d}A_0(r^2,z)$ with $A_0(0,z)>0$, and therefore,
	\[
	|r\pa_r\Omega_*|+|z\pa_z\Omega_*|+|(r\pa_r)^2\Omega_*|+|(r\pa_r)(z\pa_z)\Omega_*|+|(z\pa_z)^2\Omega_*|\leq C\Omega_*.
	\]
	Near the endpoint on $\{r>0,z=0\}$, Proposition \ref{prop:smoothness-fixed-point} ensures that  $\Omega_*(r,z)=zB_0(r,z^2)$ with $B_0(r,0)>0$, and the same first-order derivative bound holds. Away from the endpoints, $\Omega_*$ is smooth and strictly positive. Consequently, $|\mathcal L\Omega_*|\leq C(r^2+z^2)\Omega_*$ on $\Gamma_0$, and since $D_*\sim r^2+z^2$, we obtain $$\sup_{\Gamma_0}|\mathcal H|\leq C. $$
	Plugging the above inequality  into \eqref{app-eq:H-gronwall} gives $$\sup_{\Pi_+}|\mathcal H|\leq C.$$
	Therefore, $$|\mathcal L\Omega_*|\leq CD_*\Omega_*\leq C(r^2+z^2)\Omega_*,$$ which leads to \eqref{app-eq:L-Omega-bound}.
\end{proof}

Now we are ready to prove Proposition \ref{prop:second-log-derivative-Omega}. 

\begin{proof}[Proof of Proposition \ref{prop:second-log-derivative-Omega}]
	We use the algebraic formula \eqref{app-eq:XYZ-algebra}. As $|r\pa_r\Omega_*|+|z\pa_z\Omega_*|\leq C\Omega_*,$ we first deduce from  Lemma \ref{lem:L-Omega-bound} that
	\begin{equation}\label{SBeq6}
	|F_0|\leq |\mathcal L\Omega_*|+(d-1)z^2|r\pa_r\Omega_*|+r^2|z\pa_z\Omega_*|\leq C(r^2+z^2)\Omega_*.
	\end{equation}
	 Moreover, \eqref{A1} implies that $|E_1|+|E_2|\leq C\Omega_*$. As $D_*\sim r^2+z^2$,  we get, by using \eqref{app-eq:XYZ-algebra} and \eqref{SBeq6}, that
	\begin{equation}\label{Eq.recover-second-log}
		|(r\pa_r)^2\Omega_*|+|r\pa_r(z\pa_z)\Omega_*|+|(z\pa_z)^2\Omega_*|\leq C\Bigl(|E_1|+|E_2|+\frac{|F_0|}{r^2+z^2}\Bigr)\leq C\Omega_*.
	\end{equation}
	 Then \eqref{app-eq:second-derivative-final} follows directly from \eqref{Eq.recover-second-log} and the first-order derivative bound in $\mathcal B_{M_0'}$. We thus complete the proof of  Proposition \ref{prop:second-log-derivative-Omega}. \end{proof}

\section*{Acknowledgments}
  P. Zhang is partially  supported by National Key R$\&$D Program of China under grant 2021YFA1000800 and by National Natural Science Foundation of China under Grant 12421001, 12494542 and 12288201.
Z. Zhang is partially supported by  NSF of China  under Grant 12288101.

\end{document}